\documentclass{article}
\usepackage{amsmath,amsthm,amssymb}
\usepackage{enumerate}
\usepackage{graphicx}
\usepackage{cite}
\usepackage{comment}
\usepackage{subfigure}
\usepackage{oands}
\usepackage[usenames,dvipsnames]{xcolor}
\usepackage{tikz}
\usepackage{changepage}
\usepackage{bbold}
\usepackage{mathtools}
\usepackage[margin=1in]{geometry}
\usepackage{appendix}
\usepackage{authblk}
\usepackage[pagewise,mathlines]{lineno}
\usepackage{multicol}
\usepackage{microtype}
\usepackage{todonotes}

\usepackage[pdftitle={Almost sure multifractal spectrum of SLE},
  pdfauthor={Ewain Gwynne, Jason Miller, Xin Sun},
colorlinks=true,linkcolor=NavyBlue,urlcolor=RoyalBlue,citecolor=PineGreen,bookmarks=true,bookmarksopen=true,bookmarksopenlevel=2,unicode=true,linktocpage]{hyperref}

\theoremstyle{plain}
\newtheorem{thm}{Theorem}[section]
\newtheorem{cor}[thm]{Corollary}
\newtheorem{lem}[thm]{Lemma}
\newtheorem{prop}[thm]{Proposition}

\newtheorem{notation}[thm]{Notation}

\def\@rst #1 #2other{#1}
\newcommand\MR[1]{\relax\ifhmode\unskip\spacefactor3000 \space\fi
  \MRhref{\expandafter\@rst #1 other}{#1}}
\newcommand{\MRhref}[2]{\href{http://www.ams.org/mathscinet-getitem?mr=#1}{MR#2}}

\numberwithin{equation}{section}
\numberwithin{figure}{section}
 
\theoremstyle{definition}
\newtheorem{defn}[thm]{Definition}
\newtheorem{remark}[thm]{Remark}

\newcommand{\dsb}{\begin{adjustwidth}{2.5em}{0pt}
\begin{footnotesize}}
\newcommand{\dse}{\end{footnotesize}
\end{adjustwidth}}

\newcommand{\ssb}{\begin{adjustwidth}{2.5em}{0pt}}
\newcommand{\sse}{\end{adjustwidth}}

\newcommand{\aryb}{\begin{eqnarray*}}
\newcommand{\arye}{\end{eqnarray*}}
\def\alb#1\ale{\begin{align*}#1\end{align*}}
\def\allb#1\alle{\begin{align}#1\end{align}}
\newcommand{\eqb}{\begin{equation}}
\newcommand{\eqe}{\end{equation}}
\newcommand{\eqbn}{\begin{equation*}}
\newcommand{\eqen}{\end{equation*}}

\newcommand{\BB}{\mathbf}
\newcommand{\D}{{\BB D}}

\newcommand{\ol}{\overline}
\newcommand{\ul}{\underline}
\newcommand{\op}{\operatorname}

\newcommand{\im}{\operatorname{Im}}
\newcommand{\re}{\operatorname{Re}}
\newcommand{\SLE}{\op{SLE}}
\newcommand{\dimH}{\dim_{\mathcal H}}
\newcommand{\frk}{\mathfrak}
\newcommand{\eqD}{\overset{d}{=}}
\newcommand{\ep}{\epsilon}
\newcommand{\rta}{\rightarrow}

\newcommand{\Rta}{\Rightarrow}

\newcommand{\wt}{\widetilde}
\newcommand{\wh}{\widehat}
\newcommand{\mcl}{\mathcal}

\newcommand{\bdy}{\partial}

\newcommand{\av}{a}

\newcommand{\fl}{{\operatorname{f}}}
\newcommand{\ed}{{0}}
 
 \renewcommand{\C}{{\BB C}}

\definecolor{purple}{rgb}{0.7,0,0.7}
\definecolor{gray}{rgb}{0.6,0.6,0.6}
\definecolor{dgreen}{rgb}{0.0,0.4,0.0}
\definecolor{dblue}{rgb}{0.0,0.0,0.5}

\newcommand*\patchAmsMathEnvironmentForLineno[1]{  \expandafter\let\csname old#1\expandafter\endcsname\csname #1\endcsname
  \expandafter\let\csname oldend#1\expandafter\endcsname\csname end#1\endcsname
  \renewenvironment{#1}     {\linenomath\csname old#1\endcsname}     {\csname oldend#1\endcsname\endlinenomath}}\newcommand*\patchBothAmsMathEnvironmentsForLineno[1]{  \patchAmsMathEnvironmentForLineno{#1}  \patchAmsMathEnvironmentForLineno{#1*}}\AtBeginDocument{\patchBothAmsMathEnvironmentsForLineno{equation}\patchBothAmsMathEnvironmentsForLineno{align}\patchBothAmsMathEnvironmentsForLineno{flalign}\patchBothAmsMathEnvironmentsForLineno{alignat}\patchBothAmsMathEnvironmentsForLineno{gather}\patchBothAmsMathEnvironmentsForLineno{multline}}

\title{Almost sure multifractal spectrum of SLE}
\date{\today}
\author{Ewain Gwynne
\quad Jason Miller
 \quad Xin Sun}
\affil{Massachusetts Institute of Technology}

\begin{document}

\maketitle
 
\begin{abstract}
Suppose that $\eta$ is a Schramm-Loewner evolution ($\SLE_\kappa$) in a smoothly bounded simply connected domain $D \subset \C$ and that $\phi$ is a conformal map from $\D$ to a connected component of $D \setminus \eta([0,t])$ for some $t>0$.  The multifractal spectrum of $\eta$ is the function $(-1,1) \to [0,\infty)$ which, for each $s \in (-1,1)$, gives the Hausdorff dimension of the set of points $x \in \partial \D$ such that $|\phi'( (1-\epsilon) x)| = \epsilon^{-s+o(1)}$ as $\epsilon \to 0$.  We rigorously compute the a.s.\ multifractal spectrum of $\SLE$, confirming a prediction due to Duplantier.  As corollaries, we confirm a conjecture made by Beliaev and Smirnov for the a.s.\ bulk integral means spectrum of $\SLE$, we obtain the optimal H\"older exponent for a conformal map which uniformizes the complement of an $\SLE$ curve, and we obtain a new derivation of the a.s.\ Hausdorff dimension of the $\SLE$ curve for $\kappa \leq 4$. Our results also hold for the $\SLE_\kappa(\ul\rho)$ processes with general vectors of weight $\ul\rho$. 
\end{abstract}

\tableofcontents

\section{Introduction}

The Schramm-Loewner evolution ($\SLE_\kappa$) is a one-parameter family of random fractal curves in a simply connected domain in $\BB C$, indexed by $\kappa >0$.  $\SLE$ was introduced by Schramm in \cite{schramm0}, and has since become a central object of study in both probability theory and statistical physics.  See e.g.\ \cite{werner-notes, lawler-book} for an introduction to $\SLE$. Its importance is that it describes the scaling limit of the interfaces which arise in a number of discrete models in statistical physics, see, e.g., \cite{lsw-lerw-ust, smirnov-ising, ss-explorer, ss-dgff, gl-contours}. 
 
Roughly speaking, the multifractal spectrum of a domain $D\subset \BB C$ refers to one of the two functions
\[
s \mapsto \dimH \Theta^s(D) \quad\text{or}\quad s \mapsto \dimH \wt{\Theta}^s(D)
\]
where $\dimH$ denotes the Hausdorff dimension and $\wt{\Theta}^s(D)$ is the set of points $x \in \partial \D$ with the property that the modulus of the derivative $|\phi'((1-\epsilon)x)|$ of a conformal map $\phi$ from the unit disk $\BB D$ into $D$ grows like $\epsilon^{-s}$ as $\epsilon \to 0$ and $\Theta^s(D) = \phi(\wt{\Theta}^s(D))$. 
There are several more or less equivalent definitions of this concept.  See Section~\ref{multifractal def} for the precise definition we use in this paper.

The multifractal spectrum of $D$ is a means of quantifying the behavior of $|\phi'|$ near $\partial \BB D$, even though $\phi$ need not be differentiable on $\partial D$. It is closely related to various other quantities associated with $\partial D$, e.g.\ the Hausdorff dimension, H\"older regularity, and packing dimension of $\partial D$; the integral means spectrum of $D$; and the harmonic measure spectrum of the complement of a hull. See \cite{makarov-fine} for some results in this direction. Such complex analytic quantities are often difficult if not impossible to compute explicitly for specific deterministic domains. However, for random domains (like the complement of an SLE curve) explicit calculations can sometimes be more tractable. 

There has been substantial interest in the multifractal properties of $\SLE_\kappa$ (i.e.\ that of the domain obtained by excising the curve) in both mathematics and physics recent years.  For example, it is shown by Beffara in \cite{beffara-dim} that the a.s.\ Hausdorff dimension of the $\SLE_\kappa$ curve is $1+\kappa/8$ for $\kappa \in (0,8)$ and $2$ for $\kappa \geq 8$.  The optimal H\"older exponent for the $\SLE_\kappa$ curve (with the capacity parameterization) is derived in \cite{lawler-viklund-holder}, building on the work of Rohde and Schramm \cite{schramm-sle} and Lind \cite{lind-holder}.

There have also been a number of works which study various versions of the multifractal spectrum of $\SLE$. The first such works~\cite{dup-hm99,dup-mf99}, due to Duplantier, give non-rigorous predictions of the multifractal exponents for Brownian motion and self-avoiding random walk, which correspond to $\op{SLE}_\kappa$ for $\kappa = 6$ and $\kappa=8/3$, respectively. In~\cite{dup-mf-spec-bulk}, Duplantier extends this to a non-rigorous prediction of the multifractal spectrum of the $\SLE_\kappa$ curve for general values of $\kappa > 0$. Observing that the predicted multifractal spectrum for $\op{SLE}_\kappa$ in~\cite{dup-mf-spec-bulk} is invariant under the replacement $\kappa\mapsto 16/\kappa$ is what originally led Duplantier to conjecture \emph{SLE duality} (c.f.\ \cite{dup-mf-spec-bulk,dup-higher-mf}), which states the outer boundary of an $\op{SLE}_\kappa$ curve for $\kappa > 4$ is described by a type of $\op{SLE}_{16/\kappa}$ curve. Various forms of SLE duality have since been rigorously proven in~\cite{zhan-duality1,zhan-duality2,dubedat-duality,ig1,ig4}.  

In \cite{binder-dup-winding1,binder-dup-winding2}, the authors study (non-rigorously) a notion of spectrum involving the argument, rather than just the modulus, of the derivative of the $\SLE$ maps. In \cite{dup-higher-mf}, these predictions are expanded to higher multifractal spectra, e.g.\ the dimension of the set of points on the curve where the behavior of the derivative on \emph{both} sides of the curve is prescribed. See also \cite{dup-mf-spec} for additional discussion of these and other multifractal-type spectra. 
 
The first mathematical work on the multifractal spectrum of $\SLE$ is due to Beliaev and Smirnov \cite{bel-smirnov-hm-sle} in which they compute the average integral means spectrum for a whole-plane $\SLE$ curve.  Expanding on the results of \cite{bel-smirnov-hm-sle}, the authors of \cite{dup-coefficient} (see also \cite{ly-spec-avg,ly-spec-new})
use exact solutions of differential equations for the moments of the derivatives of the whole-plane $\SLE$ maps to study the integral means spectrum of certain $\SLE$ and generalized $\SLE$ processes. 
The paper~\cite{dup-log-coefficients} extends these calculations to the case of mixed moments for the modulus of an $\op{SLE}_\kappa$ Loewner map and the modulus of its derivative, and studies a generalized integral means spectrum.
In \cite{lawler-viklund-tip}, the authors rigorously compute the multifractal spectrum at the tip of the $\SLE$ curve; this is the first work in which an almost sure result for the multifractal spectrum for $\SLE$ is obtained.  
The authors of~\cite{abv-bdy-spec} compute the almost sure dimension of the set of points where an $\op{SLE}_\kappa$ curve ($\kappa > 4$) intersects the boundary at a given ``angle".
Binder and Duplantier have informed the authors in private communication~\cite{dup-binder-comm} of a forthcoming work in which they prove formulae for the average mixed integral means spectra (i.e.\ $\beta$-spectrum with complex exponent) both in the bulk and at the tip, for chordal SLE. The corresponding formulae agree after Legendre transform with the predictions from~\cite{binder-dup-winding1,binder-dup-winding2} concerning the mixed multifractal spectra for harmonic measure and rotation (equivalently, modulus and argument).

In this article, we will give the first rigorous derivation of the a.s.\ bulk multifractal spectrum of chordal $\SLE_\kappa$ (i.e.\ that of the complementary domain).  We will also obtain the a.s.\ bulk integral means spectrum of $\SLE$; the spectrum that we find confirms \cite[Conjecture~1]{bel-smirnov-hm-sle}.  Our approach differs from those used elsewhere in the literature to prove results of this type in that we make use of various couplings of $\SLE$ processes with the Gaussian free field (GFF). In the proof of the upper bound we use a coupling of the reverse $\SLE$ Loewner flow with a free boundary GFF (sometimes called the ``quantum zipper'') \cite{shef-zipper,qle,wedges}. Our proof of the lower bound will make extensive use of the coupling of $\SLE$ with a GFF with Dirichlet boundary conditions (sometimes called the ``imaginary geometry" coupling) \cite{shef-slides,dubedat-coupling,ig1,ig2,ig3,ig4}.  This latter coupling has also been used to aid in proving lower bounds for the Hausdorff dimensions of sets associated with $\SLE$ in \cite{miller-wu-dim}. Our approach at a high level is similar in spirit to the one used in \cite{miller-wu-dim}, but the technical details are rather different. 

\bigskip

\noindent{\bf Acknowledgments}
The authors thank Dapeng Zhan and an anonymous referee for helpful comments on earlier versions of this paper. EG was supported by the Department of Defense via an NDSEG fellowship.  JM was partially supported by DMS-1204894.  XS was partially supported by DMS-1209044.

\subsection{Multifractal spectrum definition}
\label{multifractal def}

We will now introduce the sets whose Hausdorff dimension we will compute, in the setting of general domains in the complex plane. Our definitions are similar to those in \cite[Section~2]{lawler-viklund-tip}, but we deal with the boundary of a domain rather than the tip of a given curve. 
 
Let $D\subset \BB C$ be a simply connected domain and let $\phi : \BB D\rta  D$ be a conformal map. For $s\in\BB R$, define
\begin{align}
\label{tilde theta}
\wt\Theta^{s  }(D) &:=  \left\{x\in \partial \BB D:  \lim_{\ep \rta 0} \frac{\log |\phi'((1-\ep)x)|}{-\log \ep }  = s \right\}    
\intertext{and}
\label{theta}
\Theta^{s }(D) &:=\phi( \wt\Theta^{s }(D)  ). 
\intertext{Also define}
\wt\Theta^{s ; \leq }(D) &:=  \left\{x\in \partial \BB D: \limsup_{\ep \rta 0} \frac{\log |\phi'((1-\ep)x)|}{-\log \ep }  \leq s \right\}   \notag\\
\Theta^{s ; \leq }(D) &:=\phi( \wt\Theta^{s ; \leq }(D)  ) \notag \\
\wt\Theta^{s ; \geq }(D) &:=  \left\{x\in \partial \BB D: \limsup_{\ep \rta 0} \frac{\log |\phi'((1-\ep)x)|}{-\log \ep }  \geq s \right\}  \notag \\
\Theta^{s ; \geq }(D) &:=\phi( \wt\Theta^{s ; \geq }(D)  ). \notag
\end{align}
The \emph{multifractal spectrum} of $D$ can be defined as one of the two functions  $s \mapsto \dimH\Theta^{s}(D)$ or $s  \mapsto \dimH \wt\Theta^{s }(D)$.  It is easy to check that these definitions do not depend on the choice of conformal map $\phi$. We note that although the sets $\Theta^s(D)$ and $\wt\Theta^s(D)$ are defined for all $s\in \BB R$, these sets are empty for $s\notin [-1,1]$ (see Lemma~\ref{s>1 empty} below).
 
\subsection{Main results}

Our main result is the following theorem. 

\begin{thm}
\label{main thm}
Let $\kappa\leq 4$. Let $\eta$ be a chordal $\SLE_\kappa$ from $-i$ to $i$ in $\BB D$. Let $D_\eta$ be the connected component of $\BB D\setminus \eta([0,\infty))$ containing 1 on its boundary. Let 
\begin{align}
\wt\xi(s) &: =1 - \frac{(4 + \kappa)^2 s^2}{8 \kappa (1 + s)}\label{tilde xi(s)} \\   
\xi(s) &:= \frac{ 8 \kappa ( 1 + s  ) -(4+\kappa)^2 s^2}{8 \kappa ( 1 - s^2) } \label{xi(s)} \\ 
s_- &:= \frac{4 \kappa -  2 \sqrt2 \sqrt{\kappa (2 + \kappa) (8 + \kappa)}}{ (4 + \kappa)^2} \label{s-} \\
s_+ &:= \frac{4 \kappa + 2 \sqrt2 \sqrt{\kappa (2 + \kappa) (8 + \kappa)}}{ (4 + \kappa)^2}. \label{s+} 
\end{align}
For $s\in (-1,1)$, a.s. 
\eqbn
\begin{aligned}[c]
	\dimH \wt\Theta^{s  }(D_\eta) &= \dimH \wt\Theta^{s; \geq }(D_\eta) = \wt\xi(s)  ,\quad &0\leq s \leq s_+ \notag\\
	\dimH \wt\Theta^{s  }(D_\eta) &= \dimH \wt\Theta^{s; \leq }(D_\eta) = \wt\xi(s)  , \quad & s_- \leq s \leq 0 \notag\\
	\dimH \Theta^{s  }(D_\eta) &= \dimH  \Theta^{s; \geq }(D_\eta) =   \xi(s)  , \quad & \frac{\kappa}{4} \leq s \leq s_+ \notag\\
	\dimH  \Theta^{s  }(D_\eta) &= \dimH  \Theta^{s; \leq }(D_\eta) =  \xi(s)  , \quad & s_- \leq s \leq \frac{\kappa}{4} .\notag
\end{aligned} 
\eqen 
Moreover, we a.s.\ have $\wt\Theta^{s }(D_\eta) =   \Theta^{s  }(D_\eta)   = \emptyset$ for each $s\notin [s_- , s_+]$.
\end{thm}

\begin{remark}
\label{critical s remark}
The significance of $s_-$ and $s_+$ is that $\wt\xi(s) \geq 0$ for $s \in [s_-  , s_+]$, and the significance of $s=\kappa/4$ is that it is the value which maximizes $\xi$. Note $s_- \in (-1,0)$ and $s_+ \in (0,1]$ for any $\kappa > 0$ and $s_+ = 1$ if and only if $\kappa=4$. We refer the reader to Remark~\ref{s=1 kappa=4} below for more detail regarding the case $\kappa=4$, $s=1$. 
\end{remark}

\begin{figure}
\begin{center}
{\includegraphics[width=0.45\textwidth]{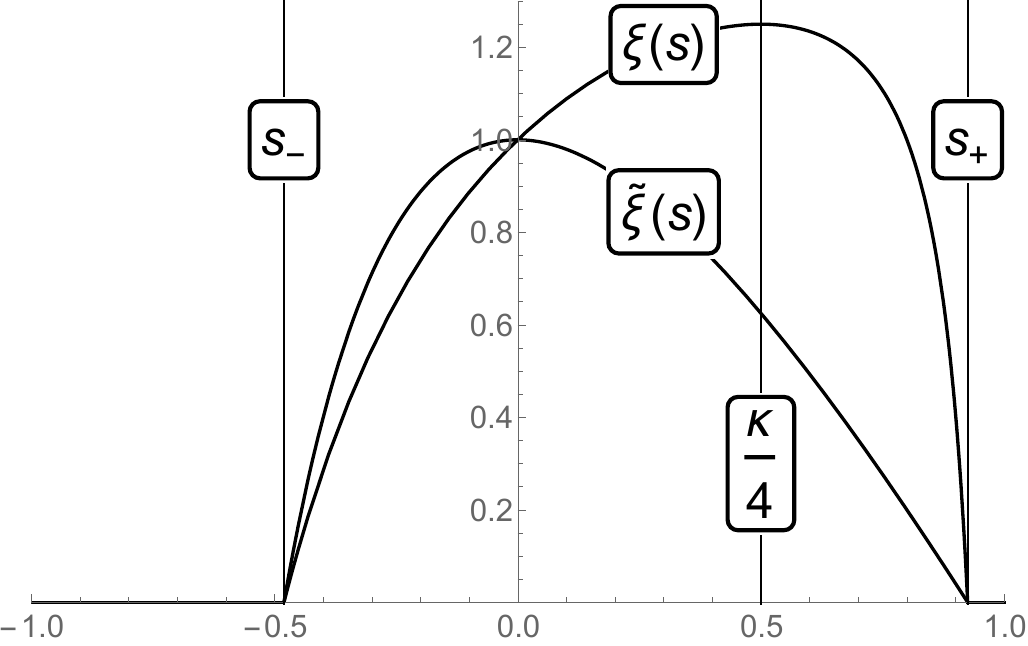}}\hspace{0.08\textwidth}{\includegraphics[width=0.45\textwidth]{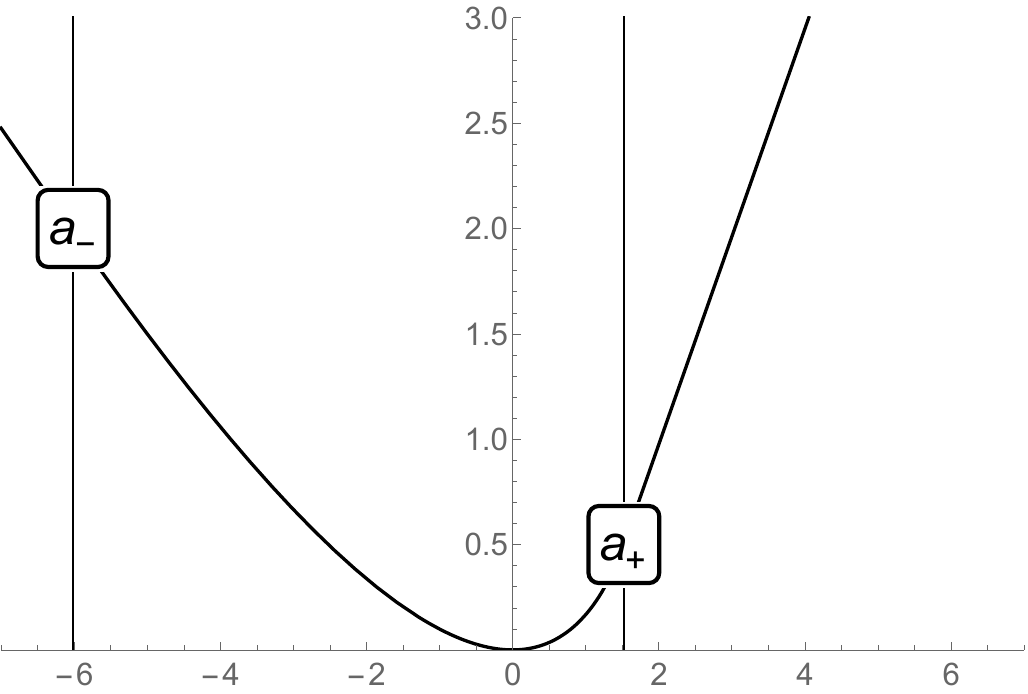}}
\end{center}
\caption{\label{xi xit graph} Left: A graph of the Hausdorff dimensions $\wt\xi(s)$ of $\wt\Theta^s(D_\eta)$ and $\xi(s)$ of $\Theta^s(D_\eta)$ from Theorem~\ref{main thm} as $s$ ranges from~$-1$ to~$1$ for $\kappa = 2$.  The value of $s$ which maximizes $\wt{\xi}$ is $0$ and the value of $s$ which maximizes $\xi$ is $\kappa/4=1/2$. Note that $\xi(\kappa/4) = 1+\kappa/8$ which is the almost sure Hausdorff dimension of $\SLE_\kappa$ \cite{beffara-dim}.  Right: a graph of the bulk integral means spectrum $\op{IMS}_{D_\eta}(\av)$ of $D_\eta$ from Corollary~\ref{ims cor} as $\av$ ranges from $-7$ to $7$ for $\kappa = 3$.}  
\end{figure}

The $\op{SLE}_\kappa(\ul{\rho})$ processes are an important variant of $\op{SLE}$ in which one keeps track of extra marked points --- so-called force points.  The force points can be either on the domain boundary or in its interior and are respectively referred to as boundary and interior force points.  These processes were first introduced by Lawler, Schramm, and Werner in \cite[Section~8.3]{lsw-restriction} and, just like ordinary $\op{SLE}_\kappa$, the $\op{SLE}_\kappa(\rho)$ processes naturally arise in many different contexts.  Since $\op{SLE}_\kappa(\ul\rho)$ for different vectors of weights $\ul\rho$ has the same behavior when it is not interacting with its force points, one expects an analog of Theorem~\ref{main thm} to be true for such processes provided we exclude points near the boundary of the domain and stop the path before interacting with an interior force point. Furthermore, by SLE duality, one expects an analog of Theorem~\ref{main thm} for $\kappa >4$. Such results do indeed hold true, as described in the following corollary.

\begin{cor}
\label{dim for K_t}
Let $D \subset \BB C$ be a smoothly bounded domain. Let $\kappa >0$ and let $\ul\rho$ be a vector of real weights. Let $\eta$ be a chordal $\op{SLE}_\kappa(\ul\rho)$ process in $D$, with any choice of initial and target points and force points located anywhere in $\ol{  D}$, run up until the first time it either hits an interior force point or hits the continuation threshold (c.f.\ \cite[Section~2.1]{ig1}). Fix $s \in (-1,1)$. Almost surely, the following is true. 
Let $V$ be a connected component of $ D\setminus \eta$ or a connected component of $ D\setminus \eta([0,t])$ for any $t > 0$ before $\eta$ hits an interior force point or the continuation threshold and let $\phi : \BB D\rta V$ be a conformal map. Then
\eqbn
\begin{aligned}[c]
	\dimH \left( \wt\Theta^{s  }(V) \setminus \phi^{-1}(\partial D)\right) &= \dimH \left( \wt\Theta^{s; \geq }(V) \setminus \phi^{-1}(\partial D)  \right) = \wt\xi(s) ,\quad & 0\leq s \leq s_+  \notag\\
	\dimH\left( \wt\Theta^{s  }(V) \setminus \phi^{-1}(\partial D)\right) &= \dimH\left( \wt\Theta^{s; \leq }(V) \setminus \phi^{-1}(\partial D)\right) = \wt\xi(s) , \quad & s_- \leq s \leq 0  \notag\\
	\dimH \left(\Theta^{s  } (V) \setminus \partial D\right) &= \dimH \left(\Theta^{s; \geq }(V) \setminus \partial D\right) =   \xi(s)  ,\quad & \frac{\kappa}{4} \leq s \leq s_+ \notag\\
	\dimH\left(  \Theta^{s  } (V) \setminus \partial D\right) &= \dimH\left(  \Theta^{s; \leq }(V) \setminus \partial D\right) =  \xi(s) , \quad & s_- \leq s \leq \frac{\kappa}{4} \notag
\end{aligned}
\eqen	
That is, the conclusion of Theorem~\ref{main thm} holds a.s.\ away from the domain boundary at all times simultaneously for an $\op{SLE}_\kappa(\ul\rho)$ with a general $\kappa > 0$ and vector of weights $\ul\rho$ up until the process either hits an interior force point or the continuation threshold. 
\end{cor}
\begin{proof}
This follows from Theorem~\ref{main thm} combined with Proposition~\ref{theta zero one} below. Note that the functions $\wt \xi(s)$ and $\xi(s)$ are unaffected if we replace $\kappa$ by $16/\kappa$, as one would expect from SLE duality \cite{zhan-duality1,zhan-duality2,dubedat-duality,ig1,ig4}.
\end{proof}

\begin{remark}
We believe that the techniques developed in this paper could also be employed to describe the multifractal behavior of the $\op{SLE}_\kappa(\ul{\rho})$ processes even near their intersection points with the domain boundary and near their tip, though we will not carry this out here.
\end{remark}

Roughly speaking, the \emph{harmonic measure spectrum} of a hull $A\subset \BB H$ gives, for each $\alpha \in (1/2,\infty)$, the Hausdorff dimension of the set $\Theta^\alpha_{\op{hm}}(A)$ of points $x \in \partial A$ for which the harmonic measure from $\infty$ of $B_\ep(x)$ in $\BB H\setminus A$ decays like $\ep^\alpha$ as $\ep\rta 0$ (or in the pre-image $\wt\Theta^\alpha_{\op{hm}}(A)$ of $\Theta^\alpha_{\op{hm}}(A)$ under a conformal map $\BB D\rta \BB H\setminus A$). In \cite[Section~2.3]{lawler-viklund-tip}, the authors give a rigorous treatment of the harmonic measure spectrum at the tip of a curve. A nearly identical construction works for the harmonic measure spectrum of a whole hull in $\BB H$. Similar constructions also work for hulls in $\BB D$ or $\BB C$. In particular, one has (see \cite[Lemma~2.3]{lawler-viklund-tip}) 
\eqb \label{mf hm spec}
\Theta^s ( A) = \Theta^{\frac{1}{1-s}}_{\op{hm}}(\BB H\setminus A) \quad \forall s \in (-1,1) . 
\eqe 

\begin{remark}
\label{bm spectrum}
In light of the relationship between $\SLE_6$ and Brownian motion \cite{lsw-frontier}, we see that Corollary~\ref{dim for K_t} with $\kappa = 6$ yields the harmonic measure spectrum for the Brownian frontier computed in \cite{lawler-frontier, lsw-frontier, lsw-bm-exponents1, lsw-bm-exponents2, lsw-bm-exponents3}. 
\end{remark}

\begin{remark}
In \cite{dup-mf-spec-bulk} (see in particular \cite[Equation~6]{dup-mf-spec-bulk}), Duplantier predicts that the harmonic measure spectrum for the bulk of the $\SLE_\kappa$ curve is given by 
\eqb \label{dup f(a)}
f(\alpha) = \alpha + \frac{25-c}{24} \left(   1 - \frac12 \left(2\alpha - 1 + \frac{1}{2\alpha-1} \right) \right) ,
\eqe
where
\eqbn
c = \frac{(6-\kappa) (6-16/\kappa)}{4}
\eqen
is the central charge. The exponent~\eqref{xi(s)} is related to the exponent~\eqref{dup f(a)} by
\eqbn
\xi(s ) = f\left(\frac{1}{1-s} \right). 
\eqen
This is what we would expect in light of~\eqref{mf hm spec}.  
\end{remark}

The dimension $\xi(s)$ attains a unique maximum value of $1+\kappa/8$ on $[-1,1]$ at $s = \kappa/4$. This maximum value coincides with the Hausdorff dimension of the $\SLE_\kappa$ curve \cite{beffara-dim}, which suggests that near a ``typical point'' of $\eta$, the modulus of the derivative of a conformal map from $D_\eta$ to $\BB D$ grows like $\op{dist}(z , \eta)^{ \frac{\kappa}{4-\kappa}}$. Hence Theorem~\ref{main thm} gives an alternative proof of the following. 
   
\begin{cor}
Let $\kappa\leq 4$. The Hausdorff dimension of an $\SLE_\kappa$ curve $\eta$ is a.s.\ equal to $1+\kappa/8$. 
\end{cor}

We remark that we believe that the methods that we use to establish the lower bound in Theorem~\ref{main thm} could be employed to give an independent derivation of the lower bound of the dimension of $\SLE_\kappa$ for all $\kappa > 0$, however we will not carry this out here.
 
\subsection{Optimal H\"older exponent for map uniformizing an $\SLE$}
\label{ohe sec}

Another consequence of Theorem~\ref{main thm} is that it is allows us to determine the optimal bulk H\"older exponent for the conformal map which uniformizes the complement of an $\SLE_\kappa$ curve.  (Note that this result concerns a different problem than \cite{lawler-viklund-holder}, which gives the optimal H\"older exponent for the $\SLE_\kappa$ curve itself with the capacity parameterization.)

\begin{cor}
\label{ohe cor}
Suppose that we have the same setup as in Theorem~\ref{main thm} and let $\phi \colon \D \to D_\eta$ be a conformal map taking $-i$ and $i$, respectively, to the start and end points of $\eta$.  On any subset of $\BB D$ lying at positive distance from $\{-i,i\}$, the function $\phi$ is $\alpha$-H\"older continuous for every $\alpha < (1-s_+)$ and is not $\alpha$-H\"older continuous for every $\alpha > (1-s_+)$.	
\end{cor}
\begin{proof}
Suppose that $s > s_+$. By Theorem~\ref{main thm}, $\wt \Theta^{s; \geq}(D_\eta) = \emptyset$ a.s. In fact, the proof of Theorem~\ref{main thm} gives a slightly stronger statement, namely that for each $\delta>0$, it is a.s.\ the case that $|\phi'(z)| \leq \ep^{-s}$ for each sufficiently small $\ep > 0$ and each $z \in (1-\ep) \bdy\BB D$ lying at distance at least $\delta$ from $\{-i,i\}$ (the relation~\eqref{bdy haus max} from Proposition~\ref{bdy haus upper} shows this with $\phi$ replaced by the inverse of the centered forward Loewner map for $\eta$ stopped at time $t >0$, and this is easily transferred to $\phi$). Consequently, if $x\in\bdy\BB D$ lies at distance at least $\delta$ from $\{-i,i\}$ then $|\phi'(z)| \preceq (1-|z|)^{-s}$ for each $z$ in the line segment $[(1-\ep) x , x]$. Integrating this relation gives $|\phi(x)  -\phi((1-\ep) x)| \preceq \ep^{1-s}$. Similarly, if $z,w\in (1-\ep)\bdy\BB D$ each lie at distance at least $\delta$ from $\{-i,i\}$, then $|\phi(z) - \phi(w)| \preceq |z-w| \ep^{-s}$. Combining these relations with $\ep = |x-y|$ and applying the triangle inequality shows that $|\phi(x) - \phi(y)| \preceq |x-y|^{1-s}$ whenever $x,y\in\bdy\BB D$ lie at distance at least $\delta $ from $\{-i,i\}$. 
This proves the upper bound.

Now suppose $s < s_+$.  Theorem~\ref{main thm} implies that $\wt \Theta^s(D_\eta) \neq \emptyset$ a.s.  Fix $x \in \wt \Theta^s(D_\eta)$ and for $\ep > 0$, let $y_\ep = (1-\ep)x$.  Then we know that $|\phi'(y_\ep)| \succeq \ep^{-s + o_\ep(1)}$.  Standard distortion estimates for conformal maps then imply that $|\phi'(z)| \succeq \ep^{-s + o_\ep(1)}$ for all $z \in B_{\ep/2}(y_\ep)$, which in turn implies that $\phi$ is not $(1-s)$-H\"older continuous.  This proves the lower bound.
\end{proof}

As explained above in the context of Theorem~\ref{main thm}, the statement of Corollary~\ref{ohe cor} also applies for $\SLE_\kappa$ curves with $\kappa > 4$ away from intersections with the domain boundary (by $\SLE$ duality) and for $\SLE_\kappa(\ul{\rho})$ curves for all $\kappa > 0$, also away from intersection with the domain boundary (by absolute continuity).

\subsection{Integral means spectrum} 
\label{ims sec}

The \emph{integral means spectrum} of a simply connected domain $D\subset \BB D$ is the function $\op{IMS}_D : \BB R\rta \BB R$ defined by 
\eqb
\label{IMS def}
\op{IMS}_D(\av) := \limsup_{\ep\rta 0} \frac{\log \int_{\partial B_{1-\ep}(0)} |\phi'(z)|^\av \, dz }{-\log \ep} ,
\eqe 
where $\phi : \BB D\rta D$ is a conformal map.  (There is a three parameter family of such conformal maps, but $\op{IMS}_D(\av)$ does not depend on the specific choice of $\phi$.)  The integral means spectrum is of substantial interest in complex analysis, primarily in the form of the \emph{universal integral means spectrum}, which is defined by
\[
\op{IMS}^U(\av) := \sup_{D} \op{IMS}_D(\av)
\]
where the supremum is over all simply connected domains $D\subset \BB C$. It has been conjectured by Kraetzer \cite{kraetzer-ims} that $\op{IMS}^U(\av) = t^2/4$ for $|t|\leq 2$ and $\op{IMS}^U(\av) = |t|-1$ for $|t|\geq 2$. This conjecture has several important consequences in complex analysis. See, e.g., \cite{pom-im-survey, bel-smirnov-survey, hs-ims-survey, pom-book} for more details. The integral means spectrum is often very difficult to compute in practice for deterministic domains. However, domains bounded by random fractals (e.g.\ the complement of an $\SLE_\kappa$ curve) are sometimes more tractable. For example, in \cite{bel-smirnov-hm-sle} Beliaev and Smirnov give an explicit calculation of the average integral means spectrum of the complement of a whole plane $\SLE_\kappa$ curve (which is defined as in~\eqref{IMS def} but with $|\phi'(z)|^a$ replaced by $\BB E(|\phi'(z)|^a)$).

In this paper we shall be interested in a slight refinement of the definition of the integral means spectrum for the complement of a curve which negates possible pathologies arising from unusual behavior at its endpoints or when it intersects itself or the boundary of the domain. Namely, let $D\subset\BB C$ be a bounded simply connected domain with smooth boundary and let $\eta  : [0,T]\rta \ol{ D}$ be a non-self-crossing curve (we allow $T = \infty$). Let $V$ be a connected component of $D\setminus \eta$. Let $x_V$ be the first (equivalently last) point of $\partial V$ hit by $\eta$ and let $\phi : \BB D\rta V$ be a conformal map. 

For $\zeta > 0$, let 
\eqb \label{I_zeta def}
I^\zeta(\phi) := \phi^{-1}\left( \partial V \setminus (B_\zeta(\eta(T)) \cup B_\zeta(x_V)  \cup B_\zeta(\partial D) ) \right)  . 
\eqe 
Let $A_\ep^\zeta(\phi)$ be the set of $z\in \partial B_{1-\ep}(0)$ with $z/|z| \in I^\zeta(\phi)$. 
The \emph{bulk integral means spectrum} of $V$ is the function $\op{IMS}_V : \BB R\rta \BB R$ defined by 
\eqb \label{bulk IMS def}
\op{IMS}_V^{\op{bulk}}(\av) := \sup_{\zeta  > 0} \limsup_{\ep\rta 0} \frac{\log \int_{A_\ep^\zeta(\phi) } |\phi'(z)|^\av \, dz }{-\log \ep} .
\eqe 
One can check that the definition~\eqref{bulk IMS def} does not depend on the choice of $\phi$.

We extract the following from the proof of Theorem~\ref{main thm}.

\begin{cor}
\label{ims cor}
For $a\in \BB R$ with $a < \frac{ (4+\kappa)^2}{8\kappa}$, let  
\eqb \label{s_*(a)} 
s_*(a) := -1 + \frac{ 4 + \kappa }{\sqrt{ (4 + \kappa)^2 - 8 a \kappa}} .
\eqe
Also let $s_-$ and $s_+$ be as in~\eqref{s-} and~\eqref{s+} and let $a_-$ (resp.\ $a_+$) be the value of $a$ for which $s_*(a) = s_-$ (resp.\ $s_*(a) = s_+$). Set
\eqb \label{IMS* def}
\xi_{\op{IMS}} (a):=
\begin{dcases} 
-1+ s_- a ,\qquad &a  < a_- \\
 -a  +   \frac{(4+\kappa)(4+\kappa - \sqrt{(4+\kappa)^2 -8a \kappa})}{4\kappa} ,\qquad &a\in [a_- , a_+] \\
-1+ s_+ a ,\qquad &a   >a_+ .
\end{dcases}
\eqe
Suppose we are in the setting of Corollary~\ref{dim for K_t}. Almost surely, the following is true. Let $a \in \BB R$ and let $V$ be a complementary connected component of either $D\setminus \eta$ or of $D\setminus \eta^t$ for any $t > 0$ (before $\eta$ hits an interior force point or the continuation threshold if it is an $\op{SLE}_\kappa(\ul{\rho})$ process). Then  
\eqb \label{ims eqn}
\op{IMS}_{V}^{\op{bulk}}(a) = \xi_{\op{IMS}}(a) .
\eqe 
\end{cor}
 
The result of Corollary~\ref{ims cor} is in agreement with the (rigorously proven) formula\footnote{The formula appearing in \cite[Theorem~1]{bel-smirnov-hm-sle} for the bulk integral means spectrum is actually equal to 5 plus the formula~\eqref{IMS* def}; the $5$ in their formula is a misprint.} 
for the average bulk integral means spectrum of whole-plane $\SLE$ in \cite[Theorem~1]{bel-smirnov-hm-sle} for $a \in [a_-, a_+]$,
and with \cite[Conjecture~1]{bel-smirnov-hm-sle} for the a.s.\ bulk integral means spectrum for all values of $a\in\BB R$.

\begin{remark}
As conjectured in~\cite{bel-smirnov-hm-sle}, the a.s.\ bulk integral means spectrum of Corollary~\ref{ims cor} differs from the average integral means spectrum computed in~\cite{bel-smirnov-hm-sle} for values of $a\notin [a_-,a_+]$. We explain why this is the case. First, as noted in~\cite{bel-smirnov-hm-sle}, we expect the average and a.s.\ bulk integral means spectra to differ because the function which gives the average bulk integral means spectrum does not satisfy Makarov's~\cite{makarov-fine} characterization of possible integral means spectra. At a more heuristic level, the average integral means spectrum for $a\notin[a_-,a_+]$ is distorted by the occurrence of the small (but still positive) probability event that a conformal map $\phi : \BB D\rta V$ satisfies $|\phi'(z)| \approx (1-|z|)^{-s}$ for some $z$ close to $\partial\BB D$ and some $s\notin [s_- , s_+]$. However, this event a.s.\ does not occur in the limit (c.f.\ Theorem~\ref{main thm}) so does not affect the a.s.\ bulk integral means spectrum.
\end{remark}

\subsection{Outline} 

There is a systematic approach to computing Hausdorff dimensions of random fractal sets of the sort we consider here. One first gets a sharp estimate for the probability that a single point is contained in the set (the ``one-point estimate") and uses this to get an upper bound on the Hausdorff dimension. One then defines a subset of the set of interest (the ``perfect points") and obtains an estimate for the probability that any two given points are perfect (the ``two-point estimate"). This enables one to define a Frostman measure on the set of perfect points and thereby obtain a lower bound on the Hausdorff dimension of the set of interest (see \cite[Section~4]{peres-bm} for more on Frostman measures and their connection to Hausdorff dimension). We will follow this outline here. See, e.g., \cite{miller-wu-dim,mww-extremes,lawler-viklund-tip,msw-gasket} for more examples of this technique.

We will now give a moderately detailed outline of the remainder of this paper. The reader should note that this section does not constitute a precise description of all of the proofs in our paper, but rather is only a heuristic guide. For the sake of brevity, many technical details have been omitted, especially in regards to proof of the two-point estimate. 
 
In Section~\ref{prelims sec}, we will give some background on the objects which appear in our proofs, including SLE, the GFF, and the various couplings between them. We will also establish some notation, introduce the main regularity conditions we will use in our estimates, and prove some elementary lemmas which we will need in the sequel.  
 
Next we will prove our one-point estimate. This is done in two stages. In Section~\ref{inverse sec}, we will establish pointwise derivative estimates for the inverse centered Loewner maps $(f_t^{-1})$ for an $\op{SLE}_\kappa$. Roughly, our estimates will take the form
\eqb
\label{rough inverse estimate}
\BB P(|(f_t^{-1})'(z)| \approx \ep^{-s} , \: \text{regularity conditions}) \approx \ep^{\alpha(s)} ,\qquad \forall s \in (-1,1) ,\qquad \forall z\in \BB H \:\op{with} \: \im z = \ep ,
\eqe 
with $\alpha(s) =\frac{(4 + \kappa)^2 s^2}{8 \kappa (1 + s)} $. The proof of these estimates is based on a family of non-negative martingales for the reverse Loewner flow $(g_t)$, analogous to the martingales for the forward $\op{SLE}_\kappa$ flow in \cite[Section~5]{sw-coord}. The reverse Loewner flow is of interest because we have $g_t \eqD f_t^{-1}$ for each fixed $t$ (see, e.g., \cite[Lemma~3.1]{schramm-sle}).  For a given $z\in \BB H$ with $\im z = \ep$, one can find a martingale $M_t^z$ with the property that $M_t \BB 1_{\ul E(z) }  \approx \ep^{-\alpha(s)}$, where $\ul E(z)$ denotes the event in the probability in~\eqref{rough inverse estimate} with $g_t$ in place of $f_t^{-1}$. We then arrive at
\[
\BB P(\ul E(z))\approx  \ep^{\alpha(s)}  \BB P^z_*(\ul E(z)),
\]
where $\BB P^z_*$ denotes the measure obtained by re-weighting the law of the original $\op{SLE}_\kappa$ process by $M$ (which will be the law of a reverse chordal $\op{SLE}_\kappa(\rho)$ for an appropriate $\rho$). Hence we just need to show $\BB P_*^z(\ul E(z))  $ is uniformly positive, independent of $\ep$. This is done in two steps.  First, to obtain $\BB P_*^z(|g_t'(z)| \approx \ep^{-s}) \rta 1$ as $\ep\rta 0$, we use a coupling of $g_t$ with a GFF together with a coordinate change argument similar in spirit to the proof of \cite[Theorem~8.1]{qle}. To obtain that the auxiliary regularity conditions hold with uniformly positive probability under $\BB P_*^z$, we use a combination of stochastic calculus, forward/reverse (in the sense of Loewner flows) SLE symmetry, and GFF coupling arguments. 

In Section~\ref{time infty sec} we use the estimate of Section~\ref{inverse sec} to establish pointwise derivative estimates for the ``time infinity" conformal map $\Psi_\eta$ associated with an $\op{SLE}_\kappa$ process $\eta$ from $-i$ to $i$ in the unit disk $\BB D$, defined as follows. Let $D_\eta$ be the right connected component of $\BB D\setminus \eta$, as in Theorem~\ref{main thm}. Let $\Psi_\eta : D_\eta \rta \BB D$ be the unique conformal map fixing $-i$, $i$, and 1. Our estimates for $\Psi_\eta$ take the form
\eqb \label{rough forward estimate}
\BB P\left(\op{dist}(z , \eta) \approx \ep^{1-s} , \: |\Psi_\eta'(z)| \approx \ep^{ s} , \: \text{regularity conditions}\right) \approx \ep^{\gamma(s)} ,\qquad \forall s \in (-1,1) ,\qquad \forall z\in \BB D 
\eqe 
where $\gamma(s) = \alpha(s) -2s + 1$ and $\alpha(s)$ as above. 
The idea of the proof of~\eqref{rough forward estimate} is as follows. First we observe using the Koebe quarter theorem that for each $\ep > 0$ and each $t>0$, the set of points $\ul A_\ep(t)$ in $\BB D$ for which the analog of the event of~\eqref{rough inverse estimate} with $\BB D$ in place of $\BB H$ occurs is (approximately) the image under $f_t$ of the set $A_\ep(t)$ of points in $\BB D$ for which the event of~\eqref{rough forward estimate} holds with $\Psi_\eta$ replaced by $f_t$ and $\eta$ replaced by $\eta([0,t])$. Hence the estimate~\eqref{rough inverse estimate} together with an elementary change of variables yields $\BB E(\op{Area} A_\ep(t)) \approx \ep^{\gamma(s)}$. We are then left to (a) transfer this area estimate from finite time to infinite time and (b) argue that the probability of the event~\eqref{rough forward estimate} does not depend too strongly on $z$. Both tasks will be accomplished by means of various conditioning arguments which rely crucially on the regularity conditions involved in the estimate~\eqref{rough inverse estimate}. 

In Section~\ref{haus upper sec}, we will use the estimates~\eqref{rough inverse estimate} and~\eqref{rough forward estimate} to prove upper bounds for the Hausdorff dimensions of the sets $\wt\Theta^{s;*}(D_\eta)$ and $\Theta^{s;*}(D_\eta)$, where $*$ stands for $\geq$ or $\leq$ as well as an upper bound for the bulk integral means spectrum of $D_\eta$, as claimed in Corollary~\ref{ims cor}. 

Before proving our two-point estimate, we need a modification of the estimate~\eqref{rough forward estimate}, which we prove in Section~\ref{2pt setup sec}. Namely, let $\ol\eta$ denote the time reversal of $\eta$, which has the law of a chordal $\op{SLE}_\kappa$ from $i$ to $-i$ \cite{zhan-reversibility}. Let $\tau_\beta$ (resp.\ $\ol\eta_\beta$) be the first time $\eta$ (resp.\ $\ol\eta$) hits the ball of radius $e^{-\beta}$ centered at the origin. Let $\eta^{\tau_\beta} = \eta([0,\tau_\beta])$, $\ol\eta^{\ol\tau_\beta} = \ol\eta([0,\ol\tau_\beta])$, and let $\phi_\beta$ be the conformal map from $\BB D\setminus (\eta^{\tau_\beta} \cup \ol\eta^{\ol\tau_\beta})$ to $\BB D$ which fixes $-i$, $i$, and 1.  Then we will use the one-point estimate~\eqref{rough forward estimate} to show
\eqb \label{rough hitting estimate}
\BB P\left(|\phi_\beta'(z)|   \approx e^{-\beta q}, \: \text{regularity conditions}\right) \approx  e^{-\beta \gamma^*(q)}  ,\qquad \forall q \in (-1/2,\infty)  .
\eqe 
Here $q = s/(1-s)$ and $\gamma^*(q) = \gamma(s)/(1-s) = (q+1)\gamma(q)$, with $\gamma$ as in~\eqref{rough forward estimate}.
 
In Section~\ref{2pt sec} we prove our two-point estimate. This section contains the most technical, but also the most novel, arguments in the paper; see Section~\ref{sec-2pt-outline} for a more detailed outline of this section than the one given here. 
The estimate~\eqref{rough hitting estimate} allows us to break the event that $|\Psi_\eta'(0)| \approx e^{-n \beta}$ down into several stages and estimate each individually. Indeed, if we apply a conformal map from $\BB D\setminus (\eta^{\tau_\beta} \cup \ol\eta^{\ol\tau_\beta})$ to $\BB D$ which fixes 0, then the rest of the curve will be mapped to another curve whose law is the same as that of $\eta$ (modulo perturbations of its endpoints, which can be dealt with in various ways). In this manner we can construct two approximately independent events $E_{0,1}$ and $E_{0,2}$ whose intersection is contained in the event $\{|\Psi_\eta'(0)| \approx e^{- 2\beta q}  \}$. By iterating this procedure we construct a sequence of approximately independent events $E_{0,j}$ such that $|\Psi_\eta'(0)| \approx e^{-n\beta q}$ on $E_n(0) := \bigcap_{j=1}^n E_{0,j}$ and $\BB P(E_{z,j}) \approx e^{-\beta \gamma^*(q)}$.\footnote{Actually, we will need to increase $\beta$ by a little bit at each stage for technical reasons, but the basic idea of the argument is the same if we consider a fixed but large $\beta$.}
We can similarly construct events $E_{z,j}$ and $E_n(z)$ for any $z\in\BB D$ by first mapping $z$ to 0. 

For the lower bound on $\dim_{\mathcal H} \Theta^s(D_\eta)$, the perfect points will be, roughly speaking, the set of $z\in\BB D$ for which $E_n(z)$ occurs for every $n\in\BB N$.  In order to obtain a lower bound on the Hausdorff dimension of the set of perfect points, we need to estimate the probability that $E_n(z)$ and $E_n(w)$ both occur for $z,w\in\BB D$, depending on $|z-w|$. To this end, suppose $|z-w| \approx e^{-\beta k}$. We condition on the event $E_k(z)$, corresponding to what happens before we get near $z$ and $w$. After we map out the part of the curve which is grown before the $k$th stage, $z$ and $w$ will be at constant order distance from each other. See Figure~\ref{2pt zoom in fig}.

\begin{figure}[ht!]
 \begin{center}
\includegraphics[scale=.8]{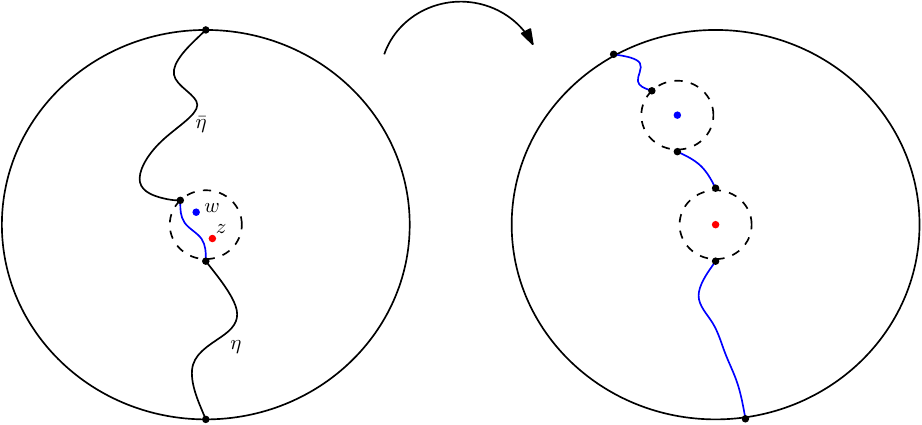} 
\caption{\label{2pt zoom in fig} If $|z-w| \approx e^{-\beta k}$, then after applying a conformal map which takes the complement of the parts of $\eta$ and $\ol\eta$ involved in the event $E_0^k(z)$ to $\BB D$ and takes $z$ to 0, the images of $z$ and $w$ will be at constant order distance from each other. Note, however, that in this setting the derivatives of the stage $k+1$-map near $z$ and $w$ are not approximately independent, since they each depend on the whole curve in the picture on the right.  }
\end{center}
\end{figure}

We would like to say that the behaviors of the curve near $z$ and near $w$ are approximately conditionally independent given $E_k(z) $. However, the derivatives of the maps we are interested in depend on the whole curve. Hence we need to localize our events. This is accomplished using a different coupling with a GFF, namely the forward SLE/GFF coupling, or ``imaginary geometry" coupling studied in \cite{dubedat-coupling,shef-zipper,shef-slides,ig1,ig2,ig3,ig4}. 

At each stage in the construction of the events $E_n(z)$, we can add auxiliary curves, which are all flow lines (in the sense of \cite{ig1}; c.f.\ Section~\ref{ig prelim}) of the same GFF. These auxiliary curves will form pockets surrounding $z$ with the property that the parts of $\eta$ inside different pockets are independent once we condition on the pockets, and the derivative of $\Psi_\eta$ at a point inside a pocket can be estimated by the derivative of a map which depends only on the behavior of $\eta$ inside this pocket. We then define the event $E_{z,j}$ so that it depends only on the behavior of the curve inside the $j$th pocket. See Figure~\ref{2pt zoom in flow lines fig} for an illustration.

\begin{figure}[ht!]
 \begin{center}
\includegraphics[scale=.8]{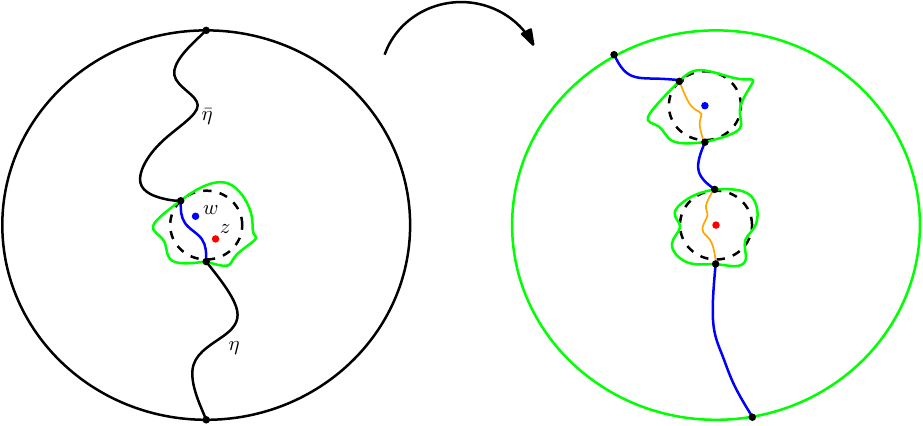} 
\caption{\label{2pt zoom in flow lines fig} A modified version of Figure~\ref{2pt zoom in fig} where we add auxiliary curves (shown in green) at each stage to form a pocket. Here we define the events at each stage in terms of only the part of the curve inside the previous pocket. This gives us the needed local independence of the events $E_{z,j}$ and $E_{w,j}$. }
\end{center}
\end{figure}

The independence of the parts of $\eta$ inside different pockets will eventually enable us to establish the two-point estimate needed for the proof of the lower bounds in Theorem~\ref{main thm}.  

We expect that arguments similar to those in Section~\ref{2pt sec} may also be useful for proving other estimates for sets related to SLE; see Section~\ref{sec-other-settings} for further discussion of this point.

In Section~\ref{haus lower sec}, we use our two-point estimate to prove lower bounds for the Hausdorff dimensions of the sets $\wt\Theta^s(D_\eta)$ and $\Theta^s(D_\eta)$ as well as for the bulk integral means spectrum of $D_\eta$.  

Appendix~\ref{auxiliary sec} contains the proof of an estimate which is needed in Section~\ref{inverse sec}. Appendices~\ref{local prelim} and~\ref{smac sec} contain some technical lemmas which are needed in Sections~\ref{2pt setup sec} and~\ref{2pt sec}.

\section{Preliminaries}
\label{prelims sec}

In this section we will establish some notation, give some background on the objects involved in the paper, and prove some elementary lemmas. We recommend that the reader familiarize themselves with Section~\ref{basic notation} and Section~\ref{G prelim} before reading the remainder of the paper, as the notation and results of these subsections will be used frequently in the sequel. Sections~\ref{sle prelim},~\ref{gff prelim}, and~\ref{ig prelim} contain background on results on SLE, Gaussian free fields, and the couplings between them. Readers who are already familiar with these topics may wish to skim these subsections to acquaint themselves with the notation, and refer back to them as needed. Sections~\ref{multifractal sets} and~\ref{zero one sec} contain some elementary lemmas about the sets whose Hausdorff dimensions we will compute. The results of these sections are not used extensively in the sequel, but are needed in Sections~\ref{haus upper sec} and~\ref{haus lower sec}. Finally, in Section~\ref{pos prob sec}, we recall some lemmas from~\cite{miller-wu-dim} which we use frequently throughout the paper.

\subsection{Basic notation}
\label{basic notation}
Given two variables $a$ and $b$, we say $b = o_a(1)$ if $b \rta 0$ as $a \rta 0$ (or as $a\rta \infty$, depending on the context) and we say $b = O_a(1)$ if $b$ is bounded above by an $a$-independent constant for sufficiently small (or sufficiently large, depending on context) values of $a$. We usually allow $o_a(1)$ and $O_a(1)$ terms to depend on certain parameters other than $a$, but not on others. We will describe this dependence as needed. 

We say that $a \preceq b$ (resp.\ $a \succeq b$) if there is a constant $c$ which does not depend on the main parameters of interest such that $a \leq c b$ (resp.\ $a \geq c b$). We say $a \asymp b$ if $a \preceq b$ and $a \succeq b$. As in the case of $o_a(1)$ and $O_a(1)$ above, we usually allow the implicit constants in $\preceq, \succeq$, and $\asymp$ to depend on certain parameters, but not on others, and we describe this dependence as needed.

For a point $z\in \BB C$ and $r> 0$, we write $B_r(z)$ for the ball of radius $r$ centered at $z$. More generally, for a set $A\subset \BB C$, we write $B_r(A) = \bigcup_{z\in A} B_r(z)$. 

For a curve $\eta : [0,T] \rta \BB C$, we will often use the abbreviation
\eqb \label{eta^t}
\eta^t = \eta([0,t]) .
\eqe 
Furthermore, when there is no risk of ambiguity we will simply write $\eta$ for the entire image of $\eta$. 

For a domain $D$ and $z\in D$, we write $\op{hm}^z( \cdot ; D)$ for the harmonic measure from $z$ in $D$.  That is, for $A\subset\partial D$, $\op{hm}^z(A ; D)$ is the probability that a Brownian motion started from $z$ exits $D$ in $A$. 

If $D' = D\setminus \eta$ for some non-self-crossing curve $\eta$ in $\ol D$ and $z$ is a point on $\eta$ which is visited only once, we will write $z^-$ (resp.\ $z^+$) for the prime end of $D'$ corresponding to the left (resp.\ right) side of $z$. When we use this notation, our curve $\eta$ will have an obvious orientation and ``left" and ``right" are as viewed by someone walking along $\eta$ in the forward direction. 

We will also use the following notation.

\begin{notation}
\label{arc notation}
Given a Jordan domain $D$ and $x,y\in\partial D$, we write $[x,y]_{\partial D} $ for the closed counterclockwise arc from $x$ to $y$ in $\partial D$. We similarly define the open arc $(x,y)_{\partial D}$ and the half-open arcs $(x,y]_{\partial D}$ and $[x,y)_{\partial  D}$. 
\end{notation}

\subsection{Reverse continuity conditions} 
\label{G prelim}

\subsubsection{In the upper half plane}
\label{G prelim H}

Here we introduce a regularity condition which will arise frequently in the remainder of the paper. This regularity event will depend on a certain increasing function (thought of as a modulus of continuity). To lighten notation when referring to such functions, we introduce the following definition. 

\begin{defn} \label{mathcal M def}
We denote by $\mathcal M$ the set of increasing functions $\mu : (0,\infty) \rta (0,\infty)$ with $\lim_{\delta\rta0}\mu(\delta) = 0$. 
\end{defn}

\begin{defn}
\label{G def}
Let $f$ be a (random) map from a subdomain $D$ of $\BB H$ into $\BB H$. For $\mu\in\mathcal M$, let $G(f,\mu)$ be the event that the following occurs. For any $\delta>0$ and any $x,y\in \BB R \cap \partial D$ with $|x| ,|y| \leq  \delta^{-1}$ and $|x-y| \geq \delta $, we have $|f (x) | , |f (y)| \leq \mu(\delta)^{-1}$ and $|f   (x) - f (y)| \geq \mu(\delta)$.   
\end{defn}

The statement that $\mathcal G(f,\mu)$ holds is the same as the statement that $f^{-1}$ has a certain $\mu$-dependent modulus of continuity on $f(\BB R \cup \infty)$, with $\BB R \cup \infty$ given the one-point compactification topology. 

We note that
\eqb \label{G compose}
G(f,\mu_1) \cap G(g,\mu_2) \quad \Rta \quad G(g\circ f , \mu_2 \circ \mu_1) .
\eqe  
 
We are interested in the condition $G(f,\mu)$ (and the analogous conditions in the next subsection) for two reasons. The first is that these conditions imply bounds on the distance from certain subsets of $\partial D$ to certain subsets of $\BB R$ (or $\partial\BB D$ in the setting of the next subsection) and on the diameter of such subsets (see Lemmas~\ref{G implies U} and~\ref{G dist} below). Such bounds are needed for several purposes in our proofs. One reason is that some of our derivative estimates do not hold if the curve gets too close to the boundary---intuitively, if the curve comes close to hitting the boundary and forming a ``bubble", then the derivative of its associated Loewner map at points inside the bubble will be very small. This manifests itself in the fact that the martingale~\eqref{M def} blows up. Another use of such estimates is in checking the hypotheses of the harmonic measure estimates from Appendix~\ref{local prelim}.
 
The second reason for our interest in $G(f,\mu)$ is as follows. We will often want to study conformal maps which are normalized by specifying the images of certain marked boundary points. When composing various maps, our marked points might be mapped to somewhere other than where we want them to go. So, we will frequently need to apply a conformal automorphism (of $\BB D$ or $\BB H$) at the end of our arguments to move the marked points to their desired positions. The condition $G(\cdot,\mu)$ ensures that the images of the marked points are not too close together, and so allows us to control the derivative of this conformal automorphism.  

Both of the above uses of our regularity events appear in numerous places throughout the paper.
  
\begin{lem}\label{G implies U}
Let $\eta$ be a simple curve started from $0$ in $\BB H$ parameterized by capacity which does not hit $\BB R$ and recall that $\eta^t := \eta([0,t])$. Let $f_t  : \BB H\setminus \eta^t  \rta \BB H$ be the centered Loewner maps for $\eta$, i.e.\ $f_t$ is the time $t$ Loewner map for $\eta$, minus a real number chosen so that it maps 0 to 0. Fix $T \in (0,\infty)$ and suppose that for some $\mu \in \mathcal M$,   
\eqb\label{f_T big}
 f_T(-\delta ) - f_T(0^-) \leq -\mu(\delta) \quad \op{and} \quad \mu(\delta) \leq f_T( \delta ) - f_T(0^+),\qquad \forall \delta> 0.
\eqe
 Then there is a $\mu'\in\mathcal M$ and a $d > 0$ depending only on $\mu$ and $T$ such that
\eqb\label{diam small}
\op{diam} \eta^T  \leq d \quad \op{and} \quad \text{$\im z \geq \mu'(\delta)$, $\forall \delta>0$, $\forall z\in \eta^T$ with $|\re z| \geq \delta $}.  
\eqe
Conversely, if~\eqref{diam small} holds for some $d > 0$ and some $\mu'\in\mathcal M$, we can find $\mu\in\mathcal M$ depending only on $d$ and $\mu'$ such that $ G(f_T ,\mu)$ holds. 
\end{lem}

Note that it is clear that $G(f_T, \mu)$ implies~\eqref{f_T big}, so Lemma~\ref{G implies U} implies in particular that~\eqref{diam small} holds for some $d$ and $\mu'$ depending only on $\mu$ whenever $G(f_T,\mu)$ occurs.

\begin{proof}[Proof of Lemma~\ref{G implies U}]
Let $\op{hm}^\infty_T = \op{hm}^\infty(\cdot ; \BB H\setminus \eta^T)$ denote harmonic measure from $\infty$ in $\BB H \setminus\eta^T$, so for a set $I\subset \partial (\BB H\setminus \eta^T )$ (viewed as a collection of prime ends), 
\eqbn
\op{hm}^{\infty}_T I := \lim_{y\rta \infty}  y\BB P^{iy} (B_\tau \in I)
\eqen
 for $B$ a Brownian motion and $\tau $ its exit time from $\BB H\setminus\eta^T$. It follows from conformal invariance of Brownian motion that for any $I\subset \partial (\BB H\setminus \eta^T)$,
 \eqb \label{hm length}
\op{hm}^\infty_T(I ) = \frac1\pi \op{length} f_T(I) , 
 \eqe 
 where by $\op{length}$ we mean Lebesgue measure.  
 
Now, assume~\eqref{f_T big} holds. 
For any $r>0$ and $x\in\BB R$, the harmonic measure from $\infty$ in $\BB H$ of the line segment $[x,x+ir]$ from $x$ to $x+ir$ is a constant depending only on $r$. For $\delta > 0$, we can find $r = r(\delta) > 0$ such that this constant is $< \pi \mu(\delta)$. If $\eta^T$ contains a point $x+i y$ with $x \geq \delta$ and $y\leq r$, then $\op{hm}^\infty_T([0,\delta] ) \leq \op{hm}^\infty_T([x , x+ir] ) <  \pi  \mu(\delta)$. This contradicts our hypothesis on~\eqref{f_T big} and the relation~\eqref{hm length}. A similar statement holds if we instead consider $x\leq -\delta$. Hence each point of $\eta^T$ with real part $\geq \delta$ in absolute value has imaginary part $\geq r$. This proves the second part of~\eqref{diam small} with $\mu'(\delta) = r$. 
 
For the first part of~\eqref{diam small}, fix $\delta>  0$. Denote by $S_{\delta}$ the set of points in $z\in \BB H$ with $|\re z| \geq \delta$. By the second part of~\eqref{diam small},  
\eqb \label{hm S_delta}
\op{hm}^\infty_T(\eta^T \cap S_{\delta})  \leq \frac{1}{\mu'(\delta)} \lim_{y\rta \infty}  y\BB E^{iy}\!\left( \im B_\tau  \BB 1_{(B_\tau \in \eta^T \cap S_{\delta})} \right) .
\eqe 
By \cite[Proposition~3.38]{lawler-book},
 \eqb \label{hcap gamma formula}
T  = \op{hcap} \eta^T   = \lim_{y\rta \infty}  y\BB E^{iy} (\im B_\tau) 
\eqe 
so~\eqref{hm S_delta} is at most $  T/\mu'(\delta)$. On the other hand,~\eqref{hcap gamma formula} and the Beurling estimate imply that $\sup_{z\in \eta^T} \im z$ is bounded above by a constant $C_0$ depending only on $T$. The harmonic measure from $\infty$ in $\BB H$ of $[-\delta , \delta] \times [0,C_0]$ is at most a constant $C_1$ depending only on $\delta$ and $T$. Therefore 
\[
\op{hm}^\infty_T ( \eta^T) \leq T/\mu'(\delta) + C_1 .
\]
By \cite[equation 3.13]{lawler-book}, this implies $\op{diam} \eta^T $ is bounded above by a constant depending only on $\mu$ and $T$.  

Conversely, suppose~\eqref{diam small} holds. For $\delta >0$, let $U_\delta$ be the set of points in $z\in\BB H$ with $|z| \leq d$ and either $|\re z | \leq \delta/2$ or $\im z \geq \mu'(\delta/2)$. Then $\eta^T \subset U_\delta$. The harmonic measure from $\infty$ of each sub-interval of $[\delta/2 , \delta^{-1}] \cup [-\delta^{-1} , -\delta/2]$ in $\BB H\setminus U_\delta$ of length $\delta/2$ is at least some constant $\mu_0(\delta)$ depending only on $\delta $ and $\mu'(\delta/2)$. By~\eqref{hm length}, this implies that the length of the image of such an interval under $f_T$ is at least a $\pi   \mu_0(\delta)$. On the other hand, \cite[Proposition~3.46]{lawler-book} implies that we can find $\mu_1(\delta) > 0$ depending only on $\delta$ and $d$ such that $|f_T(x)| \leq \mu_1(\delta)^{-1}$ for each $x\in [-\delta^{-1} , \delta^{-1}]$. This proves that $\mathcal G(f_T ,\mu)$ holds with $\mu = (\pi  \mu_0)\vee\mu_1$. 
\end{proof}

\subsubsection{In the disk} \label{G disk sec}
 
The following is the analog of Definition~\ref{G def} for the unit disk $\BB D$.

\begin{defn}\label{G infty def}
Let $D\subset \BB D$ be a subdomain and let $I\subset \partial \BB D \cap \partial D$. Let $f :  D\rta \BB D$ be a conformal map. Let $\mu\in\mathcal M$ (Definition~\ref{mathcal M def}). We say that $\mathcal G_I(f, \mu)$ occurs if the following is true. For each $\delta >0$ and each $x,y\in I$ with $|x-y| \geq \delta$, we have $|f(x) - f(y) |\geq \mu(\delta)$. We abbreviate 
\[
\mathcal G (f,\mu) =   \mathcal G_{\partial\BB D\cap \partial D} (f,\mu).
\]
\end{defn}

We also define the following event, which is closely related to $G(f,\mu)$ and is a variant of the condition~\eqref{f_T big}.

\begin{defn}\label{G' def}
Let $A\subset \ol{\BB D}$ be a closed set and $I\subset \ol{\partial \BB D \setminus A}$. (Oftentimes we will take $I$ to be a closed arc with endpoints in $A$, or a finite union of such arcs.) We say that $\mathcal G_I'(A , \mu)$ occurs if the following is true. For each $\delta > 0$, $A$ lies at distance at least $\mu(\delta)$ from $I\setminus B_\delta(I\cap A)$. We write
\[
\mathcal G'(A , \mu) = \mathcal G_{\ol{\partial \BB D \setminus A}}(A , \mu). 
\]
\end{defn}

\begin{remark} \label{G pos}
We will frequently find ourselves in the following situation. Suppose we are given a deterministic arc $I\subset \partial \BB D$, a random closed subset $A\subset \ol{\BB D}$ with $I \subset\ol{\partial \BB D \setminus A}$ a.s., and a deterministic $\ep > 0$. In this case we can find (using monotonicity) a deterministic $\mu \in \mathcal M$ for which $\BB P\!\left(\mathcal G_I(A , \mu) \right) \geq 1-\ep$ where $\BB P$ is typically the law of SLE. 
\end{remark}

The conditions of Definitions~\ref{G infty def} and~\ref{G' def} will serve as the main ``global regularity" conditions in our estimates starting from Section~\ref{time infty sec}. The relationship between the conditions $\mathcal G(\cdot)$ and $\mathcal G'(\cdot)$ is contained in the following lemma. 
 
\begin{lem} 
\label{G dist}
Let $A\subset \ol{\BB D}$ be a closed set and $I = [x,y]_{\partial\BB D}$ be an arc contained in $\ol{\partial \BB D \setminus A}$. Let $m\in (x,y)_{\partial\BB D}$ and suppose that $|x-m|$ and $|y-m|$ are each at least $\Delta>0$. Let $D $ be the connected component of $\BB D\setminus A$ containing $I$ on its boundary. Let $\Phi :  D \rta \BB D$ be the unique conformal map taking $x$ to $-i$, $y$ to $i$, and $m$ to~$1$. 
\begin{enumerate}
\item For each $\mu\in\mathcal M$, there exists $\mu'\in\mathcal M$ depending only on $\mu$ and $\Delta$ such that if $\mathcal G_I (\Phi  , \mu)$ occurs, then $\mathcal G_I'(A , \mu')$ occurs. \label{G to dist}
\item Conversely, suppose $I'\subset I$ (possibly $I'=I$) and $\mathcal G_{I'}'(A , \mu)$ occurs for some $\mu\in\mathcal M$. There is a $\mu'\in\mathcal M$ depending only on $\mu$ and $\Delta$ such that $\mathcal G_{I'}(\Phi , \mu')$ occurs. In fact, the following superficially stronger statement is true. For each $\delta > 0$, $\Phi$ is Lipschitz continuous on $I'\setminus (B_\delta(x) \cup B_\delta(y))$ and $\Phi ^{-1}$ is Lipschitz continuous on $\Phi( I'\setminus (B_\delta(x) \cup B_\delta(y)) )  $ with Lipschitz constants depending only on $\mu(\delta)$, $\delta$, and $\Delta$. \label{dist to G}
\end{enumerate}
\end{lem}
\begin{proof}
The basic idea of the proof is similar to that of Lemma~\ref{G implies U}, but we consider harmonic measure from $m$ rather than harmonic measure from $\infty$. 

Let $\wh D$ be the radial reflection of $\wh D$ across $I$, viewed as a subset of the Riemann sphere. Extend $\Phi $ to $\wh D $ by Schwarz reflection. Then $\Phi$ maps $\wh D $ into $\BB C\setminus [i,-i]_{\partial\BB D}$, and maps $I$ to $[-i,i]_{\partial\BB D}$.

For $\delta>0$, let $x_{\delta }$ and $y_{\delta }$ be the unique points of $I$ lying at distance $\delta $ from $x$ and $y$, respectively. Also let $\wh D_\delta = \wh D\setminus [y_{\delta }  ,y]_{\partial\BB D}$ and let $\wt y_\delta := \Phi(y_\delta)$. Then $\wt y_{\delta} $ is determined by the condition that the harmonic measure of $[y_{\delta},i]_{\partial\BB D}$ from $m$ in $\wh D_\delta$ equals the harmonic measure of the side of $[\wt y_{\delta}  , i]_{\partial\BB D}$ closer to 0 from 1 in $(\BB C \cup \infty)\setminus [\wt y_{\delta} ,-i]_{\partial\BB D}$. 

If $\mathcal G^*_I(\Phi  ,\mu)$ occurs, then $\wt y_{\delta} $ lies at distance at least $\mu(\delta)$ from $i$, which means that the harmonic measure of $[y_{\delta },y]_{\partial\BB D}$ from 1 in $\wh D_\delta$ is at least some constant $\ep > 0$ depending only on $\mu(\delta)$. By symmetry, the same holds for $[x ,x_{\delta }]_{\partial\BB D}$. 

By the Beurling estimate, we can find $\zeta_0 > 0$ depending only on $\ep$ such that $\op{dist}(m ,A) \geq \zeta_0$. We can also find a $\zeta_1 >0$ such that if $z\in [x_{\delta } , y_{\delta }]_{\partial\BB D}$ lies at distance at least $ \zeta_0$ from $m$, then the probability that a Brownian motion started from $m$ hits $B_{\zeta_1}( z)$ before hitting $[i,-i]_{\partial\BB D}$ is at most $\ep$. If $\op{dist}(z,A) < \zeta_1$ for such a $z$, then a Brownian motion started from 1 must hit $B_{\zeta_1}(z)$ before hitting either $[y_{\delta } , y]_{\partial\BB D}$ or $[x , x_{\delta  } ]_{\partial\BB D}$. Hence we must have $\op{dist}(z , A) \geq    \zeta_1 \wedge \zeta_0$ for each $z\in [x_{\delta } , y_{\delta }]_{\partial\BB D}$. This proves assertion~\ref{G to dist} with $\mu'(\delta) = \zeta_1 \wedge\zeta_0$. 

Conversely, suppose $I'\subset I$ and $\mathcal G_{I'}'(A , \mu)$ occurs for some $\mu\in\mathcal M$. For $\delta >0$ let $x_\delta'$ be either $x_\delta$ (as defined just above) or the endpoint of $I'$ closest to $x$, whichever is furthest from $x$. Define $y_\delta'$ similarly. A Brownian motion started from any point of $  [x_{\delta}' , y_{\delta }' ]_{\partial\BB D}$ as a positive probability depending only on $\delta $, $\mu(\delta)$, and $\Delta$ to stay within distance $\mu(\delta)$ of $I$ until it hits $ [y_{\delta }' , y]_{\partial\BB D}$ (resp.\ $[x,x_{\delta }']_{\partial\BB D}$). By the Beurling estimate there is a $\mu'(\delta) > 0$ depending only on $\mu(\delta)$, $\delta$, and $\Delta$ such that $\Phi( [x_{\delta}' , y_{\delta }' ]_{\partial\BB D} )$ lies at distance at least $\mu'(\delta)$ from $[i,-i]_{\partial\BB D}$. Thus $\mcl G_{I'}(\Phi , \mu')$ occurs.

It remains to establish the Lipschitz continuity statement. For this, we observe that for any $z\in  [x_{\delta}' , y_{\delta }' ]_{\partial\BB D}$, the Koebe quarter theorem implies
\[
 \frac{    \op{dist}(\Phi (z)  , [i,-i]_{\partial\BB D}) }{4 \op{dist}(z,A) \wedge \delta   }   \leq |\Phi '(z)| \leq  \frac{  4  \op{dist}(\Phi (z)  , [i,-i]_{\partial\BB D}) }{ \op{dist}(z,A) \wedge \delta   }  .
\]
Hence
\[
 \frac{ \mu'(\delta) }{8 }   \leq |\Phi '(z)| \leq  \frac{ 8}{ \mu(\delta)\wedge\delta  }  .
\]
So, $|\Phi '|$ is bounded above and below by positive constants on $[x_\delta' ,y_\delta']_{\partial\BB D}$ depending only on $\mu(\delta)$, $\delta$, and $\Delta$ which establishes the desired Lipschitz continuity.
\end{proof}

 \subsection{Schramm-Loewner evolution}
 \label{sle prelim} 
 
 Let $t\mapsto W_t$ be a continuous function on $[0,\infty)$. The \emph{chordal Loewner equation} is the ordinary differential equation
\eqb\label{loewner eqn}
\partial_t g_t(z) = \frac{2}{g_t(z) -  W_t}   ,\qquad g_0(z) = z  .
\eqe
A solution to~\eqref{loewner eqn} is a family of conformal maps $\{g_t : t\geq 0\}$ from subdomains of $\BB H$ to $\BB H$, satisfying the hydrodynamic normalization $\lim_{z\rta \infty} (g_t(z) - z) = 0$. The complements $(K_t)$ of the domains of $(g_t)$ in $\BB H$ are an increasing family of closed subsets of $\BB H$ called the \emph{hulls} of the process. The \emph{centered Loewner maps} corresponding to $(g_t)$ are defined by
\[
f_t := g_t -W_t .
\]

A chordal \emph{Schramm-Loewner evolution} with parameter $\kappa > 0$ ($\op{SLE}_\kappa$) is the random evolution obtained by solving~\eqref{loewner eqn} where the driving process $W$ is $\sqrt\kappa$ times a Brownian motion. It can be shown \cite{schramm-sle} that this Loewner evolution is generated by a curve which we typically denote by $\eta$.  
Chordal $\op{SLE}_\kappa$ on other domains is defined by conformal mapping. We refer the reader to \cite{lawler-book} or \cite{werner-notes} for a more detailed introduction to SLE. 

More generally, suppose we are given a vector of real weights $\ul \rho = (\rho^1 , \ldots , \rho^n  )$ and a collection of points $z^1 , \ldots , z^n \in \BB H$. 
Chordal $\op{SLE}_\kappa(\ul \rho)$ is the random evolution obtained by solving~\eqref{loewner eqn} with the driving function $W$ part of the solution to the system of SDE's 
\eqb\label{sle kappa rho}
dW_t = \sqrt\kappa dB_t +    \sum_{i=1}^{n} \re \frac{\rho^{i}}{W_t - V_t^i} dt, \qquad dV_t^{i } = \frac{2}{V_t^i - W_t} dt , \qquad W_0 = y \qquad V_0^{i } = z^{i } .
\eqe  
The points $z^i$ are called the \emph{force points}. It is shown in~\cite{ig1} that if the force points are located in $\bdy \BB H$, then the SLE$_\kappa(\ul\rho)$ curve is a.s.\ defined and continuous up until the first time it reaches the so-called \emph{continuation threshold}, i.e., the first time that the sum of the weights of the force points it has either hit or disconnected from its target point is $\leq -2$. By local absolute continuity, the same is true if the curve a.s.\ does not hit any of its interior force points. The continuity of $\op{SLE}_\kappa(\rho)$ for $\rho < -2$ is proved in \cite{ms-gff_light_cones,cle-percolations}. 
See \cite{lsw-restriction, sw-coord, ig1} for more on $\op{SLE}_\kappa(\ul\rho)$. 

We will also need to consider the \emph{reverse Loewner equation}. This is the ODE
\eqb\label{reverse loewner eqn}
\partial_t g_t(z) = -\frac{2}{g_t(z) -  W_t}   ,\qquad g_0(z) = z  ,
\eqe
whose solution is a family of conformal maps from $\BB H$ to sub-domains of $\BB H$. Reverse $\op{SLE}_\kappa$ is obtained by taking $W_t$ to be $\sqrt\kappa$ times a Brownian motion. For each time $t$, the time $t$ centered Loewner map of a reverse $\op{SLE}_\kappa$ has the same law as the inverse of the time $t$ centered Loewner map of a forward $\op{SLE}_\kappa$ \cite[Lemma~3.1]{schramm-sle}. 

\emph{Reverse $\op{SLE}_\kappa(\ul \rho)$} with force points $z^1,\ldots,z^n$ is obtained by solving~\eqref{reverse loewner eqn} with the driving function $W$ part of the solution to the system of SDE's 
\eqbn 
dW_t = \sqrt\kappa dB_t +  \sum_{i=1}^{n} \re \frac{\rho^{i}}{W_t - V_t^i} dt, \qquad dV_t^{i } = -\frac{2}{V_t^i - W_t} dt , \qquad W_0 = y \qquad V_0^{i } = z^{i } .
\eqen  
For a general $\ul\rho$ we do not have as simple a relation between forward and reverse $\op{SLE}_\kappa(\ul\rho)$ as we do for ordinary $\op{SLE}_\kappa$. However, there are various forward and reverse symmetries, some of which are discussed in \cite{wedges, shef-zipper}. 

Throughout most of the rest of this paper we will fix $\kappa \in (0,4]$ and we will not always make dependence on $\kappa$ explicit.

\subsection{Gaussian free fields}
\label{gff prelim}

For some of our results, we will make use of couplings of $\op{SLE}_\kappa$ with Gaussian free fields. In this section we give some basic background about the latter object. 

Let $D$ be a domain in $\BB C$ with harmonically non-trivial boundary (i.e.\ a Brownian motion started in $D$ a.s.\ exits $D$ in finite time). We denote by $H(D)$ the Hilbert space completion of the subspace of $C^\infty(\ol D)$ consisting of those smooth, real-valued functions $f$ such that
\[
\int_{D} |\nabla f(z)|^2 \, dz <\infty ,\qquad \int_D f(z) \,  dz = 0
\]
with respect to the Dirichlet inner product
\eqb\label{dirichlet prod}
(f,g)_\nabla = \frac{1}{2\pi} \int_D \nabla f(z)  \cdot \nabla g(z) \,  dz .
\eqe
A \textit{free-boundary Gaussian free field} (GFF) on $D$ is a random distribution (in the sense of Schwartz) on $D$ given by the formal sum
\eqb\label{h}
h = \sum_{j=1}^\infty X_j f_j
\eqe
where $\{f_j\}$ is an orthonormal basis for $H(D)$ and $(X_j)$ is a sequence of i.i.d.\ standard Gaussian random variables. It is not defined as a pointwise function, but for each $g\in H(D)$, the formal inner product
\[
(h, g)_\nabla = \sum_{j=1}^\infty (f,g)_\nabla 
\]
converges almost surely. Moreover, $(h,g)$ is a.s.\ defined for each fixed $g\in L^2(D)$ by the formula
\eqb\label{by parts}
(h , g) = (h , -\Delta^{-1} g)_\nabla
\eqe
where $\Delta^{-1}$ denotes the inverse Laplacian with Neumann boundary conditions. More generally, this formula makes sense if $g$ is any distribution whose inverse Laplacian is in $H(D)$. 

Similarly, one can define a \emph{zero-boundary GFF} on $D$ by replacing $H(D)$ with $H_0(D)$, defined as the Hilbert space completion of the space of smooth compactly supported functions on $D$ in the inner product~\eqref{dirichlet prod}. A zero boundary GFF is defined without the need to make a choice of additive constant. A Gaussian free field with a given choice of boundary data on $\partial D$ is defined to be a zero boundary GFF plus the harmonic extension of the given boundary data to $D$. 

If $V, V^\perp \subset H(D)$ are complementary orthogonal subspaces, then the formula~\eqref{h} implies that $h$ decomposes as the sum of its projections onto $V$ and $V^\perp$. In particular, we can take $V$ to be the closure $H_0(D)$ of $C_c^\infty(D)$ in the inner product~\eqref{dirichlet prod} and $V^\perp$ the set $\op{Harm}_D$ of functions in $H(D)$ which are harmonic in $D$. This allows us to decompose a free boundary GFF as the sum of a zero boundary Gaussian free field and a random harmonic function $\frk h$ on $D$, the latter defined modulo additive constant. We call these distributions the \emph{zero-boundary part} and \emph{harmonic part} of $h$, respectively. 

We refer to \cite{shef-gff} and the introductory sections of \cite{ss-contour} and \cite{qle} for more details on GFF's. 
 
\subsubsection{Reverse SLE/GFF coupling}

The following relation between free boundary GFFs and reverse $\op{SLE}_\kappa(\ul\rho)$ is established in \cite[Section~4.2]{shef-zipper}. Let $(g_t)$ be the \emph{centered} Loewner maps of a reverse $\op{SLE}_\kappa(\ul\rho)$ with force points $z^1,\ldots ,z^n$ as in Section~\ref{sle prelim}. Let $ h$ be a free boundary GFF on $\BB H$, independent from $(g_t)$.  
For $t\geq 0$ let 
\eqbn
h_t =   h \circ g_t + \frac{2}{\sqrt\kappa} \log |g_t(\cdot)| + \frac{1}{2\sqrt\kappa} \sum_{i=1}^n \rho^i G(g_t(z^i), g_t(\cdot) ) ,
\eqen
where 
\[
G(x,y) :=    -\log|x-y| - \log |\ol x- y|
\]
is the Green's function on $\BB H$ with Neumann boundary conditions. Let
\eqb\label{Q def}
Q = \frac{2}{\sqrt\kappa}  +\frac{\sqrt\kappa}{2} .
\eqe
Let $\tau$ be a stopping time for $\eta$ which is a.s.\ less than the first time $t$ that $f_t(z^i) =0$ for some $i$. Then \cite[Theorem~4.5]{shef-zipper} implies that $h_\tau + Q\log |g_\tau'| \eqD h_0$, modulo additive constant.

\begin{figure} 
 \begin{center}
\includegraphics[scale=.8]{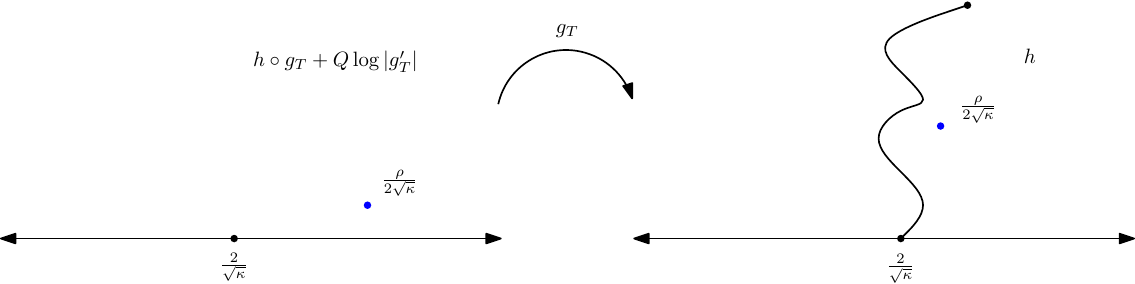} 
\caption{An illustration of the reverse SLE/GFF coupling in the case of a single force point (marked in blue). This is the case we will use in Section~\ref{inverse sec}.}
\end{center}
\end{figure}

There is also an analog of the above coupling for a zero boundary GFF paired with a forward $\op{SLE}_\kappa(\ul\rho)$, which we discuss in Section~\ref{ig prelim}.

\subsubsection{Estimates for the harmonic part}\label{gff lem sec}

In the course of proving our one-point estimate we will need some basic analytic lemmas about the harmonic part of a free boundary GFF which we will prove here.  

\begin{lem}\label{harmonic cov}
Let $\frk h$ be the harmonic part of a free boundary GFF on $\BB D$, normalized so that $\frk h(0) = 0$. Then for any $z,w \in \BB D$, $\frk h(z)$ and $\frk h(w)$ are jointly Gaussian with means zero and covariance
\[
\BB E(\frk h(z) \frk h(w))  =  -2 \log |1-z\ol w| .
\]
\end{lem}
\begin{proof}
For $n\geq 1$, let 
\eqb\label{phi psi}
\phi_n(z) = (2 / n)^{1/2} \re z^n , \qquad \psi_n(z) = (2 / n)^{1/2} \im z^n  .
\eqe
 Then $\{\phi_n , \psi_n : n\geq 1\}  $ is an orthonormal basis for the set of harmonic functions on $\BB D$ in the Dirichlet inner product. So, by definition of the free boundary GFF, we can write
\eqb\label{series formula}
 \sum_{n=1}^\infty X_n \phi_n + \sum_{n=1}^\infty Y_n \psi_n ,
\eqe
where the $X_n$'s and $Y_n$'s are i.i.d.\ $N(0,1)$. From this expression, it follows that $(\frk h(z) , \frk h(w))$ is centered Gaussian for each $z,w\in \BB D$, and one easily computes
\alb
\BB E(\frk h(z) \frk h(w)) 
&= \sum_{n=1}^\infty  \phi_n(z)\phi_n(w) + \sum_{n=1}^\infty \psi_n(z)\psi_n(w)  
= 2\sum_{n=1}^\infty  \frac{(\re z^n)(\re w^n) + (\im z^n) (\im w^n)}{n}   \\
&=  \sum_{n=1}^\infty  \frac{(z\ol w)^n + (w\ol z)^n}{n}   
=  -\log(1 - z\ol w) - \log (1-w\ol z) 
=  -2\log|1 - z\ol w| . \qedhere 
\ale
\end{proof}

We also need the following estimate for circle averages of the GFF.

\begin{lem}\label{gff prob}
Let $h$ be a free boundary GFF on $\BB H$ with additive constant chosen so that its harmonic part vanishes at $a$ for some $a\in \BB H$. Let $A\subset \BB H$ be a deterministic hull lying at positive distance from $a$ and let $g : \BB H\rta \BB H\setminus A$ be the map which takes some marked point of $a$ to 0 and looks like a translation at $\infty$. Let $\wt h = h\circ g$ and let $(\wt h_\ep)$ be the circle average process for $\wt h$ (see \cite[Section~3.1]{shef-kpz} for more on the circle average process). Fix $x\in \BB R$ and $\xi  >1/2$. For any $\delta \geq \ep > 0$,  
\eqb \label{wt h_ep prob}
\BB P\!\left( |\wt h_\ep(x + i\delta )|  > (\log \ep^{-1})^\xi \right) = o_\ep(\ep^p) \quad \forall p > 0 ,
\eqe 
at a rate depending only on $x$, $a$, $\op{diam} A$, $\xi$, and $\delta$, but uniform for $x$ in compact subsets of $\BB R$, $a$ in compact subsets of $\BB H$, and $\delta$ in compact subsets of $[\ep , \infty)$.  
\end{lem}
\begin{proof}
Write $h = h^0 + \frk h$, for $h^0$ a zero boundary GFF and $\frk h$ an independent harmonic function. Let $\frk h_A$ be the projection of $h^0$ onto the set of functions which are harmonic on $\BB H\setminus A$ and let $h^0_A = h^0|_A - \frk h_A$ be the zero-boundary part of $h^0|_A$. Then we can write
\eqb \label{h|H minus A}
h|_{\BB H\setminus A} = h^0_A + \frk h_A + \frk h|_{\BB H\setminus A} , 
\eqe 
with the three summands independent. 
The function $g$ increases imaginary parts, so it follows from Lemma~\ref{harmonic cov} and a coordinate change to $\BB D$ that $\frk h(g(x+i\delta))$ is centered Gaussian with variance $\leq 2\log \delta^{-1} + O_\ep(1)$.

By the Koebe distortion theorem, $|g'(x + i\delta)|$ is at least a constant depending only on $y$ times $\delta |g'(x + iy)|$ for any $y > \delta$. By \cite[Proposition~3.46]{lawler-book} and the Koebe quarter theorem, for large enough $y$ (depending only on $\op{diam} A$), $|g'(x  + iy)|$ is bounded above by a constant depending only on $\op{diam} A$. By another application of the Koebe quarter theorem, we therefore have 
\eqb \label{dist succeq delta}
\op{dist}(g(x+i \delta) , A) \succeq \delta^2  .
\eqe
It follows from \cite[Lemma~6.4]{ig1} that $\frk h_A(g(x+i \delta ))$ is centered Gaussian with variance at most $2 \log \delta^{-1} + O_\ep(1)$.

By conformal invariance, $h_A^0 \circ g$ has the law of a zero boundary GFF on $\BB H$. By~\eqref{dist succeq delta} and \cite[Proposition~3.1]{shef-kpz}, the circle average $(h_A^0 \circ g)_\ep(x+i\delta)$ is Gaussian with mean 0 and variance at most $2\log \ep^{-1} + O_\ep(1)$. By~\eqref{h|H minus A}, 
\[
\wt h_\ep(x+i \delta ) = (h_A^0 \circ g)_\ep(x+i \delta) + \frk h_A(g(x+i \delta))  + \frk h(g(x+i\delta))
\]
is Gaussian with mean 0 and variance at most $6 \log \ep^{-1} + O_\ep(1)$.  
We obtain~\eqref{wt h_ep prob} from the Gaussian tail bound.
\end{proof}

\subsection{Imaginary geometry}
\label{ig prelim}

The proof of the lower bounds in our main theorems will make heavy use of the so-called forward coupling of $\op{SLE}_\kappa$ or $\op{SLE}_\kappa(\ul\rho)$ with the GFF with Dirichlet boundary conditions.  In this coupling, $\op{SLE}_\kappa(\ul\rho)$ for $\kappa \in (0,4)$ can be interpreted as the flow line of the formal vector field $e^{i h / \chi}$ where $h$ is a GFF and 
\begin{equation}
\label{chi lambda}
\chi = \frac{2}{\sqrt{\kappa}} - \frac{\sqrt{\kappa}}{2}.
\end{equation}
For $\kappa > 4$, $\op{SLE}_\kappa(\ul\rho)$ can be interpreted as a ``tree'' or ``light-cone'' of $\op{SLE}_{16/\kappa}$ flow lines \cite{ig1}.  The case $\kappa=4$ is somewhat degenerate (though simpler to analyze) since $\chi \to 0$ as $\kappa \to 4$.  $\op{SLE}_4(\ul\rho)$ has the interpretation of being a level line (rather than a flow line or light cone) of the GFF. See \cite{wang-wu-level-lines} for a detailed study of this case. 

The coupling of $\op{SLE}_4$ with the GFF was actually the first coupling in this family to be discovered \cite{ss-contour} (see also \cite{ss-dgff} which gives the convergence of the contours of the discrete GFF to $\op{SLE}_4$). The existence of the forward coupling in the general setting is established in \cite{dubedat-coupling,ss-contour,shef-slides,ig1}; see \cite[Theorem~1.1]{ig1} for a precise statement. The theory of how different flow lines and light cones of the same GFF interact is developed in \cite{ig1,ig2,ig3,ig4}; these works are also where the term ``imaginary geometry'' is coined.  At this point in time, there are several places which contain short ``crash courses'' on imaginary geometry which are sufficient to understand its usage in this work.  We refer the reader to one of \cite[Section~2.2]{ig2}, \cite[Section~2.3]{ig4}, or \cite[Section~2.2]{miller-wu-dim}; \cite[Section~1]{ig1} and \cite[Section~4]{ig4} contain many of the main theorem statements in addition to more detailed overviews of the related literature.

\subsection{Properties of the multifractal spectrum sets}
\label{multifractal sets}

In this subsection we will prove some elementary deterministic properties of the sets of Section~\ref{multifractal def}, as well as a lemma which is relevant to the integral means spectrum. 
See, e.g.,~\cite[Section 2]{lawler-viklund-tip} for some similar estimates in the setting of the tip multifractal spectrum. 
Our first lemma tells us that the sets of Section~\ref{multifractal def} are only non-empty in the case $s\in [-1,1]$. 

\begin{lem}\label{s>1 empty}
Let $D\subset \BB C$ be a simply connected domain and let $\phi : \BB D\rta  D$ be a conformal map. For each $x\in\partial\BB D$, there is a constant $C > 1$ depending only on $\phi$ and $ \phi(x) $ but uniform for $\phi(x)$ in compact subsets of $\ol D$ such that for each sufficiently small $\ep>0$, 
\eqbn
C^{-1} \ep \leq |\phi'((1-\ep)x)| \leq C \ep^{-1} .
\eqen
\end{lem}
\begin{proof}
 By the Cauchy estimate,
\eqbn
|\phi'((1-\ep)x)| \leq \ep^{-1} \sup_{z\in B_\ep((1-\ep) x)} |\phi(z)| 
\eqen
which gives the upper bound. For the lower bound, we apply the Koebe distortion theorem. 
\end{proof}

Next we prove some lemmas which give that the multifractal spectrum sets are invariant under reasonable modifications of the definitions. 

\begin{lem}\label{two curves}
Let $D\subset \BB C$ be a simply connected domain, $\phi :\BB D\rta D$ a conformal map, and fix $x\in \partial \BB D$. Let $\gamma : [0,1] \rta \ol{\BB D}$ be a simple smooth curve such that $\gamma(0) = x$, $\gamma((0,1]) \subset \BB D$, and $ \gamma'(0)$ is not tangent to $\partial\BB D$ at $x$. Then
\eqb \label{limsup on curve}
\limsup_{\ep \rta 0} \frac{\log |\phi'( (1-\ep)x)|}{-\log \ep } = \limsup_{\ep \rta 0} \frac{\log |\phi'( \gamma(\ep) )|}{-\log \ep } .
\eqe
If one of the limsups is in fact a true limit, then the other is as well.
\end{lem}
\begin{proof}
This is a straightforward application of the Koebe distortion theorem. 
\end{proof}

We next show that the multifractal spectrum depends locally on the domain.

\begin{lem}\label{D D'}
Let $D$ and $D'$ be two simply connected domains in $\BB C$, bounded by curves, which share a common boundary arc $I$. Let $z$ be a prime end lying in the interior of $I$. Then for each $s \in \BB R$, we have $z\in \Theta^s(D)$ if and only if $z \in \Theta^s(D')$. The same holds with $\Theta^{s , \geq}(\cdot)$ or $\Theta^{s;\leq}(\cdot)$ in place of $\Theta^s(\cdot)$. 
\end{lem}
\begin{proof} 
By comparing $D$ and $D'$ to the connected component of $D\cap D'$ with $I$ on its boundary, it suffices to consider the case where $D'\subset D$. Let $\phi : \BB D\rta D$ and $\psi:\BB D\rta D'$ be the corresponding conformal maps. We can factor $\phi = \psi \circ \xi$, where $\xi = \psi^{-1} \circ \phi$. Then
\eqb\label{phi and xi}
\phi'((1-\ep) \phi^{-1}(z ) ) = \psi'(\xi( (1-\ep) \phi^{-1}(z ) )  )  \xi'((1-\ep) \phi^{-1}(z )) 
\eqe
 By Schwarz reflection, $\xi$ extends to be analytic in a neighborhood of $\phi^{-1}(z)$, so $|\xi'((1-\ep) \phi^{-1}(z ))|$ is bounded above and below by positive constants for small $\ep$. Let $\gamma(\ep) = \xi((1-\ep) \phi^{-1}(z )) $. Note that $\gamma$ is a simple curve in $\BB D$ with $\gamma(0) = \psi^{-1}(z)$ and $ \gamma'(0) = -\xi'(\phi^{-1}(z)) \phi^{-1}(z)$.
Since $\xi$ maps a neighborhood of $\phi^{-1}(z)$ in $\partial\BB D$ into $\partial\BB D$, it follows that $\xi'(\phi^{-1}(z))$ is a real multiple of $ \frac{ \xi(\phi^{-1}(z))}{\phi^{-1}(z)} = \frac{ \psi^{-1}(z) }{\phi^{-1}(z)}$. Hence $\gamma'(0)$ is a real multiple of $ \psi^{-1}(z)  $.
In particular $\gamma$ is not tangent to $\partial\BB D$ at $\psi^{-1}(z)$ so the stated result follows from Lemma~\ref{two curves}.  
\end{proof}

We also record the analog of Lemma~\ref{D D'} for the integral means spectrum. 

\begin{lem}
\label{ims compare}
Let $D$ and $D'$ be two bounded Jordan domains in $\BB C$ and suppose there exists a connected boundary arc $I$ shared by $D$ and $D'$.  Let $\phi \colon \BB D \to D$ and $\psi \colon \BB D \to D'$ be conformal maps.  Let $J'$ be a closed subset of the interior of $I$ and let $J$ be a closed subset of the interior of $J'$. For $\ep > 0$, let $A_\ep$ be the set of $z\in \partial B_{1-\ep}(0)$ with $z/|z| \in \phi^{-1}(J)$ and let $A_\ep'$ be the set of $z\in \partial B_{1-\ep}(0)$ with $z/|z| \in \psi^{-1}(J')$. Then 
\eqb \label{ims phi psi}
\limsup_{\ep\rta 0} \frac{\log \int_{A_\ep} |\phi'(z)|^a \, dz}{-\log \ep} \leq  \limsup_{\ep\rta 0} \frac{\log \int_{A_\ep'} |\psi'(z)|^a \, dz}{-\log \ep}  .
\eqe 
\end{lem}
\begin{proof} 
Let $\xi$ be the conformal map from a subdomain of $\BB D$ to a subdomain of $D'\cap D$ which equals $\psi^{-1} \circ \phi$ wherever the latter is defined. By Schwarz reflection $\xi$ extends to a conformal map from a neighborhood of $\phi^{-1}(J')$ to a neighborhood of $\psi^{-1}(J')$. In particular $|\xi'| \asymp 1$ on a neighborhood of $\phi^{-1}(J')$, with implicit constants independent of $\ep$. By a change of variables, for sufficiently small $\ep > 0$,
\eqb \label{ims change vars1}
\int_{A_\ep} |\phi'(z)|^a \,dz \asymp \int_{A_\ep} |\psi'(\xi(z))|^a \,dz \asymp \int_{\xi(A_\ep)} |\psi'(w)| \,dw . 
\eqe 
Let $p_\ep$ be the radial projection from $\BB D$ onto $\partial B_{1-\ep}(0)$. By the above application of Schwarz reflection (and the fact that $J$ is contained in the interior of $J'$), for sufficiently small $\ep > 0$, we have that $p_\ep$ restricts to a diffeomorphism from $\xi(A_\ep)$ to a subset $\wt A_\ep'$ of $A_\ep'$. Furthermore, since $|\xi'| \asymp 1$ on a neighborhood of $\psi^{-1}(J')$, we have $|p_\ep'| \asymp 1$ on $\xi(A_\ep)$ for sufficiently small $\ep$, and by the Koebe distortion theorem $|\psi'(p_\ep(w))| \asymp |\psi'(w)|$ for $w \in \xi(A_\ep)$ and sufficiently small $\ep$. 
Therefore, a second change of variables yields
\eqb \label{ims change vars2}
\int_{\xi(A_\ep)} |\psi'(w)| \,dw \asymp \int_{\wt A_\ep'} |\psi'(z)| \, dz \leq \int_{ A_\ep'} |\psi'(z)| \, dz .
\eqe 
We obtain~\eqref{ims phi psi} by combining~\eqref{ims change vars1} and~\eqref{ims change vars2}. 
\end{proof}

\subsection{Zero-one laws}
\label{zero one sec}

In this section we will prove that the multifractal spectrum and integral means spectrum of an $\op{SLE}_\kappa(\ul\rho)$ curve are a.s.\ deterministic and do not depend on $\ul\rho$ or on which complementary component of the curve we consider. These statements will be used to conclude the proofs of our main results in Section~\ref{haus lower sec} once we show that the desired lower bounds on the quantities we are interested in hold with positive probability for one specific type of SLE. 
  
\begin{prop} \label{theta zero one}
Let $D \subset \BB C$ be a smoothly bounded domain. Let $\kappa >0$ and let $\ul\rho$ be a vector of real weights. Let $\eta$ be a chordal $\op{SLE}_\kappa(\ul\rho)$ process in $D$, with any choice of initial and target points and force points located anywhere in $\ol{  D}$, run up until the first time it either hits an interior force point or hits the continuation threshold after which it is no longer defined (c.f.\ \cite[Section~2.1]{ig1}). Fix $s\in (-1,1)$. Almost surely, the following is true. 
Let $V$ be a connected component of $ D\setminus \eta$ or a connected component of $\  D\setminus \eta([0,t])$ for any $t > 0$ and let $\phi : \BB D\rta V$ be a conformal map.
The Hausdorff dimension of each of the multifractal spectrum sets 
\alb
&\Theta^{s}(V) \setminus \partial D , 
\quad , \Theta^{s;\leq}(V)\setminus \partial D
\quad \Theta^{s;\geq}(V)\setminus \partial D  , \\
&\wt\Theta^{s}(V) \setminus \phi^{-1}(\partial D)
\quad \wt\Theta^{s;\leq}(V)\setminus \phi^{-1}(\partial D)
\quad \op{and} \quad \wt \Theta^{s;\geq}(V) \setminus \phi^{-1}(\partial D)  
\ale
from Section~\ref{multifractal def} is a.s.\ equal to a deterministic constant which depends only on $\kappa$ and $s$. Furthermore, the a.s.\ Hausdorff dimensions of the corresponding sets for $\kappa$ and $16/\kappa$ are equal. 
\end{prop}
\begin{proof}
We will prove the proposition for the sets $\Theta^s(V)$ and $\wt\Theta^s(V)$; the statements for the sets with the $\leq$ or $\geq$ are proven similarly. 
By changing coordinates from $\BB D$ to $\BB H$, it suffices to prove the proposition with $\wt\Theta^s(V)$ and $\Theta^s(V)$ replaced by
\eqb \label{theta with psi}
\wt\Theta_{\BB H}^s(V)  =  
 \left\{x\in \BB R:   \lim_{\ep \rta 0} \frac{\log |\psi' (x + i\ep )|}{-\log \ep }  =s \right\} \quad \op{and}\quad  \Theta_{\BB H}^s(V) = \psi  ( \wt\Theta_{\BB H}^s(V)) 
\eqe 
for $\psi : \BB H \rta V$ a conformal map. This sill be more convenient since we will be working with chordal SLE$_\kappa$. 

First consider the case where $D=\BB H$, $\kappa\leq 4$, and $\eta$ is an ordinary $\op{SLE}_\kappa$ process. In this case, the statement of the proposition for a complementary connected component $V$ of $\BB H\setminus \eta$ follows from the statement for $V = \BB H\setminus \eta^t$ by Lemma~\ref{D D'} and countable stability of Hausdorff dimension, so it suffices to prove the statement with $V = \BB H\setminus \eta^t$ for a general choice of $t > 0$. This will be deduced from the domain Markov property.
 
By scale invariance the law of each $\Theta_{\BB H}^s(\BB H\setminus \eta^t)$ is independent of $t$. Since the derivative of the conformal map $f_{t/2}$ is bounded above and below by positive (random) constants in a neighborhood of each point of $\eta^t \setminus \eta^{t/2}$, we infer that $\Theta_{\BB H}^s(\BB H\setminus \eta^t) \setminus \eta^{t/2} = \Theta_{\BB H}^s(\BB H\setminus  f_{t/2}(\eta^t\setminus \eta^{t/2}))$. 

Since conformal maps preserve Hausdorff dimension of sets in the interior of their domains and by Lemma~\ref{D D'}, we thus have that the Hausdorff dimension of each $\Theta_{\BB H}^s(\BB D\setminus \eta^t)$ is equal to the maximum of $\dim_{\mathcal H} \Theta_{\BB H}^s(\BB H\setminus  \eta^{t/2})$ and $\dim_{\mathcal H} \Theta_{\BB H}^s(\BB H\setminus f_{t/2}(\eta^t\setminus \eta^{t/2}))$. These latter two sets are independent and identically distributed (by the Markov property of SLE) and their Hausdorff dimensions agree in law with that of $\Theta_{\BB H}^s(\BB H\setminus \eta^t)$ (by the scale invariance property noted above). A random variable can be equal to the maximum of two independent random variables with the same law as itself only if it is a.s.\ constant.  

To prove the analogous statement for $\wt \Theta_{\BB H}^s(\BB H\setminus \eta^t)$, we observe that $\dim_{\mathcal H} \wt\Theta_{\BB H}^s(\BB H\setminus \eta^t)$ is the maximum of $\dim_{\mathcal H} f_t^{-1} ( \wt \Theta_{\BB H}^s(\BB H\setminus \eta^t) \cap \eta^{t/2})$ and $  \dim_{\mathcal H} f_t^{-1} ( \Theta_{\BB H}^s(\BB H\setminus \eta^t) \setminus \eta^{t/2})$. By the smoothness of the map $f_{t/2} \circ f_t^{-1}$ on $ f_{t/2}( \BB H\setminus \eta^{t/2})$ and of $f_t^{-1}  $ on $\eta^t\setminus \eta^{t/2}$, respectively, these dimensions equal $\dim_{\mathcal H}   f_{t/2}^{-1}\left(  \wt \Theta_{\BB H}^s(\BB H\setminus \eta^{t/2})\right) $ and $\dim_{\mathcal H} (f_t \circ f_{t/2}^{-1} )^{-1}\left( \wt \Theta_{\BB H}^s(  \BB H\setminus f_{t/2}( \eta^t  \setminus \eta^{t/2}))\right)$, respectively. By the Markov property these latter two quantities are i.i.d., and we conclude as above. 

The case when $\kappa \leq 4$ and $\ul\rho$ and $D$ are arbitrary follows from the above case, Lemma~\ref{D D'}, and the local absolute continuity of the laws of SLE$_\kappa(\ul\rho)$ and SLE$_\kappa$ away from the boundary. 
The case for $\kappa > 4$ follows from the statement for $16/\kappa  <4$ together with Lemma~\ref{D D'} and SLE duality (see, e.g.\ \cite{zhan-duality1,zhan-duality2,dubedat-duality,ig1,ig4}). 
\end{proof}

For the proof of Corollary~\ref{ims cor}, we will also need the analog of Proposition~\ref{theta zero one} for the integral means spectrum. 
 
\begin{prop} \label{ims zero one}
Suppose we are in the setting of Proposition~\ref{theta zero one}. Fix $\av \in\BB R$. Almost surely, the following is true. Let $V$ be a complementary connected component of either $D\setminus \eta$ or of $D\setminus \eta^t$ for any $t > 0$. Then $\op{IMS}^{\op{bulk}}_{V}(\av) $ is equal to a deterministic constant which depends only on $\kappa$ and $\av$. This deterministic constant is the same if we replace $\kappa$ with $16/\kappa$.  
\end{prop}
\begin{proof}
The is proven similarly to Proposition~\ref{theta zero one} but with Lemma~\ref{ims compare} used in place of Lemma~\ref{D D'}. 
\end{proof}

\subsection{SLE stays close to a fixed curve with positive probability}
\label{pos prob sec}

The paper~\cite{miller-wu-dim} proves several estimates which give that SLE$_\kappa$ curves have a positive chance of staying in a small ``tube" around a deterministic curve until getting close to its endpoint. These estimates will be used frequently throughout the paper, so we re-state these estimates here.
  
Suppose $\ul \rho = (\ul \rho^L ; \ul\rho^R) = (\rho^L_l, ... , \rho^L_0 ; \rho^R_0 , ... , \rho^R_r)$ is a vector of $l+r$ weights with $\rho^L_0 , \rho^R_0  > -2$ and let $\eta$ be a chordal SLE$_\kappa(\ul\rho^L ; \ul\rho^R)$ from 0 to $\infty$ in $\BB H$ with force point located at points $x_l^L < \dots < x_0^L = 0^-$ and $0^+ = x_0^R < \dots < x_r^R$.
The following is~\cite[Lemma~2.3]{miller-wu-dim}.  

\begin{lem} \label{miller-wu-dim-2.3}
Let $\ep > 0$ and let $\gamma : [0,T] \rta \ol{\BB H}$ be a deterministic simple curve started from 0 which stays in $\BB H$ after time 0. Let $A_\ep$ be the $\ep$-neighborhood of $\gamma$. Then with positive probability, $\eta$ hits $B_\ep(\gamma(T))$ before exiting $A_\ep$. 
\end{lem}

We will also need the analog of Lemma~\ref{miller-wu-dim-2.3} for curves which hit the boundary, which is~\cite[Lemma~2.5]{miller-wu-dim}.  

\begin{lem} \label{miller-wu-dim-2.5}
Suppose $k\in \{1,...,r-1\}$ with $\ol\rho_k^R: = \sum_{j=1}^k \rho_j^R \in (\kappa/2-4,\kappa/2-2)$, so that $\eta$ can hit $[x_k^R ,x_{k+1}^R]$. 
Let $\gamma$ be a simple curve from 0 to a point in $[x_k^R , x_{k+1}^R]$ which stays in $\BB H$ except at its endpoints. Let $\ep > 0$ and let $A_\ep$ be the $\ep$-neighborhood of $\gamma$. There exists $p = p(\ep , \ul\rho,\kappa , \gamma) > 0$ such that the following is true. Suppose $|x_{k+1}^R - x_k^R| \geq \ep$ and $|x_{k+1}^R  | \leq \ep^{-1}$. Let $A_\ep$ be the $\ep$-neighborhood of $\gamma$. Then with probability at least $p$, $\eta$ hits $[x_k^R , x_{k+1}^R]$ before exiting $A_\ep$. 
\end{lem}

\begin{remark}
Lemma~\ref{miller-wu-dim-2.5} can also be used to control the behavior of an $\op{SLE}_\kappa(\ul\rho)$ curve in a bounded domain for all time, as follows. First we observe that the statement of Lemma~\ref{miller-wu-dim-2.5} is also valid if the interval $[x_k^R  ,x_{k+1}^R]$ is replaced by a single point which is a.s.\ hit by $\eta$, with the same proof as in \cite{miller-wu-dim}. 
Suppose now for concreteness that we have changed coordinates to $\BB D$ in such a way that the start and end points of $\eta$ are $-i$ and $i$, respectively, and the vector of weights $\ul\rho$ is such that $\eta$ a.s.\ does not hit the continuation threshold in finite time (so is defined for all time). 
If we let $f : \BB D\rta \BB H$ be a conformal map taking $-i$ to $0$ and $i$ to 1, then by the main result of \cite{sw-coord}, the law of $f(\eta)$ is a certain $\op{SLE}_\kappa(\ul\rho')$ from $0$ to $\infty$ in $\BB H$, with force points located at 1 and the images of the force points for $\eta$ run until the a.s.\ finite time at which it hits 1. 
By applying Lemma~\ref{miller-wu-dim-2.5} to $f(\eta)$, we infer that for an appropriate choice of $\ul\rho$, $\eta$ has positive probability to stay in the $\ep$-neighborhood of a curve from $-i$ to $i$ in $\BB D$ for all time. 
\end{remark}

\section{One point estimates for the inverse maps}
\label{inverse sec}

In this section we will prove derivative estimates for the inverse centered Loewner maps of a chordal $\op{SLE}_\kappa$ process, which we state just below. 
Let  $\kappa \in(0,4]$. Let $\eta$ be a chordal $\op{SLE}_\kappa$ process from $0$ to $\infty$ in $\BB H$. Let $(f_t)$ be its centered Loewner maps. For $z\in\BB H$ with $\im z = \ep$, $u > 0$, $  s\in (-1,1]$, $c>0$, and $r>0$, let $\ul E^{s;u}(z;t) = \ul E^{s;u}(z; t,c ,  r   )$ be the event that 
\eqb \label{1pt chordal event}
 c^{-1} \ep^{-s + u} \leq |(f_t^{-1})'(z)| \leq c \ep^{-s-u}  \quad \op{and} \quad  \im f_t^{-1}(z) \geq r .
\eqe

\begin{thm} \label{1pt chordal}
Let $z\in \BB H$ with $\im z = \ep \in (0,1) $ and $R^{-1} \leq |\re z| \leq R$ for some $R > 1$. Define the event $\ul E^{s;u}(z;t) = \ul E^{s;u}(z; t,c ,  r )$ as above and define the exponents
\eqb \label{alpha def}
\alpha(s) = \frac{(4 + \kappa)^2 s^2}{8 \kappa (1 + s)} ,\qquad \alpha_0(s) =    \frac{ (4 + \kappa)^2 s  (2 + s)}{8 \kappa (1 + s)^2} . 
\eqe 
Also let $G( f_t  , \mu)$ be the event of Definition~\ref{G def}.   
For each $t , c,r >0$, each $\mu\in\mathcal M$, each $s\in (-1,1]$, and each $R > 1$,  
\eqb \label{alpha(s) asymp}
 \BB P\left(\ul E^{s;u}(z ;t )  \cap G( f_t  , \mu) \right)  \preceq \ep^{\alpha(s )  - \alpha_0(s) u}   .
\eqe
Furthermore, for each $ r  >0$, there exists $    t_* = t_*(r) >0$, such that for each $t\geq t_*$, we can find $\mu = \mu(t,r) \in\mathcal M$ such that for each $c  , u  > 0$, there exists $\ep_0 = \ep_0(t,r,c,u) > 0$ such that for $\ep \in (0,\ep_0]$,  
\eqb \label{alpha(s) asymp'}
 \BB P\left( \ul E^{s;u}(z ;t  )  \cap G( f_t  , \mu) \right)  \succeq  \ep^{\alpha(s ) +\alpha_0(s) u  }    .
\eqe
In both~\eqref{alpha(s) asymp} and~\eqref{alpha(s) asymp'}, the implicit constants in $\preceq$ and $\succeq$ depend on the other parameters but not on $\ep$, and are uniform for $z \in \BB H$ with $R^{-1} \leq |\re z| \leq R$. 
\end{thm} 

\begin{remark}
The reason for the condition $\im f_t^{-1}(z) \geq r$ in the definition of the event $\ul E^{s;u}(z;t)$ is because we are interested in the bulk of the curve, not the behavior near the starting point, so we want to eliminate contributions to $\BB P\left( c^{-1} \ep^{-s + u} \leq |(f_t^{-1})'(z)| \leq c \ep^{-s-u} \right)$ coming from the event that $f_t^{-1}(z)$ is near 0. The purpose of the condition $G(f_t,\mu)$ is as explained in Section~\ref{G prelim H}.
\end{remark}
 
\begin{remark}
Estimates similar to Theorem~\ref{1pt chordal} can be deduced in a somewhat more efficient manner from the results in \cite[Section~3]{schramm-sle} and those of~\cite{bel-smirnov-hm-sle}. In particular, \cite[Lemma~3.3]{schramm-sle} implies the upper bound~\eqref{alpha(s) asymp} for a restricted range of parameter values and an estimate similar to~\eqref{alpha(s) asymp'} can be deduced from \cite[Corollary 3.5]{schramm-sle}. Additionally, a version of Theorem~\ref{1pt chordal} for whole-plane SLE can be obtained using the moment estimates of~\cite{bel-smirnov-hm-sle}. These estimates lead to a.s.\ upper bounds for the integral means spectrum of SLE and for the dimension of the set $\wt\Theta^s(D_\eta) \subset \partial\BB D$ (at least for certain parameter values) via arguments similar to those given in Section~\ref{circle upper sec} and~\ref{ims upper sec}. However, these results do not include the additional regularity conditions on the event in the lower bound of Theorem~\ref{1pt chordal}, so do not lead to proofs of the lower bounds in Theorem~\ref{main thm} and Corollary~\ref{ims cor}. Most of the work in the proof of Theorem~\ref{1pt chordal} comes from obtaining a lower bound with these regularity conditions.  
\end{remark}

The proof of Theorem~\ref{1pt chordal} proceeds by way of a martingale re-weighting argument. The upper bound~\eqref{alpha(s) asymp}, explained in Section~\ref{reverse upper sec}, is straightforward, but the lower bound is more involved. For this one has to show that the event $\ul E^{s;u}(z ;t  )  \cap G( f_t  , \mu)$ holds with uniformly positive probability under the law when we re-weight by our martingale. It is shown in Section~\ref{gff deriv control} that the main derivative condition in~\eqref{1pt chordal event} holds with high probability under this weighted law using a coupling with the GFF and a coordinate change trick reminiscent of arguments in~\cite[Section 8]{qle} (we expect that this can also be proven via a longer argument which does not involve the GFF, but we do not carry out such an argument here). To check that the auxiliary conditions hold with uniformly positive re-weighted probability, we use a rather involved stochastic calculus argument which is mostly given in Appendix~\ref{auxiliary sec}.

\subsection{Reverse SLE martingales and upper bound} 
\label{reverse upper sec}

Let $(g_t)$ be the centered Loewner maps of a reverse $\op{SLE}_\kappa$ flow, so
\eqb \label{reverse Loewner SDE} 
d g_t(z) = -\frac{2}{g_t(z)} \, dt -   dW_t ,\qquad g_0(z) = z
\eqe 
for $W_t = \sqrt\kappa B_t$ and $(B_t)$ a standard linear Brownian motion. Our interest in $(g_t)$ stems from the fact that if $(f_t)$ is as in Theorem~\ref{1pt chordal}, then $g_t \eqD f_t^{-1}$ for each $t$ (see, e.g. \cite[ Lemma~3.1]{schramm-sle}). 

Let $K_t = \BB H\setminus g_t(\BB H)$ be the hulls corresponding to $(g_t)$. Since $f_t^{-1} \eqD g_t$ for each $t$, it is only a minor abuse of notation to replace $f_t^{-1}$ with $g_t$ in the definition of the events of Theorem~\ref{1pt chordal}, and we do so in the remainder of this section. 

\subsubsection{Reverse SLE martingales} \label{reverse sle sec}

We state here a result originally due to Lawler~\cite[Proposition~2.1]{lawler-reverse-multifractal}, but in a form which is more convenient for our purposes. 
 
\begin{lem}\label{reverse mart}
Let $\kappa > 0$. Let $(g_t)$ be as above, $\rho\in \BB R$, $z\in \BB H$, and
\eqb \label{M def}
M_t^z =  |g_t'(z )|^{\tfrac{ (8 +2 \kappa - \rho )\rho}{8\kappa }}  (\im g_t(z ))^{-\tfrac{\rho^2}{8\kappa}} |g_t(z) |^{\rho/\kappa} .
\eqe  
Then $M_t^z$ is a martingale. Let $\BB P_*^z$ be the law of $(g_t)$ weighted by $M^z$. The law of $(g_t)$ under $\BB P_*^z$ is that of the centered Loewner maps of a reverse $\op{SLE}_\kappa(\rho)$ with a force point at $z$. That is, under the reweighted law,
\eqb \label{dW_t}
dW_t =  -   \re \frac{\rho}{g_t(z)} \, dt  + \sqrt\kappa dB_t^z 
\eqe 
for $B_t^z$ a $\BB P_*^z$-Brownian motion. 
\end{lem}

\begin{remark}
The martingale~\eqref{M def} is the reverse SLE analog of the local martingale of \cite[Section~5]{sw-coord} in the case of a single force point.
\end{remark}

 \subsubsection{Proof of the upper bound}
 
In this subsection we will prove~\eqref{alpha(s) asymp} of Theorem~\ref{1pt chordal}. We will actually prove something a little stronger which is needed to get an upper bound for the dimensions of the sets $\Theta^{s;\leq }(D_\eta)$ and $\Theta^{s;\geq}(D_\eta)$ from Section~\ref{multifractal def}.
 
\begin{prop} \label{upper bound infty}
Let $\alpha(s)$ be as in~\eqref{alpha def} and let $(g_t)$ be the centered Loewner maps of a reverse $\op{SLE}_\kappa$ as above. Fix $c , d> 0$. For $s\in [0,1]$, a time $t>0$, and $z\in \BB H$ with $\im z = \ep \in (0,1)$, let 
\eqbn
\ul E^{s;\infty}(z;t)
=  \ul E^{s;\infty}(z;t,c, d )
:= \begin{dcases}
&\left\{  |g_t'(z)| \geq c^{-1} \ep^{-s }  ,\, |g_t(z)| \geq d^{-1} \right\} ,\quad \text{if} \: s\in [0,1] \\
&\left\{  |g_t'(z)| \leq c  \ep^{-s }  ,\, |g_t(z)| \leq d^{-1} \right\} ,\quad \text{if} \: s\in (-1,0) .
\end{dcases}
\eqen
For any bounded stopping time $\tau$ for $(g_t)$, 
\eqb \label{upper bound infty eqn}
\BB P\left(  \ul E^{s;\infty}(z;\tau )  \right) \preceq \ep^{\alpha(s )   } .
\eqe 
For any $R >1$, the implicit constant in~\eqref{upper bound infty eqn} is uniform for $z\in \BB H$ with $R^{-1} \leq |\re z| \leq R$.  
\end{prop}
The estimate~\eqref{alpha(s) asymp} is immediate from Proposition~\ref{upper bound infty} in the case $s\in [0,1]$. To extract~\eqref{alpha(s) asymp} from Proposition~\ref{upper bound infty} in the case $s \in (-1,0)$, we observe that Lemma~\ref{G implies U} implies that $\op{diam} K_t$ is bounded by a constant depending only on $t $ and $\mu$ on the event $G(g_t^{-1} , \mu)$ (c.f. the discussion following Definition~\ref{G def}). For $R^{-1} \leq |\re z|\leq R$, \cite[eqn. 3.14]{lawler-book} then implies that $|g_t(z)|$ is bounded by a constant depending only on $t, \mu $, and $ R$ on $\ul E^{s;u}(z;t) \cap G(g_t^{-1} , \mu)$. Thus $\ul E^{s;u}(z;t) \cap G(g_t^{-1} , \mu) \subset \ul E^{s + u ; \infty }(z;t,c, d)$ for a suitable choice of $d$. 

\begin{proof}[Proof of Proposition~\ref{upper bound infty}]
This is a standard martingale re-weighting argument. 
Throughout, we fix $R >1$ and require all implicit constants to be uniform for $z\in \BB H$ with $R^{-1} \leq |\re z| \leq R$.
Let
\eqb \label{optimal rho}
\rho = \rho(s) := \frac{(4 + \kappa) s}{1 + s }.
\eqe
and denote by $\BB P_*^z$ the law of $(g_t)$ re-weighted by the martingale of Lemma~\ref{reverse mart} with this choice of $\rho$.  
By the Loewner equation, $\im g_\tau(z)$ is bounded above by a constant depending only on the essential supremum of $\tau$. Therefore,
\eqb \label{M_tau succeq}
M_{\tau }^z \BB 1_{\ul E^{s;\infty}(z;\tau) }  \succeq \ep^{\tfrac{-s(8 +2 \kappa - \rho )\rho}{8\kappa} } \BB 1_{\ul E^{s;\infty}(z;\tau)}   
\eqe 
(we can replace the $\succeq$ with an $\asymp$ if we assume that $\im g_t(z)$ is bounded below and $|g_t(z)|$ is bounded above). 
Furthermore, if $R^{-1} \leq |\re z| \leq R$ then
\eqb \label{M_0^z}
M_0^z \asymp \ep^{-\tfrac{\rho^2}{8\kappa}} .
\eqe
Thus the optional stopping theorem implies
\eqbn
\ep^{\tfrac{-s(8 +2 \kappa - \rho )\rho}{8\kappa} } \BB P(\ul E^{s;\infty}(z;\tau) )\asymp  \BB E\left(M_{\tau  }^z \BB 1_{\ul E^{s;\infty}(z;\tau)} \right)\preceq \ep^{-\rho^2/8\kappa} \BB P_*^z(\ul E^{s;\infty}(z;\tau) ) .  
\eqen
Therefore 
\eqb \label{M proportionality}
\BB P (\ul E^{s;\infty}(z;\tau) )   \preceq \ep^{ \tfrac{ s (8 +2 \kappa - \rho )\rho}{8\kappa } - \tfrac{\rho^2}{8\kappa}} \BB P_*^z(\ul E^{s;\infty}(z;\tau) )  .
\eqe 
The value of the exponent on the right is maximized by taking $\rho = \rho(s)$, as in~\eqref{optimal rho}. 
Choosing this value of $\rho$ yields the upper bound~\eqref{upper bound infty eqn}.
\end{proof}

\subsection{Reduction of the lower bound to a result for a stopping time} 
\label{reverse lower sec}

Now we turn our attention to the lower bound~\eqref{alpha(s) asymp'} in Theorem~\ref{1pt chordal}. We continue to assume that we have replaced $f_t^{-1}$ with $g_t$ in the definition of the events of Theorem~\ref{1pt chordal}, as in Section~\ref{reverse upper sec}. 

Let $T_r^z$ be the first time $t$ that $ \im g_t(z)  \geq r$ and fix a time $ \ol t > 0$. Put
\eqb \label{tau def}
\tau = \tau_r^z := T_r^z \wedge   \ol t,
\eqe 
so that up to an event of probability zero, 
\[
\{\tau < \ol t\} = \{\im g_\tau(z) \geq r\} = \{\im g_{\ol t}(z) \geq r\} .
\]
We claim that to prove that~\eqref{alpha(s) asymp'} holds with $\ol t$ in place of $t$, and hence to finish the proof of Theorem~\ref{1pt chordal}, it is enough to prove the following statement.

\begin{prop} \label{P pos}
Let $\rho = \rho(s)$ be as in~\eqref{optimal rho}. Let $\BB P_*^z$ be the law of a reverse $\op{SLE}_\kappa(\rho)$ process $(g_t)$ with hulls $(K_t)$, with an interior force point located at $z \in \BB H$ with $\im z = \ep$. Let $\tau=\tau_r^z$ be as in~\eqref{tau def}. Define the events $\ul E^{s;u}(z;\tau )$ as in~\eqref{1pt chordal event}, but with $(g_t)$ in place of $(f_t)$ and the time $\tau $ hull $K_{\tau}$ for $(g_t)$ in place of $\eta^{\tau}$. For each $R>1$ there exists $r_* > 0$ such that for each $r \geq r_*$, we can find $\mu\in\mathcal M$ and $t_* > 0$ such that for each $u>0$ there exists $\ep_0 > 0$ such that for each $z\in \BB H$ with $\im z = \ep \leq \ep_0$ and $R^{-1} \leq |\re z| \leq R$ and each $\ol t \geq t_*$,  
\eqb  
\BB P_*^z \left(  \ul E^{s;u}(z;\tau) \cap G(g_{\tau}^{-1} , \mu)   \right)\succeq 1  .
\eqe
Here the implicit constant is independent of $\ep$ and uniform for $z$ with $R^{-1} \leq |\re z| \leq R$ (but may depend on $r$, $R$, $\mu$, $\ol t$, $u$, and $s$). 
\end{prop}
 
We will prove Proposition~\ref{P pos} in the subsequent subsections. In the remainder of this subsection we deduce Theorem~\ref{1pt chordal} from Proposition~\ref{P pos}. To lighten notation, in what follows we write $\tau = \tau_r^z$. 

First we note that the probability of the event of Theorem~\ref{1pt chordal} is decreasing in $r$, so it suffices to prove~\eqref{alpha(s) asymp'} for  $r\geq r_*$, with $r_*$ as in Proposition~\ref{P pos}. 
Observe that $|g_{\tau}(z)|$ is a.s.\ bounded above by a positive constant on the event $\ul E^{s;u}(z;\tau) \cap \mathcal G(g_{\tau}^{-1} , \mu) $ (c.f. Section~\ref{reverse upper sec}). By combining this with the definition of $ \ul E^{s;u}(z;\tau) $ we see that
\eqbn
M_{\tau  }^z \BB 1_{\ul E^{s;u}(z;\tau) \cap \mathcal G(g_\tau^{-1} , \mu) }  \preceq \ep^{\tfrac{-( s+u) (8 +2 \kappa - \rho )\rho}{8\kappa} } \BB 1_{\ul E^{s;u}(z;\tau) \cap   G(g_\tau^{-1} , \mu)} .
\eqen  
By~\eqref{M_0^z} and our choice~\eqref{optimal rho} of $\rho$,
\eqb \label{P and P*}
\ep^{\alpha(s) + \alpha_0(s) u}  \BB P_*^z\left(\ul E^{s;u}(z;\tau) \cap   G(g_\tau^{-1} , \mu)\right) \preceq   \BB P\left(\ul E^{s;u}(z;\tau) \cap G(g_\tau^{-1} , \mu) \right) .
\eqe  
 
Assuming that Proposition~\ref{P pos} holds,~\eqref{P and P*} implies~\eqref{alpha(s) asymp'} with $\tau$ in place of $t$. To get the desired bound at the deterministic time $ \ol t$, for $t \geq \tau$ let $g_{\tau, t}$ be the conformal map defined on $\BB H$ which satisfies $g_{\tau , t} \circ g_\tau = g_t$. 
By the strong Markov property the conditional law given $\{g_t \,:\, t\leq \tau\}$ of the family of conformal maps $\{g_{\tau , v + \tau} \,:\, v \geq 0\}$ is the same as the law of the $\{g_v \,:\, v \geq 0\}$. For $w\in \BB C$, $\mu' \in\mathcal M$ and $C>1$, let $F  = F_{\tau,\ol t}(w; C , \mu')$ be the event that the following is true.
\begin{enumerate}  
\item $C^{-1} \leq |g_{\tau,t}'(w)| \leq C$ for each $t \in [\tau ,\ol t]$. 
\item $G(g_{\tau, \ol t}^{-1} , \mu' )$ occurs.
\end{enumerate}
If $C$ is chosen sufficiently large and $\mu' \in\mathcal M$ is chosen sufficiently small, depending on $ \ol t  $ but uniform for $w$ in compact subsets of $\BB H$, then $\BB P(F )$ is at least a positive constant depending uniformly on $w$ in compact subsets of $\BB H$. Furthermore, since we have a bound on $\op{diam} K_\tau$ on the event $\ul E^{s;u}(z;\tau) \cap G(g_\tau^{-1} , \mu) $ (see Lemma~\ref{G implies U}), it follows from the Markov property that 
\eqbn 
\BB P\left( F \cap \ul E^{s;u}(z;\tau) \cap G(g_\tau^{-1} , \mu)   \right) \succeq \BB P\left(\ul E^{s;u}(z;\tau) \cap G(g_\tau^{-1} , \mu)  \right) .
\eqen
On the other hand, the definition of $F$ implies that 
\eqbn
F \cap  \ul E^{s;u}(z;\tau) \cap G(g_\tau^{-1} , \mu)    \subset \ul E^{s;u}(z;  \ol t , c' , r   )\cap G(g_{\ol t}^{-1}  , \mu \circ \mu') 
\eqen
for some $  c'  >0 $ depending on the other parameters (here we use that $\im g_t(z)$ is increasing in $t$ for the condition involving $r$). By making  $c$ sufficiently small, we can make $c'$ as small as we like. We conclude that~\eqref{alpha(s) asymp'} with $\tau$ in place of implies~\eqref{alpha(s) asymp'} with $\ol t$ in place of $t$. 

Thus to prove Theorem~\ref{1pt chordal} it remains to prove Proposition~\ref{P pos}. The proof is separated into two major steps: first we prove that the derivative condition in the definition of $\ul E^{s;u}(z)$ holds at time $\tau$ with $\BB P_*^z$-probability tending to 1 as $\ep = \im z \rta 0$. This is done in Section~\ref{gff deriv sec} via a coupling with a Gaussian free field. Then we prove that $\BB P_*^z\left(  \{\tau  < \ol t\} \cap  G(g_\tau^{-1} , \mu)  \right)$ is uniformly positive for sufficiently small $\mu$ and sufficiently large $\ol t$. This is done in Appendix~\ref{auxiliary sec} via a stochastic calculus argument. 

 \subsection{Derivative estimate via reverse SLE/GFF coupling}
\label{gff deriv sec}

Assume we are in the setting of Proposition~\ref{P pos}. In this subsection we will prove that $| g_\tau'(z)|  \approx \ep^{-s}$ with high probability under $\BB P_*^z$. 
Throughout this subsection, we fix $R>1$, $c>0$, $r>0$, $\mu \in \mcl M$, $\ol t > 0$, and $z\in \BB H$ with $\im z = \ep$ and require all implicit constants to be independent of $\ep$ and uniform for $R^{-1} \leq |\re z | \leq R$ and all $o_\ep(1)$ errors to be uniform for $R^{-1} \leq |\re z | \leq R$. These quantities are, however, allowed to depend on $R$, $c$, $r$, $\mu$, $\ol t$, $s$, and $u$. 

\begin{prop}
\label{gff deriv control}
In the setting of Proposition~\ref{P pos}, 
\begin{align} \label{P* deriv}
 \BB P_*^z\left( \{ |g_\tau'(z)| \notin [c^{-1 } \ep^{-s+u} ,  c \ep^{-s -u}] \} \cap G(g_\tau^{-1} ,\mu)  \cap \{ \tau <  \ol t \}  \right)   = o_\ep(1) .
\end{align}
\end{prop}

We prove Proposition~\ref{gff deriv control} using a coupling with a Gaussian free field (we expect that one could also do this without using the GFF---perhaps via a longer argument).

Let $ h$ be a free boundary GFF on $\BB H$, independent from $(g_t)$, normalized so that its harmonic part $\frk h$ vanishes at $i y$ for some $y > 0$ (which we will specify below in such a way that it depends on $ \ol t$, but not $\ep$). Let $\BB P_h$ be the law of $h$. 
For $t\geq 0$ let 
\eqb \label{h_t def}
h_t =   h \circ g_t + \frac{2}{\sqrt\kappa} \log |g_t(\cdot)| + \frac{\rho}{2\sqrt\kappa} G(g_t(z), g_t(\cdot) ) ,
\eqe 
where 
\[
G(x,y) :=    -\log|x-y| - \log |\ol x- y|
\]
is the Green's function on $\BB H$ with Neumann boundary conditions. 

Let $\tau$ be as in~\eqref{tau def}. By \cite[Theorem~2.5]{shef-zipper}, $h_\tau + Q\log |g_\tau'| \eqD h_0$, modulo additive constant, where $Q = \frac{2}{\sqrt\kappa}  +\frac{\sqrt\kappa}{2} $
is as in~\eqref{Q def}. Let $b_\tau$ be this additive constant, so 
\eqb\label{eqD with constant}
h_\tau + Q\log |g_\tau'|  - b_\tau \eqD h_0 . 
\eqe 
The idea of the proof of~\eqref{P pos} is to estimate the terms other than $\log |g_\tau'|$ in~\eqref{eqD with constant}, and thereby obtain an estimate for $|g_\tau'|$. See the proof of \cite[Theorem~8.1]{qle} for another argument using a similar idea.

Let 
\eqb \label{h_0' def}
\wt h_0   = h_\tau + Q\log |g_\tau'|  - b_\tau 
\eqe 
so that by~\eqref{eqD with constant}, $\wt h_0  \eqD h_0$. Rearranging the definition of $\wt h_0 $ gives
\begin{align} \label{gff diff}
&Q\log |g_\tau'(w)|  =\wt h_0  - h_\tau  + b_\tau \nonumber \\
&\quad = \wt h  - h \circ g_\tau + \frac{2}{\sqrt\kappa}  \log \frac{|w|}{|g_\tau(w)|} +\frac{\rho}{2\sqrt\kappa} \left(  \log \frac{  |g_\tau(w) - g_\tau(z)|}{ |w-  z| }  + \log \frac{ |g_\tau(w) - \ol{g_\tau(z)}|}{ |w-\ol z|  } \right) + b_\tau ,
\end{align}
where here $\wt h $ is a field with the same law as $h$ and we use $w$ instead of $\cdot$ as a dummy variable. 
Since all of the non-GFF terms in~\eqref{gff diff} are harmonic away from $z$, the equation still holds for $w\not=z$ if we replace $ \wt h $ and  $ h \circ g_\tau$ with the circle average processes $\wt h_\ep$ and $(  h\circ g_\tau )_\ep$ for these two fields. We will use~\eqref{gff diff} to estimate $b_\tau$ and then to estimate $|g_\tau'(z)|$.

\begin{lem} \label{b control}
Let $\xi >1/2$. If $y$ is chosen sufficiently large (independently of $\ep$ and uniform for $R^{-1} \leq |\re z| \leq R$) then
\eqb \label{b_tau to 0}
(\BB P_*^z \otimes \BB P_h)\left( \{ |b_\tau|  >  (\log \ep^{-1} )^\xi  \} \cap   G(g_\tau^{-1} , \mu)   \cap \{ \tau < \ol t \} \right) = o_\ep(1) .
\eqe  
\end{lem}
\begin{proof}
If we replace the GFF terms with circle averages in~\eqref{gff diff} and evaluate at $w = i y$, we get
\begin{align} \label{gff diff at iy}
 Q\log |g_\tau'(i y)|  &=\wt h_\ep(i y) - (h \circ g_\tau)_\ep( i y) + \frac{2}{\sqrt\kappa}  \log \frac{y}{|g_\tau(i y)|}\nonumber\\
 & +\frac{\rho}{2\sqrt\kappa} \left(  \log \frac{  |g_\tau(i y) - g_\tau(z )|}{ |i y-  z| }  + \log \frac{ |g_\tau(i y) - \ol{g_\tau(z)}|}{ |i y -\ol z|  } \right) + b_\tau .
\end{align}
By Lemma~\ref{G implies U} $\op{diam} K_\tau \preceq 1$ on $G(g_\tau^{-1} , \mu)$. By \cite[Proposition~3.46]{lawler-book}, $\im g_\tau(i y) \asymp |g_\tau(i y)| \asymp 1$ on $ G(g_\tau^{-1} , \mu)$. By the Koebe quarter theorem we also have $|g_\tau'(i y)| \asymp 1$ on $G(g_\tau^{-1} , \mu)$ provided $y$ is chosen sufficiently large, depending only on $\mu$, $\ol t$, and $R$. Hence each of the terms in~\eqref{gff diff at iy} except for $b_\tau$ and the two circle averages is $\asymp 1$ on $ G(g_\tau^{-1} , \mu) \cap \{\tau < \ol t\}$ (implicit constants also depending on $y$) if $y$ is chosen sufficiently large, depending only on $\mu$, $\ol t$, and $R$. By Lemma~\ref{gff prob}, for $\xi > 1/2$,
\[
(\BB P_*^z \otimes \BB P_h)\left(|\wt h_\ep(i y) - (h \circ g_\tau)_\ep( i y)| > (\log \ep)^\xi  \right)  = o_\ep(1). 
\]
Note that we took $A = \emptyset$ in that lemma to estimate $\wt h_\ep(iy)$ and we took $A = K_\tau$ and used that $ K_\tau$ is independent of $h$ to estimate $(h\circ g_\tau)_\ep(iy )$. 
By re-arranging~\eqref{gff diff at iy} we conclude. 
\end{proof}

\begin{proof}[Proof of Proposition~\ref{gff deriv control}]
Since the circle average process is continuous \cite[Proposition~3.1]{shef-kpz}, we can take the limit as $w\rta z$ in~\eqref{gff diff} to get
\begin{align}  \label{gff limit}
Q\log |g_\tau'(z)| &=  \wt h_\ep(z) - (  h\circ g_\tau)_\ep(z) + \frac{\rho}{2\sqrt\kappa} \log|g_\tau'(z)| -\frac{\rho}{2\sqrt\kappa} \log \ep \nonumber \\
&+ \frac{2}{\sqrt\kappa}  \log \frac{|z|}{|g_\tau(z)|}  + \frac{\rho}{2\sqrt\kappa} \log |  \im g_\tau(z) |  + b_\tau .
\end{align}

Since we have a uniform upper bound on $\op{diam} K_\tau$ on the event $ G(g_\tau^{-1} , \mu) $ and $\im g_\tau(z) = r$ on the event $\{\tau  < \ol t\}$, the absolute value of the sum of the fifth and sixth terms in the right in~\eqref{gff limit} is $\preceq 1$ on $G(g_\tau^{-1} , \mu) \cap \{\tau <  \ol t\}$.

By Lemma~\ref{gff prob} (applied as in the proof of Lemma~\ref{b control}), for any $\xi > 1/2$, 
\eqbn
(\BB P_*^z \otimes \BB P_h )\left( | \wt h_\ep(z) - (  h\circ g_\tau)_\ep(z)|   \geq (\log \ep^{-1})^\xi  \right) =  o_\ep(1).
\eqen 

By Lemma~\ref{b control}, the probability that the last term in~\eqref{gff limit} is $\geq (\log \ep)^{1/2}$ and $G(g_\tau^{-1} , \mu) \cap \{\tau <  \ol t\}$ occurs is of order $o_\ep(1)$. Hence except on an event of $\BB P_*^z \otimes \BB P_h $-probability of order $o_\ep(1)$, on the event $G(g_\tau^{-1} , \mu) \cap \{\tau<  \ol t\}$ it holds that
\eqbn
Q\log |g_\tau'(z)|  = \frac{\rho}{2\sqrt\kappa} \log|g_\tau'(z)| +\frac{\rho}{2\sqrt\kappa} \log  \ep^{-1} +   o_\ep( \log  \ep^{-1})    .
\eqen
Rearranging, we get that except on an event of $\BB P_*^z \otimes \BB P_h $-probability of order $o_\ep(1)$, on the event $G(g_\tau^{-1} , \mu) \cap \{\tau <  \ol t\}$,
\begin{align} \label{log g_tau}
  \log |g_\tau'(z)| =   \frac{\rho }{\kappa + 4  - \rho } \log \ep^{-1}  +   o_\ep( \log  \ep^{-1})  .
\end{align}
With $\rho$ as in~\eqref{optimal rho},
\[
\frac{\rho }{\kappa + 4  - \rho } = s,
\]
so integrating out $\BB P_h$ yields~\eqref{P* deriv}. 
\end{proof}

\subsection{Proof of Proposition~\ref{P pos}}
\label{1pt chordal proof}
 
In light of Proposition~\ref{gff deriv control}, to prove Proposition~\ref{P pos}, and hence Theorem~\ref{1pt chordal}, it remains to prove that $\BB P_*^z\left(  G(g_\tau^{-1} ,\mu)  \cap \{ \tau <  \ol t \}  \right)$ is uniformly positive. In particular, we will prove the following.
  
\begin{prop}   \label{d control}
Let $(g_t)$ be as in~\eqref{reverse Loewner SDE}. and let $(K_t)$ be the associated hulls. Let $z\in \BB H$. For $r > \im z$ let $ T_r^z$ be the first time $t$ that $\im g_t(z) = r$. 
Let $\rho \in (-\infty , \kappa/2+2)$ and let $\BB P_*^z$ be the law of $(g_t)$ weighted by $M^z$, as in Lemma~\ref{reverse mart}. For any given $R > 1$, there exists $r_* > 0$ such that for each $r \geq r_*$, we can find $\mu\in\mathcal M$, $t_* >0$, $\ep_0 > 0$, and $  p >0$ such that for $z\in \BB H$ with $  |\re z |\leq R$ and $\im z \leq \ep_0$,  
\eqb \label{d control eqn'}
\BB P_*^z\left(   \{ T_r^z < t_* \}  \cap G(g_{T_r^z}^{-1} , \mu)  \right) \geq p  .
\eqe
\end{prop}

The proof of Proposition~\ref{d control} is given in Appendix~\ref{auxiliary sec}. In the remainder of this section, we use Proposition~\ref{d control} to conclude the proof of Proposition~\ref{P pos}, and hence (recall Section~\ref{reverse lower sec}) the proof of Theorem~\ref{1pt chordal}. 

\begin{proof}[Proof of Proposition~\ref{P pos}]
Fix $R > 1$ and $c>0$. Let $r_* > 0$ be as in Proposition~\ref{d control} for this choice of $R$. Given $r \geq r_*$, let $\mu\in\mathcal M$, $\ol t >0$, $\ep_0 > 0$, and $  p >0$ be as in Proposition~\ref{d control}, so that~\eqref{d control eqn'} holds. Given $\ol t \geq t_*$, let $\tau$ be as in~\eqref{tau def}. By Proposition~\ref{gff deriv control}, we can find $\ep_0' \in (0,\ep_0]$ (depending on $c, R, \ol t, r, \mu, s,$ and $u$) such that whenever $z\in \BB H$ with $R^{-1} \leq |\re z| \leq R$ and $\im z = \ep \in (0, \ep_0']$,  
\eqbn
 \BB P_*^z\left( \{ |g_\tau'(z)| \notin [c^{-1 } \ep^{-s+u} ,  c \ep^{-s -u}] \} \cap G(g_\tau^{-1} ,\mu)  \cap \{ \tau <  \ol t \}  \right) \leq p/2 .
\eqen
If $T_r^z < t_* \leq \ol t$, then $\tau < \ol t$ and $\im g_\tau(z) \geq r$. 
By~\eqref{d control eqn'}, it follows that for such a choice of $z$,
\eqbn
\BB P_*^z\left(      \ul E^{s;u}(z;\tau) \cap G(g_{\tau}^{-1} , \mu)    \right) \geq p/2  . \qedhere
\eqen
\end{proof}

\subsection{Estimates for chordal SLE in the disk}
\label{disk sec}

In the sequel we will work mostly in the unit disk $\BB D$ rather than in the upper half plane $\BB H$. In this brief subsection we make some trivial remarks about how Theorem~\ref{1pt chordal} generalizes to this setting. 

Suppose $\eta$ is a chordal $\op{SLE}_\kappa$ from $-i$ to $i$ in $\BB D$. Let $\psi : \BB D\rta \BB H$ be the conformal map taking $-i$ to $0$, $i$ to $\infty$, and having positive real derivative at 0. Suppose $\eta$ is parameterized in such a way that $\psi(\eta)$ is parameterized by half-plane capacity. For each time $t \geq 0$, let 
\eqbn
f_t : \BB D\setminus \eta^t \rta \BB D
\eqen
be defined so that $\psi \circ  f_t \circ \psi^{-1}$ is the time $t$ centered forward Loewner map for $\psi(\eta)$. 

For $s\in(-1,1)$, $u>0$, $z\in \BB D$ with $1-|z|=\ep$ and $t,c,d > 0$, let $\ul E_{\BB D}^{s;u}(z;t) =  \ul E_{\BB D}^{s;u}(z ; t    , c , d ) $ be the event that  
\eqbn
\ep^{-s+u} \leq |(f_t^{-1})'(z)| \leq \ep^{-s-u}  \quad \op{and} \quad f_t^{-1}(z) \in B_d(0) .
\eqen

Then in this context Theorem~\ref{1pt chordal} reads as follows. 

\begin{cor}[Theorem~\ref{1pt chordal} for the disk] \label{1pt chordal disk}
Suppose we are in the setting described just above. Let $\delta > 0$ and let $z\in \BB D$ with $|z-i| , |z + i| \geq \delta$ and $1-|z| = \ep$. Define the events $\mathcal G(\cdot)$ as in Definition~\ref{G infty def}. 
For each $t , c, d  , \delta >0$, each $s\in (-1,1]$, and each $\mu\in\mathcal M$,  
\eqb \label{disk asymp}
\BB P\left(\ul E_{\BB D}^{s;u}(z;t) \cap \mathcal G( f_t  , \mu) \right)  \preceq \ep^{\alpha(s )  - \alpha_0(s) u}   .
\eqe 
Furthermore, there exists $ t_* >0$ such that for each $t\geq t_*$, we can find $\mu \in\mathcal M$ and $d \in (0,1)$ such that for each $c >0$ and each $u  > 0$, there exists $\ep_0 > 0$ such that for $\ep \in (0,\ep_0]$,
\eqb \label{disk asymp'}
 \BB P\left( \ul E_{\BB D}^{s;u}(z ;t  ) \cap \mathcal G( f_t  , \mu) \right)  \succeq  \ep^{\alpha(s ) +\alpha_0(s) u  }    .
\eqe 
In both~\eqref{disk asymp} and~\eqref{disk asymp'}, the implicit constants in $\preceq$ and $\succeq$ depend on the other parameters but not on $\ep$, and are uniform for $z\in \BB D$ with $|z-i| , |z + i| \geq \delta$. 
\end{cor}
\begin{proof}
This is immediate from Theorem~\ref{1pt chordal} and a coordinate change. Note that we use Lemma~\ref{G implies U} to obtain a $d \in (0,1) $, depending on $\mu$, such that~\eqref{disk asymp'} holds.
\end{proof}

\section{One point estimates for the forward maps}
\label{time infty sec}

\subsection{Statement of the estimates}
\label{time infty setup sec}

In this section we transfer the estimates of Theorem~\ref{1pt chordal} to estimates for certain ``time infinity" forward Loewner maps, which we will define shortly. We work in the setting of $\BB D$, rather than $\BB H$, as this setting will be more convenient for our two-point estimates. 
We emphasize that, in contrast to Section~\ref{inverse sec}, all of the Loewner maps considered in this section go in the forward, rather than the reverse, direction. 
 
We start by defining the events whose probabilities we will estimate. 
Let $x,y\in\partial\BB D$ be distinct and let $m$ be the midpoint of the counterclockwise arc connecting $x$ and $y$ in $\partial\BB D$. Suppose we are given a simple curve $\eta$ in $\BB D$ connecting $x$ and $y$. Let $D_\eta$ be the connected component of $\BB D\setminus \eta$ containing $m$ on its boundary. Let $\Psi_\eta : D_\eta \rta \BB D$ be the unique conformal map taking $x$ to $-i$, $y$ to $i$, and $m$ to 1. For $s\in\BB R$, $u  >0$, $\ep > 0$, $c>1$, and $z \in \BB D$, let $\mathcal E_\ep^{s;u}(\eta , z ; c )$ be the event that
\begin{enumerate}
\item $z\in D_\eta$; 
\item $ c^{-1} \ep^{1-s +u } \leq \op{dist}(z, \partial D_\eta  )  \leq  c \ep^{1-s-u }$; and 
\item $ c^{-1} \ep^{ s + u } \leq |\Psi_\eta'(z)|  \leq  c \ep^{ s -u }$.
\end{enumerate}

For technical reasons it will also be convenient to consider the counterclockwise arc of $\partial\BB D$ from $y$ to $x$. We denote by $m^-$ the midpoint of this arc. Let $D_\eta^-$ be the connected component of $\BB D\setminus \eta$ containing $m^-$ on its boundary and we let $\Psi_\eta^- : D_\eta^- \rta \BB D$ be the unique conformal map taking $x$ to $i$, taking $y$ to $-i$, and taking $m^-$ to $-1$. See Figure~\ref{domains fig} for an illustration.

\begin{figure}\label{domains fig}
\begin{center}
\includegraphics[scale=.9]{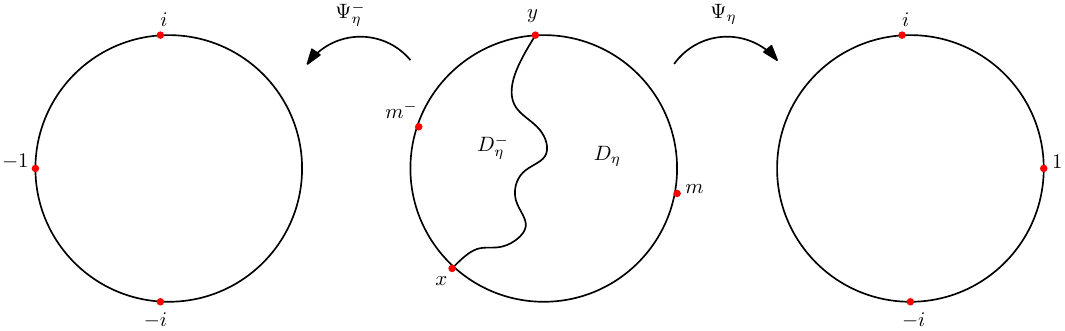}
\caption{An illustration of the domains and maps used in Theorem~\ref{1pt forward}.}
\end{center}
\end{figure}

\begin{thm} \label{1pt forward}
Suppose $\kappa \in (0,4]$ and $\eta$ is a chordal SLE$_\kappa$ from $x$ to $y$ in $\bdy\BB D$. Define the domains $D_\eta$ and $D_\eta^-$ and the event $\mcl E_\ep^{s;u}(\eta,z,;c)$ as above; and with $\alpha(s)$ and $\alpha_0(s)$ as in~\eqref{alpha def}, define
\eqb\label{gamma def}
\gamma(s) := \alpha(s) - 2s + 1 = \frac{(4 + \kappa)^2 s^2}{8 \kappa (1 + s)} -2s + 1 ,\qquad \gamma_0(s) := 2\alpha_0(s) + 2 = \frac{ 2(4 + \kappa)^2 s  (2 + s)}{8 \kappa (1 + s)^2} +2 .
\eqe 
Also define the events $\mathcal G(\cdot , \mu)$ as in Definition~\ref{G infty def}. For each $d\in (0,1)$, $\mu\in\mathcal M$, $c>0$, and $z\in B_d(0)$,  
\begin{align} \label{E probs upper} 
 \BB P\left(  \mathcal E_\ep^{s;u}(\eta , z ; c )  \cap \mathcal G(\Psi_\eta , \mu) \cap \mathcal G(\Psi_\eta^- , \mu)    \right)   \preceq   \ep^{\gamma(s)  - \gamma_0(s) u  } .
 \end{align}
Furthermore, for each $d\in (0,1)$ there exists $\mu\in \mathcal M$ such that for each $c>0$ and $u>0$ we can find $\ep_0 > 0$ such that for each $\ep \in (0,\ep_0]$ and each $z\in B_d(0)$, 
\begin{align}\label{E probs lower}
  \BB P\left(   \mathcal E_\ep^{s;u}(\eta , z ; c  )  \cap \mathcal G(\Psi_\eta , \mu) \cap \mathcal G(\Psi_\eta^- , \mu)    \right)\succeq   \ep^{\gamma(s)  +  \gamma_0(s) u  } .
\end{align} 
In~\eqref{E probs upper} and~\eqref{E probs lower} the implicit constants are independent of $\ep$ and uniform for $z\in B_d(0)$ and for $|x-y|$ bounded below by a positive constant. 
\end{thm}

The proof of Theorem~\ref{1pt forward} proceeds as follows. First we use Theorem~\ref{1pt chordal} and a change of variables to prove estimates for the area of the sets where certain finite-time analogs of the sets of Theorem~\ref{1pt forward} occur. This is done in Section~\ref{area sec}. This subsection also contains a result which allows us to extend the estimate for deterministic times to estimates for certain stopping times, which will be needed in the sequel. Then, in Section~\ref{compare sec}, we prove several lemmas comparing finite time and infinite time maps and use these lemmas to obtain estimates for the area of the set of points where the events of Theorem~\ref{1pt forward} occur. Finally, we complete the proof of Theorem~\ref{1pt forward} in Section~\ref{time infty proof sec} by proving a lemma which gives that the probabilities of the events of Theorem~\ref{1pt forward} do not depend too strongly on $z$, so that pointwise estimates can be deduced from area estimates. In Section~\ref{finite time sec} we deduce an analog of Theorem~\ref{1pt forward} for the curve stopped at a finite time.

\subsection{Area estimates and stopping estimates for finite time maps}
\label{area sec}

In this section we will prove estimates for the expected area of the set of points where finite-time analogs of the events of Theorem~\ref{1pt forward} occur. We will also prove a result which allows us to compare probabilities for events at stopping times whose difference is bounded. Suppose we are in the setting of Theorem~\ref{1pt forward}. 
  
\begin{defn} \label{E def D}
Let $\eta$ be a chordal $\op{SLE}_\kappa$ from $-i$ to $i$ in $\BB D$. Define its forward centered Loewner maps $(f_t)$ as in Section~\ref{disk sec}. For $t,\ep,u,\delta,c>0$, $s \in (-1,1)$, and $z\in \BB D$, let $ E_\ep^{s;u}(\eta , z;t,\delta, c)$ be the event that the following hold. 
\begin{enumerate}
\item $c^{-1}  \ep^{ s + u } \leq |f_t'(z)| \leq  c \ep^{ s -u }$.\label{A cond dist}
\item $ c^{-1}  \ep^{1-s  + u} \leq\op{dist}(z, \eta^t )  \leq c\ep^{1-s -u}$.\label{A cond deriv} 
\item $|f_t(z)  - i|$ and $ |f_t(z) +i| $ are both at least $ \delta$. \label{A cond delta}
\end{enumerate}  
Let $A_\ep^{s;u}(\eta ; t,\delta,c)$ be the set of $z\in \BB D$ for which $E_\ep^{s;u}(\eta , z;t,\delta,c)$ occurs.
\end{defn}

\begin{lem} \label{coord change area H}
Suppose we are in the setting of Theorem~\ref{1pt forward} with $x=-i$ and $y=i$. Fix $\delta>0$. Define the sets $A_\ep^{s;u}(\eta ; t,\delta , c)$ as in Definition~\ref{E def D} and the events $\mathcal G(f_t , \mu)$ as in Definition~\ref{G infty def}. For any choice of parameters $t,  c ,    \mu$ and any $d\in (0,1)$,
\begin{align} \label{A areas} 
 \BB E \left[ \op{Area}  ( A_\ep^{s,u}(\eta; t ,\delta, c)  \cap B_d(0))   \BB 1_{\mathcal G(f_t , \mu)  }  \right] \preceq  \ep^{\gamma(s)  -   \gamma_0(s)  u} 
\end{align}
with the implicit constants independent of $\ep$ and uniform for $z\in B_d(0)$. Moreover, there exists $t_*  > 0$ such that for each $t\geq t_*$, there exists $\mu\in\mcl M$ and $d \in (0,1)$ such that for each $c > 0$ and each $u>0$, there exists $\ep_0 > 0$ such that for $\ep \in (0,\ep_0]$, 
\begin{align} \label{A areas'} 
 \BB E \left[ \op{Area} (   A_\ep^{s;u}(\eta ; t,\delta , c ) \cap B_d(0))  \BB 1_{\mathcal G(f_t , \mu)  }    \right]  \succeq \ep^{\gamma(s)   +   \gamma_0(s)  u} ,
\end{align}
with the implicit constants independent of $\ep$ and uniform for $z\in B_d(0)$. 
\end{lem}
\begin{proof}
This will follow by integrating the estimate of Corollary~\ref{1pt chordal disk} and performing a change of variables.
Let $\ul{ A}_\ep^{s;u} =  \ul{ A}_\ep^{s;u}(\eta ; t  ,\delta , c , d ) $ be the set of $z\in \BB D$ such that 
\begin{enumerate}
\item $c^{-1} \ep^{1+u} \leq 1-|z| \leq c \ep^{1-u}$; 
\item $|z  - i| $ and $ |z  +i|$ are each at least $\delta$;
\item The event $\ul E_{\BB D}^{s;u}(z ; t    , c , d )$ of Section~\ref{disk sec} occurs. 
\end{enumerate}  
 
By~\eqref{disk asymp} in Corollary~\ref{1pt chordal disk}, if the first two conditions in the definition of $\ul{ A}_\ep^{s;u}  $ hold for some $z\in \BB D$, then 
\eqbn
  \BB P\left( \ul{E}_{\BB D}^{s;u}(z ; t    , c , d ) \cap \mathcal G(f_t  ,\mu)  \right)  \preceq  \ep^{ \alpha(s)  - \alpha_0(s) u } .
\eqen  
By integrating this over all such $z$, we get 
\eqb  \label{ul A area}
  \BB E\left[ \op{Area}(\ul{ A}_\ep^{s;u}  ) \BB 1_{G(f_t  ,\mu)}  \right] \preceq \ep^{ \alpha(s) + 1 - (\alpha_0(s)+1) u } . 
\eqe 
Similarly, suppose $t$, $d$, $\mu$, and $\ep_0$ are chosen so that~\eqref{disk asymp'} in Corollary~\ref{1pt chordal disk} holds. Then for $\ep \in (0,\ep_0]$,  
\eqb  \label{ul A area'}
  \BB E\left[ \op{Area}( \ul{ A}_\ep^{s;u}   ) \BB 1_{G(f_t  ,\mu)}  \right] \succeq \ep^{ \alpha(s) + 1 + (\alpha_0(s)+1) u } . 
\eqe 

By the change of variables formula,
\eqb\label{area integral}
  \op{Area}\left(  A_\ep^{s;u}(\eta ; t  ,\delta , c   ) \cap B_d(0)  \right) = \int_{  f_t  (A_\ep^{s;u}(\eta ; t  ,\delta , c   ) \cap B_d(0)  )} |(f_t  ^{-1})'(z)|^2 \, dz  .
\eqe
The Koebe quarter theorem implies 
\[
\ul{ A}_\ep^{s;u/2}(\eta ; t  ,\delta , c'  ,d )   \subset  f_t  \left(A_\ep^{s;u}(\eta ; t  ,\delta , c   ) \cap  B_d(0)\right)  \subset  \ul{ A}_\ep^{s;2u}(\eta ; t  ,\delta , c '' , d ) 
\]
for appropriate $c', c'' > 0$, depending only on $c$. 
Thus~\eqref{ul A area} implies~\eqref{A areas}. Similarly~\eqref{ul A area'} implies~\eqref{A areas'}.
\end{proof}

In the remainder of this subsection we record a straightforward estimate which allows us to transfer estimates between stopping times and deterministic times.  

\begin{lem}\label{stopping estimate}
Let $\eta$ be a chordal $\op{SLE}_\kappa$ from $-i$ to $i$ in $\BB D$ with centered Loewner maps $(f_t)$. Let $  \tau , \tau'$ be stopping times for $\eta$ and suppose there is a deterministic time $T > 0$ such that a.s.\ $ \tau \leq \tau' \leq T$. For any $c > 0$, $\mu\in\mathcal M$, and $\delta   > 0$, we can find $c' > 0$, $\delta' > 0$, and $\mu'\in\mathcal M$ such that for each $u>0$, there is an $\ep_0 = \ep_0(u,c,\mu,\delta) > 0$ such that for each $z\in \BB D$ and each $\ep \in (0,\ep_0]$, 
\eqb \label{stopping compare eqn}
 \BB P\left( E_\ep^{s;u}(\eta , z;\tau , \delta , c) \cap \mathcal G(f_\tau ,\mu)      \right) \preceq \BB P\left( E_\ep^{s;u}(\eta , z;\tau' , \delta' ,  c' ) \cap G(f_{\tau'}  ,\mu')   \right) ,
\eqe
with the implicit constant uniform for $z$ in compact subsets of $\BB D$ and independent of $\ep$. 
\end{lem}
\begin{proof}
Let $H$ be the event that the SLE$_\kappa$ curve $f_\tau(\eta\setminus\eta^\tau)$ stays in the tube $\{z \in \BB D : -\delta/100 \leq \re z \leq \delta/100\}$ until time $T$. 
By Lemma~\ref{miller-wu-dim-2.3} and the strong Markov property, $\BB P(H \,|\, \eta^\tau) \succeq 1$, with deterministic implicit constant depending only on $\delta$. On the other hand, if $\ep$ is sufficiently small relative to $\delta$ (so that $f_\tau(z)$ is within distance $\delta/100$ of $\bdy\BB D$ on $E_\ep^{s;u}(\eta , z;\tau , \delta , c) \cap \mathcal G(f_\tau ,\mu) $, say) then $f_\tau(z)$ lies at distance at least $\delta/2$ from this tube on the event $E_\ep^{s;u}(\eta , z;\tau , \delta , c) \cap \mathcal G(f_\tau ,\mu) $. 
Since $\tau' -\tau \leq T$, it follows easily that 
\eqbn
E_\ep^{s;u}(\eta , z;\tau , \delta , c) \cap \mathcal G(f_\tau ,\mu) \cap H \subset E_\ep^{s;u}(\eta , z;\tau' , \delta' ,  c' ) \cap G(f_{\tau'}  ,\mu')  
\eqen
for appropriate $c'$, $\delta'$, and $\mu'$ as in the statement of the lemma. Thus
\eqbn
\BB P\left(  E_\ep^{s;u}(\eta , z;\tau' , \delta' ,  c' ) \cap G(f_{\tau'}  ,\mu')  \,|\,   E_\ep^{s;u}(\eta , z;\tau , \delta , c) \cap \mathcal G(f_\tau ,\mu)     \right) \succeq 1,
\eqen
so~\eqref{stopping compare eqn} holds.
\end{proof}

\subsection{Comparison lemmas}
\label{compare sec}

In this subsection we prove several lemmas comparing probabilities of sets associated with the finite time Loewner maps to probabilities of sets associated with the infinite time Loewner maps of Theorem~\ref{1pt forward}, and use these results to estimate the area of the set where the event of Theorem~\ref{1pt forward} occurs.
 
The next lemma is needed for the proof of the lower bound in Theorem~\ref{1pt forward}. 

\begin{lem}\label{infty compare lower}
Suppose we are in the setting of Theorem~\ref{1pt forward} with $x=-i$ and $y= i$. Fix $d\in (0,1)$. For each $ \delta > 0 $, $\mu\in\mathcal M$, and $c>0$, there exists $\mu'\in\mathcal M$ and $c'  >0$ such that for each $u > 0$, there exists $\ep_0  = \ep_0(c , c' , u ,\delta , \mu,\mu',d)   > 0$ such that for $z\in B_d(0)$ and $\ep \in (0,\ep_0]$,
\begin{align} \label{infty compare lower eqn}
 \BB P\left(  \mathcal E_\ep^{s;u}(\eta , z; c'      ) \cap  \mathcal G(\Psi_\eta , \mu') \cap \mathcal G(\Psi_\eta^-,\mu')  \right) \succeq  \BB P\left( E_\ep^{s;u}(\eta , z; t, \delta  ,  c      ) \cap \{ \re f_t(z) \geq 0 \} \cap  \mathcal G(f_t,\mu)   \right) ,
\end{align} 
with implicit constants independent of $\ep$ and uniform for $z\in B_d(0)$.  
\end{lem}
\begin{proof} 
The idea of the proof is that if we condition on the event on the right side of~\eqref{infty compare lower eqn}, then with uniformly positive conditional probability the curve $\eta|_{[t,\infty)}$ will behave nicely and hence the event on the left in~\eqref{infty compare lower eqn} will also occur (this is similar to the idea of the proof of Lemma~\ref{stopping estimate}, but slightly more involved since we have to go all the way to time $\infty$). 

To explain this formally, let $f_t : \BB D\setminus \eta^t \rta \BB D$ be the centered forward Loewner maps for $\eta$ as in Section~\ref{area sec}. For $t\geq 0$, let $\eta_t = f_t(\eta|_{[t,\infty)})$. Also let $  D_t$ be the connected component of $\BB D\setminus \eta_t$ containing 1 on its boundary and let $D_t^-$ be the other connected component of $\BB D\setminus \eta_t$. Let $\Psi_t :  D_t \rta \BB D$ (resp.\ $\Psi_t^- : D_t\rta\BB D$) be the unique conformal maps fixing $-i,i,1$ (resp.\ $-i,i,-1$). Let $b_t$ (resp.\ $b_t^-$) be the image of the right (resp.\ left) side of $-i$ under $f_t$. Finally, let $\psi_t  $ (resp.\ $\psi_t^-  $) be the conformal automorphism of $\BB D$ fixing $i$, taking $\Psi_t(b_t)$ to $-i$, and taking $\Psi_t(f_t(1))$ to $1$ (resp.\ fixing $i$, taking $\Psi_t^-(b_t^-)$ to $-i$, and taking $\Psi_t^-(f_t(-1))$ to -1). Then for each $t$, 
\eqb \label{Phi decomp'}
  \Psi_\eta = \psi_t \circ \Psi_t\circ f_t ,\qquad \Psi_\eta^- = \psi_t^- \circ \Psi_t^- \circ f_t .
\eqe 
Moreover, $(\Psi_t , \Psi_t^-)$ and $f_t$ are independent and $\Psi_t\eqD  \Psi_\eta$, $\Psi_t^- \eqD \Psi_\eta^-$. See Figure~\ref{f-phi-psi} for an illustration of some of these maps. 

\begin{figure}
\begin{center}
\includegraphics[scale=.7]{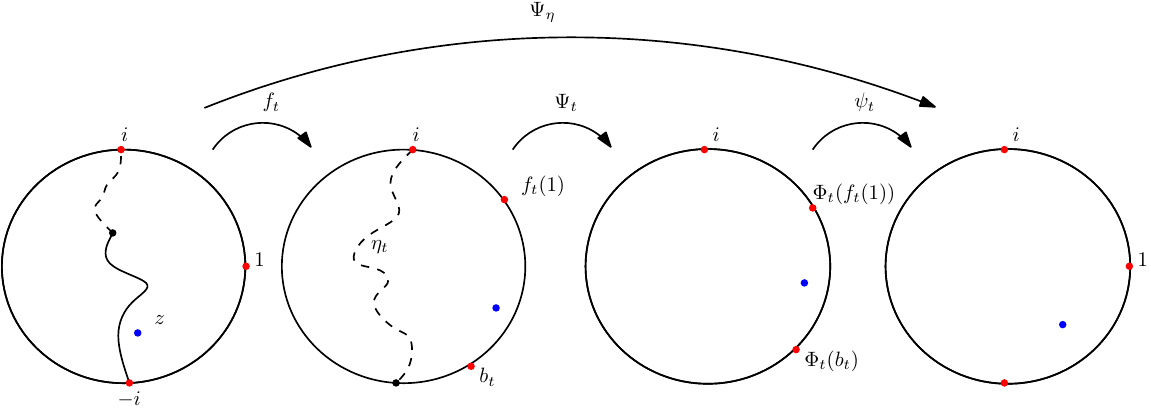}
\caption{An illustration of the maps used in the proof of Lemma~\ref{infty compare lower} for the right side of $\BB D$. The red points are the images of $-i,i$, and $1$ under the various maps. The last map $\psi_t$ takes these points back to their original positions so that by composing all three maps we recover the original map $\Psi_\eta$.} \label{f-phi-psi}
\end{center}
\end{figure} 

For $C>1$, $\mu'\in\mathcal M$, and $w\in \BB D$, let $ F(w) = F(w;t ,C , \mu' )$ be the event that 
$w\in D_t$,
$C^{-1} \leq |\Psi_t'(w)| \leq C$,
$  \op{dist}(w , \eta_t) = \op{dist}(w , \partial\BB D)$, and
$\mathcal G(\Psi_t ,\mu' ) \cap \mathcal G(\Psi_t^- , \mu')$ occurs.  
By Lemma~\ref{miller-wu-dim-2.3}, for each $\delta > 0$, we can find $C>1 $ and $\mu'\in\mathcal M$ such that for each $w\in \BB D$ lying at distance at least $\delta $ from $\pm i$ with $\re w \geq 0$, we have that $\BB P(F(w) ) \succeq 1$, with the implicit constant independent of $\ep$ and uniform for $w$ satisfying the conditions above.  

If we let
\eqbn
F^*(z ) 
 :=  E_\ep^{s;u}(\eta , z; t,\delta , c)  \cap \{\re f_t(z) \geq 0\} \cap \mathcal G(f_t, \mu  ) \cap F(f_t(z)  )   ,
\eqen
then by independence of $f_t$ and $\eta_t$ and our choice of parameters for $F(\cdot)$,
\eqb \label{P(W)}
 \BB P\left( F^*(z) \right) \asymp \BB P\left( E_\ep^{s;u}(\eta , z; t,\delta , c   ) \cap \{ \re f_t(z) \geq 0 \} \cap \mathcal G(f_t,\mu) \right)  .
\eqe 
By the ``$\mcl G$" condition in the definition of $F(f_t(z))$, we have that $|\psi_t'|$ and $|(\psi_t^-)'|$ are bounded above and below by positive $\ep$-independent constants on the event $F^*(z )$. Hence it follows from~\eqref{Phi decomp'} that  $F^*(z )\subset  \mathcal E_\ep^{s;u}(\eta , z; c'   ) \cap  \mathcal G(\Psi_\eta , \mu'') \cap \mathcal G(\Psi_\eta^-,\mu'') $ for some $c'>0$ and some $\mu''\in\mathcal M$ which do not depend on $\ep$ and are uniform for $z\in B_d(0)$. By combining this with~\eqref{P(W)} we get~\eqref{infty compare lower eqn} (with $\mu''$ in place of $\mu'$). 
\end{proof}

Our next lemma is needed for the proof of the upper bound in Theorem~\ref{1pt forward}. The proof in this case is much more involved than the proof of Lemma~\ref{infty compare lower}. Intuitively, the reason for this is that it is easy to construct a full SLE curve which contains a given segment of an SLE curve run up to finite time (just grow the rest of the curve) but harder to construct an SLE run up to a finite time which has nice behavior and contains a conformal image of a given full SLE curve (one has to use reversibility and define appropriate regularity conditions for an SLE and its time reversal in order to successfully ``splice in" the given full SLE curve).

\begin{lem} \label{infty compare upper'}
Suppose we are in the setting of Theorem~\ref{1pt forward} with $x=-i$ and $y= i$. Fix $d\in (0,1)$. There is a $\delta>0$ such that for each $\mu\in\mathcal M$ and $c>0$, there exists $\mu'\in \mathcal M$ and $c'>0$ such that for each $u>0$, there exists $\ep_0 > 0$ and a bounded stopping time $\tau$ for $\eta$ such that for each $z\in B_d(0)$ and each $\ep \in (0,\ep_0]$, 
\begin{align} \label{infty compare upper' eqn}
\BB P\left( \mathcal E^{s;u}(\eta,z;c) \cap \mathcal G(\Psi_\eta , \mu) \cap \mathcal G(\Psi_\eta^- , \mu) \right) \preceq \BB P\left(  E_\ep^{s;u}(z; \tau , \delta , c ' )  \cap\mathcal G(f_\tau , \mu ')  \right)  
 \end{align}
 with the implicit constants independent of $\ep$ and uniform for $z \in B_d(0)$. 
\end{lem} 
\begin{proof}
Suppose $ \mathcal E^{s;u}(\eta,z;c) \cap \mathcal G(\Psi_\eta , \mu) \cap \mathcal G(\Psi_\eta^- , \mu)$ occurs. 
We will prove the lemma by growing some more of the curve out from $-i$ and $i$ to get a new curve $\wt\eta \eqD \eta$ with the property that $E_\ep^{s;u}(\wt\eta , z; \tau , \delta , c ' )  \cap\mathcal G(f_\tau , \mu ')$ occurs for an appropriate bounded stopping time $\tau$ and the derivatives of the conformal maps associated with $\wt\eta^\tau$ and with $\eta$ at $z$ are comparable. 

To this end, let $ \eta_0$ be a chordal $\op{SLE}_\kappa$ from $-i$ to $i$ in $\BB D$, independent of $\eta$. Let $\ol\eta_0$ be its time reversal. Then $\ol\eta_0$ has the law of a chordal $\op{SLE}_\kappa$ from $i$ to $-i$~\cite{zhan-reversibility}. Fix parameters $ \delta_0,  C  , \beta , \zeta , r , a >0$, and $\mu_0 \in\mathcal M$ and suppose $\zeta \ll 1-d$. 
Let $P  $ be the event that the following is true. 
\begin{enumerate} 
\item Let $\ol T$ be the first time $\ol\eta_0$ gets within distance $e^{-\beta}$ of $z$. Then $\ol T < \infty$ and $\ol\eta_0^{\ol T}$ is disjoint from $(\BB D \setminus \BB H) \cup B_{1/2}(1)$. \label{P hit}
\item For each $t\geq 0$, let $\phi_t : \BB D\setminus (\eta_0^t \cup \ol\eta_0^{\ol T})$ be the unique conformal map fixing $z$ and taking $\ol\eta_0(\ol T)$ to $i$. Let $T$ be the first time $t$ that $\phi_t(\eta_0(t)) = -i$ and $|\eta_0(t) - z| \leq 2e^{-\beta}$. Then $T < \infty$ and $\eta_0^T$ is disjoint from $(\BB D\cap \BB H) \cup B_{1/2}(1)$. \label{P forward}
\item Henceforth put $\phi = \phi_T$. We have $C^{-1} \leq |(\phi^{-1})'(w)| \leq C$ for each $w\in B_{(1+d)/2}(0)$.\label{P deriv}
\item We have $\phi^{-1}\left(B_{\delta_0}(-i) \cup B_{\delta_0}(i) \cup B_{1-r}(0)\right)\subset B_{(1-d)/2}(z)$. \label{P ball}
\item Let $\ol\sigma$ be the last exit time of $\ol\eta_0$ from $B_\zeta(i)$ before time $\ol T$. Then $\ol\eta_0^{\ol\sigma} \subset B_{2\zeta}(i)$. \label{P stay}
\item Let 
\eqb \label{K set def}
K:= \eta_0^T \cup \ol\eta_0([\ol\sigma,\ol T]) \cup B_{(1-d)/2}(z)  .
\eqe  
The harmonic measure from $i$ of each side of $K \cap B_{(1-d)/2}(i)$ and each side of $K \cap B_{(1-d)/2}(-i)$ in the Schwarz reflection of $\BB D\setminus K$ across $[-1,1]_{\partial\BB D}$ is at least $a$. \label{P hm} 
\item $\mathcal G'(K , \mu_0 )$ occurs (Definition~\ref{G' def}). \label{P G} 
\end{enumerate} 

See Figure~\ref{infty pushing fig} for an illustration of the event $P$. In what follows, all implicit constants are required to depend only on $\mu$, $d$, and the parameters for $P$. 

First we will argue that for any choice of the parameters $d , \zeta $, and $r$, we can choose the other parameters for $P$ in such a way that $\BB P(P) \succeq 1$. It follows from Lemma~\ref{miller-wu-dim-2.3} and reversibility of SLE that conditions~\ref{P hit},~\ref{P forward}, and~\ref{P stay} hold with positive probability depending only on $\beta$, $\zeta$, and $d$. By the Koebe growth theorem, if $\beta$ is chosen sufficiently large (depending on $r$ and $d$) and $\delta_0$ is chosen sufficiently small (depending only on $d$) then condition~\ref{P ball} also holds simultaneously with positive probability depending only on $\beta ,\zeta , d$, $\delta_0$, and $r$. By choosing $C$ sufficiently large and $a$ and $\mu_0$ sufficiently small (see Lemma~\ref{G pos}), depending only on $d$ and the other parameters for $P$, we can arrange that the remaining conditions in the definition of $P$ hold with probability arbitrarily close to 1. Thus $\BB P(P)\succeq 1$. 
  
\begin{figure}
\begin{center}
\includegraphics[scale=.9]{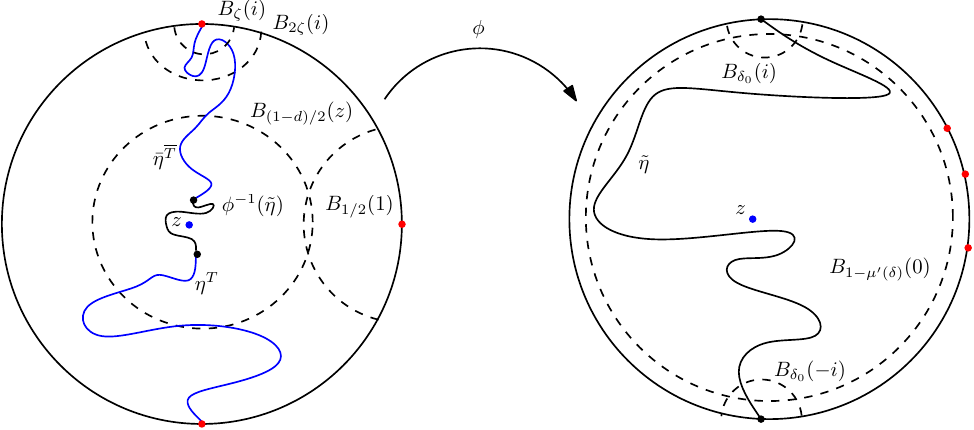}
\caption{An illustration of the event $P$ and the curve $\wt\eta$ used in the proof of Lemma~\ref{infty compare upper'}. The red points are $-i$, $i$, and 1 and their images under $\phi$. } \label{infty pushing fig}
\end{center}
\end{figure} 

Let $\wt\eta = \eta_0$ on the event that $P$ does not occur. On $P$, let $\wt \eta = \phi^{-1}(\eta) \cup  \eta_0^T \cup \ol\eta^{\ol T}$, parameterized in such a way that its image under the conformal map from $\BB D$ to $\BB H$ taking $-i$ to $0$, $i$ to $\infty$, and $0$ to $i$ is parameterized by capacity. By the Markov property and reversibility of SLE, $\wt\eta$ has the same law as $\eta$. Let $(\wt f_t)$ be the centered Loewner maps for $\wt\eta$. Let
\[
\wt{\mcl E} = \mathcal E^{s;u}_\ep(\eta,z;c) \cap \mathcal G(\Psi_\eta , \mu) \cap \mathcal G(\Psi_\eta^- , \mu) \cap P  .
\] 
Let $\tau$ be the hitting time of $B_{ \zeta}(i)$ by $\wt\eta$. Then $\tau$ is a bounded stopping time for $\wt\eta$. Furthermore, if we choose $\zeta$ sufficiently small relative to $d$ (independently of $\ep$) then on the event $\wt{\mcl E}$ we have $\wt\eta\setminus \wt\eta^\tau = \ol\eta_0^{\ol\sigma}$, with $\ol\sigma$ as in condition~\ref{P stay} in the definition of $P$. 

We claim that if the parameters for $P$ are chosen appropriately (independently of $\ep$ and $z \in B_d(0)$) then for sufficiently small $\ep > 0$, 
\eqb \label{E G P contain}
\wt{\mcl E} \subset E_\ep^{s;u}(\wt\eta , z; \tau , \delta , \wt c   )  \cap\mathcal G(\wt f_\tau , \wt \mu) 
 \eqe
for some $\wt\mu\in\mcl M$ depending only on $d$ and some $\wt c > 0$ and $\wt \mu \in \mathcal M$, depending only on $d$, $\mu$, $c$, and the parameters for $P$. Given the claim~\eqref{E G P contain}, our desired result~\eqref{infty compare upper' eqn} follows by taking probabilities and noting that $P$ is independent from $\eta$. 

By condition~\ref{P ball} in the definition of $P$, on the event $\wt{\mcl E}$ we have $\wt\eta^\tau \subset K$, as in~\eqref{K set def}, provided $r$ is chosen sufficiently small, depending only on $\mu$ and $\delta_0$. By condition~\ref{P G} in the definition of $P$ and Lemma~\ref{G dist}, we can find $\wt\mu\in\mathcal M$ depending only on $\mu$, $d$ and the parameters for $P$ such that $\wt{\mcl E}\subset \mathcal G(\wt f_\tau , \wt \mu)$. By condition~\ref{P hm} in the definition of $P$, we can find $\delta > 0$ depending only on $a$ such that $\wt f_\tau(z)  $ lies at distance at least $\delta $ from $\pm i$ on $\wt{\mcl E}$. That is, condition~\ref{A cond delta} in the definition of $E_\ep^{s;u}(\wt\eta , z; \tau , \delta ,\wt c )$ holds on $\wt{\mcl E}$. 

By condition~\ref{P deriv} in the definition of $P$, we have $\op{dist}(z , \wt \eta) \asymp \op{dist}(z , \eta)$ on $P$. It therefore follows that condition~\ref{A cond dist} in the definition of $E_\ep^{s;u}(\wt\eta , z; \tau , \delta , \wt c )$ holds on $\wt{\mcl E}$ for some $\wt c \asymp 1$. 
  
It remains to show that condition~\ref{A cond dist} in the definition of $E_\ep^{s;u}(\wt\eta , z; \tau , \delta , \wt c )$ holds on $\wt{\mcl E}$ provided $\wt c \asymp 1$ is chosen sufficiently large. It is enough to show $|\wt f_\tau'(z)| \asymp |\Psi_{ \eta}'(z)|$ on $\wt{\mcl E}$. We will do this in two stages. Let $\Psi_{\wt\eta }$ be as in Section~\ref{time infty setup sec} with $\wt\eta$ in place of $\eta$. First we will show that $|\Psi_\eta'(z)|\asymp |\Psi_{\wt\eta}'(z)|$, and then we will show that $|\Psi_{\wt\eta}'(z)| \asymp |\wt f_\tau'(z)|$. 

For the first stage, let $g $ be the conformal automorphism of $\BB D$ taking $\Psi_\eta(\phi(-i^+))$ to $-i$, $\Psi_\eta(\phi( i^-))$ to $i$, and $\Psi_\eta(\phi(1))$ to 1. Then 
\eqb \label{Psi_hat eta decomp}
\Psi_{\wt\eta} = g \circ \Psi_{ \eta} \circ \phi .
\eqe
By condition~\ref{P G} in the definition of $P$, together with the definition of $\wt{\mcl E}$,  
$|g '|\asymp 1$ uniformly on $\BB D$ on $\wt{\mcl E}$, so by condition~\ref{P deriv} in the definition of $P$, we have $|\Psi_{\wt \eta}'(z)| \asymp |\Psi_\eta'(z)|$ on $\wt{\mcl E}$.   

For the second stage, let $ \Psi_{\wt\eta^\tau}$ be the conformal map from $\BB D\setminus \wt \eta^\tau$ to $\BB D$ taking $-i^+$ to $-i$ and fixing $i$ and 1. Then $\Psi_{\wt\eta^\tau}$ differs from $\wt f_\tau$ by a conformal automorphism of $\BB D$ taking $\wt f_\tau(-i^+)$ to $-i$ and $\wt f_\tau(1)$ to 1. Since $\mathcal G(\wt f_\tau , \wt \mu)$ holds on $\wt{\mcl E}$,  
\eqb \label{f asymp Psi}
|\Psi_{\wt\eta^\tau} '(z)| \asymp |\wt f_\tau'(z)|  .
\eqe

Let $I$ be the arc of $\partial\BB D$ of length $\zeta$ centered at 1. By condition~\ref{P G} in the definition of $P$ (c.f. Remark~\ref{phi hm hypotheses}), the lengths of $\Psi_{\wt\eta}(I)$ and $\Psi_{\wt\eta^\tau}(I)$ are $\succeq 1$ on $\wt{\mcl E}$. By conditions~\ref{P hit},~\ref{P ball}, and~\ref{P stay} in the definition of $P$ and a study of the harmonic measure from $1$ in the Schwarz reflection of $D_{\wt\eta}$, the distances from $\Psi_{\wt\eta}(z)$ to $\Psi_{\wt\eta}(I)$ and from $\Psi_{\wt\eta^\tau}(z)$ to $\Psi_{\wt\eta^\tau}(I)$ are $\succeq 1$ on $\wt{\mcl E}$ provided $\zeta$ is chosen sufficiently small relative to $d$. 
By Lemma~\ref{phi hm}, it holds on $\wt{\mcl E}$ that 
\eqb \label{Psi_wt eta hm}
|\Psi_{\wt\eta}'(z)| \asymp \frac{\op{hm}^z(I ; D_{\wt\eta} )}{ \op{dist}(z , \wt\eta) }  \quad \op{and}\quad |\Psi_{\wt\eta^\tau}'(z)| \asymp \frac{\op{hm}^z(I ;\BB D\setminus \wt\eta^\tau )}{ \op{dist}(z , \wt\eta^\tau) } .
\eqe  
By the conformal invariance of harmonic measure, $\op{hm}^z(I ; D_{\wt\eta} )$ is the same as the probability that a Brownian motion started from $\Psi_{\wt\eta^\tau}(z)$ exits $\BB D$ in $\Psi_{\wt\eta^\tau}(I)$ before hitting $\Psi_{\wt\eta^\tau}(\wt\eta([\tau,\infty))$. 
By conditions~\ref{P stay} and~\ref{P hm} in the definition of $P$, if $\zeta$ is chosen sufficiently small, independently of $\ep$, then on $\wt{\mcl E}$, the distance from $\Psi_{\wt\eta^\tau}(\wt\eta([\tau,\infty))$ to $\Psi_{\wt\eta^\tau}(z) \cup \Psi_{\wt\eta^\tau}(I)$ is at least a deterministic $\ep$-independent constant; and the diameter of $\Psi_{\wt\eta^\tau}(\wt\eta([\tau,\infty))$ is smaller than $1/100$ times this constant (here we again use harmonic measure from 1). 
Therefore, the probability that a Brownian motion started from $\Psi_{\wt\eta^\tau}(z)$ exits $\BB D$ in $\Psi_{\wt\eta^\tau}(I)$ before hitting $\Psi_{\wt\eta^\tau}(\wt\eta([\tau,\infty))$ is proportional to the probability that a Brownian motion started from $\Psi_{\wt\eta^\tau}(z)$ exits $\BB D$ in $\Psi_{\wt\eta^\tau}(I)$. 
That is, $\op{hm}^z(I ; D_{\wt\eta} ) \asymp \op{hm}^z(I ;\BB D\setminus \wt\eta^\tau )$ on $\wt{\mcl E}$. By combining this with~\eqref{f asymp Psi} and~\eqref{Psi_wt eta hm}, we conclude.  
\end{proof}

Now we can transfer our area estimates for the finite time sets to area estimates for the time infinity sets.  

\begin{lem}  \label{infty area}
Suppose we are in the setting of Theorem~\ref{1pt forward} with $x = -i$ and $y=i$. 
Let $\mathcal A_\ep^{s;u}(\eta,  c)$ be the set of $z \in \BB D$ for which $\mathcal E_\ep^{s;u}(\eta , z ; c )$ occurs.
For each $d\in(0,1)$, each $\mu\in\mathcal M$, and each $c>0$,
\begin{align} \label{A areas upper} 
 \BB E\left( \op{Area} (\mathcal A_\ep^{s;u}(\eta   ; c ) \cap B_d(0) )  \BB 1_{\mathcal G(\Psi_\eta , \mu) \cap \mathcal G(\Psi_\eta^- , \mu)}    \right)   \preceq   \ep^{\gamma(s)  - \gamma_0(s) u  } .
 \end{align} 
Furthermore, there exists $d\in (0,1)$ such that for each $c>0$, there exists $\mu \in\mcl M$ and $\ep_0 > 0$ such that for each $\ep \in (0,\ep_0]$,
\begin{align} \label{A areas lower} 
 \BB E\left( \op{Area}( \mathcal A_\ep^{s;u}(\eta  ; c    ) \cap B_d(0) )   \BB 1_{\mathcal G(\Psi_\eta , \mu) \cap \mathcal G(\Psi_\eta^- , \mu)}    \right)   \succeq \ep^{\gamma(s)  +  \gamma_0(s) u  }.
 \end{align} 
In both~\eqref{A areas upper} and~\eqref{A areas lower} the implicit constants depend on the other parameters but not on $\ep$. 
\end{lem}
\begin{proof} 
The relation~\eqref{A areas upper} follows by integrating the estimate from Lemma~\ref{infty compare upper'} over $B_d(0)$, applying Lemma~\ref{stopping estimate} to replace the stopping time $\tau$ with a deterministic time, then applying~\eqref{A areas} from Lemma~\ref{coord change area H}. 
The relation~\ref{A areas lower} similarly follows from Lemma~\ref{infty compare lower}. 
\end{proof}

\subsection{Proof of Theorem~\ref{1pt forward}}
\label{time infty proof sec}
 
To deduce Theorem~\ref{1pt forward} from the area estimate of Lemma~\ref{infty area}, we need to argue that the probabilities of the events of Theorem~\ref{1pt forward} do not depend too strongly on $z$. This is accomplished in the next lemma.

\begin{lem}\label{infty log continuity}
Suppose we are in the setting of Theorem~\ref{1pt forward} with $x=-i$, $y=i$. Fix $d\in (0,1)$. For any $\mu\in \mathcal M$ and $c>0$, we can find $\mu'\in\mathcal M$ and $c'>0$ such that for each $z,w\in B_d(0)$ and $\ep \in (0,1)$, 
\eqb \label{infty log continuity eqn}
\BB P\left(   \mathcal E_\ep^{s;u}(\eta , w ; c   )  \cap \mathcal G(\Psi_\eta , \mu) \cap \mathcal G(\Psi_\eta^- , \mu)    \right) \preceq \BB P\left(  \mathcal E_\ep^{s;u}(\eta , z ; c'   )  \cap \mathcal G(\Psi_\eta , \mu') \cap \mathcal G(\Psi_\eta^- , \mu')    \right) 
\eqe 
with implicit constants independent of $\ep$ and uniform in $B_d(0)$.  
\end{lem}
\begin{proof} 
The basic idea of the proof is as follows. First we apply a conformal map taking $z$ to $w$ and fixing $-i$. The image of $\eta$ under such a map will be an $\op{SLE}_\kappa$ with a new target point $b$. To compare such a curve to our original curve, we grow a carefully chosen segment of the new curve backward from $b$ in such a way that when we map back to $\BB D$, we get a chordal $\op{SLE}_\kappa$ from $-i$ to $i$. We now commence with the details. 
 
For $z,w\in B_d(0)$, let $\phi = \phi_{z,w} : \BB D \rta \BB D$ be the unique conformal map fixing $-i$ and taking $z$ to $w$. Let $b := \phi(i)$ and $ \eta^b = \phi(\eta)$. The law of $\eta^b $ is that of a chordal $\op{SLE}_\kappa$ process from $-i$ to $b$ in $\BB D$. 

The map $\phi$ depends continuously on $z$ and $w$ in the topology of uniform convergence on compact subsets of $\BB D$. 
It follows that for any $\mu\in\mathcal M$ we can find a deterministic constant $c' > 0$ depending only on $ c$, $\mu$, and $d$, linearly on $c$, and a deterministic $\mu' \in \mathcal M$ depending only on $\mu$ and $d$ such that for $z,w \in B_d(0)$,
\eqb \label{z w compare} 
 \mathcal E_\ep^{s;u}(\eta^b , w ; c   ) \cap \mathcal G(\Psi_{\eta^b} , \mu   )  \cap \mathcal G(\Psi_{ \eta^b}^- , \mu)  \subset  \mathcal E_\ep^{s;u}(   \eta , z ; c'   )  \cap \mathcal G(\Psi_{\eta} , \mu'   ) \cap \mathcal G(\Psi_{ \eta^b}^- , \mu')    .
\eqe 

Let $\ol\eta^b $ be the time reversal of $\eta^b $. Then $\ol\eta^b $ is a chordal $\op{SLE}_\kappa$ from $b$ to $-i$ in $\BB D$ \cite{zhan-reversibility}.  We give $\ol \eta^b$ the usual chordal parameterization, so that it is the conformal image of a chordal $\op{SLE}_\kappa$ parameterized by capacity from $0$ to $\infty$ in $\BB H$. For each $t\geq 0$, let $\ol g_t : \BB D\setminus  \ol\eta^b([0,t]) \rta \BB D$ be the unique conformal map fixing $-i$ and $w$. Let $\tau $ be the first time $t$ that $\ol g_t (\ol\eta^b(t)) = i$.  

Fix $\mu^b \in\mathcal M$ and let $\ol E^b $ be the event that $\tau  $ is less than or equal to the first time $t$ that $\ol\eta^b$ hits $B_{d^*}(0)$, where 
\[
d^* := 1 - \frac14 \inf_{z,w\in B_d(0)} \op{dist}( \phi_{z,w}(B_d(0) ), \partial\BB D    ) ;
\]
 and the event $\mathcal G(\ol g_\tau , \mu^b )$ occurs.  
By Lemma~\ref{miller-wu-dim-2.3}, if $\mu^b$ is chosen sufficiently small then $\BB P(\ol E^b)$ is a positive constant depending only on $\mu^b$ and $B_d(0)$. 

By the Markov property, conditional on $\ol E^b$, the law of $\ol g_\tau (\ol\eta^b|_{[\tau,\infty)})$ is that of a chordal $\op{SLE}_\kappa$ process from $i$ to $-i$ in $\BB D$. Therefore its time reversal $\wh \eta := \ol g_\tau^b(\eta|_{[0,\tau^b]})$, where $\tau^b$ is the time corresponding to $\tau$ under the time reversal, has the law of a chordal $\op{SLE}_\kappa$ from $-i$ to $i$ in $\BB D$. In particular, $\wh\eta \eqD \eta$.  

Define the open sets $D_{\eta^b} , D_{\wh\eta}$ and the maps $\Psi_{ \eta^b} , \Psi_{\wh \eta}$ as in Section~\ref{time infty setup sec} with $ \eta^b , \wh \eta$, resp., in place of $\eta$, except that in the definition of $\eta^b$ we use the points $\phi(-1)$ and $\phi(1)$ instead of the midpoints $m^-$ and $m$. Also let $\psi$ and $\psi^-$ be the conformal automorphisms of $\BB D$ such that
\eqbn
\Psi_{\eta^b} = \psi \circ\Psi_{\wh\eta} \circ \ol g_\tau \quad \op{and} \quad \Psi_{\eta^b}^-  =   \psi^- \circ \Psi_{\wh\eta}^-\circ \ol g_\tau .
\eqen 
See Figure~\ref{cont maps fig} for an illustration of some of these maps. 

\begin{figure}
\begin{center}
\includegraphics[scale=.9]{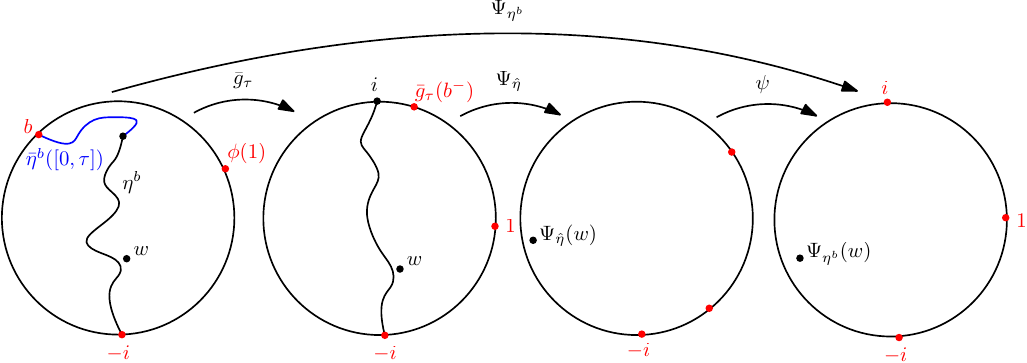}
\caption{An illustration of the maps used in the proof of Lemma~\ref{infty log continuity} on the event $\ol E^b$.} \label{cont maps fig}
\end{center}
\end{figure}

Since $\ol E^b\subset \mathcal G(\ol g_\tau , \mu^b ) $, on the event $\ol E^b  \cap   \mathcal E_\ep^{s;u}(\wh\eta , w ; c  ) \cap \mathcal G(\Psi_{\wh \eta}, \mu ) \cap \mathcal G(\Psi_{\wh\eta}^- , \mu)$, it holds that $|\psi'|$ and $|(\psi^-)'|$ are bounded above and below by deterministic positive constants depending only on $\mu^b$ and $\mu$. Furthermore, $\mathcal G(\psi , \mu_2) \cap \mathcal G(\psi^- , \mu_2)$ holds for some $\mu_2\in\mathcal M$ depending on $\mu^b  , \mu$. The Koebe distortion theorem and the definition of $\ol E^b$ imply that $|g_\tau'(w)|$ is bounded above and below by positive constants depending only on $d$ on the event $\ol E^b$. 
Hence for some $c_0  > 0$, independent of $\ep$ and uniform for $z,w\in B_d(0)$,  
\eqb \label{E^b contained'}
\ol E^b  \cap   \mathcal E_\ep^{s;u}(\wh\eta , w ; c ) \cap  \mathcal G (\Psi_{\wh \eta} , \mu) \cap \mathcal G (\Psi_{\wh \eta}^- , \mu)   \subset  \mathcal E_\ep^{s;u}( \eta^b , w ; c_0) \cap \mathcal G(\Psi_{\eta^b} , \mu_2 \circ \mu \circ \mu^b )  \cap \mathcal G(\Psi_{\eta^b}^- , \mu_2 \circ \mu \circ \mu^b ) .
\eqe  
By the Markov property and the fact that $\BB P(\ol E^b)$ is uniformly positive,  
\eqb \label{E^b Markov'}
\BB P\left(  \ol E^b  \cap  \mathcal E_\ep^{s;u}(\wh\eta , w ; c) \cap  \mathcal G (\Psi_{\wh \eta} , \mu)   \cap \mathcal G(\Psi_{\wh\eta}^- , \mu) \right) \asymp \BB P\left(   \mathcal E_\ep^{s;u}(\wh\eta , w ; c  ) \cap  \mathcal G (\Psi_{\wh \eta} , \mu)\cap \mathcal G(\Psi_{\wh\eta}^- , \mu) \right) .
\eqe  
Since $\wh \eta \eqD \eta$,~\eqref{infty log continuity eqn} now follows from~\eqref{z w compare} (applied with $\mu_2\circ \mu\circ\mu^b$ in place of $\mu$, $c_0$ in place of $c$, and a possibly larger choice of $c'$ and $\mu'$),~\eqref{E^b contained'}, and~\eqref{E^b Markov'}.
\end{proof}

\begin{proof}[Proof of Theorem~\ref{1pt forward}]
By applying a coordinate change it is enough to consider the case $x=-i,$ $y=i$. By Lemma~\ref{infty log continuity}, for any $z\in B_d(0)$, we have, in the notation of that lemma,
\alb
\BB P\left(  \mathcal E_\ep^{s;u}(\eta , z ; c )  \cap \mathcal G(\Psi_\eta , \mu) \cap \mathcal G(\Psi_\eta^- , \mu)    \right) \preceq \BB E\left( \op{Area} (\mathcal A_\ep^{s;u}(\eta , z ; c' ) \cap B_d(0))  \BB 1_{\mathcal G(\Psi_\eta , \mu') \cap \mathcal G(\Psi_\eta^- , \mu')}    \right)    \\
\BB P\left(  \mathcal E_\ep^{s;u}(\eta , z ; c  ' )  \cap \mathcal G(\Psi_\eta , \mu') \cap \mathcal G(\Psi_\eta^- , \mu')    \right) \succeq \BB E\left( \op{Area} (\mathcal A_\ep^{s;u}(\eta , z ; c    )  \cap B_d(0))  \BB 1_{\mathcal G(\Psi_\eta , \mu) \cap \mathcal G(\Psi_\eta^- , \mu)}    \right) ,
\ale
where here $\mcl A_\ep^{s;u}(\cdot)$ is the set where $\mcl E_\ep^{s;u}(\cdot)$ occurs, as in Lemma~\ref{infty area}. 
We conclude by combining this with Lemma~\ref{infty area} (and slightly decreasing $u$ and shrinking $\ep_0$ as in the proof of Lemma~\ref{infty area} to get a small enough constant in the event for lower bound). 
\end{proof}
 
\subsection{Finite time estimates}
\label{finite time sec}
 
In this subsection we use Theorem~\ref{1pt forward} and the comparison lemmas of Section~\ref{compare sec} to prove estimates for the finite time Loewner maps. The result of this subsection is not needed for the proof of our main result, and is stated only for the sake of completeness.

\begin{thm} \label{1pt forward finite}
Let $\kappa \in (0,4]$. Let $( f_t)$ be the centered Loewner maps of a chordal $\op{SLE}_\kappa$ process $  \eta$ from $-i$ to $i$ in $\BB D$. Fix $d\in (0,1)$. Define the events $E_\ep^{s;u}(z;t,\delta,c)$ as in Definition~\ref{E def D} and the sets $G(f_t ,\mu)$ as in Definition~\ref{G infty def}. For any $\mu\in\mathcal M$, $t,\delta,c> 0$, $\ep>0$, and $z\in B_d(0)$,  
\eqb \label{finite upper}
 \BB P\left( E_\ep^{s;u}(\eta , z;t,\delta ,c) \cap   \mathcal G(f_t ,\mu ) \cap \{\re f_t(z) \geq 0\}   \right)  \preceq    \ep^{\gamma(s) - 2\gamma_0(s) u     }  .
\eqe 
Moreover, there exists $t_* > 0$, $\delta>0$, and $\mu\in\mcl M$ such that for each $c > 0$ and each $u>0$, there exists $\ep_0 > 0$ such that for $\ep \in (0,\ep_0]$ and $z\in B_d(0)$, 
\eqb \label{finite lower}
 \BB P\left(  E_\ep^{s;u}(\eta , z;t,\delta ,c) \cap   \mathcal G(f_t ,\mu )   \right)  \succeq   \ep^{\gamma(s) + 2\gamma_0(s) u   }  .
\eqe 
In~\eqref{finite upper} and~\eqref{finite lower} the implicit constants are independent of $\ep$ and uniform for $z \in B_d(0)$. The estimate~\eqref{finite upper} holds with $t$ replaced by a bounded stopping time. The estimate~\eqref{finite lower} holds with $t$ replaced by a bounded stopping time which is a.s.\ $\geq t_*$. 
\end{thm} 
\begin{proof}
The statement for deterministic times follows by combining Theorem~\ref{1pt forward} with Lemmas~\ref{stopping estimate},~\ref{infty compare upper'} and~\ref{infty compare lower}. The statement for stopping times follows from this and Lemma~\ref{stopping estimate}. 
\end{proof}

\section{Upper bounds for multifractal and integral means spectra}
\label{haus upper sec}

In this section we will use the upper bounds in Theorems~\ref{1pt chordal} and~\ref{1pt forward} to prove the Hausdorff dimension upper bounds in Theorem~\ref{main thm} as well the upper bound in Corollary~\ref{ims cor}.

\subsection{Upper bound for the Hausdorff dimension of the subset of the circle}
\label{circle upper sec}
 
In this subsection we use Theorem~\ref{1pt chordal} to obtain upper bounds on the Hausdorff dimension of the sets $\wt \Theta^s(\BB D \setminus K_t)$ of Section~\ref{multifractal def} for the hulls $(K_t)$ of a chordal $\op{SLE}_\kappa$ from $-i$ to $i$ in $\BB D$. In light of Lemma~\ref{theta zero one}, Proposition~\ref{bdy haus upper} implies the upper bounds for $\dim_{\mathcal H} \wt \Theta^{s;\geq}(D_\eta)$ and $\dim_{\mathcal H} \wt \Theta^{s;\leq}(D_\eta)$ in Theorem~\ref{main thm}. 

\begin{prop} \label{bdy haus upper}
Let $\eta$ be a chordal $\op{SLE}_\kappa$ process from $-i$ to $i$ in $\BB D$ with forward centered Loewner maps $(f_t)$ (defined as in Section~\ref{disk sec}) and hulls $(K_t)$.  
Let $\wt \xi(s)$, $s_-$, and $s_+$ be as in~\eqref{tilde xi(s)}. For each $t > 0$ and $s\in [-1,1]$, a.s.\ 
\begin{align}
\label{bdy haus upper eqn}
& \dim_{\mathcal H} \wt\Theta^{s; \geq }(\BB D\setminus K_t) \leq \wt\xi(s)  ,\qquad  0\leq s \leq s_+ \nonumber \\
& \dim_{\mathcal H} \wt\Theta^{s; \leq }(\BB D\setminus K_t) \leq \wt\xi(s)  ,\qquad    s_-\leq s \leq 0  .
\end{align}   
Almost surely, for each $s\notin [s_- , s_+]$ we have $\wt\Theta^{s }(\BB D\setminus K_t)   = \emptyset$. In fact, for each $\delta >0$ and each $s > s_+$, it is a.s.\ the case that for small enough $\ep > 0$,
\eqb \label{bdy haus max}
|(f_t^{-1})'((1-\ep) x )| \leq \ep^{-s} ,\quad \forall x \in \bdy\BB D \: \text{with $|x-i| , |x+i| \geq \delta$ and $1-|f_t^{-1}(x)| \geq \delta$} ;
\eqe 
and similarly for $s < s_-$. 
\end{prop}

\begin{remark}
If $\alpha(s)$ is as in~\eqref{alpha def} in the statement of Theorem~\ref{1pt chordal}, then $\wt{\xi}(s) = 1-\alpha(s)$. 
 \end{remark}

\begin{proof}[Proof of Proposition~\ref{bdy haus upper}]
For $\delta>0$ and $s\in (-1,1)$, let
\eqbn
\wt\Theta^{ s;*}_\delta(\BB D\setminus K_t)  := \wt\Theta^{ s;*}(\BB D\setminus K_t)  \cap \left\{x\in  \partial \BB D : |x-i| , |x+i| \geq \delta ,\quad 1- |f_t^{-1}(x)| \geq \delta \right\}  ,
\eqen  
where $*$ stands for $\geq$ in the case $s \geq 0$ or $\leq$ in the case $s<0$.
The reason for this definition is that it will allow us to apply the estimates of Proposition~\ref{upper bound infty} after a change of coordinates from $\BB D$ to $\BB H$. 
By countable stability of Hausdorff dimension, to prove~\eqref{bdy haus upper eqn}, it is enough to show that a.s.\ 
\eqbn \label{dim lower show}
 \mathcal H^\beta(\wt\Theta^{ s;*}_\delta(\BB D\setminus K_t) )  = 0 \quad \forall \:\delta >0, \quad \forall \: \beta > \wt\xi(s) .
\eqen 
 
Henceforth fix $\delta$, $\beta$, and $s$ as above. Also let $s' \in [0,s)$ (if $s \geq 0$) or $s' \in (s , 0)$ (if $s < 0$) be chosen in such a way that $\wt\xi(s') < \beta$. 

For $n\in \BB N$ and $k \in \{1,\dots, 2^n\}$, let
\eqb  \label{B_n^k}
B_n^k : = \left\{w \in \BB D \,:\,  \frac{  \pi (k-1) }{2^{n-1} } \leq  \op{arg} w  \leq \frac{ \pi k}{2^{n-1}} , \quad   2^{-n } \leq 1-|w| \leq  2^{-n + 1}  \right\} .
\eqe 
Let $E_n^k$ be the event there is a $w \in B_n^k$ with $1-|f_t^{-1}(w)| \geq   \delta /2$ and
\eqb \label{deriv at w}
\begin{dcases}
&|(f_t^{-1})'(w)| \geq 2^{n s'} ,\quad\text{if}\quad s\geq 0\\
&|(f_t^{-1})'(w)| \leq 2^{n s'} ,\quad\text{if}\quad s  <0 .
\end{dcases}
\eqe 
Each $B_n^k$ can be covered by at most an $(n,k)$-independent constant number of balls of radius $<2^{-n-1}$, and each point of $B_n^k$ lies at distance at least $2^{-n}$ from $\partial\BB D$. So, the Koebe distortion and growth theorems imply that for sufficiently large $n$, on the event $E_n^k$ if $z$ is the center of one of these balls then $|(f_t^{-1})'(z)|$ is at least (if $s \geq 0$) or at most (if $s < 0$) an $(n,k)$-independent constant times $2^{n s'}$ and $1-| f_t^{-1}(z)| \geq \delta/4$. 

For $n\in\BB N$, let $\mcl K_n$ be the set of those $k\in \{1,\dots, 2^n\}$ such that $\exp(i  \pi k/2^{n-1}) $ lies at distance at least $\delta/2$ from $-i$ and $i$. By Proposition~\ref{upper bound infty} and a change of coordinates to $\BB H$, whenever $k\in\mcl K_n$,  
\eqb \label{P(E_k^n)}
\BB P(E_n^k ) \preceq 2^{-n( 1-\wt\xi(s') )}    
\eqe 
where the implicit constant is independent of $n$ and uniform for $k\in\mcl K_n$.
 
For $n\in \BB N$ and $k \in \{1,\dots,2^n\}$, let
\eqbn
I_n^k := \left\{x\in\partial \BB D \,:\, \frac{\pi (k-1) }{2^{n-1} } \leq  \op{arg} x  \leq \frac{ \pi k}{2^{n-1}} \right\} .
\eqen
For $m\in \BB N$, let $\mathcal I_m$ be the collection of those intervals $I_n^k$ for pairs $(n,k)$ such that $n\geq m$, $k\in\mcl K_n$, and $E_n^k$ occurs. We claim that for each $m\in\BB N$, $\mathcal I_m$ is a cover of $\wt\Theta^{s;*}_\delta(\BB D\setminus K_t)$. Indeed, if $x  \in \wt\Theta^{s;*}_\delta(\BB D\setminus K_t)$, then for any $m\in \BB N$ we can find $n\geq m$ and $w\in \BB D$ with $1-|w| \leq 2^{-n}$, $\op{arg} w =  \op{arg} x $, $|(f_t^{-1})'(w)| \geq (1-|w|)^{- s'}$ (resp.\ $|(f_t^{-1})'(w)| \leq (1-|w|)^{- s'}$ if $s < 0$), and $1-|f_t^{-1}(w)| \geq \delta/2$. The point $w$ lies in $B_n^k$ for some pair $(n,k)$ with $I_{n,k} \in \mcl I_m$. Since $\op{arg} w =  \op{arg} x $, we have $x \in I_{n,k}$ for this choice of $(n,k)$. 

Now, observe that~\eqref{P(E_k^n)} implies
\eqb \label{haus bdy sum}
  \BB E\left( \sum_{I\in\mathcal I_m} (\op{diam} I)^\beta \right)   \asymp \sum_{n=m}^\infty \sum_{k\in\mcl K_n} 2^{-n  \beta  } \BB P(E_n^k)  
 \preceq \sum_{n=m}^\infty 2^{-n( \beta - \wt\xi(s')   )}   .
\eqe
This tends to $0$ as $m\rta \infty$ since $\beta > \wt\xi(s')$ (by our choice of parameters above). Since $\mathcal I_m$ is a covering of $\wt\Theta^{s;*}_\delta(\BB D\setminus K_t)$ by intervals of diameter tending to zero as $m\rta \infty$, this proves $\mathcal H^\beta(\wt\Theta_\delta^{ s;*}(\BB D\setminus K_t) ) = 0$. 

If $s \in [-1,1] \setminus [s_- , s_+]$, then $\wt\xi(s) < 0$, so the right side of~\eqref{haus bdy sum} for $\beta = 0$ decays exponentially fast in $m$. Thus the expected number of sets in $\mathcal I_m$ tends to zero exponentially fast, and it follows from the Borel Cantelli lemma that a.s.\ $\mcl I_m = \emptyset$ for sufficiently large $m$.
Hence a.s.\ $\wt\Theta^{ s;*}_\delta (\BB D\setminus K_t)=  \emptyset$ for each $\delta> 0$. In fact, it is clear from the definition of $\mcl I_m$ and the definition of the event $E_n^k$ from~\eqref{deriv at w} that~\eqref{bdy haus max} also holds.
\end{proof}

\subsection{Upper bound for the Hausdorff dimension of the subset of the curve}
 \label{curve upper sec}

In this subsection we will use Theorem~\ref{1pt forward} to give an upper bound for the Hausdorff dimension of the sets $\Theta^{s;\geq}(D )$ and $\Theta^{s;\leq}(D )$ of Section~\ref{multifractal def} with $D = D_\eta$ as in Theorem~\ref{main thm}. We will work with a slight variant of the sets of Section~\ref{multifractal def}. For a domain $D\subset \BB C$, a conformal map $\phi : \BB D\rta D$, $s\in \BB R$, and $u >  0$, let 
\eqb \label{theta^su def}
 \Theta^{s;u}(D) := \left\{x\in\partial D :  s-u\leq   \limsup_{\ep\rta 0}  \frac{\log |\phi'((1-\ep )\phi^{-1}(x) )|}{-\log \ep  }  \leq s+u  \right\}.   
\eqe  
  
\begin{lem} \label{haus curve upper u}
Let $\eta$ be a chordal $\op{SLE}_\kappa$ from $-i$ to $i$ in $\BB D$ and let $D_\eta$, $\xi(s)$, $s_-$, and $s_+$ be as in Theorem~\ref{main thm}. Then a.s.\
\eqb \label{haus curve eqn}
\dim_{\mathcal H} \Theta^{s;u}(D_\eta) \leq \xi(s ) + o_u(1), 
\eqe  
whenever $s \in [s_- , s_+]$ and $s < 1$, and a.s.\ $ \Theta^{s;u}(D_\eta) =\emptyset$ for sufficiently small $u$ otherwise. The $o_u(1)$ in~\eqref{haus curve eqn} tends to $0$ as $u\rta 0$ and can be taken to be uniform for $s$ in compact subsets of $(-1,1)$. 
\end{lem} 

\begin{remark}
If $\alpha(s)$ is as in~\eqref{alpha def}, $\gamma(s)$ is as in~\eqref{gamma def}, and $  \xi(s)$ is as in~\eqref{xi(s)},  
\eqb\label{xi choice}
 \xi(s ) = 2 - \frac{\gamma(s)}{1-s} =  \frac{1-\alpha(s)  }{1-s }  . 
\eqe
\end{remark}

To prove Lemma~\ref{haus curve upper u} we first need the following lemma.

\begin{lem}\label{ab and su}
Let $D\subset\BB C$ be a simply connected domain and let $\phi : \BB D\rta D $ be a conformal map. Suppose $x\in \Theta^{s;u}(D)$ for some $s\in (-1,1)$ and $u \in (0,1-|s|)$.  
There is a sequence of points $(w_k)$ in $  D$ converging to $x$ such that
\eqb \label{ab and su deriv}
  \frac{-s-u}{1-s+u}  \leq  \liminf_{k\rta\infty}  \frac{\log |(\phi^{-1})'(w_k)|}{-\log\op{dist}(w_k,\partial D)}  \leq \limsup_{k\rta\infty}  \frac{\log |(\phi^{-1})'(w_k)|}{-\log\op{dist}(w_k,\partial D)} \leq   \frac{-s+u}{1-s-u }  .
\eqe 
and   
\eqb \label{ab and su dist}
\limsup_{k\rta \infty} \frac{ \log |w_k-x|}{ -\log \op{dist}(w_k , \partial D)}  \leq -  \frac{1-s-u }{1-s+u }   .
\eqe 
\end{lem}
\begin{proof} 
Let $x\in \Theta^{s;u}(D)$ and for $\ep > 0$, put $z_\ep =  \phi((1-\ep ) \phi^{-1}(x))$. 
By the definition~\eqref{theta^su def} of $\Theta^{s;u}(D)$, $|\phi'((1-\ep ) \phi^{-1}(x))| \leq  \ep^{ -s +u - o_\ep(1)}$
and for any $k\in \BB N$, we can find $\ep_k  > 0$ with $\ep_k\rta 0$ as $k\rta\infty$ such that 
\eqb \label{z_ep deriv}
|(\phi^{-1} )'(z_{\ep_k} )| = |\phi'((1-\ep_k ) \phi^{-1}(x))|^{-1} \in \left[ \ep_k^{  s + u  + 1/k}  , \ep_k^{s-u- 1/k} \right] .
\eqe  
By the Koebe quarter theorem,  
\eqb \label{z_ep dist}
\op{dist}(z_{\ep_k} , \partial D) \asymp \ep_k |(\phi^{-1} )'(z_{\ep_k} )|^{-1} \in \left[ \ep_k^{1- s + u  + 1/k}  , \ep_k^{1 - s-u- 1/k} \right] .
\eqe 
Hence~\eqref{ab and su deriv} holds with $w_k = z_{\ep_k}$. 
By \cite[Proposition~2.7]{lawler-viklund-tip}, $v (x ;\ep)  \leq \ep^{1-s-u -o_\ep(1)}$, where $v(x ;\ep)$ is the length of the image of the curve $t\mapsto z_t$ for $t\in [0,\ep]$. Consequently, $|z_\ep - x| \leq \ep^{1-s-u - o_\ep(1)} $. Combining this with~\eqref{z_ep dist} yields~\eqref{ab and su dist}.
\end{proof}

We note that in verifying~\eqref{ab and su dist} we used that the definition of~\eqref{theta^su def} of $\Theta^{s;u}(D)$ involves a limsup instead of a liminf. This is the reason why the sets $\Theta^{s;\geq}(D)$ and $\Theta^{s;\leq}(D)$ from~\eqref{theta} are defined with a limsup rather than a liminf.

\begin{proof}[Proof of Lemma~\ref{haus curve upper u}] 
The statement for $s \notin [s_- , s_+]$ follows from the analogous statement in Proposition~\ref{bdy haus upper}, so we henceforth assume $s \in [s_- , s_+]$. 

By countable stability of Hausdorff dimension, to prove~\eqref{haus curve eqn}, it is enough to show that a.s.\ $ \mathcal H^\beta( \Theta^{s ,u }(D_\eta) \cap B_d(0) )  = 0$ for each $\beta >  \xi(s  ) + o_u(1)$, and each $d\in (0,1)$. It is moreover enough to prove the result restricted to the event $\mathcal G(\Psi_\eta , \mu) \cap \mathcal G(\Psi_\eta^- , \mu)$ (in the notation of Theorem~\ref{1pt forward}) for an arbitrary choice of $\mu\in\mathcal M$. 
 
Fix $u\in (0,1-|s|)$ and let
\[
r  >   \frac{1-s-u}{1-s+u} .
\]
Note that we can take $r = 1 - o_u(1)$. For $n\in \BB N$ let $\mathcal D^n = 2^{-n(1-s) -4} \BB Z^2$ be the dyadic lattice of mesh size $2^{-n(1-s) -4}$.  
For $z\in \mathcal D^n$, let $B_0^n(z) , B_1^n(z)$, $B_2^n(z)$, and $B_3^n(z)$ be the disks centered at $z$ of radii $2^{-n(1-s) - 4 }$, $2^{-n(1-s )-2}$, $2^{-n(1-s )+2}$, and $ 2^{-n(1-s ) r +1 }$, respectively. 

Define $\Psi_\eta$ as in Section~\ref{time infty setup sec}. For $z\in \BB D$ let $E^n(z)$ be the event that the following occurs.
\begin{enumerate}
\item $\eta \cap B^n_2(z) \not=\emptyset$ and $\eta \cap B^n_1(z) = \emptyset$. \label{3disks hit} 
\item There is a $w\in B_0^n(z)$ with $2^{-n(s + 2u  )} \leq  |\Psi_\eta '(w)| \leq   2^{-n(s  -  2u  )}$.\label{3disks deriv} 
\end{enumerate}

On $E^n(z)$,  
\eqbn
  \op{dist}(z , \partial D_\eta ) \asymp 2^{-n(1-s )} \quad \op{and} \quad  2^{-n(s + 2u )} \preceq  |\Psi_\eta '(z)| \preceq 2^{-n(s -2u )} ,
\eqen 
with constants uniform in $B_d(0)$ (the inequality for $|\Psi_\eta'|$ follows from the Koebe distortion theorem). 
So, by Proposition~\ref{1pt forward}, 
\eqb\label{E^n(z) prob}
\BB P\left( E^n(z)  \cap  \mathcal G(\Psi_\eta , \mu) \cap \mathcal G(\Psi_\eta^- , \mu) \right) \preceq  2^{-n(\gamma(s ) - 2\gamma_0(s) u )}   
\eqe  
with constants uniform in $B_d(0)$.

Let $\mathcal U^n$ be the set of disks $B^{ n}_3(z)$ for $z\in \mathcal D^n$ such that $z\in B_d(0)$ and $E^n(z)$ occurs. Note that the cardinality of the set of disks which can belong to $\mathcal U^n$ is at most a universal constant times $ 2^{2n(1-s  )}$.  
We claim that $ \Theta^{s;u}(D_\eta) \cap B_d(0) \subset \bigcup_{n \geq N} \bigcup_{B^n_3(z) \in\mathcal U^n} B^n_3(z)$ for each $N\in \BB N$.

Indeed, suppose $x\in   \Theta^{s;u} (D_\eta ) \cap B_d(0)$. 
By Lemma~\ref{ab and su}, we can find a sequence $n_k \rta \infty$ and a sequence of points $w_k \in D_\eta$ converging to $x$ such that for each $k$, $2^{-n_k (1-s   )  -2} \leq \op{dist}(w_k , \partial D_\eta ) \leq 2^{-n_k(1-s    ) }$, $|w_k - x| \leq   2^{-n_k(1-s  )r}$, and $2^{-n_k(s +2u )} \leq |\Psi_\eta '(w_k)| \leq 2^{-n_k (s  -2u)}$.   

Each $w_k$ belongs to $B^{n_k}_0(z)$ for some $z\in \mathcal D^{n_k}$. Our hypothesis on the distance from $w_k$ to $\partial D_\eta$ implies that condition~\ref{3disks hit} in the definition of $E^{n_k}(z)$ hold for this $z$. Clearly, condition~\ref{3disks deriv} also holds for this $z$. 
Thus for such a $z$, $E^n(z)$ holds and $x\in B_3^n(z)$ (here we use the condition on $|w_k - x|$). This proves our claim. 

Thus, for any $m\in \BB N$, $\bigcup_{n\geq m} \mathcal U^n$ is a cover of $ \Theta^{s;u}(\partial D_\eta ) \cap B_d(0)$. Each set in this cover has diameter $\preceq 2^{-m(1-s  ) r}$ and by~\eqref{E^n(z) prob},  
\begin{align} \label{haus curve sum}
\BB E\left( \BB 1_{\mathcal G(\Psi_\eta , \mu) \cap \mathcal G(\Psi_\eta^- , \mu)} \sum_{n=m}^\infty \sum_{U\in \mathcal U^n} (\op{diam} U)^\beta \right) &\preceq \sum_{n=m}^\infty \sum_{z\in\mathcal D_n \cap B_d(0)} 2^{-n \beta(1-s )r    } \BB P\left( E^n(z)\cap\mathcal G(\Psi_\eta , \mu) \cap \mathcal G(\Psi_\eta^- , \mu) \right) \nonumber \\
&\preceq \sum_{n=m}^\infty 2^{2n(1-s  )} 2^{-n\beta (1-s  ) r}  2^{-n(\gamma(s)- 2\gamma_0(s) u   )}    .
\end{align}
This tends to $0$ as $m\rta \infty$ provided 
\[
\beta >  \frac{ 2(1-s) - (  \gamma(s) + 2\gamma_0(s) u )}{(1-s) r} = \xi(s) + o_u(1) ,
\]
where the $o_u(1)$ can be taken to be uniform for $s$ in compact subsets of $(-1,1)$. Since $\mu$ is arbitrary we conclude that $ \mathcal H^\beta( \Theta^{s;u } (\partial D_\eta ) \cap B_d(0))  = 0$ for any such $\beta$. 
\end{proof} 

From Lemma~\ref{haus curve upper u}, we can deduce the upper bounds on $\dim_{\mathcal H} \Theta^{s;\geq}(D_\eta)$ and $\dim_{\mathcal H}(\Theta^{s;\leq}(D_\eta))$ in Theorem~\ref{main thm}. 

\begin{prop} \label{haus curve upper}
Suppose we are in the setting of Theorem~\ref{main thm}. Then a.s.\
\alb
& \dim_{\mathcal H}  \Theta^{s; \geq }(D_\eta) \leq   \xi(s)  ,\qquad  \frac{\kappa}{4} \leq s \leq s_+ \\
&  \dim_{\mathcal H}  \Theta^{s; \leq }(D_\eta) \leq  \xi(s)  ,\qquad   s_- \leq s \leq \frac{\kappa}{4}  .
\ale
\end{prop}
\begin{proof}
For $s\leq \kappa/4$ and any $n\in\BB N$,  
\eqb \label{theta dim union}
\Theta^{s;\leq} (D_\eta )\subset \bigcup_{j=m_0}^{m_1} \Theta^{ j/n ; 1/n}(D_\eta) ,
\eqe 
where $m_0$ is the greatest integer such that $m_0/n \leq s_-$ and $m_1$ is the least integer such that $m_1/n\geq s$. The dimension function $s'\mapsto \xi(s')$ is increasing on $[s_-,\kappa/4]$. In the case when $s \leq \kappa/4$ and $s < 1$ (this latter condition is only relevant when $\kappa = 4$), our desired upper bound for $\dim_{\mathcal H}  \Theta^{s; \leq }(D_\eta) $ therefore follows from Lemma~\ref{haus curve upper u} and~\eqref{theta dim union} upon sending $n\rta\infty$. In the case when $\kappa = 4$ and $s=1$, the upper bound instead follows from the fact that $\dim_{\mcl H} \eta \leq 3/2 = \xi(1)$~\cite{beffara-dim}. A similar argument gives the upper bound for $\dim_{\mathcal H}  \Theta^{s; \geq }(D_\eta)$ when $s\geq \kappa/4$. 
\end{proof}
 
\subsection{Upper bound for the integral means spectrum}
\label{ims upper sec}

In this subsection we will prove the upper bound for the bulk integral means spectrum of the SLE curve in Corollary~\ref{ims cor}. In light of Lemma~\ref{ims zero one}, it will be enough to prove an upper bound for the bulk integral means spectrum of $\BB D\setminus \eta^t$ for given $t\geq 0$ in the case of an ordinary $\op{SLE}_\kappa$ from $-i$ to $i$ in $\BB D$ for $\kappa \leq 4$. 

\begin{prop} \label{ims upper}
Let $\kappa \in (0,4]$ and let $ \xi_{\op{IMS}}(a)$ be defined as in Corollary~\ref{ims cor}. Let $\eta$ be a chordal $\op{SLE}_\kappa$ from $-i$ to $i$ in $\BB D$. For each $t > 0$ and each $a \in \BB R$, a.s.\ $\op{IMS}_{\BB D\setminus \eta^t}^{\op{bulk}}(a) \leq \xi_{\op{IMS}}(a) $.
\end{prop}
\begin{proof}
Let $(f_t)$ be the centered Loewner maps for $\eta$, as defined in Section~\ref{disk sec}. 
The basic idea of the proof is to split up $\bdy B_{1-\ep}(0)$ into the sets where $(f_t^{-1})'(z) \approx \ep^{-s}$ for specified $s$; bound the expected Lebesgue measure of each such set using Proposition~\ref{upper bound infty}; then for each $a$ look at which value of $s$ makes the greatest contribution to the integral defining the integral means spectrum. 

For $\delta>0$, let $U_t(\delta)$ be the set of $z\in \BB D\setminus \eta^t$ with $1-|f_t^{-1}(z)| \geq \delta$ and $|z-i| , |z+i | \geq \delta$. Also define the sets $A_\ep^\zeta(f_t^{-1} )$ as in Section~\ref{ims sec} (immediately following~\eqref{I_zeta def}). For any given $\zeta  > 0$ there a.s.\ exists (random) $\delta>0$ such that $A_\ep^\zeta(f_t^{-1}) \subset \partial B_{1-\ep}(0) \cap U_t( \delta)$ for sufficiently small $\ep$. Therefore, it is enough to show that for each $\delta > 0$ and each $\beta   >  \xi_{\op{IMS}}(\av) $, a.s.\
\eqb \label{U(delta) sup}
\limsup_{\ep\rta 0} \frac{\log \int_{\partial B_{1-\ep}(0) \cap  U_t( \delta)} |(f_t^{-1})'(z)|^\av \, dz }{-\log \ep} \leq \beta .
\eqe

Fix $\delta>0$ and $\beta >  \xi_{\op{IMS}}(\av)$ as above. Also fix $t> 0$ and let $s_-$ and $s_+$ be as in the statement of Theorem~\ref{main thm}. For $n\in \BB N$ and $k\in \{0,\dots,n\}$, let 
\[
u_n = \frac{s_+-s_-}{n} \quad\op{and} \quad s_k^n = s_0 + k u_n  .
\]
For $n\in\BB N$, $\ep > 0$, and $k \in\{0,\dots,n\} $, let
\eqbn
A^n_\ep(k) := \left\{ z \in \bdy B_{1-\ep}(0) \cap U_t(\delta) : \ep^{-s_k^n  + u_n} \leq |(f_t^{-1})'(z)| \leq \ep^{-s_k^n -u_n} \right\} .
\eqen
Also let $A_\ep^n(-)$ (resp. $A_\ep^n(+)$) be the set of $z\in \partial B_{1-\ep}(0) \cap U_t(\delta) $ such that $|(f_t^{-1})'(z)| \leq \ep^{-s_- + u_n}$ (resp. $|(f_t^{-1})'(z)| \geq \ep^{-s_+ - u_n}$). 
Let $\ell_\ep^n(k)$ be the Lebesgue measure of $A_\ep^n(k)$ and let $\ell_\ep^n(\pm)$ be the Lebesgue measure of $A_\ep^n(\pm)$. 

In what follows, we require implicit constants to be independent of $\ep$, but not of $n$ or $k$, and we denote by $o_n(1)$ a term which tends to $0$ as $n\rta\infty$ and does not depend on $k$ or $\ep$. 
 
By construction, we have $\partial B_{1-\ep}(0) \cap U_t(\delta) = A_\ep^n(-)\cup A_\ep^n(+) \cup \bigcup_{k=0}^n A^n_\ep(k) $, whence
\alb
\int_{\partial B_{1-\ep}(0) \cap  U_t( \delta) }   |(f_t^{-1})'(z)|^a \, dz &\preceq \sum_{k=0}^n \ep^{-a s_k^n + o_n(1) } \ell_\ep^n(k)  + \ep^{-a s_-} \ell_\ep^n(-)  + \ep^{-a s_+} \ell_\ep^n(+) .
\ale  
By~\eqref{bdy haus max} of Lemma~\ref{bdy haus upper}, for each $n\in\BB N$ there a.s.\ exists a random $\ep_0^n > 0$ such that for $\ep \in (0, \ep_0^n]$, the sets $A_\ep^n(-)$ and $A_\ep^n(+)$ are empty. Hence for $\ep \in (0, \ep_0^n]$,  
\eqb \label{int < sup}
 \int_{\partial B_{1-\ep}(0) \cap  U_t( \delta)  } |(f_t^{-1})'(z)|^{a} \, dz   \preceq   \max_{k\in\{0,\dots,n\} }  \ep^{-a s_k^n + o_n(1) } \ell_\ep^n(k)   .
\eqe  
By Proposition~\ref{upper bound infty} and a change of coordinates to $\BB D$, for $k\in \{0,\dots,n\}$, 
\eqbn
\BB E(\ell_\ep^n(k) ) \preceq \ep^{\alpha(s_k^n)  + o_n(1) }  ,
\eqen
where $\alpha(s) = 1-\wt\xi(s)$ is the exponent from Theorem~\ref{1pt chordal}. 
By Chebyshev's inequality,  
\eqb \label{chebyshev k}
\BB P\left( \ep^{-a s_k^n} \ell_\ep^n(k) > \ep^{-\beta  } \right) \preceq \ep^{\alpha(s_k^n)    - a s_k^n  +\beta  + o_n(1)}. 
\eqe 

We have
\eqb \label{ims inf}
\inf_{s \in [s_- , s_+]} \left(\alpha(s_k^n) - a s_k^n \right) = -\xi_{\op{IMS}}(a).
\eqe 
Note that the range $(a_- , a_+)$ in Corollary~\ref{ims cor} is precisely the set of $a\in \BB R$ for which the minimizer in~\eqref{ims inf} is not equal to $s_-$ or $s_+$. 
It follows that for sufficiently large $n\in\BB N$, depending only on $\beta$, 
\eqbn
\BB P\left(  \max_{k\in\{0,\dots,n\} }  \ep^{-a s_k^n } \ell_\ep^n(k) > \ep^{-\beta} \right) \preceq \ep^{\beta - \xi_{\op{IMS}}(a) + o_n(1)} .
\eqen
Since $\beta > \xi_{\op{IMS}}(a)$, if $n\in\BB N$ is chosen sufficiently large (depending only on $\beta$ and $a$), then the Borel-Cantelli lemma together with~\eqref{int < sup} implies that a.s. 
\[
\int_{\partial B_{1-2^{-j}}(0) \cap  U_t(\delta) }   |(f_t^{-1})'(z)|^a \, dz \leq 2^{-j\beta} 
\]
for sufficiently large $j\in\BB N$. By the Koebe distortion theorem, it follows that a.s.
\[
\limsup_{\ep\rta 0} \frac{\log \int_{\partial B_{1-\ep}(0) \cap  U_t(\delta) }   |(f_t^{-1})'(z)|^a \, dz }{-\log \ep} \leq \beta .
\]
This proves (\ref{U(delta) sup}), and hence the statement of the proposition.
\end{proof}

\section{Event at the hitting time} 
\label{2pt setup sec}

In this section we introduce an event which will serve as the basic building block for the ``perfect points" which we will use to prove our lower bounds on the Hausdorff dimensions of $\Theta^s(D_\eta)$ and $\wt\Theta^s(D_\eta)$ in Section~\ref{2pt sec}, and prove upper and lower bounds for the probability of this event. Roughly speaking, this amounts to transferring the derivative estimates of Theorem~\ref{1pt forward} from the setting where we grow the \emph{entire} curve $\eta$ to the setting where we only grow $\eta$ and its time reversal until they hit a small ball centered at the origin.

\subsection{Definitions and statement of estimates}
\label{1pt hitting event sec}

Let $\wt d \in (0,1)$ and let $x,y\in \bdy\BB D$ with $|x-y| \geq \wt d$. 
Suppose $\eta : [0,\infty ] \rta \ol{\BB D}$ is a random simple curve in $\ol{\BB D}$ from $x$ to $y$. 
We recall the notation 
\eqbn
\eta^t = \eta([0,t]),\qquad \eta = \eta([0,\infty])  
\eqen 
from Section~\ref{basic notation}. 
Let $\ol \eta$ be the time reversal of $\eta$. 
We also introduce the abbreviation
\eqb \label{ball abbrv}
\mcl B_\beta := B_{e^{-\beta}}(0) ,\quad \forall \beta > 0.
\eqe 

Let $\beta > 0$, $q\in (-1/2,\infty)$, $a\in (0,1/4)$, $u,   c   > 0$, and $\mu\in\mathcal M$. The parameter $\beta$ corresponds to $\log\ep^{-1}$ (so we will eventually be sending $\beta\rta\infty$); the parameter $q$ corresponds to $s/(1-s)$ for $s$ the parameter of Theorem~\ref{main thm}; and $a,c,$ and $\mu$ are auxiliary parameters used in regularity events.

Let $E = E_\beta^{q;u}(\eta ;   a, c , \mu )$ be the event that the following holds.
\begin{enumerate} 
\item Let $\tau_\beta$ (resp.\ $\ol\tau_\beta$) be the first time that $\eta$ (resp.\ $\ol\eta$) hits $\partial \mcl B_\beta$. Then $\tau_\beta , \ol\tau_\beta < \infty$. \label{E hit}
\item Let $\phi_\beta :\BB D \setminus (\eta^{\tau_\beta} \cup \ol\eta^{\ol\tau_\beta}) \rta\BB D$ be the unique conformal transformation which takes $x^+$ to $-i$, $y^-$ to $i$, and the midpoint $m$ of $[x,y]_{\partial\BB D}$ to 1.  Then $c^{-1} e^{-\beta (q+u)} \leq |\phi_{\beta }'(0)| \leq c e^{-\beta  (q-u)}$. \label{E phi}
\item The harmonic measure from $0$ in $\BB D \setminus (\eta^{\tau_{\beta }} \cup \ol\eta^{\ol\tau_{\beta }})$ of each of the two sides of $\eta^{\tau_{\beta }} $ and each of the two sides of $ \ol\eta^{\ol\tau_{\beta }}$ is at least $a$. \label{E top}
\item $\mathcal G'(\eta^{\tau_{\beta }} \cup \ol\eta^{\tau_{\beta }}, \mu)$ occurs (Definition~\ref{G' def}). \label{E G} 
\end{enumerate}   

The goal of this section is to estimate the probability of the event $E$.  
 
\begin{prop} \label{1pt at hitting}
Suppose $x,y\in\bdy\BB D$ with $|x-y| \geq \wt d$. Let $\eta$ be a chordal $\op{SLE}_\kappa$ from $x$ to $y$ in $\BB D$ and define $E =E_\beta^{q;u}(\eta;    a ,c  , \mu   )$ as above. Let $\gamma(s)$ be the exponent from~\eqref{gamma def} and let
\begin{align} \label{gamma* def}
\gamma^*(q) &:=  (q+1) \gamma\left(\frac{q}{1+q} \right)  = \frac{8 \kappa + 8 \kappa q + (4 - \kappa)^2 q^2}{8 (\kappa + 2\kappa q)} .  
\end{align} 
There exists a function $\gamma_0^* : (-1/2 , \infty) \rta (0,\infty)$ (with $\gamma_0^*(q)$ depending only on $q$) and a $u_* = u_*(q) > 0$ such that the following is true for each $q \in (-1/2, \infty)$ and $u \in (0,u_*]$. For any choice of parameters $\beta , \mu , a , c$ as above,
\eqb\label{1pt hitting upper}
\BB P(E) \preceq e^{-\beta (\gamma^*(q) - \gamma^*_0(q) u )} .
\eqe
Moreover, there exists $\mu = \mu(\wt d) \in\mathcal M$ such that for each $a\in (0,1/4)$, $c>0$, and $u \in (0,u_*]$, there exists $\beta_*  = \beta_*(u,a,c) > 0$ such that for $\beta \geq \beta_* $,
\eqb \label{1pt hitting lower}
\BB P(E ) \succeq e^{-\beta(\gamma^*(q)     + \gamma^*_0(q) u ) }.
\eqe 
The implicit constants in~\eqref{1pt hitting upper} and~\eqref{1pt hitting lower} are independent of $\beta$ and uniform for $x,y\in\bdy\BB D$ with $|x-y| \geq \wt d$, but may depend on the other parameters.
\end{prop}

We will prove the estimates~\eqref{1pt hitting upper} and~\eqref{1pt hitting lower} in the next two subsections. The upper bound~\eqref{1pt hitting upper} is a straightforward consequence of the upper bound in Theorem~\ref{1pt forward} and the Markov property, but the lower bound will take more work. For the proof, we write
\eqb \label{1pt at hitting sigma algebra}
\mathcal F_\beta := \sigma\left( \eta|_{[0,\tau_\beta]} ,\, \ol\eta|_{[0,\ol\tau_\beta]} \right) .
\eqe 
 
\subsection{Upper bound}
\label{1pt at hitting upper sec}
 
Here we will prove the upper bound~\eqref{1pt hitting upper} in Proposition~\ref{1pt at hitting}, which is a straightforward consequence of Theorem~\ref{1pt forward}. 
 
\begin{proof}[Proof of Proposition~\ref{1pt at hitting}, upper bound]
This will follow by growing the middle part of $\eta$ connecting $\eta^{\tau_\beta}$ and $\ol\eta^{\ol\tau_\beta}$, noting that it behaves in a regular manner with positive probability, then applying the upper bound of Theorem~\ref{1pt forward}. 

More precisely, let $\wh \eta$ be the image under $\phi_{\beta }$ of the part of $\eta$ lying between $\eta(\tau_{\beta })$ and $\ol\eta(\ol\tau_{\beta })$. Let $\wh x = \phi_{\beta}(\eta(\tau_{\beta}))$ and $\wh y = \phi_{\beta}(\ol\eta(\ol\tau_{\beta}))$, so that the conditional law of $\wh\eta$ given the $\sigma$-algebra $\mathcal F_{\beta}$ of~\eqref{1pt at hitting sigma algebra} is that of an $\op{SLE}_\kappa$ from $\wh x$ to $\wh y$ in $\BB D$. Note that $|\wh x - \wh y|$ is typically small when $\beta$ is large. For $C>1$, let $\wh E = \wh E(C)$ be the event that the following occurs. 
\begin{enumerate}
\item $\wh \eta$ does not exit $\phi_{\beta}(\mcl B_{1})$. \label{hat E stay}
\item Let $D_{\wh \eta}$ be the domain lying to the right of $\wh \eta$, as in Section~\ref{time infty setup sec}. Then $\phi_{\beta}(0) \in D_{\wh\eta}$ and $C^{-1} (1-|\phi_{\beta}(0)|) \leq \op{dist}(\phi_{\beta}(0) ,\partial D_{\wh\eta}) \leq C (1-|\phi_{\beta}(0)|)$. \label{hat E dist}
\item Let $\Phi_{\wh\eta} : D_{\wh\eta} \rta \BB D$ be the conformal map taking fixing $-i$, $i$, and 1. Then $C^{-1} \leq |\Phi_{\wh\eta}'(\phi_{\beta}(0))| \leq C$. 
\end{enumerate} 
It follows from condition~\ref{E top} in the definition of $E$ and Lemma~\ref{miller-wu-dim-2.3} that we can find a $C > 0$ depending only on $a$ such that for sufficiently large $\beta$, $\BB P(\wh E | E) \succeq 1$. Thus
\eqb \label{E' and hat E}
\BB P(E) \asymp \BB P(E \cap \wh E) .
\eqe
So, it will suffice to prove an upper bound for $\BB P (E \cap \wh E )$. 
 
Let $s \in (-1,1)$ and $\ep > 0$ be chosen so that
\eqb \label{s ep choice}
\frac{s}{1-s} = q   ,\qquad \ep^{1-s  } = e^{-\beta }.
\eqe 
Let $D_\eta$, $\Psi_\eta$, $\Psi_\eta^-$, and $\mathcal E_\ep^{s; u}(\eta ,0 ; c)$ be as in Section~\ref{time infty setup sec}. It follows from Lemma~\ref{G dist} and condition~\ref{E G} in the definition of $E$ that
\eqb \label{E implies G}
E\subset \mathcal G(\phi_{\beta} , \mu') 
\eqe 
for some $\mu' \in \mathcal M$ depending only on $\mu$. By combining this with condition~\ref{hat E stay} in the definition $\wh E$ we see that $ E\cap \wh E \subset \mathcal G(\Psi_\eta , \mu') \cap \mathcal G(\Psi_{ \eta}^- , \mu')$ for some (possibly smaller) $\mu'\in\mathcal M$ depending only on $\mu$. We furthermore have
$\Psi_\eta = \Psi_{\wh\eta} \circ \phi_{\beta}$.
Hence 
\[
E\cap \wh E \subset  \mathcal E_\ep^{s; u}(\eta ,0 ; c) \cap  \mathcal G(\Psi_\eta , \mu') \cap \mathcal G(\Psi_\eta^- , \mu') 
\]
for suitable choice of $\mu'$ and $c$. Thus~\eqref{1pt hitting upper} follows from~\eqref{E' and hat E} and the upper bound in Theorem~\ref{1pt forward}.   Note that we can take the dependence on $u$ to be linear (with slope depending on $q$) since the exponent in the upper bound in Theorem~\ref{1pt forward} depends smoothly on $s\in (-1,1)$ and $u > 0$ sufficiently small. 
\end{proof}

\subsection{Lower bound}
\label{1pt at hitting lower sec}

The proof of the lower bound in Proposition~\ref{1pt at hitting} will take substantially more work than the proof of the upper bound. The basic idea is to stop $\eta$ and $\ol\eta$ at times $t_0$ and $\ol t_0$ for which the following is true. On the event $\mathcal E_\beta^{s;u}(\cdot)$ of Theorem~\ref{1pt forward}, the conformal map from $\BB D\setminus (\eta^{t_0} \cup \ol\eta^{\ol t_0})$ to $\BB D$ which takes $x^{+}$ to $-i$, $y^{-}$ to $i$, and $m$ to $1$ has the same derivative behavior at $0$ as the conformal map $\Psi_\eta : D_\eta \rta \BB D$ with the same normalization; the points $\eta(t_0)$ and $\ol\eta(\ol t_0)$ are at distance slightly less than $e^{-\beta}$ from 0; and the conditional law of the remainder of the curve given $\eta^{t_0} \cup \ol\eta^{\ol t_0}$ is that of a chordal $\op{SLE}_\kappa$. We also need to require that $\eta(t_0)$ and $\ol\eta(\ol t_0)$ are sufficiently far apart in a conformal sense, so that they do not immediately link up after times $t_0$ and $\ol t_0$. We then condition on $\eta^{t_0} \cup \ol\eta^{\ol t_0}$ and use standard arguments to get that the curves reach $\mcl B_\beta$ without any pathological behavior. The main difficulty in the proof is constructing the times $t_0$ and $\ol t_0$. 

We start by inductively defining a means of growing $\eta$ and $\ol\eta$ in an alternating fashion to get an increasing family of hulls $K_t\subset \BB D$. Assume $\eta$ (resp.\ $\ol\eta$) is parameterized in such a way that its image under the conformal map $\BB D\rta \BB H$ taking $-i$ to 0, $i$ to $\infty$, and $0$ to $i$ (resp.\ the reciprocal of this conformal map) is parameterized by half plane capacity. Let $\sigma_1$ be the first time $t$ that $\op{hm}^0(\eta^t ; \BB D\setminus \eta^t) = 1/2$. This time is a.s.\ finite since a Brownian motion started from $0$ has probability at least $1/2$ to hit $\eta$ before $\partial\BB D$. For $t\leq \sigma_1$, let $K_t = \eta^t$. Let $\ol\sigma_1$ be the first $\ol t$ that either $\op{hm}^0(\eta^{\ol t} ; \BB D\setminus (\eta^{\sigma_1} \cup \ol\eta^{\ol t}) ) =1/2$ or $\ol\eta(\ol t) = \eta(\sigma_1)$. For $t \in [\sigma_1  ,\sigma_1 + \ol\sigma_1 ]$ let $K_t = \eta^{\sigma_1} \cup \ol\eta^{t -\sigma_1}$. 

Inductively, suppose $n \geq 2$ and $\sigma_{n-1}$, $\ol\sigma_{n-1}$, and $K_t$ for $t\leq \sigma_{n-1 } + \ol\sigma_{n-1}$ have been defined. If $K_{\sigma_{n-1} + \ol\sigma_{n-1} } =\eta$ we let $\sigma_n = \sigma_{n-1}$ and $\ol\sigma_n = \ol\sigma_{n-1}$. Otherwise, let $\sigma_n$ be the least $t \geq \sigma_{n-1}$ such that either $\op{hm}^0( \eta^t   ; \BB D\setminus ( \eta^t \cup \ol\eta^{\ol\sigma_{n-1}}))  =1/2$ or $\eta(t) =\ol\eta(\ol\sigma_{n-1})$. Let $K_t = \eta^{t-\ol\sigma_{n-1}} \cup \ol\eta^{\ol\sigma_{n-1}}$ for $t\in [\sigma_{n-1}  + \ol\sigma_{n-1} ,  \sigma_n + \ol \sigma_{n-1}]$. Let $\ol\sigma_n$ be the first time $\ol t \geq \ol \sigma_{n-1}$ such that either $\op{hm}^0(  \ol\eta^{\ol t} ; \BB D\setminus ( \eta^{\sigma_n} \cup \ol\eta^{\ol t}))  =1/2$ or $\ol\eta(\ol t ) = \eta(\sigma_n)$. Let $K_t = \eta^{\sigma_n} \cup \ol\eta^{  t - \sigma_n}$ for $t\in [\sigma_n  +\ol\sigma_{n-1}, \sigma_n + \ol \sigma_n]$. 

For each $t \geq 0$, let $T_t$ (resp.\ $\ol T_t$) be the time such that $\eta(T_t)$ (resp.\ $\ol\eta(\ol T_t)$) is the tip of the part of $\eta$ (resp.\ $\ol\eta$) included in $K_t$. Observe that the Markov property and reversibility of SLE imply that for each $t$, the conditional law of $\eta \setminus K_t$ given $K_t$ is that of a chordal $\op{SLE}_\kappa$ from $\eta(T_t)$ to $\ol\eta(\ol T_t)$ in $\BB D\setminus K_t$.  

It is not immediately obvious from the construction that the curves $\eta$ and $\ol\eta$ grown according to the above procedure will a.s.\ link up in finite time. To show that this is indeed the case, we first need the following endpoint continuity property. 

\begin{lem} \label{sigma hm converge}
Let $\sigma_\infty = \lim_{n\rta\infty} \sigma_n$ and $\ol \sigma_\infty = \lim_{n\rta\infty} \ol \sigma_n$ (the limits necessarily exist by monotonicity).  Let $K_\infty = \eta^{\sigma_\infty} \cup \ol\eta^{\ol\sigma_\infty}$. Then a.s.\ 
\[
\lim_{n\rta\infty} \op{hm}^0(\eta^{\sigma_n} ; \BB D\setminus K_{\sigma_n  +\ol\sigma_n}) = \lim_{n\rta\infty} \op{hm}^0(\eta^{\sigma_n} ; \BB D\setminus K_{\sigma_n  +\ol\sigma_{n-1}})  = \op{hm}^0(\eta^{\sigma_\infty}  ; \BB D\setminus K_\infty) 
\]
and
\[
\lim_{n\rta\infty} \op{hm}^0(\ol\eta^{\ol\sigma_n} ; \BB D\setminus K_{\sigma_n  +\ol\sigma_n}) = \lim_{n\rta\infty} \op{hm}^0(\ol\eta^{\ol\sigma_{n-1}} ; \BB D\setminus K_{\sigma_n  +\ol\sigma_{n-1}})  = \op{hm}^0(\ol\eta^{\ol\sigma_\infty}  ; \BB D\setminus K_\infty)  .
\]
\end{lem}
\begin{proof} 
We a.s.\ have $0\notin\eta$ so it is a.s.\ the case that for each $\ep > 0$, we can find a random $\delta > 0$ such that for any $z\in \eta$, the probability that a Brownian motion started from $0$ hits $B_\delta(z)$ before leaving $\BB D$ is at most $\ep$. 
By a.s.\ continuity of $\eta$, we can a.s.\ find a (random) $N\in\BB N$ such that for $n\geq N$, $\eta([\sigma_n , \sigma_\infty])\subset   B_\delta(\eta(\sigma_\infty))$ and $\ol\eta([\ol\sigma_n , \ol\sigma_\infty])\subset  B_\delta(\ol\eta(\ol\sigma_\infty))$. Hence with probability at least $1-\ep$, a Brownian motion started from $0$ exists $\BB D\setminus K_{\sigma_n + \ol\sigma_n}$ at the same place it exits $\BB D\setminus K_\infty$. This proves the limits involving $K_{\sigma_n + \ol\sigma_n}$. The limits involving $K_{\sigma_n + \ol\sigma_{n-1}}$ are proven similarly.
\end{proof}

We now check that the curves a.s.\ meet in finite time and that the meeting point divides the curve into two segments whose harmonic measure from $0$ is approximately the same. 

\begin{lem} \label{sigma limit}
We a.s.\ have $K_\infty = \eta$. Let $z_\infty = \eta(\sigma_\infty) = \ol\eta(\ol \sigma_\infty)$ be the meeting point. On the event that $0$ lies to the right of $\eta$ and $\op{dist}(0 , \eta) \leq e^{-\beta}$, it holds a.s.\ that $\op{hm}^0(\eta^{\sigma_\infty} ; D_\eta)$ and $\op{hm}^0(\ol\eta^{\ol \sigma_\infty} ; D_\eta)$ are each at least $1/2 - o_\beta(1)$, where the $o_\beta(1)$ is a deterministic quantity which tends to $0$ as $\beta\rta 0$. 
\end{lem}
\begin{proof} 
First we argue that $K_\infty = \eta$. Suppose not. Almost surely, either $\op{hm}^0(\eta^{\sigma_\infty} ; \BB D\setminus K_\infty)$ or $\op{hm}^0(\ol\eta^{\ol\sigma_\infty} ; \BB D\setminus K_\infty)$ is $< 1/2$. Suppose $\op{hm}^0(\eta^{\sigma_\infty} ; \BB D\setminus K_\infty) < 1/2$. The other case is treated similarly. By Lemma~\ref{sigma hm converge} we a.s.\ have $\op{hm}^0(\eta^{\sigma_n} ; \BB D\setminus K_{\sigma_n  +\ol\sigma_{n-1}}) < 1/2$ for sufficiently large $n$. By definition of $\sigma_n$ this can be the case only if $\eta(\sigma_n) = \ol\eta(\ol\sigma_{n-1})$ which implies $K_\infty = \eta$. 

It is immediate from Lemma~\ref{sigma hm converge} and the definition of the times $\sigma_n$ and $\ol\sigma_n$ that $\op{hm}^0(\eta^{\sigma_\infty} ; D_\eta)$ and $\op{hm}^0(\ol\eta^{\ol \sigma_\infty} ; D_\eta)$ are each at most $1/2$. Furthermore, the Beurling estimate implies $\op{hm}^0(\partial\BB D ; D_\eta)  = o_\beta(1)$. Hence
\[
\op{hm}^0(\eta^{\sigma_\infty} ; D_\eta) = 1 - \op{hm}^0(\ol\eta^{\ol \sigma_\infty} ; D_\eta)  -  \op{hm}^0(\partial\BB D ; D_\eta) \geq 1/2 - o_\beta(1)
\]
and similarly for $\eta^{\ol\sigma_\infty}$. 
\end{proof}

The following lemma is what allows us to compare conformal maps defined on the domains $\BB D\setminus K_t$ to those defined on the domains $D_\eta$ (the derivative behavior of conformal maps on the latter domain can be controlled using Theorem~\ref{1pt forward}).

\begin{lem} \label{sigma good time}
For $t\geq 0$, let $\Phi_t $ be the conformal map from the connected component of $ \BB D\setminus K_t$ with $1$ on its boundary (this component is all of $\BB D\setminus K_t$ if the curves have not linked up before time $t$) to $ \BB D$ taking $x^+$ to $-i$, $y^-$ to $i$, and $m$ to 1 and let $\wt\Phi_t$ be the conformal map from this same connected component to $\BB D$ which fixes $0$ and takes $m$ to 1. 
Also let $\Psi_\eta : D_\eta \rta \BB D$ be as in Section~\ref{time infty setup sec}. For $\mu\in\mathcal M$, there is a $C>1$ and a $\beta_*  > 0$, depending only on $\mu$ such that if $\beta \geq \beta_* $ then on the event $\mathcal G(\Psi_\eta , \mu) \cap \{\op{dist}(0 , \eta) \leq e^{-\beta}\} \cap \{0\in D_\eta\}$, there a.s.\ exists a time $\tau > 0$ such that the following holds.
\begin{enumerate}
\item $\op{dist}(0,K_\tau ) \leq C \op{dist}(0,\eta)$. \label{sigma good dist} 
\item $C^{-1} |\Psi_\eta'(0)|  \leq   |\Phi_\tau '(0)|  \leq C |\Psi_\eta'(0)|$.  \label{sigma good deriv}
\item $\wt\Phi_\tau(\eta(T_\tau))$ and $\wt\Phi_\tau(\ol\eta(\ol T_\tau))$ lie in the left semi-circle $[i,-i]_{\partial\BB D}$. 
\item $  \op{hm}^0(\eta\setminus K_\tau  ; D_\eta) \geq 1/4 + o_\beta(1) $, with the $o_\beta(1)$ deterministic and depending only on $\beta$. \label{sigma good mid} 
\end{enumerate}
\end{lem}
\begin{proof}
Throughout, we assume we are working on the event $\mathcal G(\Psi_\eta , \mu) \cap \{\op{dist}(0 , \eta) \leq e^{-\beta}\} \cap \{0\in D_\eta\}$ and we require all implicit constants to be deterministic and depend only on $\mu$. 

Let $\wt\Psi_\eta  : D_\eta \rta \BB D$ be the conformal map which fixes $0$ and takes 1 to 1. If $z_\infty$ is as in Lemma~\ref{sigma limit} then by the conformal invariance of harmonic measure,
\eqb\label{z_infty to 1}
|\wt\Psi_\eta(z_\infty)  + 1| = o_\beta(1)    ,
\eqe
at a deterministic rate.  

Let $\tau$ be the first time $t$ that $\wt\Psi_\eta(\eta(T_t))$ and $\wt\Psi_\eta(\ol \eta(\ol T_t))$ are both in $[i,-i]_{\partial\BB D}$. By Lemma~\ref{sigma limit} such a $t$ necessarily exists provided $\beta$ is at least some universal constant. Let $\wt A_\tau  =   [\wt\Psi_\eta(\ol \eta(\ol T_\tau)) , \wt\Psi_\eta(\eta(T_\tau)) ]_{\partial\BB D}$  be the arc of the left side of $\bdy\BB D$ separating these two points. By continuity one of the two endpoints of $\wt A_\tau$ is $-i$ or $i$ so by~\eqref{z_infty to 1}, $\op{hm}^0(\wt A_\tau ; \BB D )\geq 1/4-o_\beta(1)$. Furthermore, the harmonic measure from $0$ in $\BB D$ of each of the two arcs connecting $\wt A_\tau$ and 1 is at least $1/4 - o_\beta(1)$. 

Let $A_\tau = \wt\Psi_\eta^{-1 } (\wt A_\tau) = \eta \setminus K_\tau$.
By conformal invariance of harmonic measure, $\op{hm}^0(\eta^{T_\tau} ; D_\eta)$, $\op{hm}^0(\ol\eta^{\ol T_\tau} ; D_\eta)$, and $\op{hm}^0(A_\tau; D_\eta)$ are each at least $1/4 - o_\beta(1)$. 
By Lemma~\ref{phi hm arc} (applied with $I = [-i,i]_{\partial\BB D}$ and $\phi = \Phi_\tau$) we have $\op{dist}(0 , K_\tau) \asymp \op{dist}(0 , \eta)$ and $|\Phi_\tau'(0)| \asymp |\Psi_\eta'(0)|$. Since $\wt\Psi_\eta(\eta(T_\tau))$ and $\wt\Psi_\eta(\ol\eta(\ol T_\tau))$ lie in $[i,-i]_{\partial\BB D}$ and removing $A_\tau$ can only increase the harmonic measure from $0$ of parts of $\partial D_\eta$ outside of $A_\tau$, we find that $\wt\Phi_\tau(\eta(T_\tau))$ and $\wt\Phi_\tau(\ol\eta(\ol T_\tau))$ must lie in $[i,-i]_{\partial\BB D}$.  
Thus, the conditions of the lemma hold for this choice of~$\tau$. 
\end{proof}

\begin{figure}\label{E0 sigma fig}
\begin{center}
\includegraphics[scale=.8]{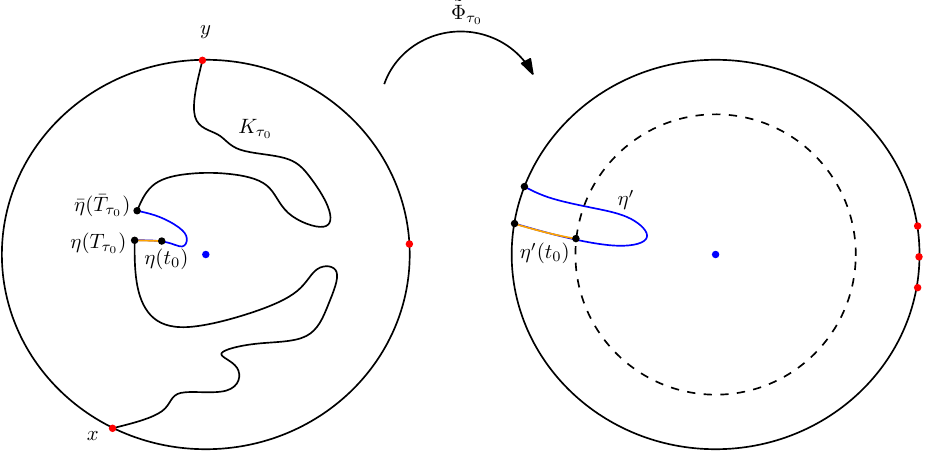}
\end{center}
\caption{An illustration of the argument of Lemma~\ref{E0 event} in the case $\{T' < \infty\}$. The hull $K_{\tau_0}$ is shown in black. the curve $\eta'$ and its pre-image under $\wt\Phi_{\tau_0}$ are shown in blue. The extra part of the curve which we grow after growing $K_{\tau_0}$ is shown in orange.}
\end{figure}

The following lemma is the main input in the proof of the lower bound in Proposition~\ref{1pt at hitting}: it provides times $t_0 , \ol t_0 > 0$ for which $|\eta(t_0) - \ol\eta(\ol t_0)|$ is of order $e^{-\beta}$, the derivative of a conformal map $\BB D\setminus (\eta^{t_0} \cup \ol \eta^{\ol t_0}) \rta\BB D$ with the same normalization as $\phi_\beta$ is of order $e^{-\beta q}$, the points $\eta(t_0)$ and $\ol\eta(\ol t_0)$ are well separated in the harmonic measure sense, and the conditional law of the ``middle" segment of $\eta$ given $\eta^{t_0} \cup \ol \eta^{\ol t_0}$ is that of an SLE$_\kappa$. Once we have these times, we just need to grow a little bit more of $\eta$ and $\ol\eta$ after times $t_0$ and $\ol t_0$, respectively, to get the estimate of Proposition~\ref{1pt at hitting}.

\begin{lem} \label{E0 event}
Let $v > 0$, $\zeta > 0$, and $\mu_0\in\mathcal M$. For $\beta>0$ and two times $t , \ol t> 0$, let $E_\beta^0(t, \ol t)  = E_\beta^0(t, \ol t ; v , \zeta , \mu_0)$ be the event that the following occurs.
\begin{enumerate}
\item $32 e^{-  \beta }  \leq \op{dist}(0 , \eta^t \cup \ol\eta^{\ol t}) \leq e^{-\beta(1-v)}$. \label{E0 dist}
\item Let $ \phi_{t,\ol t} : \BB D\setminus (\eta^t \cup \ol\eta^{\ol t}) \rta \BB D$ be the conformal map which takes $x^+$ to $-i$, $y^-$ to $i$, and $m$ to 1. Then $ e^{-\beta(q+v)} \leq    |\phi_{t,\ol t}'(0)| \leq e^{- \beta(q-v)}$.\label{E0 deriv}
\item Let $ \psi_{t,\ol t}  : \BB D\setminus (\eta^t \cup \ol\eta^{\ol t}) \rta \BB D$ be the conformal map which fixes $0$ and takes 1 to 1. Then $|\psi_{t,\ol t}(\eta(t )) - \psi_{t,\ol t}(\ol\eta(\ol t))| \geq \zeta$. \label{E0 separated}
\item $\mathcal G'(\eta^t \cup \ol\eta^{\ol t} , \mu_0)$ occurs. \label{E0 delta}
\end{enumerate}
There is a deterministic $\zeta>0$ and $\mu_0\in\mathcal M$, independent of $v$ and $\beta$, such that for each $v>0$, there exists $\beta_* = \beta_*(v , \wt d)   > 0$ such that for each $\beta \geq \beta_* $, there exist random times $t_0 $ and $ \ol t_0$ such that
\eqb\label{P(E_0)}
\BB P(E_\beta^0(t_0, \ol t_0)   ) \succeq e^{-\beta(\gamma^*(q)  + \gamma_0^*(q) v) } ,
\eqe
where here $\gamma^*(q) $ and $ \gamma_0^*(q)$ are as in Proposition~\ref{1pt at hitting} and the implicit constant is independent of $\beta$. Furthermore, we can choose $t_0$ and $\ol t_0$ in such a way that the conditional law given $\eta^{t_0} \cup \ol\eta^{\ol t_0}$ of the part of $\eta$ between $\eta(t_0)$ and $\ol\eta(\ol t_0)$ on the event $E_\beta^0(t_0 , \ol t_0)$ is that of a chordal $\op{SLE}_\kappa$ from $\eta(t_0)$ to $\ol\eta(\ol t_0)$ in $\BB D\setminus ( \eta^{t_0} \cup \ol\eta^{\ol t_0})$. 
\end{lem}
\begin{proof}
We will deduce the lemma from Theorem~\ref{1pt forward} and Lemma~\ref{sigma good time}. 
Fix $v' \in (0,v/4)$, to be chosen later in a manner depending only on $v$ and $q$, and let $s := q/(q+1)$. If $\beta > 0$ is chosen sufficiently small, in a manner depending only on $v'$ and $q$, then we can find $\ep  =\ep(s , v' , \beta) > 0$ such that
\[
\ep^{1-s } = e^{-\beta(1 -o_{v'}(1) )} \quad \op{and} \quad \ep^{1-s+2v'} \geq 32 e^{-\beta} .
\]
Let $c    >0$ and let $ \mathcal E^{s;v'}_{\ep}(\eta, 0 ; c)$ be the event of Section~\ref{time infty setup sec} (with $v'$ in place of $u$). Let $\Psi_{\eta} : D_{\eta} \rta \BB D$ and $\Psi_{\eta}^- : D_{\eta}^- \rta\BB D$ be as in that subsection. Let $\mu'\in\mathcal M$ and let
\eqbn
\mathcal E :=   \mathcal E^{s;v'}_{\ep}(\eta, 0 ; c ) \cap \mathcal G(\Psi_{\eta} , \mu') \cap \mathcal G(\Psi_{\eta}^- , \mu')  .
\eqen
By Theorem~\ref{1pt forward}, if the parameter $\mu'$ is chosen appropriately (in a manner depending only on $q$) then we can find $\beta_*  > 0$ as in the statement of the lemma such that for each $\beta\geq \beta_*$, 
\[
\BB P(\mathcal E  ) \succeq e^{-\beta ( \gamma^*(q) + \gamma_0^*(q) v' ) } ,
\] 
for an appropriate choice of $\gamma_0^*(q)$ as in Proposition~\ref{1pt at hitting}. Lemma~\ref{G dist} implies that we can find $\mu_0\in\mathcal M$ depending only on $\mu'$ such that 
\eqb \label{t ol t G union}
  \mathcal G(\Psi_{\eta} , \mu') \cap \mathcal G(\Psi_{\eta}^- , \mu')  \subset \bigcap_{t , \ol t \geq 0} \mathcal G'(\eta^t \cup \ol\eta^{\ol t} , \mu_0)  .
\eqe 
 
Let $\tau_0$ be the first time $\tau$ that the first two conditions in the definition of $E_\beta^0(T_\tau , \ol T_\tau)$ are satisfied and that $\wt\Phi_\tau(\eta(T_\tau))$ and $\wt\Phi_\tau(\ol\eta(\ol T_\tau))$ (as defined just above Lemma~\ref{sigma hm converge}) both lie in $[i,-i]_{\partial\BB D}$. By Lemma~\ref{sigma good time} and the definition of $\mathcal E$, if $c$ is chosen sufficiently large then $\tau_0 < \infty$ a.s.\ on $\mathcal E$. Moreover, decreasing $\tau$ only increases $\op{hm}^0(\eta\setminus K_{\tau }  ; D_\eta)$, so on $\mathcal E$ a.s.\
\eqb\label{diam lower bound}
\op{hm}^0(\eta\setminus K_{\tau_0}  ; D_\eta) \geq 1/4 - o_\beta(1).
\eqe

Let $\eta' = \wt\Phi_{\tau_0}(\eta\setminus K_{\tau_0})$, with the parameterization it inherits from $\eta$. By the strong Markov property, the conditional law of $\eta'$ given $K_{\tau_0}$ is that of a chordal $\op{SLE}_\kappa$ from $x':= \wt\Phi_{\tau_0}(\eta(T_{\tau_0} ))$ to $y' := \wt\Phi_{\tau_0}(\ol\eta(\ol T_{\tau_0}))$ in $\BB D$ (here we used that we made $\tau_0$ the \emph{smallest} time for which our desired conditions are satisfied). 

The definition, the event $E_\beta^0(t_0,\ol t_0)$ almost holds with $t_0 = T_{\tau_0}$ and $\ol t_0 = \ol T_{\tau_0}$, but $\wt\Phi_{\tau_0}(\eta(T_{\tau_0}))$ and $\wt\Phi_{\tau_0}(\ol\eta(\ol T_{\tau_0}))$ may be too close together. To this end, we will choose slightly larger times at which the images of the tips of $\eta$ and $\ol\eta$ are separated. Note that~\eqref{diam lower bound} implies  $\op{diam} \eta' \geq\zeta_0$ on $\mathcal E$ for some universal constant $\zeta_0\in (0,1/4)$. 
Let $\ol\eta'$ be the time reversal of $\eta'$, with the parameterization it inherits from $\ol\eta$.

Let $T'$ (resp.\ $\ol T'$) be the first time that $\eta'$ (resp.\ $\ol\eta'$) enters $B_{1-\zeta_0/4}(0)$. Let $T''$ be the first time $t \geq T_{\tau_0}$ that $\op{arg} \eta'(t) \geq \op{arg} x' + \zeta_0/8$. Let $\ol T''$ be the first time $\ol t \geq \ol T_{\tau_0}$ that $\op{arg} \ol \eta'(\ol t) \leq \op{arg} y' - \zeta_0/8$. Since $\op{diam} \eta' \geq\zeta_0$ a.s.\ on $\mathcal E$, either $|x' - y'| \geq \zeta_0/8$ or one of $T', \ol T' ,  T''  $ or $\ol T''$ is finite on this event (if not, then $\eta'$ is contained in the wedge $\{z\in\BB D:   \op{arg} y' - \zeta_0/8 \leq \op{arg} z \leq \op{arg} x' + \zeta_0/8 ,\: |z| \geq 1-\zeta_0/8\}$ and this wedge has diameter $<\zeta_0$). 
Hence the intersection with $\mathcal E$ of at least one of the events $\{|x' - y'| \geq \zeta_0/8\}$, $\{T' < \infty\}$, $\{T''< T'\}$, or $\{\ol T'' < \ol T'\}$ has probability at least $\frac14 \BB P(\mathcal E  ) \succeq e^{-\beta ( \gamma^*(q) + \gamma_0^*(q) v' ) }$. 

It is therefore enough to show that the conclusion of the lemma is true in each of the four possible cases (provided $\beta$ is sufficiently large). We will do this by choosing $t_0$ to be one of $T_{\tau_0} , T', $ or $T''$ and $\ol t_0$ to be one of $\ol T_{\tau_0} , \ol T'$, or $\ol T''$. By the strong Markov property, the last statement of the lemma holds for any such choice. 
Clearly, condition~\ref{E0 dist} in the definition of $E_\beta^0(t_0, \ol t_0)$ holds a.s.\ on $\mathcal E$ for any such choice of $t_0$ and $\ol t_0$ and any $v' \in (0,v)$. By~\eqref{t ol t G union}, condition~\ref{E0 delta} holds for any such choice. By conditions~\ref{sigma good dist} and~\ref{sigma good deriv} in Lemma~\ref{sigma good time}, on $\mcl E$,  
\eqbn
 \frac{ |\Phi_{\tau_0}'(0) |  }{ |\Psi_\eta'(0)|   } \asymp 1   \quad \op{and} \quad 
  \frac{ \op{dist}(0 , K_{\tau_0} ) }{ \op{dist}(0 , \eta)  } \asymp 1
\eqen
with deterministic, $\beta$-independent proportionality constants.
By combining this with Lemma~\ref{phi hm} and condition~\ref{E0 delta} (c.f.\ Remark~\ref{phi hm hypotheses}), we infer that on $\mcl E$, 
\eqb \label{E0 hm compare}
  \frac{ \op{hm}^0(I ; D\setminus K_{\tau_0} )  }{  \op{hm}^0(I ; D\setminus  \eta)   }  \asymp 1 ,
\eqe 
for $I$ a sub-arc of $[-i,i]_{\partial \BB D}$ which is slightly smaller than $[-i,i]_{\partial \BB D}$. For any choice of $t_0$ and $\ol t_0$ as above, we have $K_{\tau_0} \subset (\eta')^{t_0} \cup (\ol\eta')^{\ol t_0}$. Since $4v' < v$,~\eqref{E0 hm compare} and a second application of Lemma~\ref{phi hm} yield condition~\ref{E0 deriv} for large enough $\beta$. 

Finally, we will verify that condition~\ref{E0 separated} holds in each of the four cases (for an appropriate choice of $\zeta >0$ depending only on $\zeta_0$). Here we note that $|x'-y'|$ is proportional to the harmonic measure from $0$ of the boundary arc of $\BB D\setminus ((\eta')^{t_0} \cup (\ol\eta')^{\ol t_0})$ separating $\eta'(t_0)$ from $\ol\eta'(\ol t_0)$. 
\begin{enumerate}
\item If $\BB P( |x' - y'| \geq \zeta_0/8 ,\:\mathcal E) \succeq  e^{-\beta ( \gamma^*(q) + \gamma_0^*(q) v' ) }$ then we can just set $t_0 = T_{\tau_0}$, $\ol t_0 = \ol T_{\tau_0}$, and $\zeta = \zeta_0/8$.  
\item If $\BB P(T' < \infty ,\:\mathcal E) \succeq e^{-\beta ( \gamma^*(q) + \gamma_0^*(q) v' ) }$ then we set $t_0 = T'$ and $\ol t_0 = \ol T_{\tau_0}$. A Brownian motion has probability at least a constant $\zeta >0$ depending only on $\zeta_0$ to exit $B_{1-\zeta_0/16}(0)$ within distance $\zeta_0/4$ of 1 and then make a counterclockwise loop around the origin before leaving $\BB D\setminus B_{1-\zeta_0/8}(0)$. In this case it necessarily exits $\BB D\setminus (\eta')^{T'}$ on the left side of $(\eta')^{T'}$. See Figure~\ref{E0 sigma fig} for an illustration in this case. 
\item If $\BB P(T'' < T' ,\:\mathcal E) \succeq e^{-\beta ( \gamma^*(q) + \gamma_0^*(q) v' ) }$ then we set $t_0 = T'\wedge T''$ and $\ol t_0 = \ol T_{\tau_0}$. A Brownian motion has probability at least a constant $\zeta >0$ depending only on $\zeta_0$ to exit $\BB D$ before hitting any point outside of $\BB D\setminus B_{1-\zeta_0/8}(0)$ whose argument is not between $\op{arg} x'$ and $\op{arg} x'  + \zeta_0/8$. If this is the case and $T' \leq T''$, then a Brownian motion necessarily exits $\BB D\setminus (\eta')^{t_0}$ on the left side of $(\eta')^{t_0}$.  
\item The case for $\{\ol T'' < \ol T'\}$ is treated in the same manner as the case for $\{T'' < T'\}$. 
\end{enumerate}
Thus we have exhausted all possible cases and we conclude that condition~\ref{E0 separated} holds. 
\end{proof}
 
\begin{proof}[Proof of Proposition~\ref{1pt at hitting}, lower bound]
Suppose $\zeta>0$, $\mu_0 \in \mcl M$, and random times $t_0$, $\ol t_0$ are chosen so that the conclusion of Lemma~\ref{E0 event} holds. Let $v  >0$ and let $\beta_*  >0$ be chosen as in Lemma~\ref{E0 event}. Let $\beta \geq \beta_* $ and let $E_\beta^0 = E_\beta^0(t_0 , \ol t_0 , v , \zeta , \mu_0)$ be as in Lemma~\ref{E0 event}. We need to transfer the estimate of Lemma~\ref{E0 event} from the setting when we stop at times $t_0$ and $\ol t_0$ to the setting when we stop at times $\tau_\beta$ and $\ol\tau_\beta$. The idea of the proof is to consider the hitting times of $\eta$ and $\ol\eta$ of logarithmically many balls centered at $0$ whose radii differ by an exponential factor and argue that at each scale, there is a positive probability that the curves continue to behave nicely. We then apply the strong Markov property and multiply over all of the scales. 

To this end, let $\wt\beta = -\log \op{dist}(0 , \eta^{t_0} \cup \ol\eta^{\ol t_0} )$. Note that on $E_\beta^0$,
\[
\beta(1-v) \leq \wt\beta \leq \beta - \log 32 .
\]
Also fix $r \in ( \log16 , \log 32)$. We will consider the hitting times of the balls $\mcl B_{\wt\beta + k r}$ for $k\in\BB N$. 

We start with the case $k=1$, which is slightly different.  
Let $\eta_1$ be the image under the map $\psi_{t_0 , \ol t_0} : \BB D\setminus (\eta^{t_0} \cup \ol\eta^{\ol t_0}) \rta \BB D$ which fixes $0$ (defined as in Lemma~\ref{E0 event}) of the part of $\eta$ between $\eta(t_0)$ and $\ol\eta(\ol t_0)$ and let $x_1$ and $y_1$ be its endpoints. 
Let $\tau_1'$ (resp.\ $\ol\tau_1'$) be the first time $\eta_1$ (resp.\ $\ol\eta_1$) hits $\psi_{t_0 , \ol t_0}(\mcl B_{\wt\beta + r})$, so that $\psi_{t_0,\ol t_0}(\eta(\tau_{\wt\beta + r})) = \eta_1(\tau_1')$ and similarly for $\ol\eta$. 
Let $G_1$ be the event that the following holds.
\begin{enumerate}
\item $|\eta_1(\tau_1') - \ol\eta_1(\ol\tau_1')| \geq (1/32) e^{-r}$.
\item $\eta_1^{\tau_1'} \cup \ol\eta_1^{\ol\tau_1'} \subset \psi_{t_0 , \ol t_0}(\mcl B_1)$.  
\item $\eta_1^{\tau_1'} \cup \ol\eta_1^{\ol\tau_1'}$ is disjoint from the $\zeta/2$-neighborhood of the segment connecting $0$ and the midpoint of the shorter arc between $x_1$ and $y_1$. 
\end{enumerate}  
By the Koebe quarter theorem,
\[
\mcl B_{r+\log 16}  \subset \psi_{t_0 , \ol t_0}(\mcl B_{ \wt\beta + r}) \subset \mcl B_{r -  \log 16}  .
\]
Hence by Lemma~\ref{miller-wu-dim-2.3}, condition~\ref{E0 separated} in the definition of $E_\beta^0$, and the last statement of Lemma~\ref{E0 event}, $\BB P(G_1 | E_\beta^0)$ is at least a $\beta$-independent positive constant.

Now we consider the case $k\geq 2$. 
For $k =   1,2,3, \ldots$, let $\wt\psi_k$ be the map from $\BB D\setminus ( \eta^{\tau_{\wt\beta  + k r}} \cup \ol\eta^{\ol\tau_{\wt\beta  + k r}}      )$ to $\BB D$ with $\wt\psi_k(0) = 0$ and $\wt{\psi}_k'(0)>0$.
For $k\geq 2$, let $  \eta_k$ be the image under $\wt\psi_{k-1}$ of the part of $\eta$ which lies between $\eta(\tau_{\wt\beta +(k-1)r})$ and $\ol\eta(\ol\tau_{\wt\beta  + (k-1) r})$. Then the law of $\eta_k$ given $\mathcal F_{\wt\beta  + (k-1) r }$ (defined as in~\eqref{1pt at hitting sigma algebra}) is that of a chordal $\op{SLE}_\kappa$ from $x_k := \wt\psi_{k-1}( \eta(\tau_{\wt\beta  + (k-1) r})   )$ to $y_k:=\wt\psi_{k-1}(\ol\eta(\ol\tau_{\wt\beta  + (k-1) r}))$. 
Let $\ol\eta_k$ be the time reversal of $\eta_k$. 

Let $\tau_k'$ and $\ol\tau_k'$ be the hitting times of $\wt\psi_{k-1}(   \mcl B_{\wt\beta +k r} )$ by $\eta_k$ and $\ol\eta_k$, respectively, so that $\wt\psi_{k-1}(\eta(\tau_{\wt\beta +k r})) = \eta_k(\tau_k')$ and similarly for $\ol\eta$. Fix $\delta>0$ and for $k\geq 1$ let $G_k$ be the event that $\eta^{\tau_k}$ (resp.\ $\ol\eta^{\ol\tau_k}$) is contained in the $\delta$-neighborhood of the segment $[x_k , 0]$ (resp.\ $[y_k , 0]$).  

By the Koebe quarter theorem, whenever $\wt\psi_{k-1}$ is defined we have 
\[
\mcl B_{r+\log 16}  \subset \wt\psi_{k-1}(\mcl B_{ \wt\beta +kr}) \subset \mcl B_{r -  \log 16}  .
\] 
By conformal invariance of harmonic measure, on $G_{k-1} $ for $k\geq 2$, $|  x_k -  y_k|$ is at least a universal constant provided $\delta$ is taken sufficiently small. 
It now follows from Lemma~\ref{miller-wu-dim-2.3} that for each $k\geq 2$, 
\eqb \label{hitting G prob}
\BB P\left(G_k \,|\, E_\beta^0 \cap \bigcap_{j=1}^{k-1} G_j \right) \geq p
\eqe 
 for some $p>0$ which depends only on $\delta$. 

Let $k_*$ be the least integer $k$ such that $k r + \wt\beta  \geq \beta$. Note that $k_* \leq  \beta v/r$. Let
\eqbn
G^* := \bigcap_{k= 1}^{k_*} G_k  .
\eqen 
We will now argue that $E_\beta^0 \cap G^* \subset E$, then complete the proof by establishing an appropriate lower bound for $\BB P(E_\beta^0 \cap G^*)$ provided $v \ll u$ is chosen appropriately. 
 
It is clear that on the event $E_\beta^0 \cap G^*$, conditions~\ref{E hit},~\ref{E top}, and~\ref{E G} in the definition of $E  $ hold provided we take $\delta$ sufficiently small, depending on $a$. 
It remains to deal with condition~\ref{E phi}. For $k\geq 1$, let $\wh\eta_k$ be the curve obtained by connecting  $\eta(\tau^*_{\wt\beta +k r})$ and $\ol\eta(\ol\tau^*_{\wt\beta  + k r})$ via the arc of $  \mcl B_{\wt\beta +k r}  $ which does not disconnect $0$ from $[x_* , y_*]_{\partial\BB D}$. Let $\Psi_{\wh \eta_k}$ be the conformal map from the connected component of $\BB D\setminus \wh\eta_k$ containing $[x_* , y_*]_{\partial\BB D}$ on its boundary to $\BB D$ which takes $x_*$ to $-i$, $y_*$ to $i$, and the midpoint of $[x_* , y_*]_{\partial\BB D}$ to 1. By Lemma~\ref{phi hm arc},  
\eqb \label{hitting deriv compare}
C^{-1}  |\Psi_{\wh \eta_k}'(0)| \leq |\phi_{\beta'}'(0)| \leq C |\Psi_{\wh \eta_k}'(0)| ,\qquad\forall \beta' \in [\wt\beta  + (k-1) r ,   \wt\beta  + kr] ,\qquad \forall k\geq 2  
\eqe
on $G^*$, for some deterministic $C>1$ depending only on $a$, $r$, and $\mu$. A similar statement holds for $k=1$ provided we replace $C$ with a constant $C_1>0$ which is allowed to depend on $\zeta$, but not $\beta$. 
 
The estimate~\eqref{hitting deriv compare} implies in particular that $|\phi_{\wt\beta+(k-1)r}'(0)|$ and $|\phi_{\wt\beta + kr}'(0)|$ differ by a factor of at most $C^2$. Iterating~\eqref{hitting deriv compare} at most $\beta v/r$ times shows that on $G^*$,
\[
C_1^{-1} C^{- 2 \beta  v/r   } e^{-\beta (q+ v )}  \leq    |\phi_{\beta }'(0)| \leq C_1 C^{ 2 \beta v/r }e^{-\beta (q-  v )} .
\]
If we choose $v$ such that $v \leq u/3$ and $C^{ 2 v/r} \leq e^{(1\wedge \gamma_0^*(q)) u/3}$ and $\beta$ sufficiently large that $C_1 e^{\beta \gamma_0^*(q) u/3} \geq c$, then condition~\ref{E phi} in the definition of $E $ holds on $E_\beta^0\cap G^*$.  
By possibly further shrinking $v$, we can arrange that $p^{v/r} \leq e^{\gamma_0^*(q) u/2}$ where $p$ is the parameter from~\eqref{hitting G prob}. From Lemma~\ref{E0 event}, our estimates for the conditional probabilities of the $G_k$'s, and our choice of parameters above,
\eqbn
\BB P(E  ) \geq  \BB P(G_1 |E_0)  p^{\beta v / r -1} e^{-\beta(\gamma^*(q)  + \gamma_0^*(q) v) } \succeq e^{-\beta (\gamma^*(q) + \gamma_0^*(q) u)} .   \qedhere
\eqen 
\end{proof}

\section{Two point estimate}
\label{2pt sec}

\subsection{Outline of the two-point estimate} \label{sec-2pt-outline}

The goal of this section is to prove our two-point estimate which will lead to a lower bound for the Hausdorff dimensions of the sets $\Theta^s(D_\eta)$ and $\wt\Theta^s(D_\eta)$ in Theorem~\ref{main thm}. In particular, we will define events $E_n(z)$ for $z\in\BB D$ and $n\in\BB N$ and show that if $E_n(z)$ occurs for every $n\in\BB N$ (i.e., $z$ is a \emph{perfect point}) then $z\in\Theta^s(D_\eta)$; and that the correlation of $E_n(z)$ and $E_n(w)$ is small when $|z-w|$ is large, in a quantitative sense (Proposition~\ref{prop-2pt-estimate}). 
The proof of this latter correlation estimate uses the theory of imaginary geometry to get long-range independence for certain events. 

Throughout this section, we will consider the following setup. 
Let $\chi = 2/\sqrt\kappa - \sqrt\kappa/2$ and let $\lambda = \pi/\sqrt\kappa$ be the imaginary geometry parameters from~\eqref{chi lambda}. Let $h$ be a zero boundary GFF on $\BB D$ plus a harmonic function chosen in such a way that if $\psi : \BB H\rta \BB D$ is the conformal map taking $0$ to $-i$, $\infty$ to $i$, and $i$ to $0$, then $h\circ\psi - \chi\op{arg}\psi'$ is a GFF on $\BB H$ with boundary data $-\lambda$ on $(-\infty ,0 ]$ and $\lambda$ on $[0,\infty)$.  By \cite[Theorem~1.1]{ig1} the zero-angle flow line $\eta$ of $h$ started from $-i$ is a chordal $\op{SLE}_\kappa$ from $-i$ to~$i$ in~$\BB D$.\footnote{In the case $\kappa = 4$, we replace flow lines of $h$ with a given angle by level lines of $h$ at a given level (see \cite{ss-dgff,ss-contour,wang-wu-level-lines}). Everything that follows works identically with this replacement. In fact, since (in contrast to the situation for flow lines) the time reversal of a level line is also a level line \cite[Theorem~1.1.5]{wang-wu-level-lines}, some of the proofs are easier for $\kappa=4$. }    
Let $\ol\eta$ be the time reversal of $\eta$. 
Also fix a multifractal spectrum parameter $s\in (-1,1)$ and let $q:= s/(1-s) \in (-1/2,\infty)$. 
 
We will shortly give an outline of the content of the rest of this section, but before we do so we make some general comments about notation. 
\begin{itemize}
\item We continue to use the notation $\mcl B_\beta = B_{e^{-\beta}}(0)$ from~\eqref{ball abbrv}. We also recall the notation $\eta^\tau = \eta([0,\tau])$ and we will always denote the time reversal of a curve by an overbar. 
\item All curves in this section are assumed to have some arbitrary parameterization. The times we consider will only be used to specify certain segments of the curve, and these segments will not depend on the choice of parameterization. 
\item The notation in the remainder of this section is quite heavy, but it is easier to navigate if the reader keeps in mind several conventions.
Objects denoted with a superscript $\fl$ are associated with the \emph{full} curve $\eta$, as opposed to the curve $\eta_{z,j}$ at scale $j$.
Conformal maps denoted by the symbol $\psi$ with some decoration map the complement of some part of $\eta$ (or a conformal image thereof) to $\BB D$ and are required to fix the origin. Conformal maps denoted by $\phi$ or $\Phi$ with some decoration map the complement of some segment of $\eta$ (or a conformal image thereof) to $\BB D$ and are specified by the images of three points on the boundary. Conformal maps denoted by $\pi$ with some decoration map a ``pocket" formed by two auxiliary flow lines to $\BB D$. Conformal maps denoted by $f$ or $g$ with some decoration are automorphisms of $\BB D$. 
\item Much of the notation in this section is illustrated in Figures~\ref{fig-flow-line-full},~\ref{fig-flow-line-stages}, and~\ref{fig-flow-line-maps} and summarized in Section~\ref{sec-2pt-index}.
\end{itemize}

We start in Section~\ref{sec-perfect-setup} by defining an event $E$ depending on parameters $\beta > 0$ and $u \in (0,1)$ (which will eventually be sent to $0$ and $\infty$, respectively), a field $h$ on $\BB D$ with Dirichlet boundary data and its 0-angle flow line $\eta $ started from $x \in \bdy\BB D$ to $y \in \bdy\BB D$ (eventually, we will apply this definition inductively with $\eta $ replaced by the conformal image of a certain segment of our original SLE$_\kappa$ curve $\eta$). The definition of $E$ also involves several constant-order \emph{auxiliary parameters} which we list in Definition~\ref{def-aux-parameter}. Roughly speaking, $E$ is the event that the following hold.
\begin{enumerate}
\item If we run $\eta$ (resp.\ its time reversal) until the first time $\tau$ that it gets within distance $e^{-\beta}$ of the origin then apply a conformal map $\phi : \BB D\setminus (\eta^\tau\cup\ol\eta^{\ol\tau}) \rta \BB D$ normalized so that $\phi(x^+)=-i$, $\phi(y^-) = i$, and $\phi(\text{midpoint of $[x,y]_{\bdy\BB D}$}) = 1$, then $|\phi'(0)|$ is of order $e^{-\beta(q \pm u)}$. \label{item-2pt-outline-deriv}
\item Let $\eta^-$ and $\eta^+$ be flow lines of $h $ started from $\eta(\tau) $, with angles chosen so that they a.s.\ intersect each other. Then $\eta^-$ and $\eta^+$ form a ``pocket" surrounding the origin with diameter of order $e^{-\beta}$ and a roughly round shape. \label{item-2pt-outline-curve}
\end{enumerate}
The first of these two conditions will ensure that the behavior of the derivative of a conformal map from one side of $\eta$ to $\BB D$ has the right derivative behavior and the second condition will allow us to get the long-range independence needed for our two-point estimate. 
We will also prove an estimate (Lemma~\ref{lem-E_z-prob}) for the probability of $E$. 

The actual definition of $E$ will involve several regularity conditions which are needed to rule out various types of pathological behavior. 
We will break the definition up into four steps which each serve a particular purpose in the proof of our two-point estimate. Let us now give a more detailed outline of each of these four steps and its purpose; see Figure~\ref{fig-flow-line-full} for an illustration of the definition and the objects involved. 
 
The first step is to get away from the boundary so that our curve will look like an ordinary SLE$_\kappa$ (even if it was originally an SLE$_\kappa(\ul\rho)$). We grow the curves $\eta $ and $\ol\eta $ up to times $\sigma $ and $\ol\sigma $, respectively, which are approximately equal to the first time these curves hit a certain ball centered at $0$ with small (but $\beta$-independent) size. Our first event $L $ is a list of regularity conditions for $\eta^\sigma$ and $\ol\eta^{\ol\sigma}$. 
The purpose of most of these conditions is to ensure that we can apply Lemma~\ref{rho abs cont} to get that the segment of $\eta $ from $\eta(\sigma)$ to $ \ol\eta(\ol\sigma)$ is close in law to an SLE$_\kappa$ curve, even if $\eta $ is itself an SLE$_\kappa(\rho^L;\rho^R)$ curve for $\rho^L,\rho^R \in (-2,0)$. 
The probability of $L  $ will be of constant order, independent of $\beta$ (Lemma~\ref{lem-L-prob}). We note that the objects in the definition of $L$ are used infrequently outside the proof of Lemma~\ref{lem-wtE-prob}.

The second step takes care of the derivative behavior; in particular, we let $\wt E$ be the event described in item~\ref{item-2pt-outline-deriv} above, with the same regularity conditions appearing on the event of Proposition~\ref{1pt at hitting}.
The event $\wt E$ is the only event in the definition of $E$ whose conditional probability given the previous events is not of constant ($\beta$-independent) order; see Lemma~\ref{lem-E_z-prob} and Proposition~\ref{1pt at hitting}.

Since the behavior of the derivative of a conformal map from the complement of $\eta$ to $\BB D$ can a priori depend on the whole curve $\eta$, we next introduce auxiliary flow lines $ \eta^\pm$ to localize our events. These are flow lines of $h $ started from the point $\eta (\tau )$, with angles chosen so that they a.s.\ bounce off each other, but do not cross. We define an event $F$ which is the intersection of $\wt E$ and the event that these auxiliary flow lines make a pocket surrounding $0$ (which we call $D$) before hitting $\ol\eta^{\ol\tau}$ and satisfy certain regularity conditions.

The key property which these pockets $D$ satisfy, and which is the source of the long-range independence needed for our two-point estimate in Section~\ref{sec-flow-line-prob}, is that, conditional on a pocket, the restrictions of $h $ to the inside and outside of the pocket are conditionally independent (see Lemma~\ref{lem-F-cond}). Since $h$ determines $\eta$ in a local manner, this will lead to independence between certain segments of $\eta$. 
The regularity conditions in the definition of $F $ govern the size and shape of the pocket $ D$ and will be important in Section~\ref{sec-flow-line-analytic} when we compare derivatives of various conformal maps; and also ensure that the points where $\eta$ enters and exits the pocket are separated in the sense of harmonic measure from 0.

Finally, we define $E $ to be the intersection of $F $ and the event that $\eta$ and $\ol\eta $ do not have any pathological behavior between time they hit $\mcl B_\beta$ and the time when they enter the pocket $D$.

In Section~\ref{sec-perfect-setup'}, we define events $E_{z,j}$ for $z\in\BB D$ and $j\in\BB N$ associated with our original field/curve pair $(h,\eta)$ as follows. Fix sequences $\beta_j \rta \infty$ (at a logarithmic rate) and $u_j\rta 0$ (at a very slow rate), which are chosen in Lemma~\ref{lem-beta-u-choice}.  
In the case $j=1$, we apply a conformal automorphism $f_{z,1} : \BB D\rta\BB D$ sending $z$ to $0$ and let $E_{z,1}$ be the event $E$ of Section~\ref{sec-perfect-setup} defined with $\beta = \beta_1$, $u = u_1$, and $f_{z,1}\circ\eta_{z,1}$ in place of $\eta$. 
Inductively, for $j\geq 2$ we let $D_{z,j-1}$ be the pocket formed by the auxiliary flow lines used in the definition of $E_{z,j-1}$, let $\pi_{z,j-1} : D_{z,j-1} \rta \BB D$ be a conformal map which fixes 0, let $E_{z,j}$ be the event $E$ of Section~\ref{sec-perfect-setup} with $\eta $ replaced by the image under $\pi_{z,j-1}$ of the segment of (a conformal image of) $\eta$ which is contained in $D_{z,j-1}$ and which $\beta = \beta_j$ and $u = u_j$. We then set 
\eqbn
E_n(z) := \bigcap_{j=1}^n E_{z,j}
\eqen
See Figure~\ref{fig-flow-line-stages} for an illustration of the definitions of $E_{z,j}$ and $E_n(z)$. 

In Section~\ref{sec-flow-line-analytic}, we use a purely complex analytic argument to prove Lemma~\ref{lem-E_z-basics}, which says that the derivatives of certain conformal maps and the diameters of certain sets are of the correct order on $E_n(z)$. This will be used in Section~\ref{haus lower sec} to show that the perfect points (roughly speaking, those for which $E_n(z)$ occurs for every $n\in\BB N$) all belong to the multifractal spectrum set $\Theta^s(D_\eta)$. The proofs in this subsection are perhaps the most technical ones in this section; the reader who wishes to see only the main ideas of the proof of our two-point estimate may wish to read Lemma~\ref{lem-E_z-basics}, which is the only result from this subsection used in the rest of the proof, and skip the rest of Section~\ref{sec-flow-line-analytic}. 

In Section~\ref{sec-flow-line-prob}, we prove our two-point estimate Proposition~\ref{prop-2pt-estimate} using the auxiliary flow lines in the definitions of our events and various conditioning arguments based on results from~\cite{ig1}. The main idea of the proof is that (roughly speaking) the behavior of the field $h$, and hence also the curve $\eta$, inside the pockets $D_{z,n}^\fl$ and $D_{w,n}^\fl$ formed by the auxiliary flow lines is independent provided these pockets are disjoint, which allows us to get long-range independence for our events. 

Section~\ref{sec-other-settings} contains a discussion about what adaptations one would make to our argument when proving two-point estimates for other sets associated with SLE.

For the convenience of the reader, we have included an index of the notation used in this section in Section~\ref{sec-2pt-index}.

\subsection{Event for an SLE$_\kappa(\rho^L;\rho^R)$ curve coupled with a GFF} \label{sec-perfect-setup}

Fix $\wt d > 0$ and suppose $x,y\in\bdy\BB D$ with $|x-y| \geq \wt d$. Also let $\rho^L ,\rho^R \in (-2,0]$ and let $h$ be a GFF on $\BB D$ with Dirichlet boundary data chosen in such a way that its 0-angle flow line $\eta$ from $x$ to $y$ is an SLE$_\kappa(\rho^L ; \rho^R)$ from $x$ to $y$, with force points located immediately to the left and right of $x$. 
Also fix $u\in (0,1)$ and $\beta > 0$ (we will eventually send $u\rta 0$ and $\beta \rta\infty$). 

All objects in this subsection are allowed to depend on $\rho^L , \rho^R$, and $\kappa$ and we do not make this dependence explicit. We will, however, be careful about dependence on $x$ and $y$ which is why we introduce the parameter $\wt d$. 
 
In this subsection, we will define an event $E$ associated with the curve $\eta $, the field $h $, the parameters $\beta$ and $u$, and several constant-order auxiliary parameters. We will also record an estimate for $\BB P(E)$. 
In the next subsection we will define the events $E_{z,j}$ and the associated objects by replacing $h$ with the conformal image of the restriction of $h$ to a sub-domain and replacing $\eta$ with the corresponding conformal image of a segment of $\eta$. 
See Figure~\ref{fig-flow-line-full} for an illustration of most of the objects defined in this subsection. 

\begin{defn}[Auxiliary parameters] \label{def-aux-parameter}
The \emph{auxiliary parameters} are the objects
$  \Delta  > \wt\Delta  >  1$, $\delta_L , r ,   p_L  \in (0,1)$, $a \in (0,1/4)$, and $\mu , \mu_L , \mu_F  \in\mathcal M$, all chosen in a manner which does not depend on $\beta$ or $u$. 
\end{defn}
 
The auxiliary parameters will be used in the definition of our events below and will be chosen in the following manner. In Lemma~\ref{lem-E_z-prob}, we show that for a given choice of $ r, a$, and $\wt d$, a certain estimate holds provided $ \delta_L$, $p_L$, $\mu,\mu_L,\mu_F$ are chosen sufficiently small, $\Delta$ and $\wt \Delta$ are chosen sufficiently large, and $\beta$ is large enough (depending on all of the auxiliary parameters). 
In Section~\ref{sec-flow-line-analytic}, we make our choice of $r$. The parameter $a$ is allowed to remain arbitrary.  
 
We now proceed with the definition of the event $E$, as outlined in Section~\ref{sec-2pt-outline}. Let $\ol\eta  $ be the time reversal of $\eta $.
We first grow initial segments of $\eta$ and $\ol\eta$ in such a way that the ``middle part" of $\eta$, between these two segments, looks like an ordinary SLE$_\kappa$.  

Let $\sigma$ (resp.\ $\ol\sigma)$ be the first time $\eta$ (resp.\ $\ol\eta$) hits $\mcl B_\Delta$ (or $\infty$ if no such time exists).  
Let $[x^*,y^*]_{\bdy\BB D}$ be the largest sub-arc of $[x,y]_{\bdy\BB D}$ which is not disconnected from the origin by $\eta^\sigma \cup \ol\eta^{\ol\sigma}$. Note that $x^*  =x $ and $y^* = y$ if $\eta $ does not hit $\bdy\BB D$ except at its endpoints (e.g., if $\eta$ is an ordinary SLE$_\kappa$). 
 
Let $L $ be the event that the following occurs.
\begin{enumerate}
\item $\sigma , \ol\sigma  < \infty$ and $\eta^\sigma$ (resp.\ $ \ol\eta^{\ol\sigma}$) is contained in the $e^{-2\Delta}$-neighborhood of the segment $[x , 0]$ (resp.\ $[y , 0]$). Furthermore, $\eta$ (resp.\ $\ol\eta$) does not exit $\mcl B_{\wt\Delta}$ between the first time it enters $\mcl B_{\Delta/2}$ and time $\sigma$ (resp.\ $\ol\sigma$). \label{item-L-hit}  
\item The harmonic measure from $0$ in $\BB D\setminus (\eta^\sigma\cup\ol\eta^{\ol\sigma})$ of each of the two sides of $\eta^\sigma$ and each of the two sides of $\ol\eta^{\ol\sigma}$ is at least $a$. \label{item-L-hm}
\item Let $\psi^L : \BB D\setminus (\eta^\sigma \cup \ol\eta^{\ol\sigma}) \rta \BB D$ be the conformal map with $\psi^L(0) =0$ and $(\psi^L)'(0) > 0$. Then $(\psi^L)^{-1}$ maps $B_{1-\mu(\delta_L)}(0) \cup B_{\delta_L}(\psi^L(\eta(\sigma)) ) \cup B_{\delta_L}(\psi^L(\ol\eta(\ol\sigma)) )$ into $\mcl B_{ \wt\Delta    } $. \label{item-L-ball}  
\item $\mathcal G_{[ x^* ,  y^*]}(  \psi^L   , \mu_L )$ occurs (Definition~\ref{G infty def}). \label{item-L-G}  
\item The conditional probability given $\eta^\sigma \cup\ol\eta^{\ol\sigma} $ that the part of $\eta $ lying between $\eta(\sigma)$ and $\ol\eta(\ol\sigma)$ never exits $\mcl B_{\wt\Delta }$ is at least $p_L$.  \label{item-L-cond}
\end{enumerate}
See the middle panel of Figure~\ref{fig-flow-line-full} for an illustration. 
The main reason for most of the conditions in the definition of $L$ is so that the conditions of Lemma~\ref{rho abs cont} are satisfied, which will be used in Lemma~\ref{lem-wtE-prob} just below. The objects involved in the definition of $L$ ($\sigma$, $\ol\sigma$, $\psi^L$, etc.) are used infrequently in the rest of this section.

\begin{figure}
\begin{center}
\includegraphics[scale=.7]{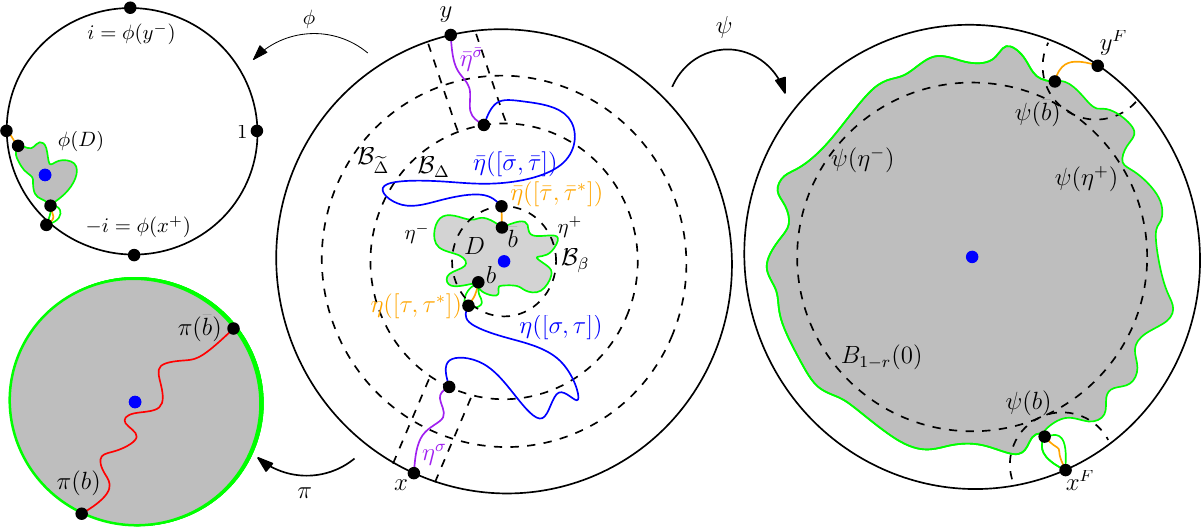}
\caption{Illustration of the definition of the event $E$. \textbf{Middle:} The full curve $\eta$. The time segments of the curve involved in the definition of $L$ are shown in purple, those involved in the definition of $\wt E$ are shown in blue, and those involved in the last part of the definition of $E$ are shown in orange. The auxiliary flow lines $\eta^\pm$ involved in the definition of $F$ are shown in green. For clarity, the disks $\mcl B_{\wt \Delta}$, $\mcl B_{\Delta}$, and $\mcl B_\beta$ here are shown larger than they actually are in practice. \textbf{Top left:} The image of the middle picture under the map $\phi : \BB D\setminus (\eta^\tau\cup\ol\eta^{\ol\tau}) \rta \BB D$. The derivative of this map at the origin is of order $e^{-\beta q}$ on $\wt E$. \textbf{Bottom left:} The image of the middle picture under the map $\pi : D \rta \BB D$ with $\pi(0) = 0$ and $\pi'(0) > 0$. In the setting of Section~\ref{sec-perfect-setup'}, if $\eta =\eta_{z,j}$, then the red curve in this picture is $\eta_{z,j+1}$. \textbf{Right:} The map $\psi$ takes $\BB D\setminus (\eta^\tau \cup \ol\eta^{\ol\tau})$ to $\BB D$ and fixes 0. The event $F$ includes several conditions which say that the flow lines $\psi(\eta^\pm)$ behave nicely. } \label{fig-flow-line-full}
\end{center}
\end{figure}

\begin{lem}\label{lem-L-prob}
For each $\wt d  \in (0,1)$, $\wt \Delta >0$, and $\mu \in \mcl M$, 
it holds for sufficiently small $\delta_L \in (0,1)$, $\mu_L \in \mathcal M$,  and $p_L \in (0,1)$ and sufficiently large $\Delta > \wt\Delta  >1 $, depending only on $\wt d , \wt\Delta$, and $\mu$, that for each $a\in (0,1/4)$, we have $\BB P(L) \succeq 1$, with implicit constant depending on $\wt d $, $\rho^L ,\rho^R,\kappa$, and the auxiliary parameters but uniform over all choices of endpoints $x,y$ with $|x-y| \geq \wt d$. 
\end{lem}
\begin{proof}
This follows from Lemma~\ref{miller-wu-dim-2.3}. Note that we can apply the Koebe growth theorem to $(\psi^L)^{-1}$ to find a $\delta_L = \delta_L(\wt\Delta,\mu)  > 0$ so that the statement of the lemma holds, no matter how large we make $\wt\Delta$. 
\end{proof}

We next define the ``part" of the definition of $E$ which gives us control of the derivatives of certain conformal maps. This is the only event in this subsection which does not occur with constant-order (i.e., $\beta$-independent) conditional probability given the earlier events. 

Recalling the auxiliary parameters from Definition~\ref{def-aux-parameter}, let $\wt E$ be the intersection of $L$ and the event $E_\beta^{q;u}(\eta ; a ,1 , \mu)$ considered in Section~\ref{2pt setup sec}, i.e., $\wt E$ is the event that the following is true. 
\begin{enumerate}
\item The event $L$ defined above occurs. Moreover, let $\tau$ (resp.\ $\ol\tau$) be the first time $\eta$ (resp.\ $\ol\eta$) hits $\mcl B_\beta$ (or $\infty$ if no such time exists). Then $\tau,\ol\tau <\infty$.   \label{item-wtE-hit}
\item The conformal map $\phi : \BB D\setminus (\eta^\tau\cup\ol\eta^{\ol\tau}) \rta \BB D$ with $\phi(x^+) =-i$, $\phi(y^-) = i$, and $\phi(\text{midpoint of $[x,y]_{\bdy\BB D}$}) = 1$ satisfies $ e^{-\beta(q+u)} \leq |\phi'(0)| \leq e^{-\beta(q-u)}$. \label{item-wtE-deriv}
\item The harmonic measure from $0$ in $\BB D\setminus (\eta^\tau \cup \ol\eta^{\ol\tau})$ of each of the two sides of $\eta^\tau$ and each of the two sides of $\ol\eta^{\ol\tau}$ is at least $a$.  \label{item-wtE-hm}
\item With $\psi^L$ as in condition~\ref{item-L-ball} in the definition of $L$, the event $\mcl G'( \psi^L( \eta^\tau \cup\ol\eta^{\ol\tau} ) , \mu)$ occurs (Definition~\ref{G' def}). \label{item-wtE-G}. 
\end{enumerate}
 The event $\wt E$ is illustrated in the middle panel of Figure~\ref{fig-flow-line-full}.  

We now record our estimate for $\BB P(\wt E)$. 

\begin{lem} \label{lem-wtE-prob}
There exists $u_* = u_*(q) \in (0,1)$ such that for each $u\in (0,u_*]$ and each $\wt d  \in (0,1)$, 
it holds for sufficiently small $\delta_L \in (0,1)$, $\mu , \mu_L \in \mathcal M$, and $p_L \in (0,1)$ and sufficiently large $\Delta > \wt\Delta$, depending only on $\wt d$, and all $a\in (0,1/4)$ that the following is true. There exists $\beta_*  > 0$ (depending on $u$, $\wt d$, and the auxiliary parameters) such that for $\beta \geq \beta_*$, 
\eqb \label{eqn-wtE-prob}
e^{-\beta(\gamma^*(q) + \gamma_0^*(q) u)} \leq    \BB P\left( \wt E \right) \leq e^{-\beta(\gamma^*(q) - \gamma_0^*(q) u)}
\eqe
where here $\gamma^*(q)$ and $\gamma_0^*(q)$ are the exponents from Proposition~\ref{1pt at hitting} and the implicit constants depend on $u$, $\wt d $, and the auxiliary parameters.
\end{lem}

Due to the Markov property and reversibility of SLE$_\kappa$, Lemma~\ref{lem-wtE-prob} is almost immediate from Lemma~\ref{lem-L-prob} and Proposition~\ref{1pt at hitting} if $\rho^L = \rho^R = 0$. In order to treat the case of general $\rho^L ,\rho^R \in (-2,0]$, we will use an absolute continuity argument based on the result of Appendix~\ref{smac sec} since Proposition~\ref{1pt at hitting} is only proven for $\rho^L = \rho^R = 0$. 

\begin{proof}[Proof of Lemma~\ref{lem-wtE-prob}]
Let $\psi^L : \BB D\setminus (\eta^\sigma \cup \ol\eta^{\ol\sigma}) \rta\BB D$ be the conformal map from condition~\ref{item-L-ball} in the definition of the event $L$. Define the curve $\eta_\ed := \psi^L(\eta \setminus (\eta^\sigma \cup \ol\eta^{\ol\sigma}))$.
Also let $H^* := \left\{\eta \setminus (\eta^\sigma \cup \ol\eta^{\ol\sigma}) \subset \mcl B_{\wt\Delta}\right\}$, as in Appendix~\ref{smac sec}. By condition~\ref{item-L-ball} in the definition of $L$ and condition~\ref{item-wtE-G} in the definition of $\wt E$, we infer that $\wt E\subset H^*$. 

By conditions~\ref{item-L-hit} and~\ref{item-L-cond} in the definition of $L$, this event is contained in the event $S$ of Lemma~\ref{rho abs cont}. By Lemma~\ref{rho abs cont}, if $\wt \Delta$ (and hence also $\Delta$) is chosen sufficiently large, then the regular conditional law of the curve $\eta_\ed$ given $\eta^\sigma \cup\ol\eta^{\ol\sigma}$ and the event $H^*$ on the event $L$ is strictly mutually absolutely continuous (\hyperref[smac]{s.m.a.c.}; Definition~\ref{smac}) with respect to the law of a chordal SLE$_\kappa$ from $\psi^L(\eta(\sigma))$ to $\psi^L(\ol\eta(\ol\sigma))$ in $\BB D$ conditioned to stay in $\psi^L(\mcl B_{\wt\Delta})$, with implicit constants depending only on $\wt d$, $\rho^L ,\rho^R,\kappa$, and the auxiliary parameters. By condition~\ref{item-L-cond} in the definition of $L$, the same is true of the regular conditional law of $\eta_\ed$ given $\eta^\sigma \cup\ol\eta^{\ol\sigma}$ on the event $L$ \emph{restricted} to the event $H^*$. By condition~\ref{item-L-cond} in the definition of $L$, $\mcl G(\eta_\ed , \mu ) \subset H^*$.

By condition~\ref{item-L-hm} in the definition of $L$, $| \psi^L(\eta(\sigma)) - \psi^L(\ol\eta(\ol\sigma))|$ is bounded below by a positive $a$-dependent constant on $L$.  
By this, Proposition~\ref{1pt at hitting}, and the absolute continuity considerations in the preceding paragraph, we find that (in the notation of Proposition~\ref{1pt at hitting}) for an appropriate choice of $u_* \in (0,1)$ and a small enough choice of $\mu\in\mcl M$, it holds on $L$ that for each $u\in (0,u_*]$ and $c > 0$, there exists $\wt\beta_* = \wt\beta_*(u,a , c )  > 0$ such that for $\wt\beta\geq \wt\beta_*$, the conditional probability of the event of Proposition~\ref{1pt at hitting} on $L$ a.s.\ satisfies
\eqb \label{eqn-wtE-prob-basic}
e^{-\wt\beta(\gamma^*(q) + \gamma_0^*(q) u)}  \leq \BB P\left( E_{\wt\beta }^{q;u}(\eta_\ed ; a , c , \mu) \,|\,  \eta^\sigma \cup\ol\eta^{\ol\sigma} \right) \leq e^{-\wt\beta(\gamma^*(q) - \gamma_0^*(q) u)}  .
\eqe 
Note that if $\mu$ is chosen sufficiently small, then $E_{\wt\beta}^{q;u}(\eta_\ed;a,c,\mu) \subset  H^*$ by condition~\ref{item-L-ball} in the definition of $L$. 
By the Koebe quarter theorem, we can find $C > 0$ depending only on $\Delta$ such that on $L$,
\eqb \label{eqn-wtE-contained}
\mcl B_{\beta + C} \subset \psi^L(\mcl B_\beta) \subset \mcl B_{\beta-C} .
\eqe 
It is clear from Lemma~\ref{miller-wu-dim-2.3}, the above absolute continuity statement, and the Markov property of ordinary SLE$_\kappa$ that for an appropriate choice of $c = c(\Delta)  \in (0,1)$, the conditional probability of $\wt E$ given $\eta^\sigma \cup\ol\eta^{\ol\sigma}$ and the event $E_{\beta - C}^{q;u}(\eta_\ed ; a , c , \mu) $; and the conditional probability of $E_{\beta +C}^{q;u}(\eta_\ed ; a , c^{-1} , \mu)$  given $\eta^\sigma \cup\ol\eta^{\ol\sigma}$ and the event $\wt E$ are each a.s.\ bounded below by positive deterministic constants depending only on $\wt d $ and the auxiliary parameters. Combining this with~\eqref{eqn-wtE-prob-basic} and Lemma~\ref{lem-L-prob} yields the statement of the lemma with $\beta_* = \wt\beta_* + C$.
\end{proof}

We next define auxiliary flow lines $\eta^{ \pm}$ started from $\eta(\tau)$ which form a ``pocket" surrounding $0$ with size of order $e^{-\beta}$ with uniformly positive probability. The reason for introducing these flow lines is as follows. Roughly speaking, the part of $\eta $ inside $D$ is conditionally independent of the part of $\eta $ which outside $ D  $ given the flow lines $\eta^\pm$ (see Lemma~\ref{lem-F-cond} below). When applied at various scales, this fact will eventually allow us to get the needed long-range independence for our two-point estimate. 
  
Fix $\theta  > 0$, to be chosen momentarily, in a manner depending only on $\kappa$.
On $\wt E $, let $ \eta^{ -}$ and $ \eta^{ +}$ be the flow lines of $h$ started from $\eta (\tau)$ with angles $\theta$ and $-\theta$, respectively.  
Note that the flow line with a negative sign has positive angle and vice versa. This is because a flow line with a negative angle a.s.\ stays to the right of $\eta $, and a flow line with a positive angle a.s.\ stays to the left of $\eta  $. See \cite[Theorem~1.5]{ig1}. 

By examining the boundary data of the field $h$ along $\eta$ and applying~\cite[Theorems~1.1 and~2.4]{ig1}, we find that the conditional law of $ \eta^{ -}$ (resp.\ $\eta^{ +}$) given $\eta $ on the event $\{\tau <\infty\}$ is that of a certain $\op{SLE}_\kappa(\ul\rho)$ process from $\eta_\ed (\tau )$ to $i$ in the right (resp.\ left) connected component of $\BB D\setminus \eta_\ed $, with force points immediately the the left and right of its starting point and at the endpoints $x$ and $y$. The weights of the force points immediately to the left and right of the starting point are given by
\eqb \label{eqn-rho0-rho1}
\rho^0 = -\frac{\theta\chi}{\lambda} \quad \op{and} \quad \rho^1 = \frac{\theta \chi}{\lambda}-2   ,
\eqe
with $\rho^0$ the force point on the side corresponding to $\eta^\tau$. 
(See~\cite[Section~2.2]{ig1} for a discussion and rigorous construction of SLE$_\kappa(\rho^0;\rho^1)$ with force points immediately to the left and right of the starting point).  
 
By~\cite[Theorem 1.5, assertion (iii)]{ig1},
$ \eta^{ \pm}$ a.s.\ intersect (but do not cross) each other provided $\theta < \pi\kappa/(4-\kappa)$.   
By~\cite[Remark 5.3]{ig1}, $\eta^\pm$ a.s.\ do not hit $\eta^\tau$ provided $-\theta\chi/\lambda \geq \kappa/2-2$.
Hence we can choose $\theta > 0$ sufficiently small, depending only on $\kappa$ in such a way that $ \eta^{ \pm}$ a.s.\ intersect each other and a.s.\ do not hit $\eta^\tau$.   
 We henceforth assume that $\theta$ has been chosen in this manner.  
 
If there is a connected component of $\BB D\setminus ( \eta^{ -} \cup \eta^{ +}) $ lying between $\eta^{ -}$ and $\eta^{ +}$ which contains 0, we take $D $ to be this connected component, and we set $D  =\emptyset$ otherwise. We also let $\pi : D\rta \BB D$ be the conformal map with $\pi(0) = 0$ and $\pi'(0) >0$. 
   
The next piece in the definition of our event $E$ is a list of regularity conditions for the flow lines $\eta^\pm$ which ensures that the pocket $D$ they form has a roughly round shape. 
Let $t^{+}$ be the first time that $\eta^{+}$ hits $\eta^{-}$ after the first time it exits the disk of radius $ e^{-\beta - 1}$ centered at $\eta (\tau)$. Let $t^{-}$ be the time such that $\eta^{-}(t^-) = \eta^{+}(t^{+})$. Let $\ol b  = \eta^-(t^-) = \eta^+(t^+)$ and let $b$ be the last intersection point of $\eta^\pm$ before hitting $\ol b $, so that if $D\not=\emptyset$, then $b$ and $\ol b$ are the first and last points of $\bdy D$ hit by $\eta^\pm$. 
Also let $\wt t^\pm$ be the first exit times of $\eta^\pm$ from the annulus $\mcl B_{\beta  - \Delta} \setminus \mcl B_{\beta  + \Delta}$. 
Let $F$ be the event that the following occurs.
\begin{enumerate}  
\item $\wt E$ occurs, $t^+ \leq \wt t^+$, $t^- \leq \wt t^-$, $D \not=\emptyset$, and $\ol b \notin \ol\eta^{\ol\tau}$. \label{item-F-contained} 
\item Let $\psi  : \BB D\setminus (\eta^{\tau} \cup \ol\eta^{\ol\tau})$ be the conformal map with $\psi(0) = 0$ and $ \psi'(0) > 0$. Let $x^F = \psi  ( \eta ( \tau))$ and $y^F  =  \psi (\ol\eta(\ol\tau) )$. Then $|\psi ( b ) - x^F|$ and $|\psi(\ol b) - y^F|$ are each at most $r$.  \label{item-F-b-bar}  
\item Each point of $\psi ( (\eta^+)^{t^+})$ (resp.\ $\psi(( \eta^-)^{  t^-})$) lies within distance $r$ of $[x^F , y^F]_{\partial\BB D}$ (resp.\ $[y^F , x^F]_{\partial\BB D}$).
\label{item-F-hm}
\item $\mathcal G'( \psi ( (\eta^+)^{t^+} \cup  (\eta^-)^{t^-}) , \mu^F)$ occurs (Definition~\ref{G' def}). \label{item-F-G}
\end{enumerate}
See the right panel in Figure~\ref{fig-flow-line-full} for an illustration of the event $F$.

The main reason for our interest in the domain $D$ is contained in the following lemma, which will be a key tool in our two-point estimate.

\begin{lem} \label{lem-F-cond}
Recall the pocket $D$ formed by the auxiliary flow lines $\eta^\pm$ and its two marked boundary points $b$ and $\ol b$. On the event $\{D\not=\emptyset\}$, if we condition on $ D$ and $h|_{\BB D\setminus D}$ then the joint conditional law of $h|_D$ and the segment of $\eta$ contained in $\ol D$ is that of a GFF with Dirichlet boundary data determined by $(D,b,\ol b)$ and its zero-angle flow line from $b$ to $\ol b$. In particular, the conditional law of this segment of $\eta$ given $ D$ and $h|_{\BB D\setminus D}$ is that of a chordal $\op{SLE}_\kappa(\rho^1 ; \rho^1)$ in $D$ from $b$ to $\ol b$, with $\rho^1$ as in~\eqref{eqn-rho0-rho1}.
\end{lem} 
\begin{proof}
By~\cite[Theorem 1.1]{ig1} and since $\tau$ is a stopping time for $\eta$, the set
\[
A := \eta^\tau \cup \eta^-\cup \eta^+
\]
is a local set for $h$ in the sense of~\cite[Section 3.3]{ss-contour}, i.e., the conditional law of $h|_{\BB D\setminus A}$ given $A$ and $h|_A$ is that of an independent zero-boundary GFF in each connected component of $\BB D\setminus A$ plus a harmonic function determined by $(h|_{\BB D\setminus A} , A)$. This harmonic function is described explicitly in~\cite[Theorem 1.1]{ig1}: in particular, the conditional law of $h|_D$ given $(A , h|_{D\setminus A})$ on the event $\{D \not=\emptyset\}$ is that of a GFF on $D$ with boundary data $\lambda - \theta \chi- \chi\cdot \op{winding}$ on $[\ol b , b]_{\partial D}$ and $-\lambda +\theta \chi - \chi\cdot \op{winding}$ on $[b , \ol b]_{\partial D}$, where $\lambda$ and $\chi$ are as in Section~\ref{ig prelim} and the term ``winding" has the meaning of~\cite[Figure~1.9]{ig1}. 

The domain $D$ is one of the connected components of $\BB D\setminus A$ and the field $h|_{\BB D\setminus D}$ is determined by $A$, $h|_A$, and the restrictions of $h$ to the other connected components of $\BB D\setminus A$. 
Since $A$ is a local set for $h$ and is a.s.\ determined by $h$ (by~\cite[Theorem~1.2]{ig1}), we infer that $A$ is a.s.\ determined by $D$ and $h|_{\BB D\setminus A}$. 
Hence we get the same conditional law for $h|_D$ if we instead condition on $ D$ and $h|_{\BB D\setminus D}$.

The statement about the conditional law of the segment of $\eta$ contained in $\ol D$ follows easily from our description of the conditional law of $h|_D$ and \cite[Theorems~1.1 and~2.4]{ig1}. 
\end{proof}

To complete the definition of our event $E$, we need one last regularity condition to rule out pathological behavior of the segments of $\eta $ and $\ol\eta  $ before they hit $D$.  
Let $ \tau^*$ (resp.\ $\ol\tau^*$) be the time at which $ \eta $ (resp.\ $\ol\eta $) hits $  b$ (resp.\ $\ol b$). Note that these times are a.s.\ finite if $F$ occurs since $\eta^-$ and $\eta^+$ a.s.\ lie to the left and right of $\eta $, respectively. Let $E $ be the event that the following occurs. 
\begin{enumerate} 
\item $F$ occurs. \label{item-E_z-F}
\item With $\psi $ as in condition~\ref{item-F-b-bar} in the definition of $F $, $ \psi  ( \eta_\ed ([\tau ,  \tau^*]) )$ (resp.\ $ \psi  (\ol\eta_\ed ( [\ol \tau  , \ol\tau^*])$) is contained in the disk of radius $2r$ centered at $x^F$ (resp.\ $y^F$) (notation as in condition~\ref{item-F-b-bar} in the definition of $F$). \label{item-E_z-stay}  
\end{enumerate}

\begin{remark}\label{remark-E-stay}
By \cite[Theorem~1.5]{ig1} $\eta $ cannot cross $\eta^\pm$. By combining this with condition~\ref{item-L-ball} in the definition of $L $, condition~\ref{item-wtE-G} in the definition of $\wt E$, and condition~\ref{item-E_z-stay} in the definition of $E$, it follows that the segment of $\eta$ between $\eta (\sigma )$ and $\ol\eta (\ol\sigma )$) is contained in $\mcl B_{\wt\Delta }$ on the event $E $. 
\end{remark} 

We now estimate the conditional probability of $E$ given the second intermediate event $\wt E$ defined above. 

\begin{lem} \label{lem-F-prob}
For each $r\in (0,1/2)$, it holds for sufficiently small $\mu_F \in \mcl M$ and sufficiently large $\Delta >1$, depending only on $r$, $a$, and $\wt d$, that $\BB P(E   \,|\, \wt E )  \succeq 1$, with the implicit constant depending only on $\wt d$ and the auxiliary parameters. 
\end{lem}
\begin{proof}
Let $  \eta^F$ be the image under $\psi$ of the part of $\eta $ between $\eta(\tau)$ and $\ol\eta(\ol\tau )$.  Note that the distance between the endpoints $ x^F$ and $y^F $ of $\eta^F$ is uniformly positive on $\wt E $ by condition~\ref{item-wtE-hm} in the definition of $\wt E$. 

Let $\wt r \in (0,r^2)$ and let $U$ be the $\wt r$-neighborhood of the line segment from $  x^F$ to $  y^F$. Also let $\mu_F' \in \mathcal M$ and let $S$ be the event that $ \eta^F \subset U$, $\mathcal G'(   \eta^F , \mu_F')$ occurs, and the time reversal of $\eta^F$ does not enter $B_{\wt r}(y^F)$ after leaving $B_{2\wt r}(y^F)$. 

The absolute continuity considerations in the proof of Lemma~\ref{lem-wtE-prob} (still applied at times $\sigma$ and $\ol\sigma$) show that the conditional law of $\eta^F$ given $\eta^\tau\cup\ol\eta^{\ol\tau}$ on the event $\wt E$, restricted to the event $S$, is \hyperref[smac]{s.m.a.c.} with respect to the law of a chordal SLE$_\kappa$ from $x^F$ to $y^F$ in $\BB D$, with implicit constants depending only on $\wt d$, $\rho^L ,\rho^R,\kappa$, and the auxiliary parameters.
By Lemma~\ref{miller-wu-dim-2.3}, we infer that $\BB P(S \,|\, \wt E) \succeq 1$. 

The conditional law of $\psi ( \eta^+ )$ given $\eta$ on the event $\wt E \cap S$ is that of an SLE$_\kappa(\ul\rho)$ process in the right connected component of $\BB D\setminus \eta^F$ from $  x^F$ to $\psi(i^-)$; it has force points with weights~\eqref{eqn-rho0-rho1} on either side of its starting point, and two other boundary force points lying at uniformly positive distance from its start and end points (this distance it uniformly positive by condition~\ref{item-wtE-hm} in the definition of $\wt E$). 
Similar statements hold with $-$ in place of $+$ and ``left" in place of ``right". By Lemma~\ref{miller-wu-dim-2.5} and the Beurling estimate (to make sure that $\psi(\mcl B_{\beta  +\Delta})$ covers most of $\BB D$) we infer that $\BB P(E \,| \, \wt E\cap S ) \succeq 1$ provided $\mu_F$ is chosen sufficiently small and $\Delta>1$ is chosen sufficiently large, in a manner depending only on $r$.\footnote{To get that the flow lines $ \eta^\pm$ intersect one another where we want them to with uniformly positive probability, we can further condition on a second pair of flow lines $\wt\eta^\pm$ with the same angles as $\eta^\pm$, started at a point near where we want the intersection to occur. We then apply Lemma~\ref{miller-wu-dim-2.5} to the conditional law of $\eta^\pm$ given $\wt\eta^\pm$ and $\eta $, and observe that $\eta^\pm$ merge with $\wt\eta^\pm$ upon intersecting~\cite[Theorem~1.5]{ig1}; and that $\wt\eta^\pm$ a.s.\ intersect one another at points arbitrarily close to their starting points. See \cite{miller-wu-dim} for several examples of similar arguments.}
Since $\wt r < r$, if $F\cap \wt E \cap S$ occurs, then so does $E$. 
We conclude by observing that
\eqbn
\BB P\left(E \,|\, \wt E \right) \geq  \BB P\left(E  \cap S  \,|\, \wt E \right) = \BB P\left(E \,|\,  \wt E  \cap S \right) \BB P\left(S  \,|\, \wt E \right) . \qedhere
\eqen 
\end{proof}

By combining Lemmas~\ref{lem-wtE-prob} and~\ref{lem-F-prob}, we infer the following one-point estimate for the event $E $. 

\begin{lem} \label{lem-E_z-prob}
Let $\wt d \in (0,1)$ and $a, r \in (0,1/4)$. There exists $u_* = u_*(q) \in (0,1)$ such that the following is true for each $u\in (0,u_*]$. If we choose $  \delta_L ,  p_L$, $\mu,\mu_L,$ and $\mu_F$ sufficiently small, and $\Delta >\wt\Delta$ sufficiently large in a manner depending only on $\wt d$, $a$, and $r$, then we can find $\beta_*(u) > 0$ (depending on $u$, $\wt d$, and the auxiliary parameters) such that for $\beta  \geq \beta_*(u )$,
\[
e^{-\beta (\gamma^*(q) + \gamma_0^*(q) u )} \preceq \BB P(E ) \preceq e^{-\beta  (\gamma^*(q) - \gamma_0^*(q) u  )} 
\]
with the implicit constants depending only on $u$, $\wt d$, and the auxiliary parameters.
\end{lem}  

The last lemma in this subsection will be used to circumvent the fact that the laws of our objects will not be exactly the same at every scale. To explain this, we observe that
Lemma~\ref{lem-F-cond} gives the objects defined in this subsection a certain self-similarity property: if $E$ occurs and we replace $(h,\eta)$ with the pushforward under the map $\pi : D \rta \BB D$ of $(h|_D , \eta\setminus (\eta^{\tau^*} \cup \ol\eta^{\ol\tau^*})$ then we end up in the same situation we started with but with $(\rho^1 , \rho^1) $ in place of $(\rho^L ,\rho^R)$ and a possibly different choice of start and end points for the curve. If we start with $\rho^L = \rho^R = \rho^1$, then we can remove the lack of stationarity coming from the change of $\rho$-values. The asymmetry coming from the change of start and end points is non-trivial, and is dealt with in the following lemma.
We note that by rotational invariance, we only care about $\op{arg}(y/x)$, not the particular values of $x$ and $y$.

\begin{lem} \label{lem-endpoint-abs-cont}
Let $r_H > 0$ and let $H = H(a,r_H )$ be the event that the following is true.
\begin{enumerate}
\item With $\tau$ as in condition~\ref{item-wtE-hit} in the definition of $\wt E$, we have $\tau < \infty$ and the harmonic measure from $0$ in $\BB D\setminus \eta^\tau$ of each side of $\eta^\tau$ is at least $a$. \label{item-H-hm}
\item Let $\psi^H : \BB D\setminus \eta^\tau \rta \BB D$ be the conformal map with $\psi^H(0) = 0$ and $\psi^H(\eta(\tau)) = -i$. Then each point of $\psi^H(\bdy D)$ lies at distance at least $r_H$ from $\bdy\BB D \setminus B_a(-i)$. \label{item-H-dist}
\end{enumerate}
Recalling the map $\pi : D\rta \BB D$ which fixes 0, let $x' := \pi(b)$ and $y' := \pi(\ol b)$, so that $x'$ and $y'$ are the start and end points of the image under $\pi$ of the segment of $\eta$ contained in $\ol D$. 
Suppose also that we are given two choices of start/end point pairs $(x_1,y_1)$ and $(x_2,y_2)$ for $\eta$ with $|x_1-y_1| ,|x_2-y_2| \geq \wt d$, and for $i\in\{1,2\}$ denote the objects defined above with $(x_i,y_i)$ in place of $(x,y)$ with a subscript $i$. The conditional law of $\op{arg}(y_1'/x_1')$ given $H_1$ and the conditional law of $\op{arg}(y_2'/x_2')$ given $ H_2$ are strictly mutually absolutely continuous (\hyperref[smac]{s.m.a.c.}; Definition~\ref{smac}), with the implicit constant depending only on $\wt d$, $a$, and $r_H$ (not on $\beta$, $u$, or the particular choice of $(x_1,y_1)$ and $(x_2,y_2)$). 
\end{lem}

With $H$ the event of Lemma~\ref{lem-endpoint-abs-cont}, it follows from condition~\ref{item-wtE-hm} in the definition of $\wt E$, condition~\ref{item-F-G} in the definition of $F$, and the Schwarz lemma applied to the map $\psi  \circ (\psi^H)^{-1} : \BB D \setminus \psi^H(\ol\eta^{\ol\tau}) \rta \BB D$ that for any choice of the auxiliary parameters $a\in (0,1/4)$ and $\mu_F\in\mcl M$, there is an $r_H = r_H(a,\mu_F)$ for which $E\subset H$. 

\begin{proof}[Proof of Lemma~\ref{lem-endpoint-abs-cont}]
We observe that $\op{arg}(y'/x')$ is equal to $2\pi$ times the harmonic measure from $0$ of $\bdy D\cap \eta^+$. Hence we need to prove an absolute continuity statement for this harmonic measure.

The conditional law of the curve $\psi^H(\eta^-)$ (resp.\ $\psi^H(\eta^+)$) given $\eta^\tau$ is that of a certain chordal SLE$_\kappa(\ul\rho)$ (resp.\ SLE$_\kappa(\ul\rho)$) from $-i$ to $\psi^H(i)$ in $\BB D$ with force points of weight $\rho^1$ (as in~\eqref{eqn-rho0-rho1}) and $\theta\chi/\lambda$ located on either side of $-i$ and additional force points located at $\psi^H(x^-)$ and $\psi^H(x^+)$. 
By condition~\ref{item-H-hm} in the definition of $H$, on $H$ each of these additional force points lies at distance at least $2a$ from $-i$. 

Let $U$ be the set of points in $\BB D$ which lie at distance at least $r_H$ from $\bdy\BB D\setminus B_a(-i)$ and let $t^{U,\pm}$ be the exit time of $\eta^\pm$ from $U$. 
By~\cite[Lemma 2.8]{miller-wu-dim} (applied once to $\eta^-$ and once to the conditional law of $\eta^+$ given $\eta^-$) we infer that, in the notation of the lemma, the joint conditional law of $( (\eta_1^-)^{t_1^{U,-}} , (\eta_1^+)^{t_1^{U,+}})$ given $\eta_1^{\tau_1}$ on the event that condition~\ref{item-H-dist} in the definition of $H_1$ holds; and the joint conditional law of $( (\eta_2^-)^{t_2^{U,-}} , (\eta_2^+)^{t_2^{U,+}})$ given $\eta_2^{\tau_2}$ on the event that condition~\ref{item-H-dist} in the definition of $H_2$ holds; are \hyperref[smac]{s.m.a.c.}, with implicit constants depending only on $\wt d$, $a$, and $r_H$. This immediately implies the statement of the lemma.
\end{proof}

\subsection{Events for the perfect points} \label{sec-perfect-setup'}

Recall the setting described at the beginning of Section~\ref{sec-2pt-outline}: $h$ is a GFF on $\BB D$ with Dirichlet boundary data chosen so that its 0-angle flow line $\eta$ from $-i$ to $i$ is an ordinary SLE$_\kappa$. 

Fix auxiliary parameters $  r ,   a  $ (to be chosen later) and assume that the other auxiliary parameters from Definition~\ref{def-aux-parameter} are chosen in such a way that the conclusion of Lemma~\ref{lem-E_z-prob} holds for this choice of $r$ and $a$. 

Fix $d\in (0,1)$; we will work on $B_d(0)$ to avoid pathologies coming from the boundary. Also fix sequences of positive numbers $\beta_j \rta\infty$ and $u_j\rta 0$ to be chosen in Lemma~\ref{lem-beta-u-choice} just below; we note that in particular $\beta_j$ will grow like $\log j$. 
 
In this subsection we will define the main events and objects we consider in the rest of this section using the construction of Section~\ref{sec-perfect-setup} and induction over scales of size $e^{-\beta_j}$. 
See Figure~\ref{fig-flow-line-stages} for an illustration of the objects defined in this subsection and Section~~\ref{sec-2pt-index} for an index of these objects.

\begin{figure}
\begin{center}
\includegraphics[scale=.7]{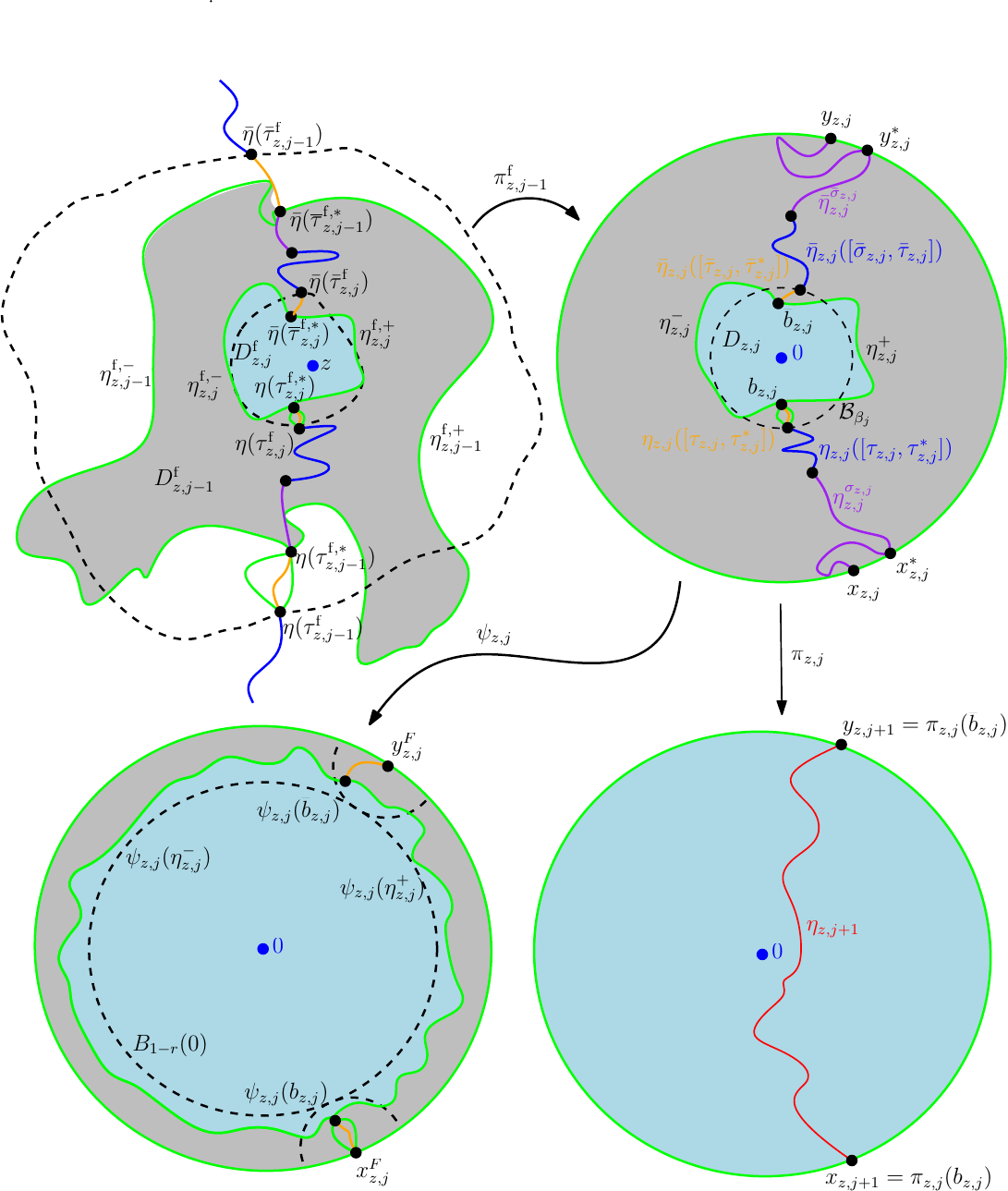}
\caption{\textbf{Top left:} Illustration of two stages of the inductive construction in Section~\ref{sec-perfect-setup'} (the picture shows a small neighborhood of the point $z \in \BB D$). Segments of $\eta$ associated with the events $L_{z,j} $ (resp.\ $\wt E_{z,j-1}$ and $\wt E_{z,j}$; the last parts of $E_{z,j-1}$ and $E_{z,j}$) are shown in purple (resp.\ blue; orange). As in Figure~\ref{fig-flow-line-full}, balls and curve segments are not shown to scale. \textbf{Top right:} The picture we obtain after applying the map $\pi_{z,j-1}^\fl : D_{z,j-1}^\fl \rta \BB D$. This is the same as the setting of the middle panel in Figure~\ref{fig-flow-line-full} with $\eta = \eta_{z,j } $. Note that here $x_{z,j}^{ *} \not=x_{z,j} $ and $y_{z,j}^{ *}\not=y_{z,j} $ since $\eta_{z,j} $ hits $\bdy\BB D$. \textbf{Bottom left:} The setting we obtain after applying the map $\psi_{z,j} $, which corresponds to the right panel in Figure~\ref{fig-flow-line-full}. \textbf{Bottom right:} The setting we obtain after applying the map $\pi_{z,j}$. The red curve $\eta_{z,j+1}$ is the image under $\pi_{z,j}$ of the segment of $\eta_{z,j}$ contained in $D_{z,j}$. } \label{fig-flow-line-stages}
\end{center}
\end{figure}

\subsubsection{Inductive definitions of events} \label{sec-perfect-event}

Here we will use the events of Section~\ref{sec-perfect-setup} with $\eta$ replaced by a conformal image of an appropriate segment of $\eta$ to define the following objects for $z\in B_d(0)$ and $j\in\BB N$. 
\begin{itemize}
\item Events $L_{z,j} $, $\wt E_{z,j}$, $F_{z,j}$, and $E_{z,j}$. 
\item Points $x_{z,j} , y_{z,j}  , x_{z,j}^* ,   y_{z,j}^* , x_{z,1}^F , y_{z,1}^F,  b_{z,j} $, and $\ol b_{z,j}$. 
\item Conformal maps $\psi_{z,j}^L , \phi_{z,j} ,  \psi_{z,j} , $ and $\pi_{z,j}$. 
\item Random times $\sigma_{z,j}  , \ol\sigma_{z,j} $, $\tau_{z,j}$, $ \ol \tau_{j,z}$, $\tau_{z,1}^*$, and $\ol\tau_{z,j}^*$. 
\item Curves $\eta_{z,j}$ and $\eta_{z,j}^\pm$. 
\item Fields $h_{z,j}$. 
\item Domains $D_{z,j}$. 
\end{itemize}
  
First we consider the case $j=1$. For $z\in B_d(0)$, let $f_{z,1}$ be the conformal automorphism of $\BB D$ satisfying $f_{z,1}(z) = 0$ and $f_{z,1}(-i)  = -i$. Let $\eta_{z,1}  := f_{z,1}(\eta )$ and let $x_{z,1}  := -i = f_{z,1}(-i)$ and $y_{z,1}  := f_{z,1}(i)$ be its start and end points. Also define the field $h_{z,1}  := h^\fl \circ f_{z,1}^{-1} - \chi \op{arg}((f_{z,1}^{-1})')$, where here $\chi  =2/\sqrt\kappa - \sqrt\kappa/2$ is the imaginary geometry parameter. 

Define the event $E_{z,1}$ and the associated objects as in Section~\ref{sec-perfect-setup} with $\beta = \beta_1$, $u= u_1$, $ \eta_{z,1} $ in place of $\eta$, and $h_{z,1}$ in place of $h$, and denote these objects with a subscript $z,1$. 
We recall in particular that $D_{z,1}$ is the domain formed by the auxiliary flow lines $\eta_{z,1}^\pm$, with marked points $b_{z,1} , \ol b_{z,1} \in \bdy D_{z,1}$, and we let $\pi_{z,1} : D_{z,1} \rta \BB D$ be the conformal map with $\pi_{z,1}(0) = 0$ and $\pi_{z,1}'(0) > 0$. 

Now suppose $j\geq 2$ and our objects have been defined for all positive integers $l \leq j-1$. 
If $D_{z,j-1}=\emptyset$, we take all of the objects defined below to be equal to a graveyard point. 
Otherwise, let $\eta_{z,j} $ be the image under $\pi_{z,j-1}$ of the segment of $\eta_{z,j-1}$ contained in $\ol D_{z,j-1}$ (equivalently, the segment of $\eta_{z,j-1}$ from $\eta_{z,j-1}(\tau_{z,j-1})$ to $\ol\eta(\ol\tau_{z,j-1}$). Then $x_{z,j} =\pi_{z,j-1}(b_{z,j-1}) $ and $y_{z,j} =\pi_{z,j-1}(\ol b_{z,j-1})$ are the initial and terminal points of $\eta_{z,j} $. Define the field
\eqbn
h_{z,j} := h_{z,j-1}  \circ \pi_{z,j-1}^{-1} - \chi \op{arg} (\pi_{z,j-1}^{-1})'. 
\eqen
Lemma~\ref{lem-F-cond} implies that $h_{z,j}$ is a GFF with Dirichlet boundary data, $\eta_{z,j}$ is its 0-angle flow line from $-i$ to $y_{z,j}$, and $\eta_{z,j}$ is an SLE$_\kappa(\rho^1;\rho^1)$ with force points located on either side of $-i$.   

Define the event $E_{z,j}$ and the associated objects as in Section~\ref{sec-perfect-setup} with $\beta=\beta_j$, $u = u_j$, $ \eta_{z,j} $ in place of $\eta$, and $h_{z,j}$ in place of $h$ and denote these objects by a subscript $z,j$. 

\begin{remark} \label{remark-endpoint-dist}
There exists $\wt d \in (0,1)$, depending only on $d$, such that if $z\in B_d(0)$ then each conformal automorphism $\BB D\rta\BB D$ taking $z$ to $0$ takes $-i$ and $i$ to a point of $\bdy \BB D$ at distance at least $\wt d$ from each other, so $|x_{z,1}-  y_{z,1}| \geq \wt d$. By conditions~\ref{item-F-b-bar} and~\ref{item-wtE-hm} in the definition of $\wt E_{z,j}$ and condition~\ref{item-F-hm} in the definition of $F_{z,j}$, after possibly shrinking $\wt d$ (in a manner depending only on $r$ and $a$) we can arrange that also $|x_{z,j} -y_{z,j}| \geq \wt d$ for $j\geq 2$. 
\end{remark}

\subsubsection{Objects associated with the full curve $\eta$} \label{sec-perfect-full}

Let
\eqb \label{eqn-E(z)-def0}
E_n(z) := \bigcap_{j=1}^n E_{z,j}   . 
\eqe
Also define the $\sigma$-algebra 
\eqb \label{eqn-mclF-def}
 \mathcal F_{z,n} :=  \sigma\left( \eta_{z,j} |_{[0, \tau_{z,j}^*]} ,\, \ol\eta_{z,j}|_{[0,\ol\tau_{z,j}^*]} , \eta_{z,j}^-|_{[0,t_{z,j}^-]} , \eta_{z,j}^+|_{[0,t_{z,j}^+]}  \,:\, j \leq n \right)  
\eqe
so that $E_n(z) \in \mcl F_{z,n}$.
 
We will also need to define a few additional objects associated with the full curve $\eta$, which are denoted with a superscript $\fl$ (recall the notational convention described at the beginning of Section~\ref{sec-2pt-outline}). For $z,j\in\BB N$, define the conformal map
\eqb \label{eqn-pi^fl-def}
\pi_{z,j}^\fl  :=   \pi_{z,j} \circ\dots \circ \pi_{z,1} \circ f_{z,1} .
\eqe
Also set $\pi_{z,0} := f_{z,1}$. 
Then $\pi_{z,j}^\fl : D_{z,j}^\fl \rta  \BB D$, for $D_{z,j}^\fl$ a domain in $\BB D$ containing $z$ and $\pi_{z,j}^\fl(z) = 0$. 

For $z \in B_d(0)$ and $j\in\BB N$, let $\tau_{z,j}^\fl$ and $\tau_{z,j}^{\fl,*}$ (resp.\ $\ol\tau_{z,j}^{\fl,*}$ and $\ol\tau_{z,j}^{\fl,*}$) be the times for $\eta$ (resp.\ $\ol\eta$) such that 
\eqb \label{eqn-tau^fl-def}
\pi_{z,j-1}^\fl(\eta(\tau_{z,j}^\fl)) = \eta_{z,j}(\tau_{z,j}) \quad \op{and} \quad
(\psi_{z,j} \circ \pi_{z,j-1}^\fl)(\eta(\tau_{z,j}^{\fl,*})) = \eta_{z,j}(\tau_{z,j}^*)
\eqe 
(resp.\ the analogous relation holds for $\ol\eta$ and $\ol\eta_{z,j}$). 

Let $\eta_{z,j}^{\fl,\pm}$ be the flow lines of $h$ with angles $\mp\theta$ started from $\eta(\tau_{z,j})$. 
Then $\eta_{z,j}^{\fl,\pm}$ trace $\bdy D_{z,j}^\fl$ and if we let $t_{z,j}^{\fl,\pm}$ be the time at which $\eta_{z,j}^{\fl,\pm}$ finishes tracing $\bdy D_{z,j}^\fl$, 
\eqb \label{eqn-eta^fl-aux-def}
(\eta_{z,j}^\pm)^{t_{z,j}^\pm} =   \pi_{z,j-1}^\fl ( (\eta_{z,j}^{\fl,\pm})^{t_{z,j}^{\fl,\pm}} ) .
\eqe    

\subsubsection{Choosing of $\beta_j$ and $u_j$}  \label{sec-beta-u}

We now choose the sequences $\beta_j \rta\infty$ and $u_j \rta 0$ which are used in place of $\beta$ and $u$, respectively, in the definitions of the events in Section~\ref{sec-perfect-setup}.

By Lemma~\ref{lem-E_z-prob} (applied with $\wt d$ as in Remark~\ref{remark-endpoint-dist}) tells us that for each $u\in (0,1)$, there exists $\beta_* (u) = \beta_*(u ,\wt d ) >0$ such that if we are in the setting of Section~\ref{sec-perfect-setup} with $\beta \geq \beta_*(u)$, either $\rho^L = \rho^R = \rho^1$ or $\rho^L = \rho^R = 0$, and $|x-y| \geq \wt d$, then
\eqb \label{eqn-C_u}
 C_u^{-1}  e^{-\beta  (\gamma^*(q)  + \gamma_0^*(q) u   )}  \leq \BB P(E  ) \leq C_{u} e^{-\beta  (\gamma^*(q)  - \gamma_0^*(q) u   )} , 
\eqe 
where for $u > 0$, $C_u$ is a constant which is allowed to depend on $u $, $\wt d$, and the auxiliary parameters but not on $\beta$ or the particular choice of $x$ and $y$. We now choose $\beta_j \rta\infty$ and $u_j \rta 0$ in such a way that~\eqref{eqn-C_u} remains true with $\beta_j$ in place of $\beta$ and $u_j$ in place of $u$. 

\begin{lem} \label{lem-beta-u-choice} 
For each choice of $\wt d$ (which we recall from Remark~\ref{remark-endpoint-dist} depends on $d$) and each choice of the auxiliary parameters, there exists $\beta_0 > 0$ such that with $\beta_j = \log j + \beta_0$, one can choose $(u_j)_{j\in\BB N}$ such that the following is true. 
\begin{enumerate}
\item $u_j$ decreases to $0$ as $j\rta\infty$. \label{item-beta-u-limit}   
\item For each $j \in\BB N$. we have $\beta_j \geq \beta_*(u_j)$ so that~\eqref{eqn-C_u} holds with $\beta_j$ in place of $\beta $ and $u_j$ in place of $u $. \label{item-beta-u-estimate}
\item For each $j\in \BB N$, $C_{u_j} \leq e^{\beta_j u_j  \gamma_0^*(q) }$. \label{item-beta-u-C}  
\item $\beta_j u_j \rta \infty$ as $j\rta \infty$. \label{item-beta-u-infty}
\end{enumerate}
\end{lem}

\begin{remark} \label{remark-beta-u-choice}
The reason we allow $\beta$ and $u$ to vary here is that we eventually want to get a lower bound for the Hausdorff dimension of the sets $\Theta^s(D_\eta)$ and $\wt\Theta^s(D_\eta)$. If we fixed $u$, we would instead get the Hausdorff dimension of the sets where the limits in the definitions of $\Theta^s(D_\eta)$ and $\wt\Theta^s(D_\eta)$ are between $s-u$ and $s+u$. In order to allow $u$ to vary, we also need to allow $\beta$ to vary, for otherwise the constants $C_u$ in~\eqref{eqn-C_u} would be larger than $e^\beta$ when $u$ is very small. The idea in Lemma~\ref{lem-beta-u-choice} below is to let $u_j \rta 0$ and $\beta_j \rta \infty$ slowly enough that our estimates are not much different than they would be with fixed $\beta$ and $u$. 
\end{remark}
 
 \begin{proof}[Proof of Lemma~\ref{lem-beta-u-choice}]
Fix $u_0 \in (0,1)$. Choose $\beta_0$ much larger than $\Delta \vee \gamma_0^*(q)^{-1} \log C_{u_0}$ and large enough that~\eqref{eqn-C_u} holds with $\beta_0$ in place of $\beta $ and $u_0$ in place of $u $. Set $\beta_j = \log j + \beta_0$ for this choice of $\beta_0$.  We now inductively choose $(u_j)_{j\in\BB N}$. 
Start with a sequence $(u_l^*)_{l\in\BB N} \subset (0,u_0)$ which decreases to 0. Let $j_1$ be the least positive integer $j$ such that $\beta_j \geq \beta_*(u_1^*)$, $C_{ u_1^* }  \leq  e^{  \beta_j u_1^*  \gamma^*_0(q) }$, and $\beta_j u_1^* \geq 1$. Such a $j$ exists since $\beta_j \rta \infty$ as $j\rta\infty$. Set $u_j = u_0$ for $j\in \{1,\dots,j_1\}$. Inductively, suppose $l \geq 1$ and $j_1 , \ldots , j_{l-1}$ and $u_j$ for $j\leq j_{l-1}$ have been defined. Let $j_l$ be the least integer $j \geq j_{l-1}+ 1$ such that $\beta_j \geq u_l^*$, $  C_{ u_l^* }  \leq  e^{  \beta_j u_l^*  \gamma^*_0(q) }$, and $\beta_j u_l^* \geq l$. Let $u_j = u_{l-1}^*$ for $j \in \{j_{l-1} + 1 , \ldots , j_l\}$. It is clear that conditions~\ref{item-beta-u-estimate},~\ref{item-beta-u-C},~and~\ref{item-beta-u-infty} hold for this choice of $(u_j)$. 
\end{proof}

We henceforth assume that the sequences $(\beta_j)$ and $(u_j)$ are chosen as in Lemma~\ref{lem-beta-u-choice}. 
We also define  
\eqb \label{eqn-bar-beta}
\ol \beta_m := \sum_{j=1 }^{m  } \beta_j \quad \op{and} \quad \ol u_m := \sum_{j= 1}^{m} \beta_j u_j ,\quad \forall m \in \BB N.
\eqe 

Due to our choice of the $\beta_j$'s and $u_j$'s, we obtain the following estimate for the probabilities of the events $E_n(z)$. 

\begin{lem} \label{lem-P(E_z)}
With $E_n(z)$ as in~\eqref{eqn-E(z)-def0}, it holds for each $n\in\BB N$ that
\eqb \label{eqn-P(E_z)}
e^{  - \ol \beta_n  \gamma^*(q)    -  2 \gamma_0^*(q)  \ol u_n      }  
\preceq  \BB P\left(E_n(z) \right) 
\preceq  e^{  - \ol \beta_n   \gamma^*(q)    +  2 \gamma_0^*(q) \ol u_n        } 
\eqe 
with the implicit constants independent of $n$ and uniform for $z\in B_d(0)$. The same is true if we replace $(\beta_j , u_j)_{j\in\BB N}$ by $(\beta_{j+m}, u_{j+m})_{j\in\BB N}$ for any $m\in\BB N$ (both in the definition of $E_n(z)$ and in~\eqref{eqn-P(E_z)}), with the implicit constants unchanged. 
\end{lem}
\begin{proof}
By Lemma~\ref{lem-F-cond},~\eqref{eqn-C_u}, and Remark~\ref{remark-endpoint-dist}, for each $j\in\BB N$, 
\eqbn
 C_{u_j}^{-1}  e^{-\beta_j  (\gamma^*(q)  + \gamma_0^*(q) u_j   )}  \leq \BB P(E_{z,j} \,|\, E_{j-1}(z)  ) \leq C_{u_j} e^{-\beta_j  (\gamma^*(q)  - \gamma_0^*(q) u_j   )} .
\eqen
The estimate~\eqref{eqn-P(E_z)} follows by multiplying this over all $j\in \{1,\dots,n\}$ and applying condition~\ref{item-beta-u-C} in Lemma~\ref{lem-beta-u-choice}. 
\end{proof}

\subsection{Analytic properties}
\label{sec-flow-line-analytic}
  
In this subsection we study some analytic properties of the events of Section~\ref{sec-perfect-setup'}. The results of this subsection are needed to analyze the correlation structure of our events in the next subsection and to show that the perfect points are in fact contained in the sets whose Hausdorff dimension we want to compute in Section~\ref{haus lower sec}. The main result of this subsection is the following proposition.

\begin{lem}\label{lem-E_z-basics}
Assume we are in the setting of Section~\ref{sec-perfect-setup'}, and recall in particular the event $E_n(z)$ for $n\in\BB N$ and $z\in B_d(0)$ from~\eqref{eqn-E(z)-def0}. On $E_n(z)$ let $\Phi^\fl_{z,n}$ be the conformal map from $\BB D\setminus (\eta^{\tau_{z,n}^{\fl,*}} \cup \ol\eta^{\ol \tau_{z,n}^{\fl,*}} )$ to $\BB D$ which takes $-i^+$ to $-i$, $i^-$ to $i$, and 1 to 1. We can choose the parameter $r$ sufficiently small, in a manner depending only on $a$, and $\beta_0$ (and hence every $\beta_j$) sufficiently large, in a manner which does not depend on $(u_j)$ and is uniform for $z\in B_d(0)$, in such a way that the following holds a.s.\ on $E_n(z)$, with all implicit constants deterministic and independent of $n$ and uniform for $z\in B_d(0)$.  
\begin{enumerate}  
\item We have   \label{item-Phi'-asymp}
\eqbn
   e^{-\ol\beta_{ n} q - 2 \ol u_{ n}   }   \preceq |(\Phi^\fl_{z,n })'(z)|  \preceq  e^{ - \ol\beta_{ n}  q +      2 \ol u_{ n}    }   .  
\eqen   
\item There is a constant $\lambda_* > 0$, independent of $n$ and uniform for $z\in B_d(0)$, such that \label{item-eta-dist} 
\eqbn
   e^{-\ol\beta_{ n}   - \lambda_* n  }   \leq \op{dist}\left(z , \eta^{\tau_{z,n}^{\fl,*}} \cup \ol\eta^{\ol \tau_{z,n}^{\fl,*}}\right)    \leq  e^{ - \ol\beta_{ n}    +      \lambda_* n }   .  
\eqen 
\item  We have \label{item-eta-tips}
\[
 |\eta(\tau_{z,n}^{\fl,*}) - z| \asymp |\ol\eta(\ol\tau_{z,n}^{\fl,*}) - z| \asymp  \op{dist}\left(z , \eta^{\tau_{z,n}^{\fl,*}} \cup \ol\eta^{\ol\tau_{z,n}^{\fl,*}} \right)  
 \] 
\item We have \label{item-D-diam}
\[
e^{ -\ol\beta_n  - \lambda_* n } \leq \op{dist}(z , \partial D_{z,n}^\fl) \leq \op{diam} D_{z,n }^\fl \leq e^{-\ol\beta_n + \lambda_* n} .
\] 
\end{enumerate}
\end{lem}

Lemma~\ref{lem-E_z-basics} is the only statement from this subsection which will be needed in later sections, and the proof is a rather technical complex analysis argument. 
The reader may wish to skip the rest of this subsection to see the more probabilistic aspects of the proofs of our main results. 
 
It may seem at first glance that Lemma~\ref{lem-E_z-basics} should be a simple consequence of the definitions in Section~\ref{sec-perfect-setup'} and the chain rule. This is not the case, however, as at each stage in our construction we restrict to the domain $D_{z,j}$ so $\Phi_{z,n}^\fl$ (which is defined on all of $\BB D\setminus (\eta^{\tau_{z,n}^{\fl,*}} \cup \ol\eta^{\ol \tau_{z,n}^{\fl,*}} )$) cannot be expressed as a composition of maps defined in Section~\ref{sec-perfect-setup'}. 
To prove the lemma, we will express $\Phi_{z,n}^\fl$ as a composition of maps corresponding to scales $j=1,\dots,n$ (see in particular~\eqref{eqn-Phi-decomp}) then argue that these maps are in some sense comparable to the maps appearing in Section~\ref{sec-perfect-setup'}. 

To prove Lemma~\ref{lem-E_z-basics} we will need to compare the derivatives of several different maps. To this end, we will define the following objects.  
\begin{itemize}
\item Conformal maps $\psi^\fl_{z,j}$, $\wt \phi_{z,j} $, $\wh \phi_{z,j}$, $f_{z,j}$, and $g_{z,j}$. \label{item-analytic-def}
\item Random times $\wt\tau_{z,j}^{ *}$ and $\ol{\wt \tau}_{z,j}^{ *}$.
\item Points $\wt x_{z,j} $ and $\wt y_{z,j} $.
\item Curves $\wt\eta_{z,j} $.
\end{itemize}
For the definitions, we recall the notational conventions discussed at the beginning of Section~\ref{sec-2pt-outline}. We assume we are working on the event $E_j(z)$ for all of these definitions.  

For $j\in\BB N$, let $\psi^\fl_{z,j}$ be the conformal map from $\BB D\setminus (\eta^{\tau_{z,j}^{\fl,*}} \cup \ol\eta^{\ol \tau_{z,j}^{\fl,*}})$ to $\BB D$ which fixes $0$ and whose derivative at $0$ has the same argument as $ (\Phi^\fl_{z,j})'(z)$ (the latter map is defined in Lemma~\ref{lem-E_z-basics}). 
 
For $j=1$, the conformal automorphism $f_{z,1}$ taking $z$ to $0$ has already been defined in Section~\ref{sec-perfect-setup}. For $j\geq 2$, we let $f_{z,j}:\BB D\rta\BB D$ be the conformal automorphism which takes $\Phi^\fl_{z,j-1}(z)$ to $0$ with $f_{z,j}'(\Phi^\fl_{z,j-1} (z) ) > 0$.  Observe that $ \psi_{z,j-1}^\fl = f_{z,j} \circ \Phi^\fl_{z,j-1}$ (here we take $\Phi^\fl_{z,0}$ to be the identity map and $ \psi_{z,0}^\fl = f_{z,1}$ in the case $j=1$). 

For $j\geq 1$, let $\wt \eta_{z,j} $ be the image under $ \psi_{z,j-1}^\fl$ of the part of $\eta$ between $\eta(\tau_{z,j-1}^{\fl,*})$ and $\ol\eta(\ol \tau_{z,j-1}^{\fl,*})$. Note that $\wt\eta_{z,j}$ is a conformal image of the same part of the curve $\eta$ as $\eta_{z,j} $, but the conformal map used to get $\wt\eta_{z,j} $ is defined on $\BB D\setminus (\eta^{\tau_{z,j}^{\fl,*}} \cup\ol\eta^{\ol\tau_{z,j}^{\fl,*}})$ rather than the pocket $D_{z,j}^\fl$. Let $\wt\tau_{z,j}^{   *} $ and $\ol{\wt \tau}_{z,j}^{  *}  $ be the times for $\wt\eta_{z,j}  $ and its time reversal $\ol{\wt \eta}_{z,j}  $ such that 
\eqbn
\psi_{z,j}^\fl(\eta(\tau_{z,j}^{\fl,*}) ) = \wt\eta_{z,j}(\wt\tau_{z,j}^*) \quad \op{and} \quad 
\psi_{z,j}^\fl(\ol\eta(\ol\tau_{z,j}^{\fl,*}) ) = \ol{\wt\eta}_{z,j}(\ol{\wt\tau}_{z,j}^*)
\eqen

Let $\wt x_{z,j} $ and $\wt y_{z,j} $ be the start and end points for $\wt\eta_{z,j} $. Let $\wt\phi_{z,j} : \BB D \setminus (\wt\eta_{z,j}^{\wt\tau_{z,j}^* } \cup \ol{\wt\eta}_{z,j}^{\ol{\wt\tau}_{z,j}} ) \rta \BB D  $ which takes $\wt x_{z,j}^+$ to $-i$, $ \wt y_{z,j}^-$ to $i$ and the midpoint of $[\wt x_{z,j}  , \wt y_{z,j} ]_{\partial\BB D}$ to 1. Let $g_{z,j}  : \BB D\rta \BB D$ be the conformal automorphism taking $(\wt\phi_{z,j} \circ f_{z,j})(b)$ to $b$ for $b = -i^+ , i^-  ,1$. Let
\eqb \label{eqn-wh-phi-def}
\wh\phi_{z,j} := g_{z,j}  \circ \wt \phi_{z,j} \circ f_{z,j} : \BB D \setminus (\wt\eta_{z,j}^{\wt\tau_{z,j}^* } \cup \ol{\wt\eta}_{z,j}^{\ol{\wt\tau}_{z,j}} ) \rta \BB D
\eqe 
and observe that that (with $\Phi^\fl_{z,j}$ as in Lemma~\ref{lem-E_z-basics}) 
\eqb \label{eqn-Phi-decomp}
\Phi^\fl_{z,j} = \wh\phi_{z,j} \circ \cdots \circ \wh \phi_{z,1}. 
\eqe 
See Figure~\ref{fig-flow-line-maps} for an illustration of these maps in the case $j=2$ (which has all of the features of the general case). 

The following straightforward lemma tells us that on $E_n(z)$, the derivatives at $0$ of the conformal maps from $D_{z,n}^\fl$ to $\BB D$ and from $\BB D\setminus (\eta_{z,n}^{\tau_{z,n}^{\fl,*}} \cup  \ol \eta_{z,n}^{\ol\tau_{z,n}^{\fl,*}}  )$ to $\BB D$ which take $z$ to $0$ are comparable (equivalently, by the Koebe quarter theorem, the distance from $z$ to $\bdy D_{z,n}^\fl$ is comparable to the distance from $z$ to $\eta_{z,n}^{\tau_{z,n}^{\fl,*}} \cup  \ol \eta_{z,n}^{\ol\tau_{z,n}^{\fl,*}}$). 
 
\begin{figure}[ht!]
\begin{center}
\includegraphics[scale=.7]{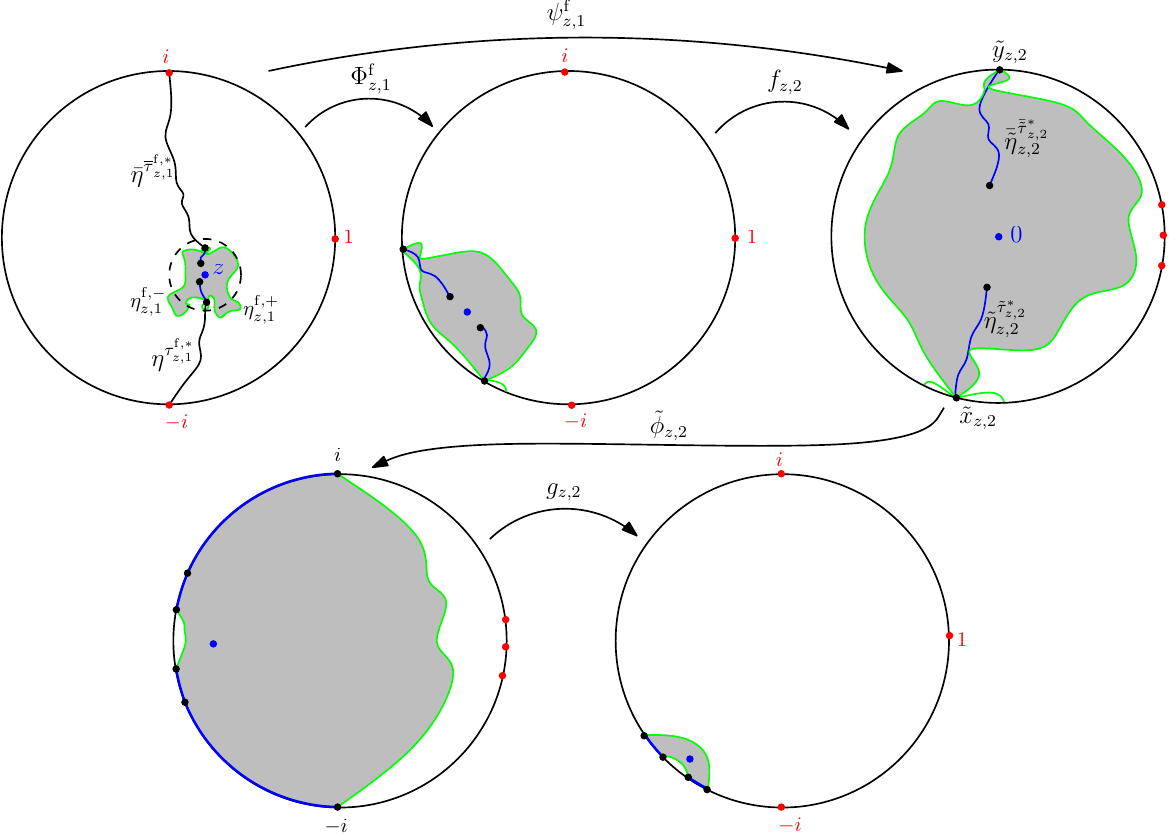}
\caption{An illustration of some of the maps associated with the events $E_{z,1}$ and $E_{z,2}$. The images of $-i$, $i$, and 1 are shown in red. The images of $z$ are shown in blue. The black curves are associated with the event $E_{z,1}$ and the blue curves are associated with the event $E_{z,2}$. The map $\wh\phi_{z,2}$ is the composition of the last three maps in the figure. The map $\Phi^\fl_{z,2}$ is the composition of all four maps.} \label{fig-flow-line-maps}
\end{center}
\end{figure} 

\begin{lem} \label{lem-pi-and-psi}
If $\beta_0$ is chosen sufficiently large, independently of everything else, then on the event $E_n(z)$, 
\eqb \label{eqn-pi-and-psi}
|(\pi_{z,n}^\fl)'(z)| \asymp |(\psi_{z,n}^\fl)'(z)| ,
\eqe 
with the implicit constants independent of $n$ and uniform for $z\in B_d(0)$. 
\end{lem}
\begin{proof} 
Assume we are working on the event $ E_n(z)$. Let $\wh \pi_{z,n-1} $ be the conformal map from $ \psi_{z,n}^\fl( D_{z,n-1}^\fl)$ to $\BB D$ with $\wh \pi_{z,n-1} (0) = 0$ and $\wh \pi_{z,n-1} '(0) > 0$ (in the case $n = 1$, we take $\wh \pi_{z,n-1} $ to be the identity). Let $\wh \pi_{z,n}^{* }$ be the conformal map from $(\wh \pi_{z,n -1}   \circ   \psi_{z,n}^\fl)(D^\fl_{z,n})$ to $\BB D$ with $\wh \pi_{z,n}^{* }(0)=0$ and $\op{arg} (\wh \pi_{z,n}^{* })'(0)$ chosen in such a way that
\eqb \label{eqn-psi-decomp}
 \pi_{z,n}^\fl =  \wh \pi_{z,n}^{* }  \circ  \wh \pi_{z,n-1}  \circ \psi_{z,n}^\fl .
\eqe
By the Beurling estimate and \cite[Exercise 2.7]{lawler-book} the diameters of the connected components of $\BB D\setminus \psi_{z,n}^\fl(D_{z,n-1}^\fl)$ each tend uniformly to $0$ as $\beta_n \rta \infty$ (and hence also as $\beta_0 \rta \infty$). Therefore, if $\beta_0$ is chosen sufficiently large, then $|(\wh \pi_{z,n-1})'(0)|\asymp 1$. 

Let $\psi_{z,n} : \BB D \setminus (\eta_{z,n}^{\tau_{z,n}} \cup \ol\eta_{z,n}^{\ol\tau_{z,n}}) \rta \BB D$ be as in condition~\ref{item-F-b-bar} in the definition of $F_{z,n}$.
The set $(\wh \pi_{z,n-1}  \circ \psi_{z,n}^\fl)(\partial D_{z,n}^\fl)$ is the image of $\psi_{z,n}(\partial D_{z, n})$ under a conformal map which fixes $0$ and maps the complement of the set $ \psi_{z,n} ( \eta_{z,1}([\tau_{z,n} ,  \tau_{z,n}^*]) ) \cup \psi_{z,n}(\ol\eta_{z,n}( [\ol \tau_{z,n} , \ol\tau_{z,n}^*])$ to $\BB D$.  
By condition~\ref{item-E_z-stay} in the definition of $E_{z,n}$, the distance from $0$ to $(\wh \pi_{z,n-1} \circ \psi_{z,n}^\fl)(\partial  D_{z,n}^\fl)$ is proportional to the distance from $0$ to $\psi_{z,n}(\partial D_{z, n})$. By condition~\ref{item-F-contained} in the definition of $F_{z,n}$, this distance is $\asymp 1$. Consequently, $|(\wh \pi_{z,n}^{*})'(0)|\asymp 1$ so~\eqref{eqn-pi-and-psi} follows from~\eqref{eqn-psi-decomp}. 
\end{proof}

\begin{lem} \label{lem-wt-phi-G}
Let $\zeta \in (0 , a/100)$. If the auxiliary parameter $r$ is at most some constant depending only on $a$ and $\zeta$, and $\beta_0$ is chosen sufficiently large (in a manner which does not depend on $(u_j)$ and is uniform for $z\in B_d(0)$) then for any $n\in\BB N$ and any sub-arc $I$ of $[\wt x_{z,n+1}  , \wt y_{z,n+1} ]_{\partial\BB D}$ lying at distance at least $\zeta $ from $\wt x_{z,n+1} $ and $\wt y_{z,n+1} $, the map $\wt\phi_{z,n+1} $ is Lipschitz on $I $ and $ \wt\phi_{z,n+1}^{-1}$ is Lipschitz on $\wt\phi_{z,n+1} (I)$ on the event $E_n(z)$ with Lipschitz constants independent of $(\beta_j)$ and $(u_j)$ and uniform for $z \in B_d(0)$ and $n\in\BB N$.
\end{lem}
\begin{proof}
See Figure~\ref{fig-phi-G-maps} for an illustration of the argument. Throughout, we work on the event $E_n(z)$.

Let $A :=  \psi_{z,n }^\fl(( \eta_{z,n}^{\fl,+})^{t_{z,n}^+})$, where here we recall that $\eta_{z,n}^{\fl,+}$ is the stage-$n$ right auxiliary flow line for $h$. Then $A$ disconnects $\wt\eta_{z,n+1}$ from $I $ in $\partial\BB D$. We claim that if~$r$ is chosen sufficiently small then there is a constant $\delta  > 0$, depending only on $\zeta$, $d$, and the auxiliary parameters from Definition~\ref{def-aux-parameter}, such that for large enough~$\beta_0$, 
\eqb\label{eqn-A-zeta-dist}
E_n(z) \subset \{\op{dist}(A , I) \geq \delta \} .
\eqe
Given the claim, the statement of the lemma follows from Lemma~\ref{G dist} and the fact that $\wt\eta_{z,n+1}$ lies to the left of $A$ due to the monotonicity of flow lines~\cite[Theorem 1.5]{ig1}.   

Let $\psi_{z,n}^*$ be a conformal map from the connected component of $\BB D\setminus (\eta_{z,n}^{\tau_{z,n}^*} \cup \ol\eta_{z,n}^{\ol\tau_{z,n}^*})$ containing 1 on it boundary to $\BB D$ which fixes 0. This map is defined only up to a rotation, which we will specify shortly. Let $\eta_{z,n+1}^*$ be the image under $\psi_{z,n}^*$ of the part of $\eta_{z,n}$ between $\eta_{z,n}(\tau_{z,n}^*)$ and $\ol\eta_{z,n}(\ol\tau_{z,n}^*)$. We can choose the normalization for $\psi_{z,n}^*$ in such a way that
\[
 \eta_{z,n+1}^*  = \wh \pi_{z,n-1} ( \wt\eta_{z,n+1}  ) ,
\]
with $\wh \pi_{z,n-1} $ as in the proof of Lemma~\ref{lem-pi-and-psi}. 

By condition~\ref{item-wtE-hm} in the definition of $\wt E_{z,n}$ and condition~\ref{item-E_z-stay} in the definition of $E_{z,n}$, the set $\BB D\setminus \psi_{z,n }^\fl( D_{z,n-1}^\fl)$ lies at distance at least a positive constant depending only on $a$ from $\wt x_{z,n}  $ and $\wt y_{z,n} $ on $E_n(z)$. Since the diameters of the connected components of $\BB D\setminus  \psi_{z,n }^\fl( D_{z,n-1}^\fl)$ and $\bdy\BB D$ each tend to $0$ uniformly as $\beta_n \rta \infty$ (by the argument of Lemma~\ref{lem-pi-and-psi}), the map $\wh \pi_{z,n-1}^{-1}$ is nearly constant near $\wt x_{z,n} $ and $\wt y_{z,n}  $ if $\beta_0$ (and hence also $\beta_n$) is sufficiently large. 
By the Schwarz lemma $\wh \pi_{z,n-1}^{-1}  $ increases distances to $\partial\BB D$. Hence the distance from $A$ to $I$ is at least an $n$-independent constant times the distance from $\wh A$ to $I$ if $\beta_n$ is chosen sufficiently large, where $\wh A := \wh \pi_{z,n-1} (A) $.
Hence it is enough to prove~\eqref{eqn-A-zeta-dist} with $\wh A$ in place of $A$. 

Let $I' \supset I$ be a slightly larger arc. By condition~\ref{item-wtE-hm} in the definition of $\wt E_{z,n}$, condition~\ref{item-E_z-stay} in the definition of $E_{z,n}$, and a harmonic measure estimate, the distance from $\wh A$ to $I$ is $\succeq$ the distance from $\psi_{z,n}( (\eta_{z,n}^+)^{t_{z,n}^+})$ to $I'$ if $r$ is chosen sufficiently small, depending only on $a$ and $\zeta$, where $\psi_{z,n} $ is as in the definition of $F_{z,n}$. 
We conclude by applying condition~\ref{item-F-G} in the definition of $F_{z,n}$.
\end{proof}

\begin{figure}
\begin{center}
\includegraphics[scale=.8]{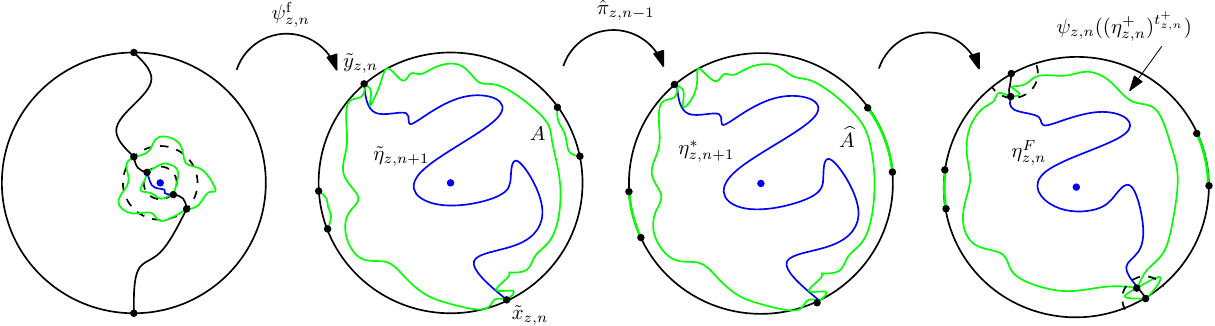}
\caption{An illustration of the maps used in the proof of Lemma~\ref{lem-wt-phi-G}. In order to control the distance from $\wt\eta_{z,n+1}  $ to an arc on the right boundary of the disk, we compare $\wt\eta_{z,n+1} $ to the curve $\eta_{z,n+1}^*$ and then to the curve $\eta_{z,n }^F $ which is the image under $\psi_{z,n}$ of the part of $\eta_{z,n}$ between $\eta_{z,n}(\tau_{z,n})$ and $\ol\eta_{z,n}(\ol\tau_{z,n})$. The distance from the last curve to an appropriate arc of the right boundary is bounded below by condition~\ref{item-F-G} in the definition of $F_{z,n}$.} \label{fig-phi-G-maps}
\end{center}
\end{figure}

We can now get an estimate for the derivatives of our $\phi$-type conformal maps (which, recall, are specified by the images of three boundary points). Iterating this estimate will eventually lead to Lemma~\ref{lem-E_z-basics}.

\begin{lem} \label{lem-phi-asymp}
If the auxiliary parameter $r$ in the definition of $E_{z,n}$ is at most some universal constant, then on $E_n(z)$,
\eqb \label{eqn-phi-asymp}
e^{- \beta_n (q + u_n) } \preceq |\phi'(w)| \preceq e^{-\beta_n (q-u_n)}
\eqe 
where the pair $(\phi , w)$ is any one of $(\phi_{z,n} , 0) ,  (  \wt \phi_{z,n}  , 0)$, or $ (\wh \phi_{z,n} , \Phi^\fl_{z,n-1}(z))$. The implicit constants are independent of $n$ and uniform for $z \in B_d(0)$. 
\end{lem}
\begin{proof}
By condition~\ref{item-wtE-deriv} in the definition of $\wt E_{z,n}$, the statement of the lemma is true for $(\phi  , w) = (\phi_{z,n} , 0)$. 
We will now transfer the estimate~\eqref{eqn-phi-asymp} from $\phi_{z,n}$ to $\wt\phi_{z,n} $ to $\wh\phi_{z,n}$. This latter map is our primary interest, mostly because of~\eqref{eqn-Phi-decomp}. Throughout, we assume that $E_n(z)$ occurs and require all implicit constants to be independent of $n$ and uniform for $z \in B_d(0)$. 

Let $\phi_{z,n}^{*}$ be the conformal map from the connected component of $\BB D \setminus (\eta_{z,n}^{\tau_{z,n}^*} \cup \ol\eta_{z,n}^{\ol\tau_{z,n}^*})$ containing $0$ to $\BB D$ which takes $(x^*_{z,n})^+$ to $-i$, $(y_{z,n}^*)^-$ to $i$, and the midpoint of $[x_{z,n}^* , y_{z,n}^*]_{\bdy\BB D}$ to 1. Intuitively, $\phi_{z,n}^*$ is a slight perturbation of $\phi_{z,n}$ (which is defined in the same manner but with $\tau_{z,n}$ and $\ol \tau_{z,n} $ in place of $\tau_{z,n}^*$ and $\ol \tau_{z,n}^* $). It is easily seen from condition~\ref{item-E_z-stay} in the definition of $E_{z,n}$ that~\eqref{eqn-phi-asymp} for $(\phi_{z,n} , 0)$ implies~\eqref{eqn-phi-asymp} for $(\phi_{z,n}^* , 0)$. 

To transfer from $\phi_{z,n}^*$ to $\wt\phi_{z,n}$, we apply Lemma~\ref{phi hm} to find that for any arc $I \subset [ x_{z,n}^* ,   y_{z,n}^* ]_{\bdy\BB D}$ with length~$\asymp 1$,
\allb \label{eqn-star-tilde-dist}
&|\wt\phi_{z,n}'(0)| 
\asymp   \op{dist}\left( 0 , \wt\eta_{z,n}^{\wt\tau_{z,n}^*} \cup \ol{\wt\eta}_{z,n}^{\ol{\wt\tau}_{z,n}^*} \right)^{-1} \op{hm}^0\left( [\wt x_{z,n} , \wt y_{z,n}]_{\bdy\BB D} ; \BB D\setminus ( \wt\eta_{z,n}^{\wt\tau_{z,n}^*} \cup \ol{\wt\eta}_{z,n}^{\ol{\wt\tau}_{z,n}^*} ) \right) 
\quad \op{and}   \notag\\
&\qquad |(\phi_{z,n}^*)'(0)| \asymp    \op{dist}\left( 0 ,  \eta_{z,n}^{ \tau_{z,n}^*} \cup \ol{ \eta}_{z,n}^{\ol{ \tau}_{z,n}^*} \right)^{-1}  \op{hm}^0\left( I ; \BB D\setminus (\eta_{z,n}^{ \tau_{z,n}^*} \cup \ol{ \eta}_{z,n}^{\ol{ \tau}_{z,n}^*} ) \right)
\alle
By Lemma~\ref{lem-pi-and-psi} (applied with $n-1$ in place of $n$) and the Koebe quarter theorem,
\eqb \label{eqn-dist-compare0}
\op{dist}\left(0 , \wt \eta_{z,n }^{\wt\tau_{z,n }^{*}} \cup  \ol{\wt \eta}_{z,n }^{\ol{\wt\tau}_{z,n }^{*}} \right) \asymp 
\op{dist}\left(0 , (\eta_{z,n}^{\tau_{z,n}^{*}} \cup \ol \eta_{z,n}^{\ol\tau_{z,n}^{*}} \right) 
\eqe 
with the implicit constant independent of $n$ and uniform for $z\in B_d(0)$. Moreover, it is easily seen from condition~\ref{item-F-hm} in the definition of $F_{z,n-1}$ that the harmonic measure terms in~\eqref{eqn-star-tilde-dist} are likewise proportional (here we recall that $[x_{z,n}^* , y_{z,n}^*]_{\bdy\BB D} = \pi_{z,n-1}^\fl(\eta_{z,n-1}^{\fl,+} \cap D_{z,n-1}^\fl)$). 
Thus we obtain~\eqref{eqn-phi-asymp} for $\wt\phi_{z,n}$ from~\eqref{eqn-phi-asymp} for $\phi_{z,n}^*$.

To transfer the estimate to $\wh\phi_{z,n}$, recall~\eqref{eqn-wh-phi-def} and write
\eqb \label{eqn-phi-decomp2}
|\wh \phi_{z,n}'(\Phi^\fl_{z,n-1}(z))| =    |g_{z,n}'( \wt \phi_{z,n} (0) )|   |\wt \phi_{z,n}'(0)|  |f_{z,n}'(\Phi^\fl_{z,n-1}(z)) | , 
\eqe
where here we take $\Phi^\fl_{z,0} $ to be the identity map in the case $n=0$. 
By condition~\ref{item-wtE-hm} in the definition of $\wt E_{n-1}$, we can find $\zeta > 0$ depending only on $a$ such that $ f_{z,n}  ([-i,i]_{\partial\BB D})$ lies at distance at least $\zeta $ from $\wt x_{z,n} $ and $\wt y_{z,n} $ on $E_{n-1}(z)$. By Lemma~\ref{lem-wt-phi-G}, on $E_{n-1}(z)$, it holds that $\wt\phi_{z,n}  $ distorts the distances between points in $f_{z,n}  ([-i,i]_{\partial\BB D})$ by at most a constant factor (here we use that $z\in B_d(0)$ in the case $n=1$). 
The maps $g_{z,n}$ and $f_{z,n}^{-1}$ are two conformal automorphisms of $\BB D$ and each takes three points in $[-i,i]_{\bdy\BB D}$ (which lie at uniformly positive distance from $\pm i$) to $-i$, $i$, and 1. 
Since the distances amongst the marked points for these two conformal maps differ by a constant factor, it follows easily that
\eqbn
|g_{z,n}'(w_1)| \asymp |(f_{z,n}^{-1})'(w_2)|
\eqen
for any points $w_1$ and $w_2$ in the left half of $\BB D$. By combining this with~\eqref{eqn-phi-decomp2} and the estimate~\eqref{eqn-phi-asymp} for $\wt\phi_{z,n}$, we conclude. 
\end{proof}

\begin{proof}[Proof of Lemma~\ref{lem-E_z-basics}]
Throughout, we require all implicit constants to be independent of $n$ and uniform for $z\in B_d(0)$. 
Assume $E_n(z)$ occurs and that $r$ and $\beta_0$ have been chosen so that the conclusion of Lemma~\ref{lem-phi-asymp} holds. Assertion~\ref{item-Phi'-asymp} is immediate from Lemma~\ref{lem-phi-asymp} and the relation~\eqref{eqn-Phi-decomp}. Note that we can absorb the implicit constants in~\eqref{eqn-phi-asymp} into an additional factor of $e^{\ol u_n}$ due to condition~\ref{item-beta-u-infty} of Lemma~\ref{lem-beta-u-choice}. 
  
To prove assertion~\ref{item-eta-dist}, we induct on $n$. The case $n=1$ is immediate from the definitions of the events.  
Now suppose $n\geq 2$ and assertion~\ref{item-eta-dist} has been proven with $n$ replaced by $n-1$. 
Since $( \psi_{z,n-1}^\fl)^{-1}$ maps $\BB D\setminus (\wt \eta_{z,n }^{\wt\tau_{z,n }^{*}} \cup  \ol{\wt \eta}_{z,n }^{\ol{\wt\tau}_{z,n }^{*}})$ to $\BB D\setminus (\eta^{\tau_{z,n }^{\fl,*}} \cup \ol{ \eta}^{\ol{\tau}_{z,n }^{\fl,*}})$ and fixes 0, the Koebe quarter theorem implies that
\eqb \label{eqn-dist-compare}
 \op{dist}\left(z , \eta^{\tau_{z,n }^{\fl,*}} \cup \ol{ \eta}^{\ol{\tau}_{z,n }^{\fl,*}} \right) \asymp    |((\psi_{z,n-1}^\fl)^{-1})'(0)|  \op{dist}\left(0 , \wt \eta_{z,n }^{\wt\tau_{z,n }^{*}} \cup  \ol{\wt \eta}_{z,n }^{\ol{\wt\tau}_{z,n }^{*}} \right) .
\eqe 
By a second application of the Koebe quarter theorem,
\eqb \label{eqn-wh-psi-koebe}
|((\psi_{z,n-1}^\fl)^{-1})'(0)|   \asymp \op{dist}\left(z , \eta^{\tau_{z,n-1}^{\fl,*} } \cup \ol\eta^{\ol \tau_{z,n-1}^{\fl,*}}\right)   .
\eqe 
By the inductive hypothesis,
\eqb \label{eqn-wh-psi-inductive}
 e^{-  \ol\beta_{n-1}    - \lambda_* (n-1)   } \preceq \op{dist}\left(z ,  \eta^{\tau_{z,n-1}^{\fl,*} } \cup \ol\eta^{\ol \tau_{z,n-1}^{\fl,*}} \right) \preceq e^{-  \ol\beta_{n-1}     + \lambda_* (n-1)  } .
\eqe 
By~\eqref{eqn-dist-compare0} and the definition of $E_{z,n}$, 
\eqb  \label{eqn-dist-compare1}
\op{dist}\left(0 , \wt \eta_{z,n }^{\wt\tau_{z,n }^{*}} \cup  \ol{\wt \eta}_{z,n }^{\ol{\wt\tau}_{z,n }^{*}} \right) \asymp e^{-\beta_n} 
\eqe 
on $E_n(z)$. Provided $\lambda_*$ is chosen sufficiently large, independently of $n$ and $z\in B_d(0)$, we can now complete the induction by combining~\eqref{eqn-dist-compare}, \eqref{eqn-wh-psi-koebe},~\eqref{eqn-wh-psi-inductive}, and~\eqref{eqn-dist-compare1}. 
  
By condition~\ref{item-wtE-hm} in the definition of $\wt E_{z,n}$ and condition~\ref{item-E_z-stay} in the definition of $E_{z,n}$, if we choose $r$ sufficiently small relative to $a$ then the harmonic measure from $z$ of each of the two sides of $\eta^{\tau_{z,n}^{\fl,*}}$ (resp.\ each of the two sides of $\ol\eta^{\ol \tau_{z,n}^{\fl,*}}$) in $\BB D\setminus ( \eta^{\tau_{z,n}^{\fl,*} } \cup \ol\eta^{\ol\tau_{z,n}^{\fl,*}})$ is at least some constant which does not depend on $n$ or the particular choice of $z \in B_d(0)$. By the Beurling estimate this implies assertion~\ref{item-eta-tips}. 

For assertion~\ref{item-D-diam}, we use assertion~\ref{item-eta-dist} (with $n-1$ in place of $n$) and the Koebe quarter theorem to see that there exists radii $\rho '  > \rho > $0$ $ such that $\rho \succeq e^{-\ol\beta_{n }  - \lambda_* n }$, $\rho' \preceq e^{-\ol\beta_{n }  +\lambda_* n    }$, $ (\psi_{z,n-1}^\fl)^{-1}(\mcl B_{\beta_n + \Delta}) \supset B_\rho(z)$, and $(\psi_{z,n-1}^\fl)^{-1}(\mcl B_{\beta_n -\Delta}) \subset B_{\rho'}(z)$. By combining this with condition~\ref{item-F-contained} in the definition of $F_{z,n}$ we see that assertion~\ref{item-D-diam} holds (after possibly increasing $\lambda_*$).
\end{proof}

\subsection{Probabilistic properties} 
\label{sec-flow-line-prob}

Continue to assume we are in the setting of Section~\ref{sec-perfect-setup'}.
In this subsection we will prove estimates for the correlations of the events $E_n(z)$ of~\eqref{eqn-E(z)-def0}. These estimates will eventually lead to our two-point estimate, which we now state. 

\begin{prop} \label{prop-2pt-estimate}
Let $z,w\in B_d(0)$. Let $\lambda_*$ be the constant from Lemma~\ref{lem-E_z-basics} and for $n\in\BB N$, defined the events $E_n(z)$ and $E_n(w)$ as in~\eqref{eqn-E(z)-def0}. Choose $ k \in \BB N$ such that $e^{- \ol \beta_{k +1} -\lambda_* (k+1) } \leq |z-w| \leq e^{-\ol \beta_{k } - \lambda_* k  }$. We can choose the auxiliary parameters in a manner depending only on $d$ such that the following is true. If $\beta_0$ is chosen sufficiently large (depending on the auxiliary parameters), then for any $n \in \BB N$ with $\ol\beta_n -\lambda_* n \geq   \ol \beta_{k +1} + \lambda_* (k+2)  $,
\eqb \label{eqn-2pt-estimate}
\BB P(  E_n(z) \cap E_n(w)) \preceq  e^{\ol\beta_k o_k(1)}  \frac{\BB P(  E_n(z)) \BB P(  E_n(w))}{\BB P(  E_k(w))}  
\eqe
with the implicit constants independent of $n$ and $k$, the $o_k(1)$ independent of $n$, and both uniform for $z,w \in B_d(0)$.  
\end{prop}

\begin{remark} \label{remark-2pt-estimate}
In the setting of Proposition~\ref{prop-2pt-estimate},
$e^{-\ol\beta_k  } = |z-w|^{1 +  o_{|z-w|}(1)}$ 
so by Lemma~\ref{lem-P(E_z)} we can rewrite the estimate~\eqref{eqn-2pt-estimate} as
\eqb \label{eqn-2pt-estimate'}
\BB P(  E_n(z) \cap E_n(w)) \preceq  |z-w|^{-\gamma^*(q) + o_{|z-w|}(1)}  \BB P(  E_n(z)) \BB P(  E_n(w))  .
\eqe
This is the form of the estimate we will use when we prove lower bounds for the Hausdorff dimensions of our sets. 
We emphasize that there is no $e^{-\ol\beta_n o_n(1)}$ error in~\eqref{eqn-2pt-estimate'}; this is important for the proofs in Section~\ref{haus lower sec}. 
\end{remark}

Throughout this subsection, we fix the auxiliary parameters from Definition~\ref{def-aux-parameter} in such a way that the conclusions of Lemmas~\ref{lem-E_z-prob} and~\ref{lem-E_z-basics} hold. The starting point of the proof of Proposition~\ref{prop-2pt-estimate} is the following absolute continuity statement. Note that to get strict mutual absolute continuity, we need to skip one scale (i.e., we condition on what happens up to stage $n-2$ and look at the objects at or after stage $n$) in order to re-randomize the locations of the endpoints of the curve.
 
\begin{lem}\label{lem-eta-abs-cont}
Suppose we are in the setting of Section~\ref{sec-perfect-setup'} and for $z \in B_d(0)$ and $j\in\BB N$, let $H_{z,j}$ be the event of Lemma~\ref{lem-endpoint-abs-cont} with $\eta = \eta_{z,j}$ and $r_H$ chosen sufficiently small that $   E_{z,j} \subset H_{z,j}$. If $\beta_0$ is chosen sufficiently large, independently of $z\in B_d(0)$, then for $n\geq 2$ and $z\in B_d(0)$, the following two laws are a.s.\ strictly mutually absolutely continuous (\hyperref[smac]{s.m.a.c.}; Definition~\ref{smac}) modulo rotations of $\BB D$, with deterministic implicit constants uniform in $n$, $(\beta_j , u_j)_{j\geq 1}$, and $z \in B_d(0)$.
\begin{enumerate}
\item The conditional joint law of $ \eta_{z,n}$ and $\left\{  (\eta_{z,j}^+ , \eta_{z,j}^-)  \right\}_{j\geq n}$ given the event $H_{z,n-1}$ and the $\sigma$-algebra $  \mathcal F_{z,n-2} $ of~\eqref{eqn-mclF-def} on the event $E_{n-2}(z)$. \label{item-eta-n-law} 
\item The conditional joint law of $\eta_{z,2}$ and $\left\{ (\eta_{z,2}^+ ,\eta_{z,2}^-) \right\}_{j \geq 1}$ given $H_{z,1}$ with the sequence $(\beta_j , u_j)_{j\in\BB N}$ replaced by $(\beta_{n+j-2} , u_{n+j-2})_{j\in\BB N}$. \label{item-eta-1-law}
\end{enumerate} 
\end{lem}
\begin{proof}  
The $\sigma$-algebra $\mcl F_{z,n}$ for $n\in\BB N$ is generated by flow lines of $h$ which lie outside of $D_{z,n}^\fl$, so since $h$ determines its flow lines in a local manner (this follows from~\cite[Theorem~1.2]{ig1} and the fact that flow lines are local sets in the sense of~\cite{ss-contour}) we infer that $\mcl F_{z,n} \subset \sigma(D_{z,n}^\fl , h|_{\BB D\setminus D})$. 

By Lemma~\ref{lem-F-cond} and induction, we infer that for $n\geq 1$ and any $\frk x , \frk y \in \bdy\BB D$, the conditional joint law of $ \eta_{z,n}$ and $\left\{ (\eta_{z,j}^+,\eta_{z,j}^-) \right\}_{j\geq n}$ given $\mcl F_{z,n-1}$ on the event that the start and end points $x_{z,n}$ and $y_{z,n}$ for $\eta_{z,n}$ are equal to $\frk x$ and $\frk y$, respectively, coincides with the conditional joint law of $\eta_{z,2}$ and $\left\{ (\eta_{z,2}^+ ,\eta_{z,2}^-) \right\}_{j \geq 1}$ given $\{x_{z,2} =\frk x , y_{z,2} = \frk y\}$ with the sequence $(\beta_j , u_j)_{j\in\BB N}$ replaced by $(\beta_{n+j-2} , u_{n+j-2})_{j\in\BB N}$.

Since we require only strict mutual absolute continuity modulo rotations of $\BB D$, in order to prove the statement of the lemma, it therefore suffices to show that the conditional law of $\op{arg}( y_{z,n} / x_{z,n})$ given $H_{z,n-1}$ and $\mcl F_{z,n-2}$ on the event $E_{n-2}(z)$ is \hyperref[smac]{s.m.a.c.} with respect to the conditional law of $(x_{z,2} , y_{z,2} )$ given $H_{z,1}$ (this is why we condition only on $\mcl F_{z,n-2}$---if we conditioned  on $\mcl F_{z,n-1}$, the endpoints $x_{z,n-1}$ and $y_{z,n-1}$ would be determined). This, in turn, follows from Lemma~\ref{lem-endpoint-abs-cont}. 
\end{proof}

In light of Lemma~\ref{lem-eta-abs-cont}, it will be convenient to consider events defined with the sequence $(\beta_j , u_j)_{j\in\BB N}$ replaced by a shifted version. 
In particular, we define $E_n^m(z)$ for $n,m \in \BB N$ in the same manner as the event $E_n(z)$ of~\eqref{eqn-E(z)-def0} but with $(\beta_j , u_j)_{j\in\BB N}$ replaced by $(\beta_{m+j-1} , u_{m+j-1})_{j\in\BB N}$. We similarly define the event $H_{z,j}^m$ as in Lemma~\ref{lem-eta-abs-cont} but with $(\beta_j , u_j)_{j\in\BB N}$ replaced by $(\beta_{m+j-1} , u_{m+j-1})_{j\in\BB N}$.

For $n_1,n_2 \in \BB N$ with $n_1 +1 \leq n_2 $, we also write
\eqb \label{eqn-E(z)-range-def}
E_{n_1,n_2}(z) := \bigcap_{j=n_1+1}^{n_2} E_{z,j} ;
\eqe 
and we define $E_{n_1,n_2}^m(z)$ in the same manner but with $(\beta_j , u_j)_{j\in\BB N}$ replaced by $(\beta_{m+j-1} , u_{m+j-1})_{j\in\BB N}$.

As a consequence of Lemma~\ref{lem-eta-abs-cont}, we get that the following approximate multiplicative property for the probabilities of the events $E_n^m(z)$.

\begin{lem} \label{lem-E-split}
For $z\in\BB D$ and $k , n  , m  \in\BB N$ with $k \leq n-2$,
\eqb \label{eqn-E-split}
 \BB P\left(E_n^m(z)\right)  = e^{O(\beta_{k+m})} \BB P\left(E_k^m (z)\right)  \BB P\left( E_{n-k}^{m+k }(z)  \right) 
\eqe 
with the rate of the $O(\beta_{k+m})$ depending only on the auxiliary parameters.
\end{lem}

We emphasize that the $O(\beta_{k+m})$ error in Lemma~\ref{lem-E-split} does \emph{not} depend on $n$; rather, it will eventually correspond to an error of order $|z-w|^{o_{|z-w|}(1)}$ in~\eqref{eqn-2pt-estimate'}. This error comes from the need to skip one scale in Lemma~\ref{lem-eta-abs-cont}.

\begin{proof}[Proof of Lemma~\ref{lem-E-split}]
We have
\eqb \label{eqn-E-split0}
\BB P\left(E_n^m (z)\right) = \BB P\left(E_k^m (z)\right) \BB P\left(E_n^m (z)  \, |\, E_k^m (z) \right)   .
\eqe  
By Lemma~\ref{lem-eta-abs-cont} and since the definitions of our events are invariant under rotations of $\BB D$,  
with $H_{z,k}^m$ be as above,
\allb \label{eqn-E-split-lower}
 \BB P\left(E_n^m (z)  \, |\, E_k^m (z) \right) 
 \geq \BB P\left( E_{k,n}^m(z)  \,|\, E_{k-1}^m(z) \cap H_{z,k}^m \right) 
 \succeq \BB P\left( E_{1,n-k+1}^{m+k -1}(z) \,|\, H_{z,1}^{m+k-1} \right)   
\alle 
and
\allb \label{eqn-E-split-upper}
 \BB P\left(E_n^m (z)  \, |\, E_k^m (z) \right) 
 \leq \BB P\left( E_{k+1,n}^m(z) \,|\, E_k^m(z) \cap H_{z,k+1}^m \right) 
 \preceq \BB P\left( E_{2,n-k+1}^{m+k-1}(z) \,|\, E_1^{m+k-1}(z)\cap  H_{z,2}^{m+k-1} \right).
\alle  
Using Lemma~\ref{lem-E_z-prob} and some straightforward algebra with conditional probabilities, we see that the right side of~\eqref{eqn-E-split-lower} (resp.~\eqref{eqn-E-split-upper}) is bounded below (resp.\ above) by $e^{O(\beta_{k+m})} \BB P\left( E_{n-k}^{m+k }(z)  \right) $. Plugging this into~\eqref{eqn-E-split0} yields~\eqref{eqn-E-split}. 
\end{proof}

The next lemma is the key input in the proof of Proposition~\ref{prop-2pt-estimate}. It reduces the problem of estimating $\BB P(E_n(z)\cap E_n(w))$ to the estimates of the preceding lemmas, and is the place where we use the local independence provided by the auxiliary flow lines.

\begin{lem} \label{lem-near-independence}
Let $z,w\in B_d(0)$ and let $\lambda_*$ be the constant from Lemma~\ref{lem-E_z-basics}. Choose $ k \in \BB N$ such that $  \frac12 e^{- \ol \beta_{k +1} -\lambda_* (k+1)  } \leq |z-w| \leq \frac12 e^{-\ol \beta_{k } - \lambda_* k   }$. For any $n \in \BB N$ with $\ol\beta_n -\lambda_* n \geq   \ol \beta_{k +1} + \lambda_* (k+1)  $,  
\eqb \label{eqn-near-independence}
\BB P\left( E_n (z) \cap   E_n (w) \,|\,  E_k(z) \cap E_k(w)\right) \preceq  e^{\ol\beta_k o_k(1)} \BB P(  E^{k }_{ n-k }(z) ) \BB P(   E^{k }_{n-k }(w))  
\eqe 
with the implicit constants independent of $n$ and $k$, the $o_k(1)$ independent of $n$, and both uniform for $z,w \in B_d(0)$.  
\end{lem}
\begin{proof} 
Throughout, we require implicit constants and $o_k(1)$ terms to satisfy the conditions of the statement of the lemma. 

Let $k'$ be the least integer such that $\ol\beta_{k'} -\lambda_* k' \geq   \ol \beta_{k +1} + \lambda_* (k+1) $. Note $k\leq k'\leq n$. 
Let $ P_{z,k'}$ be the event that the pocket $  D^\fl_{z,k'}$ formed by the auxiliary flow lines is non-empty and satisfies $\op{diam}(  D^\fl_{z,k'}  ) \leq  e^{- \ol\beta_{k'} + \lambda_* k'  }$ and the endpoints $x_{z,k'}$ and $y_{z,k'}$ for $\eta_{z,k'}$ differ by at least $\wt d$, where $\wt d$ is the constant from Remark~\ref{remark-endpoint-dist}.

By the definition~\ref{eqn-E(z)-def0} of $E_n(z)$, assertion~\ref{item-D-diam} of Lemma~\ref{lem-E_z-basics}, and our choice of $\wt d$ (c.f.\ Remark~\ref{remark-endpoint-dist}),
\eqbn  
E_n(z) \subset   P_{z,k'} \cap E_{k',n}(z) \quad \op{and} \quad E_n(w)  \subset  P_{w,k'} \cap E_{k',n}(z) ,
\eqen 
where here $E_{k',n}(z)$ is as in~\eqref{eqn-E(z)-range-def}. 
Therefore,
\eqb \label{eqn-E-to-P}
  \BB P\left(E_n(z) \cap E_n(w) \,| \, E_k(z) \cap E_k(w) \right)   
 \leq \BB P\left(E_{k',n} (z) \cap E_{k',n} (w)    \,| \, E_k(z) \cap E_k(w) \cap P_{z,k'} \cap  P_{w,k'}  \right) .
\eqe
So, we need only estimate the right side of~\eqref{eqn-E-to-P}.

Let $\mcl H$ be the $\sigma$-algebra generated by $ D^\fl_{z,k'}$, $D^\fl_{w,k'}$, and $h|_{ \BB D\setminus (D^\fl_{z,k'} \cup D^\fl_{w,k'}   )  }$. By our choices of $k$ and $k'$, on the event $P_{z,k'} \cap P_{w,k'}$, the domains $ D^\fl_{z,k'}$ and $ D^\fl_{w,k'}$ are disjoint.  
Hence $  P_{z,k'}$ and $ P_{w,k'}$ belong to $\mcl H$ (the boundary data of $h|_{\partial   D^\fl_{z,k'}}$ determines the locations of $  x_{z,k'} $ and $ y_{z,k'} $ and similarly with $w$ in place of $z$). 
By assertion~\ref{item-D-diam} of Lemma~\ref{lem-E_z-basics} (applied with $k$ in place of $n$) and our choices of $k$ and $k'$, on the event $E_k(z) \cap E_k(w) \cap   P_{z,k'} \cap  P_{w,k'}$, 
\alb
 D^\fl_{z,k'} \cup  D^\fl_{w , k'} 
 \subset B_{ e^{-\ol\beta_{k'} + \lambda_* k'}}(z) \cup B_{ e^{-\ol\beta_{k'} + \lambda_* k'}}(w) \\
 \subset B_{ e^{-\ol\beta_{k +1} - \lambda_* (k+1) }}(z) \cup B_{ e^{-\ol\beta_{k+1} - \lambda_* (k+1)}}(w) 
& \subset   B_{ e^{-\ol\beta_k - \lambda_* k }}(z) \cap  B_{ e^{-\ol\beta_k - \lambda_* k }}(w)
\subset  D^\fl_{z,k} \cap D^\fl_{w,k} .
\ale
Since flow lines are determined locally by the field, the event $E_k(z)$ is determined by $ D_{z, k}^\fl$ and $h|_{\BB D\setminus  D_{z,k}^\fl}$, and similarly with $w$ in place of $z$. 
Therefore, $E_k(z) \cap E_k(w) \cap P_{z,k'} \cap  P_{w,k'} \in \mcl H$. 
  
By~\cite[Theorem 1.2]{ig1}, the objects involved in the definition of $E_{k',n}(z)$ are a.s.\ determined by $h|_{D_{z,k'}^\fl}$ and similarly with $w$ in place of $z$. Hence the preceding paragraph together with Lemma~\ref{lem-F-cond} imply that the events $E_{k' ,n} (z)$ and $E_{k' ,n}(w)$ are conditionally independent given $\mcl H$ on the event $E_k(z) \cap E_k(w)\cap  P_{z,k'} \cap P_{w,k'}$, i.e., on this event,
\eqb \label{eqn-E-ind}
 \BB P\left(E_{k' ,n} (z) \cap E_{k',n}(w) \,|\,  \mathcal H   \right)   
 = \BB P\left(E_{k' ,n} (z) \,|\,   \mathcal H   \right)   \BB P\left(E_{k' ,n} (w) \,|\,     \mathcal H   \right)  .
\eqe
By Lemma~\ref{lem-F-cond}, the conditional law of the objects involved in the definitions of $E_{z,j}$ for $j\geq k'+1$ given $\mcl H$ is the same as the conditional law of these objects given $\mcl F_{z,k'}$ on the event $E_k(z) \cap E_k(w)\cap  P_{z,k'} \cap P_{w,k'}$. Since $E_{k'}(z) \subset H_{z,k'+1} \cap E_{k'+1,n}(z)$, Lemma~\ref{lem-eta-abs-cont} implies that (in the notation defined just above~\eqref{eqn-E(z)-range-def}),
\allb \label{eqn-E-cond-prob}
\BB P\left( E_{k' ,n}(z) \,|\,   \mathcal H    \right) \BB 1_{ E_k(z) \cap E_k(w) \cap  P_{z,k'} \cap  P_{w,k'}} \preceq \BB P\left(  E_{1,n-k' }^{k'+1}(z) \,|\, H_{z,1}^{k'+1} \right) \BB 1_{ E_k(z) \cap E_k(w) \cap  P_{z,k'} \cap  P_{w,k'}} ,
\alle  
and similarly with $z$ and $w$ interchanged. Using Lemma~\ref{lem-E_z-prob} and straightforward algebra with conditional probabilities, we get
\eqbn
 \BB P\left(  E_{1,n-k' }^{k'+1}(z) \,|\, H_{z,1}^{k'+1} \right) \leq e^{O_{k'}(\beta_{k'})}  \BB P\left(  E_{ n-k' }^{k'+1}(z)   \right) 
\eqen
so by Lemma~\ref{lem-E-split} (applied with $k$ in place of $m$ and $k'-k+1$ in place of $k$),
\eqbn
 \BB P\left(  E_{1, n-k' }^{k'+1}(z) \,|\, H_{z,1}^{k'+1} \right) \leq e^{O_{k'}(\beta_{k'})}  \frac{\BB P(E_{n-k}^k(z) ) }{\BB P\left(E_{ k'-k +1}^k(z) \right) } .
\eqen 
By Lemma~\ref{lem-P(E_z)} (applied with $(\beta_{j+k} , u_{j+k})$ in place if $(\beta_j , u_j)$) and by our choice of $k$ and $k'$, 
\eqbn
\BB P\left(E_{ k'-k + 1}^k(z) \right)  
\succeq e^{- \ol\beta_k o_k(1)} .
\eqen
Therefore,
\eqb \label{eqn-3E}
\BB P\left(  E_{1,n-k' }^{k'+1}(z) \,|\, H_{z,1}^{k'+1} \right) \leq e^{\ol\beta_k o_k(1)} \BB P(E_{n-k}^k(z) )  .
\eqe 
We also have the analog of~\eqref{eqn-3E} with $w$ in place of $z$.

By~\eqref{eqn-E-to-P},~\eqref{eqn-E-ind},~\eqref{eqn-E-cond-prob}, and~\eqref{eqn-3E}, we obtain~\eqref{eqn-near-independence}. 
\end{proof}

\begin{proof}[Proof of Proposition~\ref{prop-2pt-estimate}]
We have
\alb
\BB P\left(  E_n(z) \cap   E_n(w)\right) &= \BB P\left(    E_{n}(z) \cap   E_{n}(w) \,|\,  E_k(z) \cap   E_k(w) \right)  \BB P\left(   E_k(z) \cap E_k(w)      \right)   \quad \text{(by definition}) \\  
&\preceq   e^{\ol\beta_k o_k(1)}   \BB P\left(  E^k_{n-k}(z) \right) \BB P\left(   E^k_{n-k}(w)\right)  \BB P\left( E_k(z)  \right)  \quad  \text{(by Lemma~\ref{lem-near-independence})} .
\ale
By Lemma~\ref{lem-E-split} (applied with $m=0$ and $n-k$ in place of $n$),
\eqbn
 \BB P\left(   E^k_{n-k}(w)\right)   = e^{o_k(1) \ol\beta_k }   \frac{\BB P\left( E_n(w)\right)}{\BB P\left(   E_k(w)\right)}   \quad\op{and} \quad   \BB P\left(  E^k_{n-k}(z) \right)  \BB P\left(   E_k(z)   \right) =   e^{o_k(1) \ol\beta_k }  \BB P\left(   E_n(z)\right) .
\eqen
By combining the above relations we get~\eqref{eqn-2pt-estimate}. 
\end{proof}

\subsection{Remarks on adaptations to other settings}
\label{sec-other-settings}

We expect that the arguments in this section can be adapted to prove two-point estimates for other sets associated with SLE or CLE which can be coupled with a GFF using imaginary geometry.  
Here we make some remarks about which aspects of the definitions of our events and our proofs are also useful in other settings and which are specific to the multifractal spectrum (and hence are unnecessary when working with other sets). See also~\cite{miller-wu-dim,light-cone-dim} for other examples of Hausdorff dimension calculations using imaginary geometry. 

The regularity events $\mcl G(f;\mu)$ and $\mcl G'(A;\mu)$ of Section~\ref{G prelim H} seem to be useful in general when dealing with SLE since they allow us to avoid pathological behavior of the curve near the boundary and control how much points on the boundary are moved by conformal maps.
Other regularity conditions could be used for this purpose but this might lead to more complicated definitions of events for the two-point estimate. 

The most basic simplification one can make when computing the dimension of sets other than the multifractal spectrum sets (e.g., the dimension of the SLE$_\kappa$ curve) is that it is not always necessary to grow the curve from both the forward and reverse direction simultaneously. We need to do this in the setting of the present paper since we would get only the derivative behavior near the tip of the curve, not the derivative behavior in the bulk, if we only grew the curve in the forward direction.
This makes some definitions easier since one does not have to worry about the fact that the time reversal of a flow line is not a flow line. 
 
The main purpose of the first event $L$ from Section~\ref{sec-perfect-setup} is to allow us to apply Lemma~\ref{rho abs cont} in order to transfer the estimate for the event $\wt E$ in the case $\rho^L = \rho^R = 0$ to the case of general $\rho^L , \rho^R \in (-2,0]$ in the proof of Lemma~\ref{lem-wtE-prob} (we need the estimate to hold for $\rho^L , \rho^R \not=0$ since the segment of $\eta$ inside the pocket formed by the auxiliary flow lines is an SLE$_\kappa(\rho^L ; \rho^R)$ for non-zero $\rho^L , \rho^R$). 
If one is growing $\eta$ in only the forward direction, rather than in the forward and reverse directions simultaneously, one can simplify the definition of $L$ and apply~\cite[Lemma 2.8]{miller-wu-dim} in place of Lemma~\ref{rho abs cont}.

The event $\wt E$ from Section~\ref{sec-perfect-setup} is of course specific to the multifractal spectrum. For other dimension calculations $\wt E$ would be replaced by an entirely different event. 

In other settings, one would still need to introduce the auxiliary flow lines $\eta^\pm$ and define some variant of the regularity event $F$ for these flow lines as in Section~\ref{sec-perfect-setup}. The specific regularity conditions in the definition of $F$ can be modified somewhat depending on the situation, but one always needs to make sure that $\eta^\pm$ form a pocket containing the point of interest (0, in our case) and that the images of the points where $\eta$ enters and exits this pocket under a conformal map fixing the point of interest lie at uniformly positive distance from one another.

The proof of Lemma~\ref{lem-F-cond} and the iterative construction of Section~\ref{sec-perfect-setup'} would also remain largely unchanged in other settings. 

When using auxiliary pockets to define curves iteratively, one needs some way to deal with the fact that the laws of the curves $\eta_{z,n}$ are not exactly stationary in $n$.
In our setting, the endpoints of $\eta_{z,n}$ are different for each $n$ and we get around this issue by skipping one scale to re-randomize the endpoints (Lemma~\ref{lem-endpoint-abs-cont}). 

Most of the conditions in Lemma~\ref{lem-E_z-basics} are specific to the multifractal spectrum and are used to show that the perfect points are contained in the multifractal spectrum sets. For the proof of the two-point estimate one really only needs to show that the size of pockets $D_{z,j}$ is of the right order (i.e., condition~\ref{item-D-diam} in Lemma~\ref{lem-E_z-basics}). In other settings one would need to establish different analytic properties to show that the perfect points are contained in the sets of interest; establishing such properties would replace most of Section~\ref{sec-flow-line-analytic}.  

The argument of Section~\ref{sec-flow-line-prob} should remain largely unchanged for other two-point estimate proofs using imaginary geometry. In particular, one still has to establish strict mutual absolute continuity of the objects used to define the events at each scale (Lemma~\ref{lem-eta-abs-cont}), use this to prove approximate multiplicativity of the probabilities of the events $E_n(z)$ (Lemma~\ref{lem-E-split}), then use the independence of what happens inside disjoint pockets formed by auxiliary flow lines to conclude.

\subsection{Index of notation}
\label{sec-2pt-index}

In this subsection we list most of the notation used in Section~\ref{2pt sec}. Each symbol is linked to the location in the text where it is defined. Note that the subscript $z,j$ is dropped in Section~\ref{sec-perfect-setup}. We also recall the notational conventions discussed at the beginning of Section~\ref{sec-2pt-outline}.

\begin{multicols}{2}
\begin{itemize}
\item \hyperref[ball abbrv]{$\mcl B_\beta$} for $\beta>0$; Euclidean ball $B_{e^{-\beta}}(0)$.  
\item \hyperref[remark-endpoint-dist]{$\wt d$}; lower bound for the distance between the endpoints of the curve.
\item \hyperref[sec-perfect-setup']{$h_{z,j} $}; Intermediate GFF, equal to $h\circ (\pi_{z,j-1}^\fl)^{-1} -\chi \op{arg}( (\pi_{z,j-1}^\fl)^{-1}) '$.  
\item \hyperref[sec-perfect-setup']{$\eta_{z,j}$}; $j$th curve in construction, equal to $\pi_{z,j-1}(\eta_{z,j-1}\cap D_{z,j-1})$ for $j\geq 2$, is an SLE$_\kappa(\rho^0;\rho^0)$ for $j\geq 2$.
\item \hyperref[sec-perfect-setup']{$x_{z,j} $} and \hyperref[sec-perfect-setup']{$y_{z,j} $}; Start and end points for $\eta_{z,j} $. 
\item \hyperref[sec-perfect-setup]{$x_{z,j}^{ *}$} and \hyperref[sec-perfect-setup]{$y_{z,j}^{*}$}; Endpoints of largest arc of $[x_{z,j} , y_{z,j} ]_{\bdy\BB D}$ not hit by $ \eta_{z,j}^{\sigma_{z,j} } $ or $ \ol\eta_{z,j}^{\ol\sigma_{z,j} }$.  
\item \hyperref[sec-perfect-setup]{$\sigma_{z,j} $} and \hyperref[sec-perfect-setup]{$\ol\sigma_{z,j}$}; Hitting times of $\mcl B_\Delta$ by $\eta_{z,j} $ and $\ol\eta_{z,j}$. 
\item \hyperref[item-L-hit]{$L_{z,j}$}; Regularity event for $ \eta_{z,j}^{\sigma_{z,j}} $ and $\ol\eta_{z,j}^{\ol\sigma_{z,j} }$.
\item \hyperref[sec-perfect-setup]{$\eta_{z,j}$}; Curve close in law to ordinary SLE$_\kappa$; equal to $\psi_{z,j}(\eta_{z,j} \setminus (\eta_{z,j}^{\sigma_{z,j}} \cup \ol\eta_{z,j}^{\ol\sigma_{z,j}})$. 
\item \hyperref[item-wtE-hit]{$\wt E_{z,j}$}; Event with derivative conditions for $\eta_{z,j}$ at its hitting time of $\mcl B_\beta$. 
\item \hyperref[item-wtE-hit]{$\tau_{z,j}$} and \hyperref[item-wtE-hit]{$\ol\tau_{z,j}$}; Times when $\eta_{z,j}$ and $\ol\eta_{z,j}$ hit $\mcl B_\beta$. 
\item \hyperref[item-wtE-hit]{$\phi_{z,j}$}; Conformal map $\BB D\setminus (\eta_{z,j}^{\tau_{z,j}} \cup \ol\eta_{z,j}^{\ol\tau_{z,j}}) \rta\BB D$ with $\phi_{z,j}(x_{z,j}^-) = -i$ and $\phi_{z,j}(y_{z,j}^-)=i$. 
\item \hyperref[eqn-rho0-rho1]{$\eta_{z,j}^\pm$}; Auxiliary flow lines started from $\eta_{z,j}(\tau_{z,j})$. 
\item \hyperref[eqn-rho0-rho1]{$D_{z,j}$}; Pocket formed by $\eta_{z,j}^\pm$ containing 0. 
\item \hyperref[eqn-rho0-rho1]{$\pi_{z,j}$}; Map $D_{z,j} \rta \BB D$ fixing 0. 
\item \hyperref[item-F-contained]{$t_{z,j}^\pm$}; Time when $\eta_{z,j}^\pm$ finishes tracing $\bdy D_{z,j}$. 
\item \hyperref[item-F-contained]{$\wt t_{z,j}^\pm$}; Exit time of $\eta_{z,j}^\pm$ from $\mcl B_{\beta -\Delta} \setminus \mcl B_{\beta + \Delta}$. 
\item \hyperref[item-F-contained]{$F_{z,j}$}; Regularity event for $\eta_{z,j}^\pm$. 
\item \hyperref[item-F-contained]{$b_{z,j}$} and \hyperref[item-F-contained]{$\ol b_{z,j}$}; Intersection points of $\eta_{z,j}^\pm$ on $\bdy D_{z,j}$.
\item \hyperref[item-F-b-bar]{$\psi_{z,j}$}; Conformal map $\BB D\setminus (\eta_{z,j}^{\tau_{z,j}} \cup \ol\eta_{z,j}^{\ol\tau_{z,j}}) \rta\BB D$ fixing 0. 
\item \hyperref[item-F-b-bar]{$x_{z,j}^F$} and \hyperref[item-F-b-bar]{$y_{z,j}^F$}; Endpoints of $\psi_{z,j} (\eta_{z,j} \setminus (\eta_{z,j}^{\tau_{z,j}} \cup \ol\eta_{z,j}^{\ol\tau_{z,j}}) )$. 
\item \hyperref[item-E_z-F]{$\tau_{z,j}^*$} and \hyperref[item-E_z-F]{$\ol\tau_{z,j}^*$}; Times when $\eta_{z,j}$ and $\ol\eta_{z,j}$ hit $D_{z,j}$.  
\item \hyperref[item-E_z-F]{$E_{z,j}$}; Event containing $L_{z,j}$, $\wt E_{z,j}$, $F_{z,j}$, and conditions for $\eta_{z,j}([\tau_{z,j} , \tau_{z,j}^*])$, $\ol\eta_{z,j}([\ol\tau_{z,j} , \ol\tau_{z,j}^*])$.  
\item \hyperref[eqn-E(z)-def0]{$E_n(z)$}; $\bigcap_{j=1}^n E_{z,j}$.  
\item \hyperref[eqn-mclF-def]{$\mcl F_{z,j}$}; $\sigma$-algebra generated by objects used to define $E_n(z)$.  
\item \hyperref[eqn-eta^fl-aux-def]{$\eta_{z,j}^{\fl,\pm}$}; Flow line of $h$ corresponding to $\eta_{z,j}^\pm$.
\item \hyperref[eqn-pi^fl-def]{$D_{z,j}^\fl$}; Sub-domain of $\BB D$ containing $z$ bounded by $\eta_{z,j}^\pm$. 
\item \hyperref[eqn-pi^fl-def]{$\pi_{z,j}^\fl$}; Map $D_{z,j}^\fl \rta \BB D$ taking $z$ to 0.  
\item \hyperref[eqn-tau^fl-def]{$\tau_{z,j}^\fl ,\tau_{z,j}^{\fl,*} ,\ol\tau_{z,j}^{\fl}, \ol\tau_{z,j}^{\fl,*}$}; Times for $\eta$ corresponding to $\tau_{z,j}$, $\tau_{z,j}^*$, $\ol\tau_{z,j}$, $\ol\tau_{z,j}^*$.  
\item \hyperref[eqn-bar-beta]{$\ol\beta_m$} and \hyperref[eqn-bar-beta]{$\ol u_m$}; $\sum_{j= 1}^{m } \beta_j$ and $\sum_{j=1}^m \beta_j u_j$.  
\item \hyperref[lem-E_z-basics]{$\Phi_{z,j}^\fl$}; Conformal map $\BB D\setminus (\eta^{\tau_{z,j}^{\fl,*}} \cup \ol\eta^{\ol\tau_{z,j}^{\fl,*}}) \rta\BB D$ fixing $\pm i$ and 1. 
\item \hyperref[lem-E_z-basics]{$\lambda_*$}; Constant appearing in Lemma~\ref{lem-E_z-basics}. 
\item \hyperref[item-analytic-def]{$\psi_{z,j}^\fl$}; Conformal map $\BB D\setminus (\eta^{\tau_{z,j}^{\fl,*}} \cup \ol\eta^{\ol\tau_{z,j}^{\fl,*}}) \rta\BB D$ fixing 0.  
\item \hyperref[item-analytic-def]{$f_{z,j}$}; Conformal automorphism of $\BB D$ taking $\Psi_{z,j-1}^\fl(z)$ (if $j\geq 2$) or $z$ (if $j=1$) to 0. 
\item \hyperref[item-analytic-def]{$\wt\eta_{z,j} $}; curve equal to $\psi_{z,j}^\fl( \eta \setminus (\eta^{\tau_{z,j}^{\fl,*}} \cup \ol\eta^{\ol\tau_{z,j}^{\fl,*}}))$. 
\item \hyperref[item-analytic-def]{$\wt\phi_{z,j} $}; Conformal map $ \BB D \setminus (\wt\eta_{z,j}^{\wt\tau_{z,j}^{*} } \cup \ol{\wt\eta}_{z,j}^{\ol{\wt\tau}_{z,j}^{*}})  \rta\BB D$ taking the endpoints of $\wt\eta_{z,j}$ to $\pm i$.  
\item \hyperref[item-analytic-def]{$g_{z,j}$}; Conformal automorphism of $\BB D$ defined so that $ g_{z,j} \circ \wt\phi_{z,j}  \circ f_{z,j}$ fixes $-i$, $i$, and 1. 
\item \hyperref[eqn-Phi-decomp]{$\wh\phi_{z,j}$}; Conformal map $\BB D\setminus (\wt\eta_{z,j}^{\wt\tau_{z,j}^{*} } \cup \ol{\wt\eta}_{z,j}^{\ol{\wt\tau}_{z,j}^{*}}) \rta\BB D$ given by $ g_{z,j} \circ \wt\phi_{z,j}  \circ f_{z,j}$.   
\item \hyperref[eqn-E(z)-range-def]{$E_{n_1,n_2}^m(z)$}; $\bigcap_{j=n_1+1}^{n_2} E_{z,j}^m$, with $E_{z,j}^m$ defined with $(\beta_{j+m-1} ,u_{j+m-1})$ in place of $(\beta_j,u_j)$. 
\end{itemize}
\end{multicols}

\section{Lower bounds for multifractal and integral means spectra}
\label{haus lower sec}

\subsection{Setup}\label{lower setup sec}

Let $\eta$ be a chordal $\op{SLE}_\kappa$ from $-i$ to $i$ in $\BB D$. Let $D_\eta$ be the right connected component of $\BB D\setminus \eta$, as in Theorem~\ref{main thm}, and define the multifractal spectrum sets $\wt\Theta^{s }(D_\eta)$ and $\Theta^s(D_\eta)$ as in Section~\ref{multifractal def}. The goal of this section is to obtain lower bounds on $\dim_{\mathcal H} \wt\Theta^s( D_\eta)$ and $\dim_{\mathcal H} \Theta^s(D_\eta)$, and thereby complete the proof of Theorem~\ref{main thm}. We accomplish this using the estimates of Section~\ref{2pt sec}. 

Throughout this section we fix $d\in (0,1)$ and work in $B_d(0)$. We use the notation defined in Section~\ref{sec-perfect-setup'}, with $q = s/(1-s)\in (-1/2, \infty)$ (see Section~\ref{sec-2pt-index} for an index of this notation), and we assume that the auxiliary parameters have been chosen in such a way that the conclusions of Lemmas~\ref{lem-P(E_z)} and~\ref{lem-E_z-basics}, and Proposition~\ref{prop-2pt-estimate} are all satisfied.
We also continue to use the notation $\mcl B_\beta  = B_{e^{-\beta}}(0)$ from~\eqref{ball abbrv}.  

In the next two subsections we will use the events $E_n(z)$ of~\eqref{eqn-E(z)-def0} to define various notions of ``perfect points" which are contained in the sets we are interested in and which will allow us to obtain lower bounds on their Hausdorff dimensions. 
In the remainder of this subsection, we will prove the following technical lemma which is needed to prove that the perfect points are contained in our sets of interest. 
For the statement of the lemma, we recall the pocket $ D_{z,n}^\fl$ formed by the auxiliary flow lines $\eta_{z,n}^{\fl,\pm}$ from Section~\ref{sec-perfect-setup'}. 

\begin{lem} \label{I stuff}
Let $\Psi_\eta : D_\eta \rta \BB D$ be the conformal map fixing $-i$, $i$, and 1. Suppose $z\in \mathcal P_k \cap D_\eta$. For $n \leq k-1$ let $ I_{z,n}$ be the image under $\Psi_\eta$ of the segment of $\eta$ contained in $D_{z,n}^\fl$. Then the following holds.
\begin{enumerate} 
\item We have $ e^{-\ol\beta_n (q+1 ) - 3\ol u_n}  \preceq \op{length} I_{z,n} \preceq  e^{-\ol\beta_n (q+1 ) + 3\ol u_n}$. \label{I length}
\item If $n\leq k-2$ the distance from $\partial I_{z,n+1}$ to $\partial I_{z,n}$ is at least a constant times the length of $I_{z,n}$. \label{I nest}
\item If $x\in   I_{z,n}$ then there exists $\delta_n > 0$ such that $|(\Psi_\eta^{-1})'((1-\delta_n)x)| = \delta_n^{-s+o_n(1)}$ and $\delta_n =  e^{-\ol\beta_n (q+1+o_n(1))}$.\label{I deriv} 
\end{enumerate}   
The implicit constants are independent of $n$ and both the $o_n(1)$ and the implicit constants are deterministic and independent of $k$, $x$, and $z\in B_d(0)$
\end{lem}
\begin{proof} 
Fix $n$, $k$, and $z$ as in the statement of the lemma. Throughout the proof we assume $E_k(z)$ occurs and require all constants (either referred to as such or implicit in $\asymp$, etc.) to be deterministic and independent of $n , k $, and $z\in B_d(0)$. See Figure~\ref{I stuff fig} for an illustration of the argument.

The map $\pi^\fl_{z,n } : D^\fl_{z,n  } \rta \BB D$ defined in Section~\ref{sec-perfect-full} takes $z$ to $0$ and $\eta \cap D^\fl_{z,n}$ to the curve $\eta_{z,n+1} $, whose endpoints are $x_{z,n+1} $ and $y_{z,n+1}$. Note that condition~\ref{item-wtE-hm} in the definition of $\wt E_{z,n}$ together with condition~\ref{item-F-hm} in the definition of $F_{z,n}$ implies a lower bound on $|x_{z,n+1}  -y_{z,n+1} |$, depending only on the parameter $a$.  

Recall that $[x_{z,n+1}^{ *} , y_{z,n+1}^{ *}]_{\partial\BB D}$ is the largest arc of $\bdy \BB D$ to the right $\eta_{z,n+1} $ which does not contain a point of $\eta_{z,n+1} $ in its interior.
By conditions~\ref{item-L-hit} and~\ref{item-L-ball} in the definition of $L_{z,n+1} $ and condition~\ref{item-wtE-G} in the definition of $\wt E_{z,n+1}$, there is a unique arc $A^0$ of $\partial \mcl B_{\wt\Delta/2}$ which lies to the right of $\eta_{z,n+1} $ and which disconnects $\eta_{z,n+1}  \cap \mcl \mcl B_{\wt\Delta}$ from $[x_{z,n+1}^{ *} , y_{z,n+1}^{ *}]_{\partial\BB D}$ in $\BB D\setminus \eta_{z,n+1} $ (c.f.\ Remark~\ref{remark-E-stay}). Let $w^0$ be the point of $A^0$ closest to the midpoint of $[x_{z,n+1}^{*} , y_{z,n+1}^{ *}]_{\partial\BB D}$ and let $D^0$ be the connected component of $\BB D\setminus \eta_{z,n+1} $ containing $[x_{z,n+1}^{ *} , y_{z,n+1}^{ *}]_{\partial\BB D}$ on its boundary. 

From the definitions of $L_{z,n+1} $ and $\wt E_{z,n+1}$, we find that the harmonic measure from $w^0$ in $D^0$ of any sub-arc of $[x_{z,n+1}^{ *} , y_{z,n+1}^{ *}]_{\partial\BB D}$ lying at distance at least $e^{-2\Delta}$ from the endpoints is proportional to the length of that sub-arc. Furthermore, $\op{hm}^{w^0}(\eta_{z,n+1}  ; D^0) \asymp 1$. Define $\psi_{z,n } : \BB D\setminus (\eta^{\tau_{z,n}} \cup \ol\eta^{\ol\tau_{z,n}}) \rta \BB D$ as in condition~\ref{item-F-b-bar} in the definition of $F_{z,n}$.
By condition~\ref{item-wtE-hm} in the definition of $\wt E_{z,n}$, the arc of $\partial\BB D$ which is the image of the right side of $\eta_{z,n}^{\tau_{z,n} }$ (resp. the left side of $\ol\eta_{z,n}^{\ol\tau_{z,n} }$) under $\psi_{z,n}$ has length $\asymp 1$. By the conformal invariance of Brownian motion and condition~\ref{item-F-hm} in the definition of $F_{z,n}$, the harmonic measure from $(\psi_{z,n}  \circ \pi_{z,n}^{-1} ) (w^0) $ in the right connected component of $\BB D\setminus \psi_{z,n} (\eta_{z,n} \setminus (\eta_{z,n}^{\tau_{z,n}} \cup\ol\eta_{z,n}^{\ol\tau_{z,n}})$ of each of these two sub-arcs is $\asymp 1$. 
 
Let $w = (\pi_{z,n }^\fl)^{-1} ( w^0)$. It follows from the above considerations and conformal invariance of Brownian motion that (notation as in Section~\ref{basic notation})
\eqb  \label{hm w}
\op{hm}^w ( \eta^{\tau_{z,n}^{\fl,*} } ; D_\eta) \asymp \op{hm}^w(\ol \eta^{\ol \tau_{z,n}^{\fl,*} } ;D_\eta) \asymp \op{hm}^w(\eta \cap D_{z,n}^\fl ; D_\eta) \asymp 1 . 
\eqe 
By Lemma~\ref{phi hm arc} and condition~\ref{item-L-hit} in the definition of $L_{z,n} $, we thus have 
\eqb \label{dist and deriv}
|\Psi_\eta'(w)| \asymp |(\Phi^\fl_{z,n})'(w)|,\quad \op{and} \quad  \op{dist}(w ,\eta)  \asymp  \op{dist}\left( w, \eta^{ \tau_{z,n}^{\fl,*}} \cup \ol\eta^{\ol\tau_{z,n}^{\fl,*}}\right)   
\eqe
with $\Phi^\fl_{z,n}$ the map from Lemma~\ref{lem-E_z-basics}.

By the Koebe growth theorem applied to $(\pi_{z,n}^\fl)^{-1}$, we have $|w - z| \leq \frac{1}{100} \op{dist}(z , \eta^{\tau_{z,n}^{\fl,*}} \cup \ol\eta^{\ol \tau_{z,n}^{\fl,*} })$ provided $\beta_n$ is chosen sufficiently large. By the Koebe distortion theorem, $|(\Phi^\fl_{z,n})'(w)| \asymp |(\Phi^\fl_{z,n})'(z)|$, so by~\eqref{dist and deriv} and assertion~\ref{item-Phi'-asymp} of Lemma~\ref{lem-E_z-basics},
\eqb \label{Psi_eta at w}
e^{-\ol\beta_n q- 2\ol u_n} \preceq |(\Phi^\fl_{z,n})'(w)| \preceq e^{-\ol\beta_n q + 2\ol u_n} .
\eqe
Moreover, by assertion~\ref{item-eta-dist} of Lemma~\ref{lem-E_z-basics} and assertion~\ref{item-beta-u-infty} of Lemma~\ref{lem-beta-u-choice}, $\op{dist}(z , \eta^{\tau_{z,n}^{\fl,*}} \cup \ol\eta^{\ol\tau_{z,n}^{\fl,*}})$ is bounded between constants times $e^{-\ol\beta_n -  \ol u_n}$ and $e^{-\ol\beta_n + \ol u_n}$, so by~\eqref{dist and deriv} also
\eqb  \label{eta dist at w}
 e^{-\ol\beta_n   -  \ol u_n} \preceq  \op{dist}( w,  \eta )\preceq e^{-\ol\beta_n   +  \ol u_n}  .
\eqe 

Let $\wt w = \Psi_\eta(w)$. By~\eqref{Psi_eta at w},~\eqref{eta dist at w}, and the Koebe quarter theorem, $e^{-\ol\beta_n (q+1 ) - 3\ol u_n}  \preceq 1-|\wt w| \preceq  e^{-\ol\beta_n (q+1 ) + 3\ol u_n}$.  
By~\eqref{hm w} and conformal invariance of harmonic measure,  
\eqb \label{length asymp}
\op{dist}(\wt w , I_{z,n}) \asymp \op{length}(I_{z,n}) \asymp 1-|\wt w| .
\eqe
 This proves assertion~\ref{I length}. 
 
 To prove assertion~\ref{I nest}, we observe that the harmonic measure from $w^0$, as defined above, of each of $ \eta_{z,n+1}^{\sigma_{z,n+1}}$ and $ \ol\eta_{z,n+1}^{\ol\sigma_{z,n+1}}$ is $\asymp 1$, where here $\sigma_{z,n+1}$ and $\ol\sigma_{z,n+1}$ are the times in the definition of $L_{z,n}$. It therefore follows from conformal invariance of harmonic measure that the distance from the endpoints of $I_{z,n}$ to the endpoints of $I_{z,n+1}$ is $\succeq 1-|\wt w|$. We conclude by means of~\eqref{length asymp}.

To complete the proof of assertion~\ref{I deriv}, suppose given $x\in I_{z,n}$. By~\eqref{length asymp} the angle between the tangent line to $\partial\BB D$ at $x$ and the segment $[x,\wt w]$ is bounded away from 0 and $\pi$. Hence we can find $\delta_n \asymp 1-|\wt w| =  e^{-\ol\beta_n(q+1+o_n(1))}$ and $\rho \in (0,1)$, bounded away from 0 and 1, such that $\wt w \in B_{\rho \delta_n}((1-\delta_n) x)$. By the Koebe distortion theorem we have $|(\Psi_\eta^{-1})'((1-\delta_n)x)| \asymp |(\Psi_\eta^{-1})'(\wt w)|$. By combining this with~\eqref{Psi_eta at w} we conclude that assertion~\ref{I deriv} holds.
\end{proof}

\begin{figure}\label{I stuff fig}
\begin{center}
\includegraphics{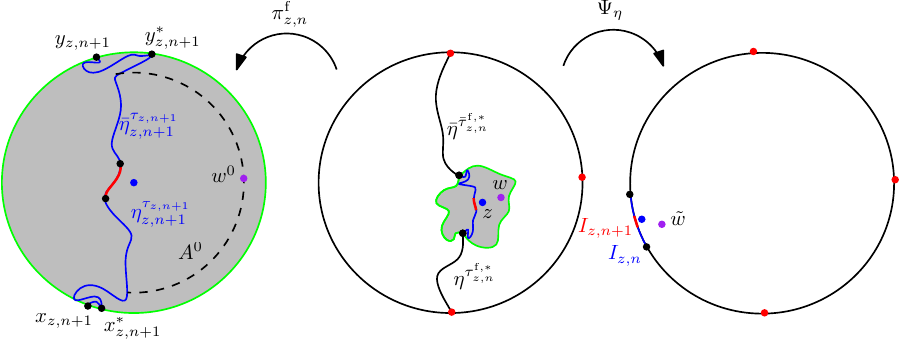}
\end{center}
\caption{An illustration of the proof of Lemma~\ref{I stuff}. The arc $I_{z,n}$ and its images under various conformal maps is shown in blue. The regularity conditions in our events imply that the harmonic measure from $w$ in the middle picture of each of the black curves is uniformly positive. This is the key step in the proof of our regularity conditions for the arc $I_{z,n}$. Also shown (in red) is the arc $I_{z,n+1}$ appearing in assertion~\ref{I nest} and its images under the various maps. }
\end{figure}

\subsection{Lower bound for the Hausdorff dimension of the subset of the curve}

In this subsection we will prove a lower bound on the Hausdorff dimension of the multifractal spectrum sets $ \Theta^{s }(D_\eta) \subset \eta$. 

\begin{prop} \label{haus curve lower}
Let $s_- , s_+$ be as in Theorem~\ref{main thm}. For each $s \in (s_- , s_+)$, a.s.\
\eqbn
\dim_{\mathcal H} \Theta^s(D_\eta) \geq \xi(s)  ,
\eqen
where $\xi(s)$ is as in~\eqref{xi(s)}. 
\end{prop}

For the proof, we assume we are in the setting of Section~\ref{lower setup sec}. 
We first define a closed subset $\mcl P$ of $\Theta^s(D_\eta)$, the so-called perfect points, whose Hausdorff dimension can be bounded below using the estimates of Section~\ref{2pt sec}. Let $\lambda_*$ be the constant from Lemma~\ref{lem-E_z-basics}. For $n\in\BB N$, let $n'$ be the greatest integer such that $\ol\beta_n -\lambda_* n \geq \ol\beta_{n'+1} +\lambda_*(n'+2)$. Let
\eqb\label{ep_n def}
\ep_n := e^{-\ol\beta_{n'+1} -\lambda_*(n'+2)  } .
\eqe 
Note that Lemma~\ref{lem-beta-u-choice} implies $e^{-\ol\beta_n  } = \ep_n^{1+o_n(1)}$. Our reason for choosing this value of $\ep_n$ is that the pockets $ D_{z,n}^\fl$ and $D^\fl_{w,n}$ are disjoint on $E_n(z) \cap E_n(w)$ provided $|z-w| \geq \ep_n$ (see Lemma~\ref{lem-E_z-basics}).

 Choose a collection $\mathcal C_n$ of $  \asymp  \ep_n^{-2} $ points in $B_d(0)$, no two of which lie within distance $ \ep_n $ of each other. Let $\mathcal C_n'  $ be the set of $z\in\mathcal C_n$ for which $E_n(z)$ occurs and define the \emph{perfect points} by
\eqb \label{mathcal P def}
 \mathcal P := \bigcap_{n\geq 1} \ol{\bigcup_{k\geq n} \bigcup_{z\in \mathcal C_k'} B_{\ep_k}(z)} .
\eqe 

\begin{lem} \label{perfect pt contained}
With $\mathcal P$ as in~\eqref{mathcal P def}, we have $\mathcal P \subset \Theta^{s }(D_\eta)$ for $s = q/(q+1)$. In fact, if $w\in\mathcal P$ then for $\ep > 0$,
\eqb \label{deriv at ep}
|(\Psi_\eta^{-1})'( (1-\ep  ) \Psi_\eta(w))| = \ep^{-s+o_\ep(1)} ,
\eqe 
with the rate of the $o_\ep(1)$ deterministic and uniform for $w\in \mcl P$. 
\end{lem}
\begin{proof} 
Fix $w\in \mathcal P$. Since $\eta$ is closed, it is clear that $w \in \eta$. It remains to prove~\eqref{deriv at ep}. By definition of $\mathcal P$, if we are given $n\in\BB N$, then we can find $k\geq n+1$ and $z\in \mathcal C_k'$ such that $|z-w| \leq  e^{- 2 \ol\beta_{n+1} } $. 
By Lemma~\ref{lem-E_z-basics}, $w \in D^\fl_{z, n}$ so $\Psi_\eta(w) \in I_{z,n}$, as defined in Lemma~\ref{I stuff}. Let $\delta_n$ be as in that lemma with $x = \Psi_\eta(w)$. 

By the Koebe distortion theorem, for $\ep \in [\delta_{n+1}  ,\delta_n]$,
\eqb \label{Psi' at ep}
\frac{1- (\delta_n - \delta_{n+1})/\delta_n }{(1+ (\delta_n - \delta_{n+1})/\delta_n )^3} \leq   \frac{ |(\Psi_\eta^{-1})'((1-\ep) \Psi_\eta(w) )| }{ |(\Psi_\eta^{-1})'((1-\delta_n) \Psi_\eta(w) )| } \leq   \frac{1+ (\delta_n - \delta_{n+1})/\delta_n }{(1- (\delta_n - \delta_{n+1})/\delta_n )^3} .
\eqe 
Since $\delta_n =  e^{-\ol\beta_n (q+1 + o_n(1))}$ (condition~\ref{I deriv} of Lemma~\ref{I stuff}),
\[
1- (\delta_n - \delta_{n+1})/\delta_n  =  e^{-\beta_{n+1} (q+1  + o_n(1))} = e^{  \ol\beta_n  o_{n}(1) } ,
\] 
which is proportional to $ \ep^{o_\ep(1)}$ by Lemma~\ref{lem-beta-u-choice}. We furthermore have $\delta_n = \ep^{1+o_\ep(1)}$. Hence~\eqref{Psi' at ep} and condition~\ref{I deriv} of Lemma~\ref{I stuff} imply $|(\Psi_\eta^{-1})'( (1-\ep  ) \Psi_\eta(w))| =  \ep^{-s+o_\ep(1)}$, as required.  
\end{proof}

\begin{proof}[Proof of Proposition~\ref{haus curve lower}]
For a Borel measure $\nu$ on a metric space $X$ and $\alpha > 0$, write 
\eqb \label{I_alpha def}
I_\alpha(\nu) = \int_X \int_X \frac{d\nu(z) \, d\nu(w)}{|z-w|^\alpha} 
\eqe
for the $\alpha$-energy of $\nu$. By standard results for Hausdorff dimension (see \cite[Theorem~4.27]{peres-bm}) a metric space which admits a positive finite measure with finite $\alpha$-energy has Hausdorff dimension at least $\alpha$. In view of Lemma~\ref{perfect pt contained}, we are led to construct such a measure $\nu$ on $\mathcal P  $ for each $\alpha < \xi(s)$. We do this using the usual argument (see, e.g. \cite{miller-wu-dim,hmp-thick-pts, beffara-dim}) and the estimates of Section~\ref{sec-flow-line-prob} 

Define the events $E_n(z)$ as in Section~\ref{sec-perfect-setup'} and the sets of points $\mathcal C_n$ and $\mathcal C_n'$ as in the definition of $\mathcal P$ (right above~\eqref{mathcal P def}). Let $\ep_n$ be as in~\eqref{ep_n def}. 

For each $n\in\BB N$, define a measure $\nu_n$ on $\BB D$ by 
\[
d\nu_n(x) = \sum_{z\in \mathcal C_n  } \frac{\BB 1_{E_n(z) }}{\BB P(E_n(z)  )} \BB 1_{(x\in B_{\ep_n}(z))} \, dx  .
\]
Then $\BB E(\nu_n(\BB D)) \asymp 1$. Moreover,
\alb
\BB E(\nu_n(\BB D)^2) &\preceq  \ep_n^4 \sum_{\substack{z,w\in\mcl C_n,\\ z\not=w} }  \frac{\BB P( E_n(z)  \cap E_n(w) ) }{\BB P(E_n(z)  ) \BB P( E_n(w) )} + \ep_n^4 \sum_{z\in\mathcal C_n } \frac{1}{\BB P(E_n(z) )} .
\ale
By Lemma~\ref{lem-P(E_z)} and Proposition~\ref{prop-2pt-estimate} (c.f.\ Remark~\ref{remark-2pt-estimate}), this is bounded by an $n$-independent constant times 
\alb
 \ep_n^4 \sum_{\substack{z,w\in\mcl C_n,\\ z\not=w} } |z-w|^{- \gamma^*(q) + o_{|z-w|}(1)  } + \ep_n^4 \sum_{z\in\mathcal C_n  }  \ep_n^{-  \gamma^*(q) + o_n(1)  } ,
\ale
with the $o_{|z-w|}(1)$ tending to $0$ as $|z-w|\rta 0$, at a rate which is independent of the particular locations of $z$ and $w$ and of $n$. For $s \in (s_- , s_+)$ we have $\gamma^*(q) = \gamma(s)/(1-s)  < 2$. Therefore, for sufficiently large $n$, $\BB E(\nu_n(\BB D)^2)$ is bounded above by a finite, $n$-independent constant. By the Vitalli convergence theorem, we can a.s.\ find a subsequence of the measures $\nu_n$ which converges weakly to a measure $\nu$ whose total mass is bounded above by some deterministic constant, and whose expected mass is positive. 

On the other hand, we have
\alb
\BB E(I_\alpha(\nu_n)) &=  \sum_{z, w\in \mathcal C_n }  \frac{\BB P(E_n(z)  \cap  E_n(w) ) }{\BB P(E_n(z) ) \BB P( E_n(w))} \iint_{B_{\ep_n}(z) \times B_{\ep_n}(w)} \frac{1}{|x-y|^\alpha} \, dx \,dy \\
&= \sum_{\substack{z,w\in\mcl C_n,\\ z\not=w} }  \frac{\BB P( E_n(z)  \cap E_n(w) ) }{\BB P(E_n(z) E_n(z) ) \BB P(E_n(z) )} \iint_{B_{\ep_n}(z) \times B_{\ep_n}(w)} \frac{1}{|x-y|^\alpha} \, dx \,dy \\
&+ \sum_{z \in \mathcal C_n }  \frac{1}{\BB P(E_n(z) ) } \iint_{B_{\ep_n}(z) \times B_{\ep_n}(z)} \frac{1}{|x-y|^\alpha} \, dx \,dy\\
 &\preceq \sum_{\substack{z,w\in\mcl C_n,\\ z\not=w} }  \frac{\BB P(E_n(z)  \cap  E_n(w) ) }{\BB P(E_n(z) ) \BB P(E_n(z) )}  \frac{\ep_n^4}{|z-w|^{ \alpha}}   +   \sum_{z \in \mathcal C_n }  \frac{\ep_n^{4-\alpha} }{\BB P(E_n(z) ) }   \\
  &\preceq    \sum_{\substack{z,w\in\mcl C_n,\\ z\not=w} } |z-w|^{-\gamma^*(q)  -  \alpha + o_{|z-w|}(1) } \ep_n^4      +     \ep_n^{2-\alpha - \gamma^*(q)  +  o_n(1)} .
\ale
We have $\gamma^*(q) + \alpha < 2$ for $s \in (s_- , s_+)$ and $\alpha < \xi(s)$, so the above expression is $\preceq 1$. 
We conclude that with positive probability, there exists a weak subsequential limit $\nu$ of the measures $(\nu_n)$ supported on $\mathcal P$ and satisfying $\nu(\mathcal P)>0$ and $I_\alpha(\nu) <\infty$. Hence \cite[Theorem~4.27]{peres-bm} and Lemma~\ref{perfect pt contained} imply that with positive probability, we have $\dim_{\mathcal H} \Theta^s(D_\eta) \geq \xi(s) $. Proposition~\ref{theta zero one} implies that this in fact a.s.\ holds.
\end{proof}

\subsection{Lower bound for the Hausdorff dimension of the subset of the circle}

In this subsection we prove the following lower bound for the set Hausdorff dimension of the set $\wt\Theta^s(D_\eta) = \Psi_\eta^{-1}(\Theta^s(D_\eta)) \subset\bdy\BB D$. 

\begin{prop} \label{haus circle lower}
Let $s_- , s_+$ be as in Theorem~\ref{main thm}. For each $s\in (s_- , s_+)$, a.s.\ 
\eqbn
\dim_{\mathcal H} \wt\Theta^s(D_\eta) \geq \wt\xi(s)  ,
\eqen
where $\wt\xi(s)$ is as in~\eqref{tilde xi(s)}. 
\end{prop}

For the proof of Proposition~\ref{haus circle lower}, we will need a different set of perfect points. Define $\ep_n$, the sets $\mathcal C_n$, $\mathcal C_n'$ as in the definition~\eqref{mathcal P def} of $\mathcal P$. For $z\in \mathcal C_n'$, let $I_{z,n-1}$ be as in the statement of Lemma~\ref{I stuff}. Let $v_{z,n}$ be the midpoint of $I_{z,n-1}$ and let $I_{z,n}'$ be the arc of length $\ep_n^{q+1}$ centered at $v_{z,n}$. By Lemma~\ref{I stuff}, $\op{length}( I_{z,n}') = \op{length}(I_{z,n-1})^{1+o_n(1)}$. Our perfect points in this case are defined by
\eqb \label{wt mathcal P def}
\wt{\mathcal P} := \bigcap_{n\geq 1} \ol{\bigcup_{k\geq n} \bigcup_{z \in \mathcal C_k'} I_{z,k-1}' } .
\eqe
Our first task is to check that $\wt{\mcl P} \subset \wt\Theta^s(D_\eta)$.

\begin{lem} \label{perfect pts contained tilde} 
Define $\wt{\mathcal P}$ as in~\eqref{wt mathcal P def}. If the auxiliary parameter $\wt\Delta$ (Definition~\ref{def-aux-parameter}) and the value $\beta_0$ are chosen sufficiently large then $\wt{\mathcal P} \subset \wt\Theta^{s }(D_\eta)$ for $s = q/(q+1)$. In fact, if $x\in \wt{\mathcal P}$, then for $\ep > 0$,
\eqbn
|(\Psi_\eta^{-1})'( (1-\ep  ) x)| =  \ep^{-s+o_\ep(1)} ,
\eqen
with the implicit constants and the $o_\ep(1)$ deterministic and uniform in $x$. 
\end{lem}
\begin{proof}
If $x\in \wt{\mathcal P}$ then for any $n \in \BB N$ we can find $k\geq n$ and $z\in\mathcal C_k'$ such that $x$ lies within distance $  \op{length}(I_{z,n}')^2$ of $I_{z,k}'$. If $k$ is chosen sufficiently large, depending on $n$, then by assertions~\ref{I length} and~\ref{I nest} of Lemma~\ref{I stuff} we have $x\in I_{z,n}$. We then conclude as in the proof of Lemma~\ref{perfect pt contained}. 
\end{proof}

In the proof of Proposition~\ref{haus circle lower}, we will break up the sum which gives the second moment of our measures into three terms, depending on the distance between the points under consideration. The following lemma is needed to bound the number of pairs of points at mesoscopic distance.

\begin{lem} \label{v_z dist}
For each $n\in\BB N$ there is an integer $m_n \leq n$ such that the following is true. We have $ \ol\beta_{n } -  \ol\beta_{m_n} =  \ol\beta_{n }  o_n(1)$ and if $z,w\in\mathcal C_n'$ with $|z-w| \geq e^{-\ol\beta_{m_n+1}}$ then $\op{dist}(I_{z,n}' , I_{w,n}' ) \succeq |z-w|^{q+1+o_{|z-w|}(1)  }$, with the $o_{|z-w|}(1)$ and implicit constants deterministic, independent of $n$, and independent of the particular choices of $z$ and $w$ in $\mcl C_n'$.  
\end{lem}
\begin{proof} 
We argue as in the proof of Lemma~\ref{lem-near-independence}. Choose $k\in \BB N$ such that $ e^{-\beta_{k+1} - \lambda_*(k+1)}  \leq |z-w| \leq e^{-\beta_k  - \lambda_* k}$.
Let $k'$ be the least integer such that $\ol\beta_{k'} -\lambda_* k' \geq   \ol \beta_{k +1} + \lambda_* (k+1) $. By our choice~\eqref{ep_n def} of $\ep_n$ we have $k'  \leq n-1$. By Lemma~\ref{lem-E_z-basics}, $D^\fl_{z,k'} \cap D^\fl_{w,k'} = \emptyset$ and hence $I_{z,k'} \cap I_{w,k'} = \emptyset$. If $\op{length}(I_{z,n}') \leq \op{length}(I_{z,k'+1 })$ then by assertions~\ref{I length} and~\ref{I nest} of Lemma~\ref{I stuff}, the midpoints of $I_{z,n'}$ and $I_{w,n}'$ satisfy
\[
\op{dist}(v_{z,n} , v_{w,n} ) \succeq  e^{-\ol\beta_{k'+1 } (q+1) - 3 \ol u_{k'+1}  } \succeq |z-w|^{q+1+o_{|z-w|}(1)  } .
\] 
On the other hand, by assertion~\ref{I length} of Lemma~\ref{I stuff} we have $\op{length}(I_{z,n}') \leq \op{length}(I_{z,k'+1 })$ provided $\ol\beta_{k' +1} (q+1) + 3 \ol u_{k'+1} \leq (\ol\beta_{n } - \lambda_* n + \ol\beta_{n } o_n(1))(q+1)$, or equivalently provided 
\eqbn
\ol\beta_{n } - \ol\beta_{k'+1}    \geq   \frac{3\ol u_{k'+1}  +  \lambda_* n + \ol\beta_{n } o_n(1)    }{q+1}  .
\eqen
It follows from Lemma~\ref{lem-beta-u-choice} that we can choose $m_n \leq n$ such that $ \ol\beta_{n }  -  \ol\beta_{m_n} = \ol\beta_{n } o_n(1) $ and $\op{length}(I_{z,n}') \leq \op{length}(I_{z,k' +1})$ whenever $k'\leq m_n$. 
\end{proof}

\begin{proof}[Proof of Proposition~\ref{haus circle lower}]
We argue as in the proof of Proposition~\ref{haus curve lower}. In particular, for any given $\alpha < \wt\xi(s)$, we will construct a positive finite measure $\wt \nu$ on $\wt{\mathcal P}$ (as defined in~\eqref{wt mathcal P def}) with finite $\alpha$-energy (as defined in~\eqref{I_alpha def}). 

Define $\ep_n$ as in~\eqref{ep_n def}. We require all implicit constants and $o_{|z-w|}(1)$ terms to be independent of $n$ and uniform for $z,w\in \mathcal C_n$. For $n\in\BB N$, define a measure $\wt \nu_n$ on $\partial \BB D$ by
\[
d\wt\nu_n(x) =    \ep_n^{ 1-q} \sum_{z\in \mathcal C_n'  } \frac{\BB 1_{E_n(z) }}{\BB P(E_n(z) )} \BB 1_{(x\in I_{z,k}')} \, dx .
\]
Then we have $\BB E(\wt \nu_n(\partial \BB D)) \asymp 1$.

As in the proof of Proposition~\ref{haus curve lower},
\alb
\BB E(\wt \nu_n(\partial\BB D)^2) &\preceq  \ep_n^{4} \sum_{\substack{z,w\in\mcl C_n,\\ z\not=w} }  \frac{\BB P( E_n(z)  \cap E_n(w) ) }{\BB P(E_n(z)  ) \BB P( E_n(w) )}    + \ep_n^4 \sum_{z \in \mcl C_n} \ep_n^{-\gamma^*(q) + o_n(1)  }  \preceq  1  .
\ale

Let $m_n$ be as in Lemma~\ref{v_z dist} and let $\mathcal K_n$ be the set of pairs $(z,w)\in\mathcal C_n \times\mathcal C_n$ with $|z-w| \leq e^{-\ol\beta_{m_n }}$ and $z\not=w$. By Lemma~\ref{v_z dist} we have $\# \mathcal K_n \leq \ep_n^{-2 - o_n(1)}$. 

By Lemma~\ref{lem-P(E_z)}, Proposition~\ref{prop-2pt-estimate}, and Lemma~\ref{v_z dist},  
\alb\label{circle energy}
\BB E(I_\alpha(\wt \nu_n ) )  
&=  \ep_n^{2- 2q} \sum_{(z,w)\in\mathcal C_n\times \mathcal C_n  } \frac{\BB P(E_n(z)  \cap  E_n(w))}{\BB P(E_n(z) ) \BB P( E_n(w))} \iint_{I_{z,k}' \times I_{w,k}'} \frac{1}{|x-y|^\alpha} \, dx \,dy \nonumber\\ 
&\preceq  \sum_{(z,w) \notin \mathcal K_n  ,\: z\not=w}  |z-w|^{-\gamma^*(q)  + o_{|z-w|}(1)   } |v_{z,n} - v_{w,n}|^{-\alpha} \ep_n^{2(q+1) + 2-2q } \nonumber \\
&  +  \sum_{(z,w)\in \mathcal K_n  }  |z-w|^{-\gamma^*(q)  + o_{|z-w|}(1)   } \ep_n^{(2-\alpha)(q+1)  +2-2q   + o_n(1)  }  \nonumber \\
& + \sum_{z\in \mathcal C_n  } \ep_n^{(2-\alpha)(q+1)  +2-2q - \gamma^*(q) + o_n(1)  } \nonumber\\
&\preceq  \ep_n^{4} \sum_{\substack{z,w\in\mcl C_n,\\ z\not=w} } |z-w|^{-\gamma^*(q) - \alpha (q+1) + o_{|z-w|}(1)      }  +  \ep_n^{(2-\alpha)(q+1)  - 2q- \gamma^*(q) + o_n(1)}+   \ep_n^{(2-\alpha)(q+1)  - 2q- \gamma^*(q)  + o_n(1)}  .
\ale
Note that for the middle term we used $|z-w| \succeq \ep_n$ and $\# \mathcal K_n \leq \ep_n^{-2 - o_n(1)}$. If $s \in (s_- , s_+)$ and $q = s/(1-s)$ we have $\gamma^*(q) + \alpha(q+1) < 2$ and $(2-\alpha)(1+q) -2q - \gamma^*(q) >0$ for $\alpha < \wt\xi(s)$. It follows that we can a.s.\ find a subsequence of the measures $(\wt\nu_n)$ which converges weakly to a finite positive limiting measure supported on $\wt{\mathcal P}$ with finite $\alpha$-energy. We then conclude using~\cite[Theorem~4.27]{peres-bm}, Lemma~\ref{perfect pts contained tilde}, and Proposition~\ref{theta zero one}. 
 \end{proof}

\subsection{Proof of Theorem~\ref{main thm}}

This follows by combining Propositions~\ref{bdy haus upper},~\ref{haus curve upper},~\ref{haus curve lower}, and~\ref{haus circle lower}.  
\qed

\begin{remark} \label{s=1 kappa=4}
In the case $\kappa=4$, we have $s_+ = 1$, so the sets $\Theta^1(D_\eta)$ and $\wt\Theta^1(D_\eta)$ for $\kappa=4$ can be non-empty. We do not explicitly mention these sets in Theorem~\ref{main thm} because our results do not apply in full in this case. However, we do prove something about these sets. In particular, we prove in Proposition~\ref{bdy haus upper} that a.s.\ $\dim_{\mcl H} \wt\Theta^1(D_\eta) = 0$. Since $  \dim_{\mcl H}(\eta) = 3/2$ for $\kappa=4$, we get a trivial upper bound of $3/2$ for $\dim_{\mcl H} \Theta^1(D_\eta)$ in the case $\kappa=4$. We do not prove a lower bound for $\dim_{\mcl H} \Theta^1(D_\eta)$ in this paper, and we are not sure if the upper bound of $3/2$ is optimal. 
\end{remark}

\subsection{Lower bound for the integral means spectrum}
\label{ims lower sec}

In this subsection we prove our lower bound for the bulk integral means spectrum of the SLE curve and thereby complete the proof of Corollary~\ref{ims cor}.

\begin{proof}[Proof of Corollary~\ref{ims cor}] 
Throughout, we consider a fixed realization and allow implicit constants to be random (but independent of the parameters of interest). 

Fix $s \in [s_- , s_+]$ (as defined in~\eqref{s-} and~\eqref{s+}) to be chosen later, and let $\wt{\mathcal P}$ be the set of perfect points defined in~\eqref{wt mathcal P def}. Also fix $\alpha < \wt\xi(s)$. By the proof of Proposition~\ref{haus circle lower}, the probability of the event  
\[
E:= \{\dim_{\mathcal H} \wt{\mathcal P}  >\alpha \}
\]
is positive. Moreover, it is clear from the definition that $\wt{\mathcal P} \subset \Psi_\eta^{-1}(\eta \cap B_d(0))$. The idea of the proof is that on $E$, we have a lower bound for the size of the set of $x \in \bdy \BB D$ where $|\Psi_\eta'((1-\ep)x)|$ grows like $\ep^{-s}$, which gives us a lower bound for the integral of $|\Psi_\eta'|^a$ over $\bdy B_{1-\ep}(0)$. We then optimize over $s$ to get a lower bound for the integral means spectrum. 

For $n\in\BB N$ let $\wh\ep_n := 2^{-n}$. 
Let $\mathcal I_n$ be the collection of arcs $ [e^{2\pi i (k-1) \wh\ep_n } , e^{2\pi i k \wh\ep_n}]_{\partial\BB D}$ for $k \in \{ 1,\dots , 2^n\}$ and let $\mathcal I_n'$ be the set of those arcs $I\in\mathcal I_n$ which intersect $\wt{\mathcal P}$. Then $\mathcal I_n'$ is a cover of $\wt{\mathcal P}$ consisting of sets of diameter at most $O_n(\wh\ep_n )$. Hence on $E$ we have 
$(\# \mathcal I_n') \wh\ep_n^\alpha   \succeq 1$ (with possibly random, but $n$-independent implicit constant) so 
$\# \mathcal I_n' \succeq \wh\ep_n^{-\alpha}$.

For $I\in \mathcal I_n'$ choose $x_I \in I \cap \wt{\mathcal P}$ and let $z_I = (1- \wh\ep_n ) x_I$. By Lemma~\ref{perfect pts contained tilde}, $|(\Psi_\eta^{-1})'(z_I )| \succeq  \wh\ep_n^{-s + o_n(1)} $, with the $o_n(1)$ and the implicit constant independent of the choice of $I$ and $x_I$. 
 
Let $J_I$ be the intersection of $(1-\wh\ep_n) I$ with the arc of $\partial B_{1-\wh\ep_n}(0)$ centered at $z_I$ of length $ \wh\ep_n^{1+r_n}$, where $(r_n)$ is a sequence of positive numbers with $r_n \rta 0$ slower than the $o_n(1)$ above. Then the arcs $J_I$ are disjoint for sufficiently large $n$ and by the Koebe distortion theorem we have $|(\Psi_\eta^{-1})'(w)| \succeq \wh\ep_n^{s + o_n(1)}$ for each $w\in J_I$. Each point of $\wt{\mathcal P}$ is mapped into $B_{1-d/2}(0)$ by $\Psi_\eta^{-1}$. Hence for sufficiently large $n$ and sufficiently small $\zeta$ (random), we have $J_I \subset A_{\wh\ep_n}^\zeta(\Psi_\eta^{-1})$ for each $I\in\mathcal I_n'$, with $A_{\wh\ep_n}^\zeta(\Psi_\eta^{-1})$ defined just below~\eqref{I_zeta def} with $\phi = \Psi_\eta^{-1}$. Hence on $E$, it holds for $a \in \BB R$ that
\[
\int_{A_{\wh\ep_n}^\zeta(\Psi_\eta^{-1}) }  |(\Psi_\eta^{-1})'(w)|^a \, dw \succeq   \sum_{I\in \mathcal I_n' } \int_{J_I}  |(\Psi_\eta^{-1})'(w)|^a \, dw \succeq \wh\ep_n^{-\alpha - a s + 1 + o_n(1)} .
\]
Therefore, for any $a \in\BB R$, on $E$ it holds that
\[
 \limsup_{n\rta \infty} \frac{ \log \int_{A_{\wh\ep_n}^\zeta(\Psi_\eta^{-1}) }  |(\Psi_\eta^{-1})'(w)|^\av \, dw }{ \log\wh\ep_n^{-1}} \geq  \alpha +\av s - 1 .
\]
Thus $\op{IMS}_{D_\eta}^{\op{bulk}}(a) \geq \alpha + \av s - 1$ with positive probability. 

By Proposition~\ref{ims zero one}, this lower bound in fact holds a.s. Since $\alpha < \wt\xi(s)$ is arbitrary, it follows that a.s.\ 
\eqb \label{ims lower'}
\op{IMS}_{D_\eta}^{\op{bulk}}(a) \geq \wt\xi(s)  + \av s - 1 
\eqe 
In the notation of Corollary~\ref{ims cor}, this quantity is maximized over all $s\in [s_-,s_+]$ by taking $s = s_*(a)$ if $a \in [a_- , a_+]$; $s = s_-$ if $a < a_-$; and $s = s_+$ if $a > a_+$. Choosing this value of $s$ in~\eqref{ims lower'} gives us that the lower bound in~\eqref{ims eqn} holds a.s.\ for each fixed $a\in \BB R$ in the case $\kappa \leq 4$, $\ul\rho= 0$, and $V = D_\eta$. 

By Proposition~\ref{ims zero one}, this lower bound in fact holds a.s.\ for each choice of $\kappa > 0$, vector of weights $\ul\rho$, $t>0$, and complementary connected component $V$ of $\eta([0,t])$. By combining this with Proposition~\ref{ims upper}, we get that~\eqref{ims eqn} holds a.s.\ for each fixed $a\in\BB R$ for each choice of $\kappa > 0$, vector of weights $\ul\rho$, $t>0$, and complementary connected component $V$ of $\eta([0,t])$.
By H\"older's inequality, it follows that the bulk integral means spectrum is a convex, hence continuous, function of $a$ (c.f.\ \cite[Theorem~5.2]{makarov-fine} for a related, but much stronger, statement for the ordinary integral means spectrum). It follows that in fact~\ref{ims eqn} holds a.s.\ for all  $a\in\BB R $ simultaneously.  
\end{proof}

\appendix

\section{Proof of Proposition~\ref{d control}}
\label{auxiliary sec}

In this appendix we will prove Proposition~\ref{d control}, which is one of the ingredients in the proof of Theorem~\ref{1pt chordal}. 
The proof will be completed in two stages. First, we will show that we can move the force point to the imaginary axis without any pathological behavior (Lemma~\ref{Y to 0}). Then, we will use a forward/reverse SLE symmetry argument to rule out pathological behavior after the force point has reached the imaginary axis. See Figure~\ref{aux condition fig} for an illustration.

\begin{figure}[ht!]
\begin{center}
\includegraphics[scale=1]{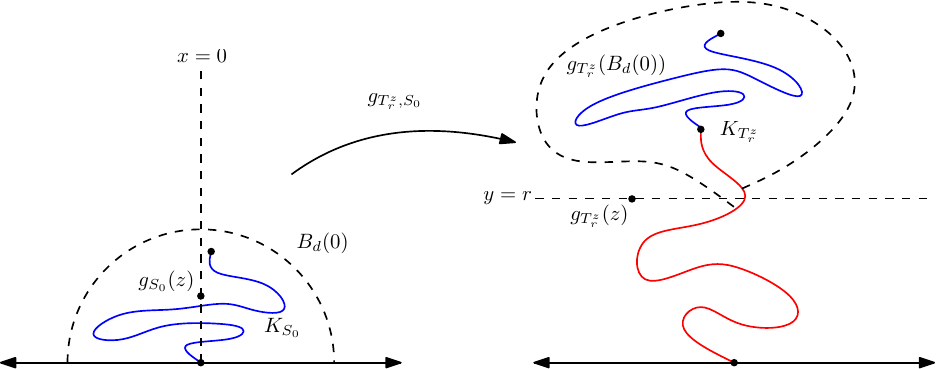}
\end{center}
\caption{An illustration of the proof of Proposition~\ref{d control}. First, we run the reverse Loewner flow with a force point at $z$ until the first time $S_0$ that $z$ is mapped to a point on the imaginary axis. We show in Section~\ref{force pt to 0} that for each $\zeta > 0$, it holds with uniformly positive probability (independent of the particular choice of $z$) that $S_0 \leq \zeta$, $Y_{S_0} = \im g_{S_0}(z) \leq 5\zeta^{1/2}$, and $K_{S_0} \subset B_d(0)$ for some $d > 0$ independent of the particular choice of $z$. Once we condition on the reverse Loewner flow up until time $S_0$, the law of the maps $g_{S_0 , v + S_0  }$ which satisfy $g_{S_0 , v + S_0 } \circ g_{S_0} = g_{v+S_0}$ for $v \geq 0$ is that of a reverse $\op{SLE}_\kappa(\rho)$ Loewner flow with force point at $Z_{S_0} = g_{S_0}(z)$. In Section~\ref{Y from 0 sec}, we show that the first time that the force point for such a Loewner flow reaches the line $\{\im w = r\}$ (i.e., $T_r^z - S_0$) is bounded independently of $Z_{S_0}$ with high probability. Furthermore, the conformal map $g_{S_0 , T_r^z}$ is likely to ``push" $B_d(0) \cap \BB H$ (and hence also $K_{S_0}$) away from the real axis; and the hull of this map (shown in red) is unlikely to be too large. These latter conditions together with Lemma~\ref{G implies U} imply that $G(g_{T_r^z}^{-1} , \mu)$ occurs with uniformly positive probability for appropriate choice of $\mu$. }  
\label{aux condition fig}
\end{figure}

We adopt the following notation. Fix $z\in \BB H$ with $|\re z|\leq R$ and $\im z = \ep$. Let 
\eqb \label{XYZ def}
 Z_t=g_t(z)  =X_t+iY_t .
 \eqe 
 By~\eqref{dW_t}, we have that under $ \BB P_*^z $, 
\begin{align} \label{X Y ODE}
dX_t=  ( \rho-2  ) \frac{ X_t }{|Z_t|^2} \, dt - \sqrt\kappa dB_t^z ,\qquad dY_t=  \frac{2Y_t  }{   |Z_t|^2} \, dt , \qquad X_0= \re z , \qquad Y_0 = \ep
\end{align}
for $B_t^z$ a $\BB P_*^z$-Brownian motion. Also let
\eqb \label{X tau}
S_0 := \inf\{t\geq 0 : X_t = 0\} .
\eqe

\subsection{Pushing the force point to the imaginary axis}
\label{force pt to 0}  

In this subsection we will prove the following lemma, which deals with the setup on the left side in Figure~\ref{aux condition fig}. 

\begin{lem}\label{Y to 0}
Suppose we are in the setting of Proposition~\ref{d control}. Let $Z_t = X_t + i Y_t$ be as in~\eqref{XYZ def} and let $S_0$ be as in~\eqref{X tau}. For each $\zeta \in (0,1)$, there exists $d > 0$ and $p_0 > 0$, independent of $\ep$ and of $X_0 \in [-R,R]$, such that the $\BB P_*^z$-probability of the event
\eqb \label{Y to 0 event}
E_0 = E_0(z,d) := \left\{ S_0 \leq \zeta , \, Y_{S_0} \leq 5\zeta^{1/2} ,\, K_{S_0} \subset B_d(0) \right\}
\eqe 
is at least $p_0$.
\end{lem}  
\begin{proof}
By symmetry we can assume without loss of generality that $\re z = X_0  >0$. We will treat the conditions in the definition of $E_0$ in order.  
Let
\eqb \label{min nu}
\nu > 1\wedge \left( \frac{2(\rho-2)}{\kappa}  + 1  \right)  
\eqe 
and let $\wt X $ be $\sqrt\kappa$ times a Bessel process driven by $-B_t^z$, started from $X_0$, of dimension $\nu$. From the form of the SDE~\eqref{X Y ODE}, one sees that a.s.\
\eqb \label{X tilde X}
\wt X_t \geq X_t, \quad \forall t\leq S_0.
\eqe  
Our choice~\eqref{optimal rho} for $\rho$ implies that~\eqref{min nu} holds for some Bessel dimension $\nu \in (0,2)$, in which case~$\wt X$ hits~$0$ before time $\zeta$ with uniformly positive probability \cite[Proposition~1.21]{lawler-book}. Hence we can find $p_0  > 0$ independent of $\ep$ and uniform for $X_0 \in [-R,R]$ such that 
\eqb \label{zeta prob}
\BB P_*^z\left(S_0 \leq \zeta \right) \geq 2p_0 .
\eqe  

By~\eqref{X Y ODE}, $Y$ is increasing and $\partial_t Y_t^2 \leq 4$. Hence $Y_t \leq  4t^{1/2} + \ep$, so on the event $\{S_0 \leq \zeta \}$ we have $Y_{S_0} \leq 5\zeta^{1/2}$.  

It remains to deal with the condition $\{K_{S_0} \subset B_d(0)\}$. 
Let $\wt X$ be the Bessel process of dimension $\nu$ started from $X_0$ driven by $-B_t^z$, as above. Since $\wt X$ and $B^z$ are a.s.\ bounded up to time $\zeta$ and their laws do not depend on $\ep$, it follows from~\eqref{X tilde X} and~\eqref{zeta prob} that we can find $C_0 > 0$, independent of $\ep$ and uniform for $X_0 \in [-R,R]$ such that the probability of the event 
\eqbn
E_0^* := \left\{ S_0 \leq \zeta ,\, Y_{S_0} \leq 5\zeta^{1/2} ,\, \sup_{t\leq \zeta} |\sqrt \kappa B_t^z| \leq C_0 ,\, \sup_{t\leq \zeta} |X_t| \leq C_0 \right\}
\eqen 
is at least $p_0$.  

By~\eqref{X Y ODE}, for $t\leq S_0$ it holds that
\eqb \label{int of Re}
| \rho-2   | \int_0^t \frac{X_v}{X_v^2 +Y_v^2}  \, dv \leq |X_0|  + |X_t| + |\sqrt\kappa B_t| .
\eqe 
In the case $\rho \not=2$, it follows from~\eqref{int of Re} that on the event $E_0^*$, 
\eqb \label{int of Re sup}
\int_0^t \frac{X_v}{X_v^2 +Y_v^2} \, dv \leq  C_1 := \frac{R  + 2C_0}{|\rho-2|} .
\eqe
In the case $\rho =2$, it follows from~\eqref{X Y ODE} that $X$ is a constant times a Brownian motion, so in this case we can (using~\eqref{zeta prob}) find a possibly larger constant $C_1$, still independent of $\ep$, such that~\eqref{int of Re sup} holds with probability at least $1-p_0/2$. In this case we add this latter condition to the event $E_0^*$ (and replace $p_0$ with $p_0/2$). 

Now consider some $b \in \BB R$ with $|b| > 1$. Let $\delta> 0$ and let $\tau_b$ be the first time $t$ that $|g_t(b)| \leq \delta$. By~\eqref{dW_t} and the reverse Loewner equation, 
\eqbn
 g_t(b)  = -\int_0^t \frac{2}{g_v(b)} \, dv + \rho \int_0^t \frac{X_v}{X_v^2 +Y_v^2} \, dv  -\sqrt\kappa B_t^z + b . 
\eqen
So, it follows from~\eqref{int of Re sup} that on $E_0^*$,
\eqbn
\inf_{t \leq S_0 \wedge \tau_b} |g_t (b)| \geq |b| - C_2,
\eqen 
where
\[
 C_2 =  2 \zeta \delta^{-1}  + |\rho| C_1 + C_0  .
\]
Hence if we take $|b| > 2C_2$, then we have $\inf_{t \leq S_0 \wedge \tau_b} |g_t (b)| \geq C_2$, which implies $\tau_b > S_0$ (provided we choose $\delta <  C_0$).

In particular, if $b>1$ is chosen sufficiently large (independent of $\ep$ and $X_0\in [-R,R]$), then $g_{S_0}(-b)$ and $g_{S_0}(b)$ lie in $\BB R$. Therefore the map $g_{S_0}^{-1}$ takes $\partial K_\tau$ into $[-b,b]$. This implies that the harmonic measure from $\infty$ of $K_\tau$ in $\BB H\setminus K_\tau$ is at most $2\pi b$, so by \cite[Equation 3.14]{lawler-book}, it follows that $\op{diam} K_{S_0}$ is bounded by a constant independent of $\ep$ and $X_0\in [-R,R]$ on $E_0^*$. Since $\BB P_*^z(E_0^*) \geq  p_0$, the lemma follows.
\end{proof}

\subsection{Pushing the force point starting from the imaginary axis}
\label{Y from 0 sec}

In light of the strong Markov property and Lemma~\ref{Y to 0}, we now need to consider the behavior of the process~\eqref{X Y ODE} if we start $(X_0 , Y_0)$ from $(0,y)$ for $y \in [\ep ,5\zeta^{1/2}]$ and $\zeta$ as in Lemma~\ref{Y to 0}. For this, we first need to review some calculations from \cite[Section~3]{wedges}. Throughout this subsection, we assume $X_0 = 0$ and $Y_0  = y \in [\ep , 5\zeta^{1/2}]$. Let 
\eqb \label{t_y def}
\theta_t = \op{arg} Z_t  \quad \op{and} \quad \frk t_y = \frac12 \log y .
\eqe 
For $\frk t \geq \frk t_y$ define $\sigma(\frk t)$ by  
\eqb \label{sigma(t) def}
\frk t= \int_{0}^{\sigma(\frk t)} \frac{ 1 }{|Z_v|^2} \, dv  + \frk t_y ,
\eqe 
so $d\sigma(\frk t) = |Z_{\sigma(\frk t)}|^2 \, d\frk t$ and $\sigma(\frk t_y) = 0$. 
Denote processes under the time change $t = \sigma(\frk t)$ by a star, so $\theta_{\frk t}^* = \theta_{\sigma(\frk t)}$, etc.  
By some elementary calculations using It\^o's formula (see the proof of \cite[Proposition~3.8]{wedges}), we have $d\log Y_{\frk t}^* = 2\, d\frk t$ and
\eqb \label{d theta*}
d \theta_{\frk t}^* = \sqrt{\kappa} \sin  \theta_{\frk t}^* d\wh B_{\frk t} + \left( 2+\frac{\kappa}{2} - \frac{\rho }{2 } \right) \sin(2\theta_{\frk t}^*)    \, d\frk t  ,
\quad \theta_{\frk t_y}^* = \frac{\pi}{2}  
\eqe 
for $\wh B_{\frk t}$ a Brownian motion. Since $Y_{\frk t_y}^* = Y_0 = y$, it follows that $Y_{\frk t}^* = e^{2\frk t}$. Furthermore, as explained in the proof of \cite[Proposition~3.8]{wedges}, there is a unique stationary distribution for the SDE~\eqref{d theta*} which takes the form
\eqb \label{stationary distribution}
 C \sin^\beta(\theta) \, d\theta ,\qquad \beta = \frac{8-2\rho }{\kappa} ,
\eqe 
where $C$ is a normalizing constant. 

Let $\wt\theta_{\frk t}^*$ be a stationary solution to~\eqref{d theta*} and set $\wt Z_{\frk t}^* = \frac{ e^{2\frk t} e^{ i \wt\theta_{\frk t}^*} }{\sin \wt\theta_{\frk t}^*}$, so that $\im \wt Z_{\frk t}^* = e^{2\frk t}$ and $\op{arg} \wt Z_{\frk t}^* = \wt\theta_{\frk t}^*$. Let $\wt W_{\frk t}^*$ be determined by $\wt Z_{\frk t}^*$ in the same manner that $W_{\frk t}^*$ is determined by $Z_{\frk t}^*$ and define
\eqbn
\wt \sigma  ( \frk t) := \int_0^{\frk t} |\wt Z_v^* |^2 \, dv .
\eqen
Denote processes under the time change $\frk t =  \wt\sigma^{-1}(t)$ by removing the star. Then we have that $(\wt \theta_t , \wt Z_t , \wt W_t)$ are related in the same manner as $(\theta_t , Z_t , W_t)$. Moreover, 
\[
\wt\sigma(\frk t) = \inf\{t \in \BB R  : \im \wt Z_t = e^{2\frk t }\} .
\]
Following \cite[Section~3]{wedges}, we define \emph{a reverse $\op{SLE}_\kappa(\rho)$ process with a force point infinitesimally above 0} to be the Loewner evolution driven by $\wt W$.  

We will eventually compare reverse SLE$_\kappa(\rho)$ with force point starting from $(0,y)$ and reverse SLE$_\kappa(\rho)$ with a force point infinitesimally above 0 by using convergence of a given solution of~\eqref{d theta*} to the stationary distribution. Before we do so, we prove an estimate which is needed to show that the hulls of the reverse SLE$_\kappa(\rho)$ with force point starting from $(0,y)$ do not get too big during the interval of times before a given solution mixes with the stationary solution.  

\begin{lem} \label{small times}
Let $\frk t_y $ be as in~\eqref{t_y def}. For any $p  \in (0,1)$ and $v >0$, there is a $b>0$ depending on $v$, $p$, and $\zeta$ but not $\ep$ or the particular choice of $y\in [\ep , 5\zeta^{1/2}]$ such that 
\eqbn
\BB P_*^z\left( K_{\frk t_y + v}^* \subset B_{b }(0) \right) \geq 1-p. 
\eqen 
Here $K_{\frk t}^* = K_{\sigma(\frk t)}$, for $(K_t)$ the hulls of the reverse Loewner evolution driven by $(W_t)$. 
\end{lem}
\begin{proof}
First note that $\theta_{\frk t}^*$ a.s.\ never hits $0$ or $\pi$. To see this, one observes that $\theta_{\frk t}^*$ is a time change of a constant multiple of the process of \cite[Section~1.11]{lawler-book} with $a = (4 +\kappa-\rho)/\kappa > 1/2$, so the claim follows from \cite[Lemma~1.27]{lawler-book}. 

Therefore there exists $\delta  > 0$ depending only on $v$ such that if $\theta_{\frk t}^*$ is started at time $\frk t_y$ with initial condition $\theta_{\frk t_y}^* = \pi/2$ then with probability at least $1-p/2$ we have $\theta_{\frk t}^* \in (\delta, 2\pi-\delta)$ for each $\frk t\in [\frk t_y , \frk t_y  + v]$. Let $G$ be the event that this occurs. 
 
We can find a constant $c >0$ depending only on $\delta$ such that on the event $G$, we have $X_{\frk t}^*/Y_{\frk t}^* \leq c$ for $\frk t \in [\frk t_y , \frk t_y+v]$. It then follows from~\eqref{X Y ODE} that on this event  
\eqbn
\partial_t Y_t \geq \frac{1}{c Y_t} ,\qquad \forall t\in [0, \sigma (\frk t_y+v)  ] ,
\eqen
for a possibly larger $c $. 
This implies 
\eqb \label{Y big}
Y_t^2 \geq c^{-1}  t  + y^2
\eqe
for a possibly larger constant $c$. In particular, $( e^{4v} -1) y^2     = Y_{\sigma(\frk t_y + v)}^2 - y^2 \geq  c^{-1} \sigma(\frk t_y+v)$, so for some possibly larger constant $c $ we have 
\eqb \label{sigma small}
\sigma(\frk t_y + v) \leq c y^2  .
\eqe 

Let $B_t^z$ be the Brownian motion of~\eqref{dW_t}. We can find a $C >0$ depending only on $\zeta$ such that with probability at least $1-p/2$, we have $|\sqrt\kappa B_t^z| \leq Cy$ for each $t \in [0, c y^2]$. Let $G'$ be the event that this occurs and that $G$ occurs, so that $\BB P_*^z(G') \geq 1-p$.
By~\eqref{sigma small} and since $Y_t \geq y$ for each $t\geq 0$, on $G'$,  
\eqbn
\left| \int_0^{\sigma(\frk t_y + v)} \re \frac{1}{Z_t} \,dt \right| \leq \int_0^{c y^2} \frac{X_t}{X_t^2 + Y_t^2} \, dt \preceq 1.
\eqen
By~\eqref{dW_t} and~\eqref{sigma small} it holds on $G'$ that
\eqbn
\sup_{t \in [0,  \sigma ( \frk t_y+v) ]}  |W_t| \preceq 1 , 
\eqen 
with the implicit constant depending only on $C$.  
By \cite[Lemma~4.13]{lawler-book} we then have $\op{diam} K_{\sigma(\frk t_y + v)} \preceq 1$. 
\end{proof}

Our next lemma controls the behavior of the Loewner transition maps $\wt g_{\ol{\frk t} , \frk t}^*$ corresponding to a stationary solution to~\eqref{d theta*} after it has been run for a certain amount of time. This estimate will eventually imply an estimate for the analogous transition maps for the Loewner evolution driven by $(W_t)$ by convergence solutions of SDE's to their stationary distribution. 

\begin{lem} \label{stationary control}
Let $(\wt g_t)$ be the reverse Loewner maps of a reverse $\op{SLE}_\kappa(\rho)$ process with a force point infinitesimally above 0, with hulls $(\wt K_t)$. We adopt the notation given just above Lemma~\ref{small times} so in particular a star denotes processes under the time change $t\mapsto \wt\sigma(\frk t)$. For $ \ol{\frk t} \in\BB R$ and $\frk t \geq \ol{\frk t}$, let $\wt g_{\ol{\frk t} , \frk t}^* $ be the map defined on $\BB H$ which satisfies $\wt g_{\frk t}^* = \wt g_{\ol{\frk t},\frk t}^* \circ \wt g_{\ol{\frk t}}^*$ and let $\wt K_{\ol{\frk t} , \frk t}^* := \wt K_{\frk t}^* \setminus \wt g_{\ol{\frk t} , \frk t}^*(\wt K_{\ol{\frk t}}^*)$ be the corresponding hull. 
For $a ,d > 0$ and $\mu\in\mathcal M$, let $F_{\ol{\frk t} , \frk t}  = F_{\ol{\frk t} , \frk t} (  a , d , \mu)$ be the event that $\wt \sigma(\frk t ) \leq a $ and for each $\delta > 0$, the harmonic measure from $\infty$ of each of $[-\delta  ,0]$ and of $[0 ,\delta ]$ in $\BB H\setminus  \left( \wt K_{\ol{\frk t} , \frk t}^* \cup \wt g_{\ol{\frk t} , \frk t}^*( B_{ d}(0) \cap \BB H) \right)$ is at least $\mu(\delta)$.  
For each $\ol{\frk t}_0 \in \BB R$, $d > 0$, and $p\in (0,1)$, we can find $\frk t_* = \frk t_*(\ol{\frk t}_0 , d , p)  \geq \ol{\frk t}_0$ such that whenever $\ol{\frk t} \leq \ol{\frk t}_0$ and $\frk t \geq  \frk t_*$, there exists $  a = a(d,p,\frk t,\ol{\frk t}_0)   > 0$ and $\mu = \mu(d,p,\frk t,\ol{\frk t}_0) \in \mcl M$ such that
\eqbn
\BB P_*^z\left(F_{\ol{\frk t} , \frk t} \right) \geq 1-p .
\eqen
\end{lem}

The reason for looking at harmonic measure in $\BB H\setminus  \left( \wt K_{\ol{\frk t} , \frk t}^* \cup \wt g_{\ol{\frk t} , \frk t}^*( B_{ d}(0) \cap \BB H) \right)$ instead of just $\BB H\setminus   \wt K_{\ol{\frk t} , \frk t}^*  $ is that for an appropriate choice of $d$, the set $B_d(0)\cap \BB H$ contains the segment of the curve $\eta$ grown before the force point gets to the imaginary axis (see Lemma~\ref{Y to 0}).

\begin{proof}[Proof of Lemma~\ref{stationary control}]
By \cite[Proposition~3.10]{wedges}, for each $\frk t >0$, the conditional law of $\wt K_{\frk t }^*$ given $\wt Z_{\frk t }^*$ is that of a forward chordal $\op{SLE}_\kappa(\rho-8)$ hull with an interior force point at $\wt Z_{\frk t}^*$ stopped at the first time it hits its force point. By \cite[Theorem~3]{sw-coord} this law is that same as that of the hull of a radial $\op{SLE}_\kappa( \kappa + 2-\rho)$ from 0 to $\wt Z_{\frk t }^*$ with a force point at $\infty$, run until the first time it hits $\wt Z_{\frk t }^*$. Since $ \kappa  +2 -\rho > \kappa/2 - 2$ (by our choice of~$\rho$) \cite[Theorem~1.12]{ig4} implies that such a process is transient (i.e., almost surely tends to its target point) and \cite[Lemma~2.4]{ig4} implies that it a.s.\ does not intersect itself or hit $\BB R \cup \{\infty\}$. In particular, $\wt K_{\frk t }^*$ is a.s.\ a simple curve which does not intersect $\BB R$ except at its starting point and has finite half-plane capacity. By stationarity the same is a.s.\ true of $\wt K_{\ol{\frk t},\frk t}^*$ for each $\ol{\frk t} \in\BB R$ and $\frk t \geq \ol{\frk t}$. 

By the uniqueness of the stationary solution to~\eqref{d theta*}, for each $v \in\BB R$ we have $\wt\theta_{\cdot +v }^* \eqD \wt\theta^*$.
Since $\wt\theta^*$ determines the driving function $\wt W^*$ and hence also the Lowener chain $(\wt g_{\frk t}^*)$, and since $ \wt Y_{\frk t}^* = e^{2\frk t}$, we have
\eqb \label{comp scaling}
\left\{ e^{-2v} \wt g_{\frk t  + v}^*(e^{ 2v} \cdot)  \,: \, \frk t\in\BB R \right\} \eqD \{ \wt g_{\frk t}^*  \,:\, \frk t\in \BB R\} ,\quad \forall v \in \BB R .
\eqe    

Now fix $\ol{\frk t}_0 \in \BB R$, $d > 0$, and $p\in (0,1)$. By~\eqref{comp scaling}, the law of the diameter of $\wt K_{\ol{\frk t}}^*$ is stochastically non-decreasing as $\ol{\frk t}$ increases. By~\cite[Proposition~3.46]{lawler-book}, it follows that we can find a deterministic $D = D(\ol{\frk t}_0 , d , p) > 0$ such that 
\eqb \label{ball image}
\BB P_*^z\left((B_d(0) \cap \BB H) \setminus \wt K_{\ol{\frk t}}^*\subset  \wt g_{\ol{\frk t}}^*\left(B_D(0) \cap \BB H  \right)   \right) \geq 1 - p/4  \quad  \forall \ol{\frk t} \leq \ol{\frk t}_0 .
\eqe 
Almost surely, the curve $\wt K_{\ol{\frk t}}^*$ does not intersect $\BB R$ except at its starting point, so there exists some deterministic $\delta > 0$ and $\lambda  > 0$ (depending only on $\ol{\frk t}$ and $p$) such that with probability at least $1-p/4$, we have $\im \wt g_0^*(w) \geq \lambda $ for each $w \in B_{\delta}(0)$.  
By~\eqref{comp scaling}, we can find $\frk t_* = \frk t_*(\ol{\frk t}_0, D , p , \lambda , \delta)  \geq \ol{\frk t}_0$ such that for $\frk t \geq \frk t_*$, it holds with probability at least $1-p/4$ that $\im \wt g_{\frk t}^*(w) \geq 1$ for each $w\in B_D(0) \cap \BB H$. 

Suppose $\ol{\frk t} \leq \ol{\frk t}_0$ and $\frk t \geq \frk t_*$. If $\im \wt g_{\ol{\frk t}, \frk t}^*(x)  < 1$ for some $x \in B_d(0) \cap \BB H$, then since $K_{\ol{\frk t}}^*$ has empty interior, there must be some $x' \in (B_d(0) \cap \BB H) \setminus \wt K_{\ol{\frk t}}^*$ for which $\im \wt g_{\ol{\frk t},  \frk t}^*(x') < 1$. If the event in~\eqref{ball image} holds, then $x' = \wt g_{\ol{\frk t}}^*(w)$ for some $w \in B_D(0) \cap \BB H$, so by definition of $\wt g_{\ol{\frk t} , \frk t}^*$ we have $\im \wt g_{\frk t}^*(w) < 1$. By our choice of $\frk t_*$, we find that 
\eqbn
\BB P_*^z\left( \im \wt g_{\ol{\frk t},\frk t}^*(w) \geq 1, \: \forall w \in B_d(0) \cap \BB H \right) \geq 1 -  p/2  .
\eqen 
Since $\wt K_{\ol{\frk t} , \frk t}^* \subset K_{\frk t}^*$ and $K_{\frk t}^*$ a.s.\ does not intersect $\BB R$ except at 0 and a.s.\ has finite half plane capacity, for each such $\frk t \geq \frk t_*$ we can find $a$ and $\mu$ as in the statement of the lemma such $\BB P_*^z(F_{\ol{\frk t} ,\frk t}) \geq 1-p$ for each $\ol{\frk t} \leq \ol{\frk t}_0$. 
\end{proof}

The following lemma together with Lemma~\ref{Y to 0} are the main inputs in the proof of Proposition~\ref{d control}.

\begin{lem}\label{Y from 0}
Suppose we are in the setting of this subsection (so that in particular $X_0 = 0$ and $Y_0 = y$). Let $\wt T_r := \inf\{t \geq 0 :   Y_t = r\} = \sigma(\frac12 \log r)$. Also let $d> 0$ and $p \in (0,1)$. There is an $r_*  > 0$ (depending on $\zeta , d$, and $p$) such that for $r \geq r_*$, there exists $A > 0$ and $\mu\in\mcl M$, independent of $\ep$ and the particular choice of $y\in [\ep , 5\zeta^{1/2}]$ such that the following is true. Let $E_1 = E_1(r ,  d , A ,  \mu )$ be the event that $\wt T_r \leq  A$ and for each $\delta > 0$, the harmonic measure from $\infty$ of each of $[-\delta  ,0]$ and of $[0 ,\delta ]$ in $\BB H\setminus \left(  K_{\wt T_r} \cup g_{\wt T_r}( B_{ d}(0) \cap \BB H ) \right) $ is at least $\mu(\delta)$. Then $\BB P_*^z(E_1) \geq 1-p$. 
\end{lem}

\begin{remark}
The purpose of the harmonic measure condition in the definition of $E_1$ is as follows. When we compose with $g_{S_0}$ on the event $E_0$ of Lemma~\ref{Y to 0}, the part of the hull grown before time $S_0$ is ``pushed" into $g_{\wt T_r}(B_d(0))$. The harmonic measure condition in the definition of $E_1$ together with Lemma~\ref{G implies U} will then imply the occurrence of $G(g_{\wt T_r}^{-1} , \mu)$ on the event $E_0\cap E_1$. See also Figure~\ref{aux condition fig}.
\end{remark}

\begin{proof}[Proof of Lemma~\ref{Y from 0}]
Define the processes $X_{\frk t}^* , Y_{\frk t}^* ,Z_{\frk t}^* , \sigma(\frk t) $, and $\theta_{\frk t}^*$ as above. 
Let $(\wt g_t)$ be the reverse Loewner maps of a reverse $\op{SLE}_\kappa(\rho)$ process with a force point immediately above 0. We adopt the notation given just above Lemma~\ref{stationary control}, so that for $\frk t > 0$, $\wt Z_{\frk t}$ is the image of the force point under $\wt g_{\frk t}$ and $\wt \theta_{\frk t}^* = \op{arg} \wt Z_{\frk t}^*$ is the corresponding stationary solution to~\eqref{d theta*}. 

By the convergence of the law of the solution of~\eqref{d theta*} to its stationary distribution, there exists $v > 0$, independent of $\ep$ and the particular choice of $y\in [\ep , 5\zeta^{1/2}]$, such that the following is true. The total variation distance between the law of $\theta_{\frk t_y +v}^*$, started from $\pi/2$ at time $\frk t_y$ and the stationary distribution~\eqref{stationary distribution} is at most $ p/4$. Let $\ol{\frk t}_y = \frk t_y +v$. We can couple $\theta^*$ with $ \wt \theta^*$ in such a way that with probability at least $1-p/3$, these two processes agree at time $\ol{\frk t}_y$ and (by the Markov property) at every time thereafter. Let $F_1$ be the event that $\theta_{\frk t}^* = \wt \theta_{\frk t}^*$ for each $\frk t \geq \ol{\frk t}_y$.   

Define the maps $\wt g_{ \ol{\frk t}_y ,\frk t}^*$ and the hulls $\wt K_{ \ol{\frk t}_y , \frk t}^*$ for $\frk t\geq \ol{\frk t}_y$ as in Lemma~\ref{stationary control}. Define $g_{\ol{\frk t}_y , \frk t}^*$ and $K_{\ol{\frk t}_y,\frk t}^*$ for $\frk t \geq \ol{\frk t}_y$ analogously but with $g_{\frk t}^*$ and $K_{\frk t}^*$ in place of $\wt g_{\frk t}^*$ and $\wt K_{\frk t}^*$. We have that $(\theta_{\frk t}^* , e^{2\frk t})$ determines $W_{\frk t}^*$ and hence also $(g_{\frk t}^*)$. Similarly for the corresponding processes under the stationary distribution. Therefore on $F_1$, we have
\eqb \label{F1 agree}
g_{\ol{\frk t}_y , \frk t}^* = \wt g_{\ol{\frk t}_y , \frk t}^* ,\qquad  K_{\ol{\frk t}_y,\frk t}^* = \wt K_{\ol{\frk t}_y,\frk t}^*      ,\qquad \forall \frk t \geq \ol{\frk t}_y .
\eqe 

By Lemma~\ref{small times} we can find a $b >0$ depending only on $v$ such that the probability of the event
\eqbn 
F_2:=\{   K_{\ol{\frk t}_y}^* \subset B_{b }(0) \}  
\eqen  
is at least $1-p/3$. By combining this with~\cite[Proposition~3.46]{lawler-book}, we find that there exists a deterministic constant $d' = d'(d ,  b ) > 0$ such that on the event $F_2$ we have
\eqb  \label{K_ol t small}
K_{\ol{\frk t}_y}^* \cup g_{\ol{\frk t}_y}^*\left(B_d(0)\cap \BB H \right)   \subset B_{d'}(0)\cap \BB H.
\eqe 

Let $\ol{\frk t}_0 = 5\zeta^{1/2} +v$, so that $\ol{\frk t}_y \leq \ol{\frk t}_0$. 
Let $\frk t_*$ be chosen so that the conclusion of Lemma~\ref{stationary control} holds with this choice of $\ol{\frk t}_0$, $d'$ in place of $d$, and $p /3$ in place of $p$. Let $\frk t \geq \frk t_*$ and let $a = a(d',p,\frk t,\ol{\frk t}_0) > 0$ and $\mu_0 = \mu_0(d',p,\frk t,\ol{\frk t}_0) \in\mcl M$ be chosen so that with $F_3 = F_{\ol{\frk t}_y , \frk t}( a, d' ,   \mu_0 )$ the event of Lemma~\ref{stationary control} we have $\BB P_*^z(F_3) \geq 1-p/3$ for each choice of $\ol{\frk t}_y \leq \ol{\frk t}_0$. Note that $a$ and $\mu_0$ do not depend on $\ep$ or the particular choice of $y\in [\ep , 5\zeta^{1/2}]$. 
Then we have
\[
\BB P_*^z\left(F_1 \cap F_2 \cap F_3 \right) \geq 1-p .
\]

If we set $r_* = e^{2\frk t_*}$ and $r = e^{2\frk t}$, then $r$ ranges over $[r_* , \infty)$ as $\frk t$ ranges over $[0,\infty)$. 
We will now conclude the proof by showing that $F_1\cap F_2 \cap F_3 \subset E_1$ for an appropriate choice of parameters. On the event $F_1 \cap F_2 \cap F_3$, we have 
\alb
\wt T_r  = \op{hcap} K_{\frk t  }^* = \op{hcap} K_{\frk t,\ol{\frk t}_y}^* + \op{hcap} K_{\ol{\frk t}_y}^* .
\ale
The first term is at most $a$ by the definition of $F_3$ together with~\eqref{F1 agree}. The second term is at most a finite constant depending only on $b$. Hence for $r\geq r_*$ we can find $A > 0$ as in the statement of the lemma such that on $F_1\cap F_2 \cap F_3$ we have $\wt T_r \leq A$. Furthermore, on $F_1 \cap F_2 \cap F_3$, 
\alb
K_{\wt T_r} \cup g_{\wt T_r}\left(B_d(0)\cap \BB H \right) 
&= K_{\frk t}^* \cup g_{\frk t}^*\left(B_d(0) \cap \BB H\right)\\
&=   K_{\ol{\frk t}_y , \frk t}^* \cup g_{\ol{\frk t}_y , \frk t }^*\left( K_{\ol{\frk t}_y}^* \cup g_{\ol{\frk t}_y}^*\left(B_d(0)\cap \BB H \right) \right) \quad \text{(by definition of $g_{\ol{\frk t}_y , \frk t}^*$)} \\
&= \wt K_{\ol{\frk t}_y , \frk t}^* \cup \wt g_{\ol{\frk t}_y , \frk t }^*\left( K_{\ol{\frk t}_y}^* \cup g_{\ol{\frk t}_y}^*\left(B_d(0)\cap \BB H \right) \right) \quad \text{(by~\eqref{F1 agree})} \\
&\subset \wt K_{  \frk t}^* \cup \wt g_{\ol{\frk t}_y , \frk t }^*\left( B_{d'}(0) \cap \BB H   \right) \quad \text{(by~\eqref{K_ol t small} and the definition of $K_{\ol{\frk t}_y , \frk t}^*$)} .
\ale
It now follows from the definition of $F_3$ (see Lemma~\ref{stationary control}) that for each $r\geq r_*$, we can find $\mu\in\mcl M$ satisfying the conditions of the lemma such that with this choice of $\mu$ and $A$ as above, the event $E_1$ holds on $F_1\cap F_2 \cap F_3$. 
\end{proof}

\subsection{Conclusion of the proof}

Now we can combine the results of the previous two subsections to complete the proof of Proposition~\ref{d control}. 

\begin{proof}[Proof of Proposition~\ref{d control}]
Let $\zeta >0$, $d>0$, and $p_0 > 0$ be as in Lemma~\ref{Y to 0}, and let $E_0 = E_0(\zeta  , d )$ be the event of that lemma, so that $\BB P_*^z(E_0) \geq p_0$. Let $S_0$ be as in~\eqref{X tau} and for $t \geq S_0$, let $g_{  S_0 , t}$ be the map defined on $\BB H$ which satisfies $g_t = g_{ S_0 , t} \circ g_{S_0}$.  

Conditional on $\{g_t : t\leq S_0\}$, the law of $\{g_{ S_0 , v + S_0} : v\geq 0\}$ is the same as that of $\{g_v : v \geq 0\}$ started from $Z_0 = (0 , Y_{S_0})$ instead of from $Z_0 = z$. Note that $Y_{S_0} \in [\ep , 5\zeta^{1/2}]$ on $E_0$. Define the time $\wt T_r$ and the events $E_1 = E_1(r , A , d, \mu)$ as in Lemma~\ref{Y from 0} but with $g_{ S_0 , \cdot + S_0 }$ in place of $g_\cdot$. Let $r_*$, $\mu$, and $A$ satisfy the conclusion of Lemma~\ref{Y from 0} for $d$ as above and $p=1/2$. Then if $r \geq r_*$ we have $\BB P_*^z(E_1 | E_0) \geq 1/2$, whence $\BB P_*^z(E_0 \cap E_1) \geq p_0/2$. 

Since $S_0 \leq \zeta$ on $E_0$ by definition and by the definition of $E_1$ we have $T_r^z = S_0 + \wt T_r \leq \zeta + A$ on $E_0\cap E_1$. Furthermore, by definition of $E_1$, on the event $E_0 \cap E_1$, the harmonic measure from $\infty$ of each of $[-\delta  ,0]$ and $[0, \delta ]$ in $\BB H\setminus  K_{T_r^z}$ is at least $\mu(\delta)$. By Lemma~\ref{G implies U} we can find $\mu'\in\mathcal M$ and $t_*  >0$ as in the proposition such that 
\[
E_0\cap E_1  \subset \{T_r^z < t_*\} \cap G(g_{T_r^z}^{-1} , \mu')    .
\]
This proves the statement of the proposition. 
\end{proof} 

\section{Comparisons of derivatives using harmonic measure}
\label{local prelim}

In this section we will prove some technical lemmas which allow us to compare conformal maps defined on different domains. We recall the notation $\op{hm}^z(I ; D)$ for the harmonic measure of $I\subset \bdy D$ from $z$ in $D$. We start with a simple geometric description of the derivative of a certain conformal map defined on a subdomain of $\BB D$.

\begin{lem} \label{phi hm}
Let $U\subset \BB D$ be a simply connected subdomain. Let $x,y\in \bdy\BB D$ such that $[x,y]_{\bdy\BB D}\subset \partial U$. Let $m\in (x,y)_{\partial\BB D}$ and let $\Psi : U\rta \BB D$ be the conformal map taking $x$ to $-i$, $y$ to $i$, and $m$ to 1. Let $z\in U$, let $I$ be a sub-arc of $[x,y]_{\partial\BB D}$, and suppose that for some $\delta > 0$, the distance from $\Psi(z)$ to $\Psi(I)$ and the length of $\Psi(I)$ are each at least $\delta$. Then 
\[
\op{hm}^z(I  ; U)  \asymp \op{dist}(z , \partial U )  |\Psi'(z)|
\]
 with the implicit constants depending only $\delta$.  
\end{lem}
\begin{proof} 
By the conformal invariance of harmonic measure, $\op{hm}^z(I  ; U)  = \op{hm}^{\Psi(z)}(\Psi(I) ;U)$. By our hypotheses on $\Psi(I)$, $\op{hm}^{\Psi(z)}(\Psi(I) ; U) \asymp \op{dist}(\Psi(z) , \partial\BB D)$, with the implicit constant depending only on~$\delta$. By the Koebe quarter theorem, $\op{dist}(\Psi(z) , \partial\BB D) \asymp \op{dist}(z , \partial U )  |\Psi'(z)|$ with a universal implicit constant. 
\end{proof}

\begin{remark} \label{phi hm hypotheses} 
We note some circumstances under which the hypotheses of Lemma~\ref{phi hm} are satisfied. Let $\wh U$ denote the Schwarz reflection of $U$ across $[x,y]_{\partial\BB D}$. Suppose $I\subset (x,y)_{\partial\BB D}$ with $m \in I$ and the distance from $\partial U \setminus \partial \BB D$ to $I$ is at least a constant $\zeta > 0$. If $z$ lies at distance at least a constant $\zeta'  >0$ from $\partial \BB D$ and is sufficiently close to $\partial U$, then by considering harmonic measure from $m$ in $\wh U$ (c.f.\ the proof of Lemma~\ref{G dist}), we get that the hypotheses of Lemma~\ref{phi hm} are satisfied with $\delta$ depending only on $\zeta , \zeta'$ and the length of $I$. In particular, if the event $\mathcal G_{[x,y]_{\partial\BB D}}(\Psi , \mu)$ of Section~\ref{G disk sec} occurs, then Lemma~\ref{G dist} implies that, under the same hypotheses on $z$, the hypotheses of Lemma~\ref{phi hm} are satisfied with $\delta$ depending only on $\mu , \zeta',$ and the length of $I$. 
\end{remark}
 
\begin{figure}
\begin{center}
\includegraphics[scale=.7]{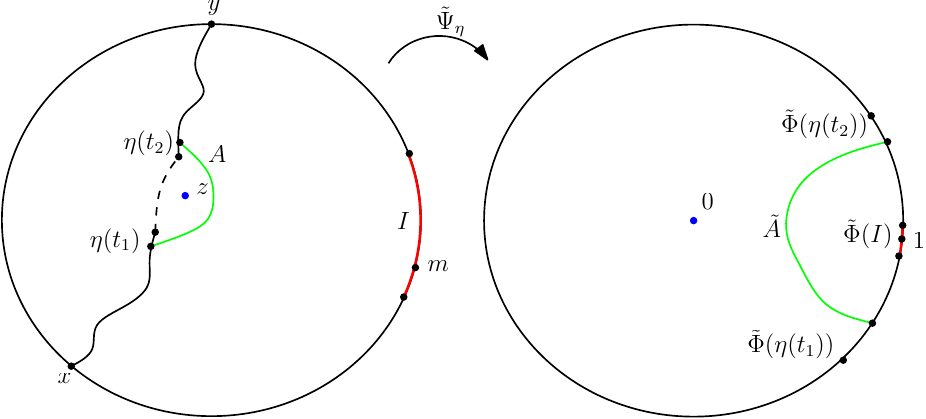}
\end{center}
\caption{An illustration of the proof of Lemma~\ref{phi hm arc}. In the left figure, the domain $D_\eta$ is the part of $\BB D$ lying to the right of the curve $\eta$ (including the dashed part $\eta([t_1,t_2])$) and the domain $D_\eta^0$ is the complement of the two solid black segments of $\eta$. The probability that a Brownian motion started from $z$ exits $D_\eta^0$ in the red arc $I $ is bounded by the supremum of the harmonic measure of $I $ in $D_\eta^0$ from any point of the green crosscut $A$. This, in turn, is bounded by a constant times the supremum of the harmonic measure of $I $ in $D_\eta$ from any point of $A$, which is bounded by the harmonic measure of $I$ from $z$ in $D_\eta$ by our choice of $\wt A$.} \label{phi hm arc fig}
\end{figure}

We now deduce a consequence of Lemma~\ref{phi hm} which allows us to compare the derivatives of conformal maps associated with an entire curve and with part of a curve. In particular, we consider a curve $\eta$ connecting two points of $\bdy\BB D$ and compare the derivative behavior of a conformal map from the right side of $\BB D\setminus \eta$ to $\BB D$ and the derivative behavior of a conformal map from the complement of a segment of $\eta$ and its time reversal to $\BB D$.

\begin{lem} \label{phi hm arc}
Fix $\delta>0$. 
Let $x,y\in\partial\BB D$ and $m\in (x,y)_{\partial\BB D}$ with $|x-m| , |y-m| \geq \delta$. Also let $\eta : [0,\infty] \rta \BB D$ be a simple curve which does not intersect $(x,y)_{\partial\BB D}$ and let $D_\eta$ be the connected component of $\BB D\setminus\eta$ containing $[x,y]_{\partial\BB D}$ on its boundary. Let $\Psi_\eta  : D_\eta \rta \BB D$ be the conformal map taking $x$ to $-i$, $y$ to $i$, and $m$ to 1. 

Fix $t_2 > t_1 \geq 0$, set $D_\eta^0 = \BB D\setminus (\eta([0,t_1]) \cup \eta([t_2 , \infty]) )$, and let $\Phi : D_\eta^0 \rta \BB D$ be the conformal map taking $x^+$ to $-i$, $y^-$ to $i$, and $m$ to 1. Suppose that the following holds for some arc $I\subset [x,y]_{\partial\BB D}$ and some point $z\in D_\eta$.
\begin{enumerate}
\item $\op{hm}^z(   \eta([0,t_1])   ;D_\eta)$ and $\op{hm}^z(\eta([t_2 , \infty]); D_\eta)$ are each at least $\delta $. \label{hm > ell}
\item The length of $\Psi_\eta(I)$ and the distance from $\Psi_\eta(z)$ to $\Psi_\eta(I)$ are each at least $\delta $.  \label{hm mu}
\item The length of $\Phi (I)$ and the distance from $\Phi (z)$ to $\Phi (I)$ are each at least $\delta $.  \label{hm mu 0}
\end{enumerate}
Then $|\Phi '(z)| \asymp   |\Psi_\eta'(z)|$ and $ \op{dist}(z , \partial D_\eta ) \asymp \op{dist}(z , \partial D_\eta^0 )$ with implicit constants depending only on $\delta$ and $z$ but uniform for $z$ in compact subsets of $\BB D$. 
\end{lem}
\begin{proof}
See Figure~\ref{phi hm arc fig} for an illustration of the proof. 

By Lemma~\ref{phi hm}, 
\eqbn
 |\Phi '(z)| \asymp \frac{\op{hm}^z( I  ;  D_\eta^0)}{ \op{dist}(z , \partial D_\eta^0) } \quad \op{and} \quad |\Psi_\eta'(z)|   \asymp \frac{\op{hm}^z( I ;  D_\eta )}{\op{dist}(z , \partial D_\eta )}
 \eqen
with the implicit constants depending only on $\delta$. We clearly have $\op{hm}^z(I  ;  D_\eta^0) \geq \op{hm}^z(I  ; D_\eta)$. By the Beurling estimate, if $r$ is chosen sufficiently large, in a manner depending only on $\delta$, then $\op{hm}^z(\eta \cap B_{r \op{dist}(z , \eta)}(z) ; D_\eta) \geq 1-\delta/2$. So, our hypothesis~\ref{hm > ell} implies that $ \op{dist}(z , \partial D_\eta ) \asymp \op{dist}(z , \partial D_\eta^0 )$. Therefore it is enough to prove
\eqb \label{hm D0 hm D}
\op{hm}^z(I  ;D_\eta^0) \preceq  \op{hm}^z(I  ; D_\eta)
\eqe 
with the implicit constant depending only on $\delta$. 

Let $\wt \Psi_\eta : D_\eta \rta \BB D$ be the conformal map taking $z$ to 0 and $m$ to 1. By conformal invariance of harmonic measure and our hypothesis~\ref{hm > ell}, the distance from each of $\wt\Psi_\eta(\eta(t_1))$ and $\wt\Psi_\eta(\eta(t_2))$ to $\wt\Psi_\eta(I )$ is at least $2\pi\delta$. Hence we can choose a crosscut $\wt A$ in $\BB D$ which disconnects 0 from $\wt\Psi_\eta(I)$ such that each point of $\wt A$ lies at distance at least $\delta$ from $\wt\Psi_\eta(I )$ and from $[\wt\Psi_\eta(\eta(t_2)) , \wt\Psi_\eta(\eta(t_t))]_{\partial\BB D}$. The harmonic measure of $\wt\Psi_\eta(I )$ from each point of $\wt A$ in $\BB D$ is bounded above by a constant depending only on $\delta$ times the length of $\wt\Psi_\eta(I )$, which in turn is proportional to $ \op{hm}^z(I  ; D_\eta)$. Furthermore, the harmonic measure of the arc $[\wt\Psi_\eta(\eta(t_2)) , \wt\Psi_\eta(\eta(t_t))]_{\partial\BB D}$ from each point of $\wt A$ in $\BB D$ is bounded above by a constant $a < 1$ depending only on $\delta$.  

Let $A = \wt\Psi_\eta^{-1}(\wt A)$. Then  
\eqb \label{A hm}
\op{hm}^w(I  ; D_\eta) \preceq \op{hm}^z(I  ; D_\eta) , \qquad  \op{hm}^w(  \eta([t_1 , t_2]) ; D_\eta ) \leq a       \qquad \forall w\in A
\eqe 
with the implicit constant depending only on $\delta$. 

A Brownian motion started from $z$ must hit $A$ before exiting $D_\eta^0$ in $I $. Therefore,
\eqb \label{hm z < hm A}
\op{hm}^z(I  ;D_\eta^0) \leq  \sup_{w\in A} \op{hm}^w(I  ; D_\eta^0) .
\eqe 
For $w\in A$, we can decompose the event that a Brownian motion $B$ started at $w$ exits $D_\eta^0$ in $I $ as the union of the event that $B$ hits $I $ before $\eta([t_1 , t_2])$ and the event that $B$ hits $\eta([t_1 , t_2])$ and then $I $. By~\eqref{A hm} the former event has probability at most a constant $C$ (depending only on $\delta$) times $  \op{hm}^z(I  ; D_\eta)$. By the Markov property the latter event has probability at most 
\[
\sup_{w\in A} \op{hm}^w(\eta([t_1 , t_2]) ; D_\eta)  \sup_{v\in \eta([t_1 , t_2]) }   \op{hm}^v(I  ; D_\eta^0) .
\]
Since $A$ disconnects $\eta([t_1 , t_2])$ from $I $ in $D_\eta^0$ we have $ \sup_{v\in \eta([t_1 , t_2]) }   \op{hm}^v(I  ; D_\eta^0)  \leq \sup_{w\in A} \op{hm}^w(I  ; D_\eta^0)$. By combining this with~\eqref{A hm} we get
\eqb \label{sup hm bound}
\sup_{w\in A} \op{hm}^w( I ; D_\eta^0) \leq  C \op{hm}^z(I ; D_\eta) + a \sup_{w\in A} \op{hm}^w(I ; D_\eta^0)  .
\eqe 
Since $a < 1$, we can re-arrange the estimate~\eqref{sup hm bound} to get
\[
\sup_{w\in A} \op{hm}^w( I ; D_\eta^0) \preceq \op{hm}^z(I ; D_\eta)  ,
\]
which together with~\eqref{hm z < hm A} yields~\eqref{hm D0 hm D}. 
\end{proof}

\section{Strict mutual absolute continuity for SLE}
\label{smac sec}

\begin{defn}
\label{smac}
We say that a measure $\mu$ is \emph{strictly mutually absolutely continuous} (s.m.a.c.) with respect to a measure $\nu$ if $\mu$ and $\nu$ are mutually absolutely continuous with Radon-Nikodym derivative a.e.\ bounded above and below by finite and positive constants.
\end{defn}

In this appendix we will prove a lemma which gives that the conditional law of the ``middle part" of an $\op{SLE}_\kappa(\rho^L ; \rho^R)$ curve given the initial and terminal segments, on a certain regularity event, is \hyperref[smac]{s.m.a.c.}\ with respect to the law of the middle part of an ordinary $\op{SLE}_\kappa$ curve (see Lemma~\ref{rho abs cont} below for an exact statement). This result is needed in the proof of our two-point estimate (see in particular Lemma~\ref{lem-wtE-prob}). We will deduce our desired result from \cite[Lemma~2.8]{miller-wu-dim} (which gives a similar strict mutual absolute continuity statement for $\op{SLE}_\kappa(\ul\rho)$ curves in domains which agree in a neighborhood of the starting point) together with the coupling results of \cite{ig1}, described in Section~\ref{ig prelim}. 

Before we can prove this result, we need to define the regularity event for the initial and terminal segments of the path which we will work on. Let $x,y\in\partial\BB D$ be distinct. Let $\eta$ be a random curve from $x$ to $y$ in $\BB D$, with time reversal $\ol\eta$. In what follows, we write $\mcl B_\beta =B_{e^{-\beta}}(0)$ and let $\tau_\beta$ (resp.\ $\ol\tau_\beta$) be the first time $\eta$ (resp.\ $\ol\eta$) hits $\mcl B_\beta$, as in Section~\ref{2pt setup sec}.  

Fix $\Delta > \Delta'  > \wt\Delta > 0$. Suppose we are given times $\sigma , \ol\sigma > 0$. Let $\eta^*$ be the part of $\eta$ between $\eta(\sigma)$ and $\ol\eta(\ol\sigma)$. Let $H^* = H^*(\eta^* ; \wt\Delta )$ be the event that $\eta^* \subset \mcl B_{ \wt\Delta}$. Let $ S =  S(\eta; \sigma , \ol\sigma ,   \Delta ,   \wt\Delta   )$ be the event that the following occur. 
\begin{enumerate}
\item $\tau_\Delta \leq \sigma  < \infty$ and $\ol\tau_{\Delta} \leq \ol\sigma < \infty$ (here, $\tau_\Delta = \tau_\beta$ and $\ol{\tau}_\Delta = \ol{\tau}_\beta$ with $\beta = \Delta$). \label{S finite}
\item $\eta^{\sigma}$ (resp.\ $\ol\eta^{\ol\sigma}$) is contained in the $e^{-2\Delta}$-neighborhood of the segment $[x ,0]$ (resp.\ $[y , 0]$).   \label{S hit} 
\item The conditional probability of $H^*$ given $\eta^\sigma \cup \ol\eta^{\ol\sigma}$ is positive.  \label{S cond}
\end{enumerate}  
Also let $S^* = S^*(\eta; \sigma , \ol\sigma ,   \Delta , \Delta' ,  \wt\Delta   )$ be the event that the following occur.
\begin{enumerate}
\item $S(\eta; \sigma , \ol\sigma ,   \Delta ,   \wt\Delta   )$ occurs.  
\item $\eta([ \tau_{\Delta'} , \sigma])$ (resp.\ $\ol\eta([ \ol\tau_{\Delta'} , \ol\sigma])$) is contained in $\mcl B_{\wt\Delta}$.   \label{S* contain}  
\end{enumerate}
 
\begin{remark} \label{L contain remark}
If the event $L $ and the times $\sigma $ and $\ol\sigma$ are defined as in Section~\ref{sec-perfect-setup}, then we have 
\eqbn
L  \subset S^*( \eta ; \sigma  , \ol\sigma  , \Delta , \Delta/2 ,\wt \Delta  ) .
\eqen
This is the primary reason for our interest in the event $S^*(\cdot)$. 
\end{remark}

\begin{remark}
In the case that $\eta$ is an $\op{SLE}_\kappa(\rho^L ;\rho^R)$ (which is what we consider in the section) one can show that condition~\ref{S cond} in the definition of $S$ is in fact implied by the other conditions in the definition of $S$. The idea to establish this is to realize $\eta$ as a flow line of a GFF, then condition on two counterflow lines (run up to a certain stopping time) with the property that the interface between them is a.s.\ equal to $\ol\eta^{\ol\sigma}$. See \cite[Section~5.4]{ig2} for a similar argument. We do not need this fact here though, so for the sake of brevity we include condition~\ref{S cond} as a condition.
\end{remark}

The main result of this section is the following. 
 
\begin{lem} \label{rho abs cont}
Let $\rho^L , \rho^R \in (-2,0]$, $\delta > 0$, and $x,y\in\partial\BB D$ with $|x-y| \geq \delta$. Let $\eta$ be a chordal $\op{SLE}_\kappa(\rho^L ; \rho^R)$ process from $x$ to $y$ in $\BB D$ with force points located at $ x^-$ and $x^+$. Let $\ol\eta$ be its time reversal.
Let $\sigma $ be a stopping time for $\eta$ and let $\ol\sigma $ be a stopping time for the filtration generated by $\eta^{\sigma }$ and $\ol\eta$. 
Let $S^* = S^*(\eta; \sigma , \ol\sigma ,   \Delta , \Delta' , \wt\Delta    )$ as above. Also let $\eta^*$ and $H^* = H^*(\eta^* ; \wt\Delta)$ be as above.  
Let $D $ be the connected component of $\BB D\setminus ( \eta^{\sigma } \cup \ol\eta^{\ol\sigma })$ containing 0.  
 
If $\wt\Delta$ (and hence also $\Delta'$ and $\Delta$) is chosen sufficiently large, in a manner depending only on $\delta$, $\rho^L$, and $\rho^R$, then a.s.\ on $S^*$ the regular conditional law of $\eta^*$ given $\eta^{\sigma } \cup \ol\eta^{\ol\sigma }$ and the event $H^*$ is \hyperref[smac]{s.m.a.c.}\ with respect to the law of a chordal $\op{SLE}_\kappa$ from $\eta(\sigma )$ to $\ol\eta(\ol\sigma )$ in $D $ conditioned on $H^*$, with deterministic constants depending only on $\rho^L ,\rho^R , \kappa , \Delta, \Delta',$ $\wt\Delta$, and $\delta$. 
\end{lem}
 
The idea of the proof of Lemma~\ref{rho abs cont} is to consider a GFF on $\BB D$ whose flow line $\eta_0$ is an ordinary SLE$_\kappa$, then grow auxiliary flow lines with the same start and endpoints in such a way that the conditional law of $\eta_0$ given these auxiliary flow lines is that of an SLE$_\kappa(\rho^L ; \rho^R)$ for the given values of $\rho^L$ and $\rho^R$. By~\cite[Lemma 2.8]{miller-wu-dim}, the conditional laws of these auxiliary flow lines given $\eta_0$ do not depend strongly on a small segment in the middle of $\eta_0$. We then apply Bayes' rule to invert the conditioning. 
See Figure~\ref{rho-abs-cont-fig} for an illustration of the argument.

For the proof of Lemma~\ref{rho abs cont}, we will assume neither $\rho^L$ nor $\rho^R$ is equal to 0; the case when one of the force points is equal to 0 is treated similarly but with only a single auxiliary flow line. 

Choose $ \Delta_0 > \wt\Delta_0 > 0$ satisfying $\wt\Delta_0 < \wt\Delta < \Delta' < \Delta_0 < \Delta$, with $\Delta , \Delta',$ and $\wt\Delta$ as in the statement of Lemma~\ref{rho abs cont}.
Let $ \eta_0$ be an ordinary chordal $\op{SLE}_\kappa $ from $x$ to $y$ in $\BB D$. Let $\ol{ \eta}_0$ be the time reversal of $ \eta_0$. Let $\sigma_0$ (resp.\ $\ol\sigma_0$) be the first time $\eta_0$ (resp.\ $\ol\eta_0$) hits $\mcl B_{\wt\Delta}$. Let $\eta_0^*$ be the part of $\eta_0$ between $\eta(\sigma_0)$ and $\ol\eta(\ol\sigma_0)$. Also let
\eqb \label{S0 def}
 S_0 := S(\eta_0 ; \sigma_0 , \ol\sigma_0 ,   \Delta_0 , \wt\Delta_0 ) ,\qquad  H^*_0 = H^*(\eta^*_0 ; \wt\Delta_0 ) .
\eqe  

We can couple $\eta_0$ with a GFF $h$ on $\BB D$ with appropriately chosen boundary data in such a way that $\eta_0$ is the zero angle flow line\footnote{In the case $\kappa=4$, we replace flow lines by level lines, as defined in \cite{ss-contour,ss-dgff}. Everything works the same with this replacement.}
 (in the sense of Section~\ref{ig prelim}) of $h$ started from $x$. Let $\theta^L > 0$ and $\theta^R < 0$ be chosen so that
\eqb \label{theta choice}
\frac{\theta^L \chi}{\lambda} - 2 = \rho^L ,\qquad  -\frac{\theta^R \chi}{\lambda} -2   =\rho^R .
\eqe 
Let $ \eta_-$ and $\eta_+$ be the flow lines of $h$ started from $x$ with angles $\theta^L$ and $\theta^R$, respectively. Since $\rho^L , \rho^R \in (-2,0)$, the flow lines $\eta_-$ and $\eta_+$ are well defined. Let $D_0$ be the connected component of $\BB D\setminus (\eta_- \cup \eta_+)$ containing the origin. Let $b$ and $\ol b$, respectively, be the first and last point on $\partial D_0$ hit by $\eta_0$. By the results of \cite[Section~7]{ig1}, the conditional law of  the part of $\eta_0$ which lies in $  D_0$ given $\eta_- \cup \eta_+$ is that of a chordal $\op{SLE}_\kappa(\rho^L ; \rho^R   )$ from $b$ to $\ol b$ in $D_0$ with force points located on either side of $b$.

We also fix a small parameter $\alpha  \in (0,1)$ and let $t_-$ and $t_+$ respectively be the first times $\eta_-$ and $\eta_+$ exit $B_{1-\alpha}(0)$.

\emph{Throughout the remainder of this subsection, we require all implicit constants, including those in s.m.a.c., to depend only on $ \Delta, \wt\Delta, \Delta' , \Delta_0 , \wt\Delta_0 , \alpha, \rho^L ,\rho^R  , \kappa $, and $\delta$ (in particular, implicit constants are not allowed to depend on the realization of whatever we are conditioning on or on the choice of stopping times $\sigma , \ol \sigma$).}

\begin{lem}\label{eta^* abs cont}
Let $\omega_0$ be a realization of $\eta_0^{ \sigma_0} \cup  \ol{ \eta}^{\ol{ \sigma}_0}$ for which $ S_0$ occurs. If $\wt\Delta_0$ (and hence also $\Delta_0$) is chosen sufficiently large and $\alpha > 0$ is chosen sufficiently small, in a manner which is uniform over values of the endpoints $x$ and $y$ such that $|x-y|$ is bounded below, then the following is true for a.e.\ such $ \omega_0$. Almost surely, the conditional law of $  \eta^*_0$ given  $\{\eta_0^{ \sigma_0} \cup  \ol{ \eta}^{\ol{ \sigma}_0} =  \omega_0\}$, $H_0^*$, and $( \eta_-^{t_-} , \eta_+^{t_+})$ is \hyperref[smac]{s.m.a.c.}\ with respect to the conditional law of $ \eta^*_0$ given only $\{\eta_0^{ \sigma_0} \cup  \ol{ \eta}^{\ol{ \sigma}_0} =  \omega_0\}$ and $H_0^*$.
\end{lem}
\begin{proof}
Let $\BB P_{ \omega_0} $ denote the regular conditional probability given $\{\eta_0^{ \sigma_0} \cup  \ol{ \eta}^{\ol{ \sigma}_0} =  \omega_0\}$ and the event $H_0^*$. Let $A^*_0$ be an event with positive $\BB P_{\omega_0}$-probability which is determined by $\eta_0^*$ and $\eta_0^{ \sigma_0} \cup  \ol{ \eta}^{\ol{ \sigma}_0}$ and is contained in $  H^*_0$. Let $A^F_0$ be the intersection of $H_0^*$ with an event which is determined by $ \eta_0^{ \sigma_0} \cup  \ol{ \eta}_0^{\ol{ \sigma}_0}$ and $( \eta_-^{t_-} , \eta_+^{t_+})$ and contained in $ S_0 $ which also satisfies $\BB P_{\omega_0}(A_0^F)>0$. By Bayes' rule, 
\eqb \label{sle smac bayes}
\BB P_{ \omega_0}(A^*_0 \,|\, A^F_0) = \frac{\BB P_{ \omega_0}(A^F_0 \,|\, A^*_0) \BB P_{ \omega_0}(A^*_0)}{\BB P_{ \omega_0}(A^F_0)} . 
\eqe 
Hence we are lead to study the conditional law of $( \eta_-^{t_-} , \eta_+^{t_+})$ given  $\{\eta_0^{ \sigma_0} \cup  \ol{ \eta}^{\ol{ \sigma}_0} =  \omega_0\}$ and $\eta_0^*$, for varying realizations of $\eta^*_0$ for which $H^*_0$ occurs. 

By the results of \cite[Section~7.1]{ig1}, the conditional law of $ \eta_+ $ given $\eta_0$ is that of a chordal $\op{SLE}_\kappa(\rho^L_F ; \rho^R_F)$ process from $x$ to $y$ in the right connected component of $\BB D \setminus \eta_0$ for certain $\rho^L_F , \rho^R_F > -2$ depending on $\rho^L$ and $\rho^R$. A similar statement holds for $\eta_-$. Furthermore, $\eta_+$ and $\eta_-$ are conditionally independent given~$\eta_0$. By \cite[Lemma~2.8]{miller-wu-dim} and the analog of condition~\ref{S hit} in the definition of $S_0$, if $\wt\Delta_0$ is chosen sufficiently large and $\alpha > 0$ is chosen sufficiently small then the conditional laws of the pair $( \eta_-^{t_-} , \eta_+^{t_+})$ given  $\{\eta_0^{ \sigma_0} \cup  \ol{ \eta}^{\ol{ \sigma}_0} =  \omega_0\}$ and $\eta^*_0$ for varying realizations of $\eta^*_0$ for which $H_0^*$ occurs are all \hyperref[smac]{s.m.a.c.}.  
By averaging over all such realizations, we get $\BB P_{ \omega_0}(A^F_0 \,|\, A^*_0) \asymp \BB P_{ \omega_0}(A^F_0 ) $. 
 By~\eqref{sle smac bayes} we therefore have $\BB P_{ \omega_0}(A^*_0 \,|\, A^F_0) \asymp \BB P_{ \omega_0}(A^*_0)$. 
\end{proof}

\begin{figure}
\begin{center}
\includegraphics[scale=.7]{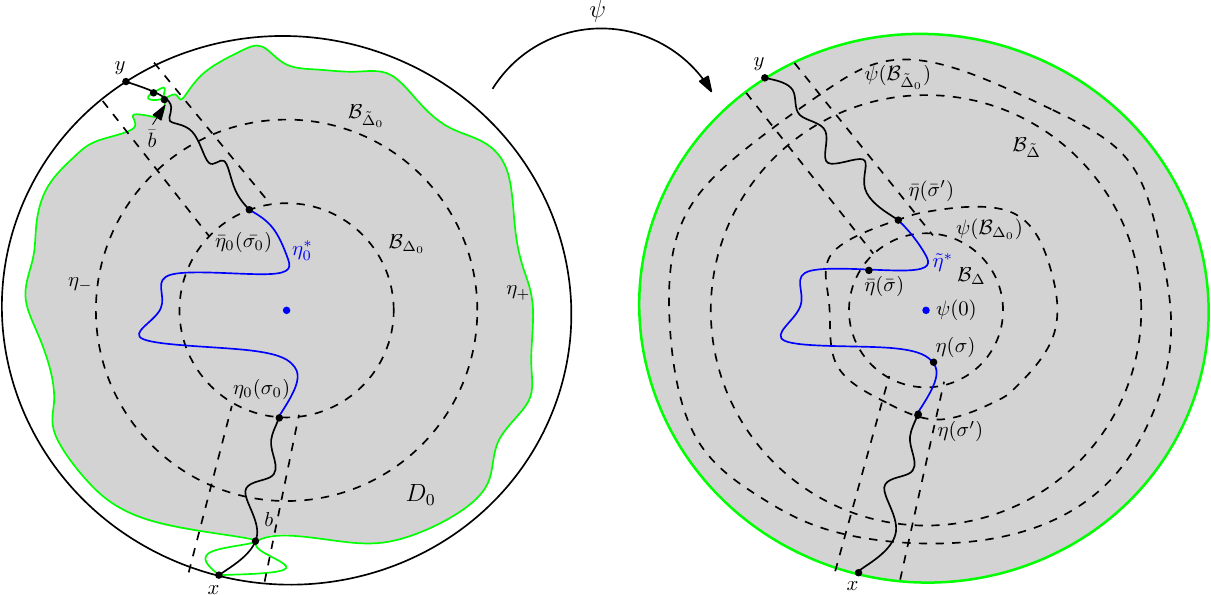}
\end{center}
\caption{An illustration of the setup for the proof of Lemma~\ref{rho abs cont}. The curve $ \eta_0$ in the left picture has the law of an ordinary chordal $\op{SLE}_\kappa $ from~$x$ to~$y$ in $\BB D$. The curve $\eta$ in the right picture (obtained by mapping the ``pocket" $ D_0$ formed green auxiliary flow lines to $\BB D$) has the law of a chordal $\op{SLE}_\kappa(\rho^L ; \rho^R)$ from $ x$ to $y$.  The amount by which $\psi$ distorts distances is exaggerated for clarity--- typically, $\psi$ is close to the identity on the event $F$.} \label{rho-abs-cont-fig}
\end{figure}

\begin{proof}[Proof of Lemma~\ref{rho abs cont}] 
Let $D_0$, $b$, and $\ol b$ be defined as in the discussion just above Lemma~\ref{eta^* abs cont}. Let $\psi :   D_0 \rta \BB D$ be the conformal map which takes $b$ to $x$ and $\ol b$ to $y$, chosen so that $|\psi (0)|$ is minimal amongst all such maps, and let
\eqbn
\eta := \psi(\eta_0 \cap D_0) ,\qquad \wt\eta^* := \psi(\eta_0^*) .
\eqen
Also let $\ol\eta$ be the time reversal of $\eta$. We define the objects in the statement of the lemma with this choice of $\eta$. 
 By the discussion just above Lemma~\ref{eta^* abs cont}, the conditional law of $\eta$ given $\eta_-$ and $\eta_+$ is that of a chordal $\op{SLE}_\kappa(\rho^L ; \rho^R)$ process from $  x$ to $y$ in $\BB D$. 
 
Fix $\ep > 0$, to be chosen later, and let $F = F(\ep)$ be the event that the following occur.
\begin{enumerate}
\item $\eta_-$ and $\eta_+$ trace all of $\partial D_0$ before times $t_-$ and $t_+$ (equivalently, since $\eta_\pm$ cannot cross themselves or each other, $t_- = t_+ = \infty$).\label{F trace}
\item $|\psi(z) - z| \leq \ep$ for each $z\in D_0$. \label{F distortion}
\end{enumerate}
By Lemma~\ref{miller-wu-dim-2.3}, for each $\ep > 0$ a.s.\ $\BB P (F \,|\, \eta_0) > 0 $.

By choosing $\ep > 0$ sufficiently small (depending only on $\Delta , \Delta' , \wt\Delta , \Delta_0$, and $\wt\Delta_0$), we can arrange that the following are true on $F$.
\begin{enumerate}
\item $\mcl B_\Delta \subset \psi(\mcl B_{\Delta_0}) \subset \psi(\mcl B_{\Delta'}) \subset  \psi(\mcl B_{\wt\Delta_0}) \subset \mcl B_{\wt\Delta}$. \label{balls contained}
\item The image under $\psi$ of the $e^{-2\Delta_0}$-neighborhood of the segment $[x ,0]$ (resp.\ $[y , 0]$) contains the $e^{-2\Delta}$-neighborhood of the segment $[x,0]$ (resp.\ $[y,0]$). \label{segments contained}
\end{enumerate}

 On the event $F$, let $\sigma' $ and $\ol\sigma'$ be the stopping times for $\eta$ and $\ol\eta$ corresponding to $ \sigma_0 $ and $\ol{\sigma}_0$, so $\psi( \eta_0(\sigma_0)) = \eta (\sigma')$, $\psi(\ol{ \eta}_0(\ol{\sigma}_0)) = \ol\eta(\ol\sigma')$, and $\wt\eta^*$ is the part of $\eta$ between $\eta(\sigma')$ and $\ol\eta(\ol\sigma')$.  
 Also let $\eta^*$ be the part of $\eta$ between $\sigma$ and $\ol\sigma$, as in the statement of the lemma.
 
 By conditions~\ref{balls contained} and~\ref{segments contained} above together with condition~\ref{S* contain} in the definition of $S^*$,
\eqb \label{F event contained}
  F\cap S^* \cap H^*  \subset F\cap S_0 \cap  H^*_0.
\eqe
(Note that the first inclusion is the only place where we use condition~\ref{S* contain} in the definition of $S^*$.) 
Furthermore, by the first inclusion in condition~\ref{balls contained} and condition~\ref{S finite} in the definition of $S$, on $F\cap S$ a.s.\
\eqb\label{sigma'<tau}
\sigma' \leq \tau_\Delta \leq \sigma  \quad \op{and} \quad \ol\sigma'\leq \ol \tau_\Delta  \leq \ol\sigma.
\eqe 

Now let $(\omega_0 , \omega_F)$ be a realization of $\left(\eta_0^{\sigma_0} \cup \ol\eta_0^{\ol\sigma_0} , \eta_+^{t_+} \cup \eta_-^{t_-}\right)$ for which $F\cap S_0$ occurs. We observe the following.
\begin{enumerate}
\item By the strong Markov property and reversibility of ordinary SLE$_\kappa$, the conditional law of $\eta_0^*$ given $\{\eta_0^{\sigma_0} \cup \ol\eta_0^{\ol\sigma_0} = \omega_0\}$ and $H_0^*$ is that of a chordal $\op{SLE}_\kappa$ from $ \eta_0( \sigma_0)$ to $\ol{ \eta}_0(\ol{ \sigma}_0)$ in $\BB D\setminus ( \eta_0^{ \sigma_0} \cup  \ol{ \eta}_0^{\ol{ \sigma}_0})$, conditioned on $H_0^*$.
\item It therefore follows from Lemma~\ref{eta^* abs cont} that the conditional law of $ \eta^*_0$ given $\{ \left(\eta_0^{\sigma_0} \cup \ol\eta_0^{\ol\sigma_0} , \eta_+^{t_+} \cup \eta_-^{t_-}\right) = ( \omega_0 , \omega_F)\}$ and $H_0^*$ is a.s.\ \hyperref[smac]{s.m.a.c.}\ with respect to the law of a chordal $\op{SLE}_\kappa$ from $ \eta_0(\sigma_0)$ to $\ol{\eta}_0(\ol{ \sigma}_0)$ in $\BB D \setminus ( \eta_0^{ \sigma_0} \cup  \ol{ \eta}_0^{\ol{ \sigma}_0})$, conditioned on $H_0^*$. 
\item By \cite[Lemma~2.8]{miller-wu-dim}, this latter law is \hyperref[smac]{s.m.a.c.}\ with respect to the law of a chordal $\op{SLE}_\kappa$ from $ \eta_0( \sigma_0)$ to $\ol{\eta}_0(\ol{ \sigma}_0)$ in the connected component of $D_0 \setminus (  \eta_0^{ \sigma_0} \cup  \ol{ \eta}_0^{\ol{ \sigma}_0})$ containing 0, conditioned on $ H^*_0$.
\item Therefore, the conditional law of $\wt\eta^* $ given $\{ \left(\eta_0^{\sigma_0} \cup \ol\eta_0^{\ol\sigma_0} , \eta_+^{t_+} \cup \eta_-^{t_-}\right) = ( \omega_0 , \omega_F)\}$ and $ H^*_0$ is \hyperref[smac]{s.m.a.c.}\ with respect to the law of a chordal $\op{SLE}_\kappa$ from $\eta(\sigma')$ to $\ol\eta(\ol\sigma')$ in the component of $\BB D\setminus (\eta^{\sigma'} \cup \ol\eta^{\ol\sigma'})$ containing 0, conditioned on $H_0^*$. \label{last law}   
\item By~\eqref{F event contained},~\eqref{sigma'<tau}, and the Markov property and reversibility of ordinary $\op{SLE}_\kappa$, assertion~\ref{last law} implies that the conditional law of $ \eta^* $ given $\{   \eta_+^{t_+} \cup \eta_-^{t_-}  =   \omega_F \}$; a realization of $\eta^\sigma \cup \ol\eta^{\ol\sigma}$ for which $S^*$ occurs; and $ H^* $ is a.s.\ \hyperref[smac]{s.m.a.c.}\ with respect to the law of a chordal $\op{SLE}_\kappa$ from $\eta(\sigma')$ to $\ol\eta(\ol\sigma')$ in the component of $\BB D\setminus (\eta^{\sigma'} \cup \ol\eta^{\ol\sigma'})$ containing 0, conditioned on $H^*$.\label{replace law}
\end{enumerate}
Since the law of $\eta$ given a.e.\ $\omega_F$ is that of a chordal $\op{SLE}_\kappa(\rho^L ; \rho^R)$ from $x$ to $y$ in $\BB D$ and there is a positive probability event of choices for $\omega_F$, assertion~\ref{replace law} implies the statement of the lemma.
\end{proof}

\bibliography{cibiblong,cibib,mf}
\bibliographystyle{hmralphaabbrv}

\end{document}